\documentclass[final,leqno,twoside]{article}
\usepackage{amsmath,amssymb,amsthm,epsfig,color,psfrag}
\usepackage{subfigure}
\usepackage{hyperref}
\hypersetup{linkbordercolor= 1 1 1}

\usepackage{fullpage}
\usepackage{tikz}
\newtheorem{proposition}{Proposition}[section]

\renewcommand{\ldots}{\dotsc}

\def\ds{\displaystyle}

\newcommand{\R}{\mathbb{R}}
\newcommand{\C}{\mathbb{C}}

\newcommand{\dd}[2]{{\displaystyle \frac{d #1}{d #2}}}

\newcommand{\circledcoloreditem}[1]{\tikz[baseline=(current bounding box.south east)] \draw node[draw,fill=white,text=black,circle,inner sep=0pt] {\scriptsize \makebox[3mm]{#1}};} 
\newcommand{\rectitem}[1]{\tikz[baseline=(current bounding box.south east)] \draw node[fill=white,text=black,rectangle,inner sep=0pt] {\scriptsize{#1}};}

\usepackage{fancyhdr}
\title{Numerical Periodic Normalization for Codim 2 Bifurcations of
Limit Cycles}

\author{F. Della Rossa$^\ddag$\thanks{DEI, Politecnico di Milano, V. Ponzio 35/5, 20133 Milano,
Italy (dellarossa@elet.polimi.it).}\and V. De Witte
\thanks{Department of Applied Mathematics and Computer Science,
Ghent University, Krijgslaan 281-S9, B-9000, Gent, Belgium
(Virginie.DeWitte@UGent.be, Willy.Govaerts@UGent.be).} \and W.
Govaerts $^\dag$ \and Yu.~A. Kuznetsov\thanks{Department of
Mathematics, Utrecht University, Budapestlaan 6, P.O. Box 80010,
3508 TA Utrecht, The Netherlands (I.A.Kouznetsov@uu.nl).} }

\begin{document}

\maketitle

\begin{abstract}
Periodic normal forms for the codim 2 bifurcations of limit cycles up to a 3-dimensional center manifold in generic autonomous ODEs and computational formulas for their
coefficients are derived. The formulas are independent of the
dimension of the phase space and involve solutions of certain
boundary-value problems on the interval $[0,T]$, where $T$ is the
period of the critical cycle, as well as multilinear functions
from the Taylor expansion of the right-hand sides near the cycle.
The formulas allow us to distinguish between various bifurcation
scenarios near codim 2 bifurcations. Our formulation makes it
possible to use robust numerical boundary-value algorithms based
on orthogonal collocation, rather than shooting techniques, which
greatly expands its applicability. The actual implementation is
described in detail with numerical examples.
\end{abstract}


\section{Introduction}
\label{sec1}
Isolated periodic orbits (limit cycles) of smooth differential
equations
\begin{equation} \label{ODE}
\dot{u}=f(u,p),\ \ u \in {\mathbb{R}}^n,\ p \in {\mathbb{R}}^m,
\end{equation}
play an important role in applications. In generic systems of the
form (\ref{ODE}) depending on one control parameter (i.e. with
$m=1$) a hyperbolic limit cycle exists for an open interval of
parameter values $p$. At a boundary of such an interval, the limit
cycle may become non-hyperbolic, so that either a cycle {\em limit
point} ({\em saddle-node}), or a {\em period-doubling} ({\em
flip}), or a {\em torus} ({\em Neimark-Sacker}) bifurcation occurs.
In two-parameter generic systems (\ref{ODE}) (i.e. with $m=2$)
these bifurcations happen at certain curves in the parameter
plane. These curves of codim 1 bifurcations can meet tangentially
or intersect transversally at some codim 2 points characterized by
a double degeneracy of the limit cycle, which play the role of
organizing centers for local dynamics, i.e. near the critical
cycle and for nearby parameter values. In some cases, such codim 2
bifurcations imply the appearance of ``chaotic motions".

The codim 2 bifurcations of limit cycles in generic systems
(\ref{ODE}) are well understood with the help of the corresponding
Poincar\'{e} maps and their normal forms (see for example
\cite{Io:79,GuHo:83,Ar:83,Ku:2004,GoGhKuMe:07}). However,
applications of these results to the analysis of concrete systems
(\ref{ODE}) are exceptional, since they require accurate
higher-order derivatives of the Poincar\'{e} map which are hardly
available numerically \cite{ToGeTe:98,ADOL-C,GuMe:2000,KuMe:2005}.

We note that there exists software, e.g. CAPD \cite{capd}, TIDES \cite{tides,barrio:06} that allows to compute up to any precision level the solution of an ODE using a Taylor series method in a variable stepsize - variable order formulation. It can also compute, up to any order, the partial derivatives of the solution with respect to the initial conditions. When applied to compute a periodic orbit by a shooting method, this will also provide the derivatives of the Poincar\'{e} map.

Though this is a valuable approach in the case of a single periodic orbit, it is neither practical nor efficient in a continuation context since the periodic orbits are not obtained as the solution to a set of equations. Also, the shooting method does not have the high order convergence properties of the method of approximation by piecewise polynomials with collocation in the Gauss points that is routinely used in standard software such as {\sc auto} \cite{AUTO97}, {\sc
content} \cite{CONTENT}, and {\sc matcont} \cite{sac2003,MATCONT}. Moreover, the number of derivatives of the Poincar\'{e} map to be computed is $O(n^k)$ if derivatives up to order $k$ are
needed (in several cases $k=5$). Even for moderate values of $n$ this involves a great deal of unnecessary work since in
our situation the normal form itself is known in advance and we need only compute its coefficients. We will show that this can be done without computing the derivatives of the Poincar\'{e} map.

Indeed, recently an alternative numerical method to analyse codim 1 limit
cycle bifurcations has been developed and implemented in
\cite{KuDoGoDh:05}. It is based on the periodic normalization
proposed in \cite{IoJo:80,Io:88,IoAd:92} and completely avoids the
numerical computation of Poincar\'{e} maps and their derivatives.
Instead, the computation of the normal form coeffcients is reduced
to solving certain linear boundary value problems (BVP), where
only the partial derivatives of the RHS of (\ref{ODE}) are used
\cite{Elf:87,ElIoTi:87,Ku:99}. In our implementation in MATCONT,
we discretize these BVPs by orthogonal collocation with
piecewise-polynomial functions.

In the present paper, we apply the approach developed in
\cite{KuDoGoDh:05} to codim 2 bifurcations of limit cycles. It
should be noted that already in \cite{ChWa:86} normal forms for
some codim 2 bifurcations of cycles in (\ref{ODE}) were derived,
while \cite{Io:88} contains the periodic normal forms for many
codim 2 bifurcations of cycles, as well as a general normalization
technique applicable at any codimension. However, in neither of
these publications explicit formulas for the normal form
coefficients were given. The derivation of such formulas is the
primary contribution of this paper.

The paper is organized as follows. In Section~\ref{Section:Normalforms}
we fix notation and list the periodic normal forms for codim 2
bifurcations of limit cycles. Then we derive explicit formulas to
compute the critical normal form coefficients for these
bifurcations, which we order by the dimension $n_c$ of the cycle
center manifold (i.e. the total number of {\em critical
multipliers} with $|\mu|=1$). We restrict to the cases $n_c=2$ and $3$. The
formulas are independent of the dimension of the phase space and
involve solutions of certain BVPs on the interval $[0,T]$, where
$T$ is the period of the critical cycle, as well as multilinear
functions from the Taylor expansion of the right-hand sides of
(\ref{ODE}) near the cycle.
A derivation of the critical periodic normal forms based on
\cite{Io:88} is given in Appendix \ref{Appendix:1}, while their
relationships with the normal forms of the Poincar\'{e} maps are
discussed in Appendix \ref{Appendix:2}.

\section{Periodic normal forms on the center manifold}
\label{Section:Normalforms}
Write (\ref{ODE}) at the critical parameter values as
\begin{equation}\label{eq:P.1}
\dot{u}=F(u),
\end{equation}
and suppose that there is a limit cycle $\Gamma$ corresponding to a
periodic solution $u_0(t)=u_0(t+T),$ where $T>0$ is its (minimal)
period. Develop $F(u_0(t)+v)$ into the Taylor series
\begin{equation} \label{eq:MULT}
\begin{array}{rcl}
F(u_0(t)+v) &=& F(u_0(t)) + \\
&&A(t)v + {\displaystyle \frac{1}{2}B(t;v,v) + \frac{1}{3!}C(t;v,v,v)} + \\
&&{\displaystyle\frac{1}{4!}D(t;v,v,v,v) +
\frac{1}{5!}E(t;v,v,v,v,v) + O(\|v\|^6)},
\end{array}
\end{equation}
where
$$
A(t)v=F_u(u_0(t))v,~~B(t;v_1,v_2)=F_{uu}(u_0(t))[v_1,v_2],~~
C(t;v_1,v_2,v_3)=F_{uuu}(u_0(t))[v_1,v_2,v_3],
$$
etc. The multilinear forms $A,B,C,D,$ and $E$ are periodic in $t$
with period $T$ but this dependence will often not be indicated
explicitly.
\par
Consider the initial-value problem for the fundamental matrix
solution $Y(t)$, namely,
\begin{equation} \label{P.2}
\frac{dY}{dt}=A(t)Y,\ \ Y(0)=I_n,
\end{equation}
where $I_n$ is the $n\times n$ identity matrix. The eigenvalues of
the monodromy matrix $M=Y(T)$ are called ({\em Floquet}) {\em
multipliers} of the limit cycle. The multipliers with $|\mu|=1$ are
called {\em critical}. There is always a ``trivial" critical
multiplier $\mu_n=1$. We denote the total number of critical
multipliers by $n_c$ and assume that the limit cycle is
non-hyperbolic, i.e. $n_c > 1$. In this case, there exists an
invariant $n_c$-dimensional {\em critical center manifold}
$W^c(\Gamma) \subset {\mathbb R}^n$ near $\Gamma$\footnote{This
manifold should not be confused with the $(n_c-1)$-dimensional center
manifold of the corresponding Poincar\'{e} map.}.

It is well known \cite{Ar:83,Ku:2004}, that in generic
two-parameter systems (\ref{ODE}) only eleven codim 2 bifurcations
occur. To describe the normal forms of (\ref{eq:P.1}) on the
critical center manifold $W^c(\Gamma)$ for these codim 2 cases, we
parameterize  $W^c(\Gamma)$ near $\Gamma$ by $(\tau,\xi)$, where
$\tau \in [0,kT]$ for $k \in \{1,2,3,4\}$ and $\xi$ is a real or
complex vector, depending on the bifurcation. It follows from
\cite{Io:88} that it is possible to select the $\tau$- and
$\xi$-coordinates so that the restriction of (\ref{eq:P.1}) to the
corresponding critical center manifold $W^c(\Gamma)$ with $n_c=2$ or $n_c=3$ will take one
of the following {\it periodic normal forms} (for derivation see Appendix
\ref{Appendix:1}). In Section \ref{Section:Theory} we will see that in two cases, namely the cusp of cycles bifurcation and the fold-flip bifurcation, a further simplification is possible in the normal form, and more specifically in the transformation of time.

\subsection{Bifurcations with a 2D center manifold}
We list here the critical periodic normal forms with $n_c=2$ and
briefly describe bifurcations in their generic unfoldings (see
\cite{Ar:83,Ku:2004} for more details).

\subsubsection{Cusp of cycles bifurcation}
The cycle has a {\em cusp of cycles} bifurcation ({\tt CPC}) if the
eigenvalue $\mu_1=\mu_n=1$ of $Y(T)$ corresponds to a
two-dimensional Jordan block and the monodromy matrix has no other critical multipliers. The two-dimensional periodic
normal form at the {\tt CPC} bifurcation is
\begin{equation} \label{eq:NF-CPC}
\begin{cases}
 \dd{\tau}{t}=1-\xi+\alpha_1 \xi^2 + \alpha_2 \xi^3 + \ldots, \smallskip \\
 \dd{\xi}{t}=c \xi^3 + \ldots,
\end{cases}
\end{equation}
where $\tau \in [0,T]$, $\xi$ is a real coordinate on
$W^c(\Gamma)$ that is transverse to $\Gamma$, $\alpha_1,\alpha_2,c
\in \R$ and the dots denote the $O(\xi^4)$-terms, which are
$T$-periodic in $\tau$. If $c \neq 0$, the limit cycle $\Gamma$ is
triple. In generic two-parameter systems (\ref{ODE}), three
hyperbolic limit cycles exist in a cuspidal wedge approaching the
codim 2 point and delimited by two bifurcations curves, where two
cycles collide and disappear via the saddle-node bifurcation.

\subsubsection{Generalized period-doubling bifurcation}
The cycle has a {\em generalized period-doub\-ling bifurcation}
({\tt GPD}) if the trivial eigenvalue $\mu_n=1$ of the monodromy
matrix $Y(T)$ is simple and there is only another critical simple
eigenvalue $\mu_1=-1$. The two-dimensional periodic normal form at
the {\tt GPD} bifurcation is
\begin{equation} \label{eq:NF-GPD}
\begin{cases}
\dd{\tau}{t}=\ds1+\alpha_1 \xi^2+\alpha_2 \xi^4 + \ldots,\smallskip \\
\dd{\xi}{t}=\ds e \xi^5 + \ldots,
\end{cases}
\end{equation}
where $\tau \in [0,2 T]$, $\xi$ is a real coordinate on
$W^c(\Gamma)$ that is transverse to $\Gamma$, $\alpha_1,\alpha_2,e
\in \R$ and the dots denote the $O(\xi^6)$-terms, which are
$2T$-periodic in $\tau$. If $e \neq 0$, at most  two period doubled
limit cycles can bifurcate from the critical limit cycle $\Gamma$.
In generic two-parameter systems (\ref{ODE}), the {\tt GPD}-point
in the period-doubling bifurcation curve separates its sub- and
super-critical branches and is the origin of a unique saddle-node
bifurcation curve, where two period doubled cycles collide and disappear.

\subsection{Bifurcations with a 3D center manifold}
We now list the critical periodic normal forms with $n_c=3$ and
briefly describe bifurcations in their generic unfoldings (see
\cite{Ar:83,Ku:2004}). In all cases, ``chaotic motions" are
possible.

\subsubsection{Chenciner bifurcation}
The cycle has a {\em Chenciner} bifurcation ({\tt CH}) if the
trivial critical eigenvalue $\mu_n=1$ is simple and there are only
two more critical simple multipliers $\mu_{1,2}=e^{\pm i\theta}$
with $\theta\neq \frac{2 \pi}{j}$, for $j=1,2,3,4,5,6$. The three-dimensional periodic normal form at the {\tt CH} bifurcation can
be written as
\begin{equation} \label{eq:NF-GNS}
\begin{cases}
 \dd{\tau}{t}=1+\alpha_1 |\xi|^2+\alpha_2 |\xi|^4 + \ldots, \smallskip \\
 \dd{\xi}{t}= i \omega \xi + i c \xi |\xi|^2+ e \xi |\xi|^4 + \ldots, \smallskip \\
\end{cases}
\end{equation}
where $\tau \in [0,T]$, $\omega = \theta/T$, $\xi$ is a complex
coordinate on $W^c(\Gamma)$ transverse to $\Gamma$,
$\alpha_1,\alpha_2, c \in \R,\, e \in \C$ and the dots denote the
$O(\|\xi^6\|)$-terms, which are $T$-periodic in $\tau$. In generic
two-parameter systems (\ref{ODE}), at the {\tt CH}-point the
Neimark-Sacker bifurcation changes its criticality (i.e. the
bifurcating invariant torus changes its stability). A complicated
bifurcation set responsible for ``collision" and destruction of
two tori of opposite stability is rooted at this codim 2 point.

\subsubsection{Strong resonance 1:1 bifurcation}
The cycle has a {\em strong resonance $1:1$} bifurcation ({\tt R1}) if the
trivial critical eigenvalue $\mu_n=1$ corresponds to a
three-dimensional Jordan block. The three-dimensional periodic
normal form at the {\tt R1} bifurcation is
\begin{equation} \label{eq:NF-11C}
\begin{cases}
 \ds\dd{\tau}{t}=1-\xi_1+\alpha \xi_1^2+\ldots, \smallskip \\
 \ds\dd{\xi_1}{t}= \xi_2+\xi_1\xi_2+\ldots, \smallskip \\
 \ds\dd{\xi_2}{t}= a \xi_1^2+ b \xi_1 \xi_2 +\ldots,
\end{cases}
\end{equation}
where $\tau \in [0,T]$, $(\xi_1,\xi_2)$ are real coordinates on
$W^c(\Gamma)$ transverse to $\Gamma$, $\alpha, a, b \in \R$ and
the dots denote the $O(\|\xi^3\|)$-terms, which are $T$-periodic
in $\tau$. In generic two-parameter systems (\ref{ODE}), in the
{\tt R1}-point is located on the saddle-node of cycles curve. At
this point, a torus bifurcation curve is rooted together with
global homoclinic bifurcation curves, along which the stable and
the unstable invariant manifolds of a saddle cycle are tangent.
The intersection of the invariant manifolds generates a
Poincar\'{e} homoclinic structure with the associated periodic and
``chaotic motions".

\subsubsection{Strong resonance 1:2 bifurcation}
The cycle has a {\em strong resonance $1:2$} bifurcation ({\tt R2}) if the
trivial critical eigenvalue $\mu_n=1$ is simple and there are only
two more critical multipliers $\mu_1=\mu_2=-1$ corresponding to a
two-dimensional Jordan block. The three-dimensional periodic
normal form at the {\tt R2} bifurcation is
\begin{equation} \label{eq:NF-12C}
\begin{cases}
 \dd{\tau}{t}=1+\alpha \xi_1^2+ \ldots , \smallskip \\
 \dd{\xi_1}{t}=\xi_2+\alpha\xi_1^2\xi_2 + \ldots , \smallskip \\
 \dd{\xi_2}{t}=a \xi_1^3+b\xi_1^2\xi_2 + \ldots,
\end{cases}
\end{equation}
where $\tau \in [0,2T]$, $(\xi_1,\xi_2)$ are real coordinates on
$W^c(\Gamma)$ transverse to $\Gamma$, $\alpha, a, b \in \R$ and
the dots denote the $O(\|\xi^4\|)$-terms, which are $2T$-periodic
in $\tau$. In generic two-parameter systems (\ref{ODE}), the {\tt
R2}-point is the end-point of a torus bifurcation curve. The
period-doubling bifurcation curve passes through this point, and
(depending on the normal form coeffcients) a torus bifurcation
curve of the period doubled limit cycle can originate there. As in
the {\tt R1}-case, global bifurcation curves related to
homoclinic tangencies can be present.

\subsubsection{Strong resonance 1:3 bifurcation}
The cycle has a {\em strong resonance $1:3$} bifurcation ({\tt R3}) if the
trivial critical eigenvalue $\mu_n=1$ is simple and there are only
two more critical simple multipliers $\mu_{1,2}=e^{\pm i\frac{2
\pi}{3}}$. The three-dimensional periodic normal form at the {\tt
R3} bifurcation can be written as
\begin{equation} \label{eq:NF-13C}
\begin{cases}
 \dd{\tau}{t}=1+\alpha_1 |\xi|^2+\alpha_2 \xi^3+\alpha_3 \bar \xi^3+ \ldots, \smallskip \\
 \dd{\xi}{t}=b \bar \xi^2+  c \xi |\xi|^2 + \ldots, \smallskip \\
\end{cases}
\end{equation}
where $\tau \in [0,3T]$, $\xi$ is a complex coordinate on
$W^c(\Gamma)$ transverse to $\Gamma$, $\alpha_1 \in \R$, $\alpha_2, \alpha_3, b, c \in
\C$ with $\alpha_3 = \bar \alpha_2$ and the dots denote the $O(\|\xi^4\|)$-terms, which are
$3T$-periodic in $\tau$. In generic two-parameter systems
(\ref{ODE}), near the {\tt R3}-point a homoclinic Poincar\'{e}
structure of the $3T$-periodic limit cycle destroys the torus that
is born at the Neimark-Sacker bifurcation curve passing through
this point. Homoclinic tangencies are rooted there.

\subsubsection{Strong resonance 1:4 bifurcation}
The cycle has a {\em strong resonance $1:4$} bifurcation ({\tt R4}) if the
trivial critical eigenvalue $\mu_n=1$ is simple and there are only
two more critical simple multipliers $\mu_{1,2}=e^{\pm
i\frac{\pi}{2}}$. The three-dimensional periodic normal form at
the {\tt R4} bifurcation can be written as
\begin{equation} \label{eq:NF-14C}
\begin{cases}
 \dd{\tau}{t}=1+\alpha_1 |\xi|^2+\alpha_2 \xi^4+\alpha_3 \bar \xi^4+ \ldots, \smallskip \\
 \dd{\xi}{t}= c \xi |\xi|^2 + d \bar \xi^3 + \ldots, \smallskip \\
\end{cases}
\end{equation}
where $\tau \in [0,4T]$, $\xi$ is a complex coordinate on
$W^c(\Gamma)$ transverse to $\Gamma$, $\alpha_1 \in \R$, $\alpha_2, \alpha_3, c, d \in
\C$ with $\alpha_3 = \bar \alpha_2$ and the dots denote the $O(\|\xi^5\|)$-terms, which are
$4T$-periodic in $\tau$. In generic two-parameter systems
(\ref{ODE}), at the {\tt R4}-point there can be eight different
situations, depending upon the values taken by the parameter $c$
and $d$. In the simplest case a homoclinic structure associated to
a $4T$-periodic cycle destroys an invariant torus that is born at
the Neimark-Sacker bifurcation curve that passes trough this
point.

\subsubsection{Fold-flip bifurcation}
The cycle has a {\em fold flip} bifurcation ({\tt LPPD}) if the
trivial critical eigenvalue $\mu_1=\mu_n=1$ is double non
semi-simple and there is only one more critical multiplier
$\mu_2=-1$. The three-dimensional periodic normal form at the {\tt
LPPD} bifurcation is
\begin{equation} \label{eq:NF-FF}
\begin{cases}
 \ds\dd{\tau}{t}=1-\xi_1 +\alpha_{20} \xi_1^2 +\alpha_{02} \xi_2^2 + \alpha_{30}\xi_1^3+\alpha_{12}\xi_1\xi_2^2+\ldots, \smallskip \\
 \ds\dd{\xi_1}{t}= a_{20} \xi_1^2 +a_{02} \xi_2^2 + a_{30}\xi_1^3+a_{12}\xi_1\xi_2^2+\ldots, \smallskip \\
 \ds\dd{\xi_2}{t}= b_{11}\xi_1\xi_2+b_{21}\xi_1^2\xi_2+b_{03} \xi_2^3+\ldots,
\end{cases}
\end{equation}
where $\tau \in [0,2T]$, $(\xi_1,\xi_2)$ are real coordinates on
$W^c(\Gamma)$ transverse to $\Gamma$, all the coefficients are
real and the dots denote the $O(\|\xi^4\|)$-terms, which are
$2T$-periodic in $\tau$. In generic two-parameter systems
(\ref{ODE}), the period-doubling and saddle-node of cycles
bifurcation curves are tangent at the {\tt LPPD}-point, where
(depending on the normal form coefficients) a Neimark-Sacker
bifurcation curve of the $2T$-periodic cycle can be rooted. Global
bifurcations of heteroclinic structures and invariant tori are
also possible.

%

\section{Computation of critical coefficients}
\label{Section:Theory}
In view of the above, we can assume that a parametrization of the
center manifold $W^c(\Gamma)$ has been selected so that the
restriction of (\ref{eq:P.1}) to this manifold has one of the
normal forms given in Section \ref{Section:Normalforms}. We then apply the so-called {\it homological equation
approach} \cite{HB:2002}: the Taylor expansions of $T$-, $2T$-, $3T$- or
$4T$-periodic unknown functions involved in these parametrizations
can be found by solving appropriate BVPs on $[0,T]$ so that
(\ref{eq:P.1}) restricted to $W^c(\Gamma)$ has the
corresponding periodic normal form. The coefficients of the normal
forms arise from the solvability conditions for the BVPs as
integrals of scalar products over $[0,T]$, involving nonlinear
terms of (\ref{eq:P.1}) near the periodic solution $u_0$, as well
as the critical eigenfunctions and higher order expansion terms of the center manifold. The Taylor expansion coefficient functions are usually unique up to the addition of a multiple of a known eigenfunction. This has to be fixed by adding an integral condition. Among other things this leads to the fact that normal form coefficients are not unique but implications for the underlying dynamical systems are independent of this. We also remark that the solvability of all the equations up to the maximal order of the normal form has to be checked. Finally, we note that the coefficients related to the transformation of time will only be computed when needed in the computation of the critical coefficients in the normal form for the state variables.

\subsection{Cusp of cycles bifurcation}\label{Section:Cusp}
The two-dimensional critical center manifold $W^c(\Gamma)$ at the
{\tt CPC} bifurcation can be parametrized locally by $(\tau,\xi)$
as
\begin{equation} \label{eq:CM_CPC}
u=u_0 + \xi v + H(\tau,\xi),\ \ \tau \in [0,T],\ \xi \in {\mathbb
R},
\end{equation}
where $H$ satisfies $H(T,\xi)=H(0,\xi)$ and has the Taylor
expansion
\begin{equation}\label{eq:H_CPC}
 H(\tau,\xi)=\frac{1}{2}h_2 \xi^2 + \frac{1}{6}h_3 \xi^3 + O(\xi^4) ,
\end{equation}
where $u_0$ and all $h_j$ are functions of $\tau$, with $h_j(T)=h_j(0)$, for $j = 2, 3$, while the generalized eigenfunction $v$ is given by
\begin{equation} \label{eq:EigenFunc_CPC}
\left\{\begin{array}{rcl}
 \dot{v}-A(\tau) v - F(u_0) & = & 0,\ \tau \in [0,T], \\
 v(T)-v(0) & = & 0,\\
 \int_{0}^{T} {\langle v, F(u_0)\rangle d\tau} & = & 0.\\
\end{array}
\right.
\end{equation}
The function $v$ exists due to Lemma~2 of \cite{Io:88}. 
Let $\varphi^*$ be a nontrivial solution of the adjoint
eigenvalue problem
\begin{equation}\label{eq:AdjEigenFunc}
\left\{\begin{array}{rcl}
 \dot{\varphi}^*+A^{\rm T}(\tau)\varphi^* & = & 0,\ \tau \in [0,T], \\
 \varphi^*(T)-\varphi^*(0) & = & 0,
\end{array}
\right.
\end{equation}
and the generalized adjoint eigenfunction $v^*$ a solution of 
\begin{equation}\label{eq:AdjGenEigenFunc}
\left\{\begin{array}{rcl}
\dot{v}^*+A^{\rm T}(\tau) v^*+\varphi^* & = & 0,\ \tau \in [0,T], \\
v^*(T)-v^*(0) & = & 0,
\end{array}
\right.
\end{equation}
which is now defined up to the addition of a multiple of $\varphi^*$.
Note that the first equation of (\ref{eq:EigenFunc_CPC}) implies
\begin{equation} \label{eq:Ortho}
 \int_0^T \langle \varphi^*, F(u_0) \rangle \; d\tau = 0
\end{equation}
for $\varphi^*$ satisfying \eqref{eq:AdjEigenFunc}. Indeed, 
\[
 \int_0^{T} \langle \varphi^*,F(u_0) \rangle \; d\tau = \int_0^{T} \langle \varphi^*,\left(\frac{d}{d\tau}-A(\tau)\right) v \rangle \; d\tau =
  - \int_0^{T} \langle \left(\frac{d}{d\tau}+A^T(\tau)\right)\varphi^*, v \rangle \; d\tau =0
\]
due to (\ref{eq:AdjEigenFunc}). 

Moreover, due to spectral
assumptions at the CPC-point, we can also assume
\begin{equation} \label{eq:Normo}
\int_{0}^{T} {\langle \varphi^*,v \rangle d\tau} = 1.
\end{equation}
Notice that this assumption gives us another
normalization for free, since because of \eqref{eq:EigenFunc_CPC} and
\eqref{eq:AdjGenEigenFunc} we have
\begin{multline}\label{eq:NormoGen}
 \int_{0}^{T} {\langle v^*,F(u_0) \rangle d\tau} = \int_{0}^{T} {\langle v^*,\left(\dd{}{\tau}-A(\tau)\right) v \rangle
 d\tau}\\
 = -\int_{0}^{T} {\langle \left(\dd{}{\tau}+A^T(\tau)\right)v^*,v \rangle d\tau} =\int_{0}^{T} {\langle \varphi^*,v \rangle d\tau} = 1,
\end{multline}
i.e. we have normalized the eigenfunction of the
adjoint problem with the generalized one of the original
problem and the generalized eigenfunction of the adjoint
problem with the eigenfunction of the original problem. So $\varphi^*$ is the unique solution of the BVP 
\begin{equation} \label{eq:AdjEigenFunc_CPC}
\left\{\begin{array}{rcl}
 \dot{\varphi}^*+A^{\rm T}(\tau)\varphi^* & = & 0,\ \tau \in [0,T], \\
 \varphi^*(T)-\varphi^*(0) & = & 0,\\
 \int_{0}^{T} {\langle \varphi^*,v \rangle d\tau} - 1 & = & 0.
\end{array}
\right.
\end{equation}
Now, we still need an integral condition for the adjoint generalized eigenfunction $v^*$. In all cases, for the computation of an adjoint generalized eigenfunction we will require the inproduct with an original eigenfunction to be zero. Here, the inproduct with $v$ is appropriate. Therefore we obtain 
\begin{equation} \label{eq:AdjGenEigenFunc_CPC}
\left\{\begin{array}{rcl}
 \dot{v}^*+A^{\rm T}(\tau)v^* + \varphi^* & = & 0,\ \tau \in [0,T], \\
 v^*(T)-v^*(0) & = & 0,\\
 \int_{0}^{T} {\langle v^*,v \rangle d\tau} & = & 0.
\end{array}
\right.
\end{equation}

Now, we substitute (\ref{eq:CM_CPC}) into (\ref{eq:P.1}), using
(\ref{eq:MULT}), (\ref{eq:NF-CPC}) and (\ref{eq:H_CPC}), as well
as
$$
 \frac{du}{dt}=\frac{\partial u}{\partial \xi}\frac{d\xi}{dt} + \frac{\partial u}{\partial \tau}\frac{d\tau}{dt}.
$$
This gives
\begin{multline*}
 \dot u_0 + \xi \left(\dot v-\dot u_0\right) + \xi^2 \left(\alpha_1\dot u_0-\dot v +\frac{1}{2}
 h_2\right) + \xi^3 \left(\alpha_2\dot u_0 + \alpha_1 \dot v - \frac{1}{2} \dot h_2 + \frac{1}{6} \dot h_3 + c
 v\right) + O(\xi^4) \\
 = F(u_0)+ \xi A(\tau) v + \frac{1}{2}\xi^2\left(A(\tau) h_2 + B(\tau;v,v)\right) + \frac{1}{6}\xi^3\left(A(\tau) h_3 + 3 B(\tau;h_2,v) + C(\tau;v,v,v)\right) + O(\xi^4),
\end{multline*}
where dots denote the derivatives with respect to $\tau$.

Collecting the $\xi^0$-terms we get the identity
$$
\dot{u}_0=F(u_0),
$$
since $u_0$ is the periodic solution of \eqref{eq:P.1}.

The $\xi^1$-terms provide another identity, namely
$$
\dot{v}-\dot u_0=A(\tau) v,
$$
due to \eqref{eq:EigenFunc_CPC}.

From collecting the $\xi^2$-terms we obtain an equation for $h_2$
\begin{equation} \label{eq:xi^2}
 \dot h_2 - A(\tau) h_2 = B(\tau;v,v) + 2 \dot v - 2 \alpha_1 \dot u_0.
\end{equation}
Now, we project the left-hand side of this equation on the adjoint null-eigenfunction and obtain
\[
 \int_0^{T} \langle \varphi^*,\left(\frac{d}{d\tau}-A(\tau)\right) h_2 \rangle \; d\tau =
  - \int_0^{T} \langle \left(\frac{d}{d\tau}+A^T(\tau)\right)\varphi^*, h_2 \rangle \; d\tau =0
\]
due to (\ref{eq:AdjEigenFunc}). We can use this result to impose the so-called {\em
Fredholm solvability condition}, i.e. also project the right-hand side of \eqref{eq:xi^2} on $\varphi^*$ which then has to be equal to $0$
\[
  \int_0^{T} \langle \varphi^*, B(\tau;v,v) + 2 \dot v - 2 \alpha_1 \dot u_0 \rangle \; d\tau =
  \int_0^{T} \langle \varphi^*, B(\tau;v,v) + 2 A(\tau) v \rangle \; d\tau=0.
\]
Notice that this condition is actually trivially satisfied, due to the fact
that we are at a cusp of cycles point, so that the second order normal form
coefficient
\[
 b=\frac{1}{2} \int_0^{T} \langle \varphi^*, B(\tau;v,v) + 2 A(\tau) v \rangle \; d\tau
\]
(see \cite{KuDoGoDh:05}) vanishes. Hence equation
\eqref{eq:xi^2} is solvable, independent of the value of $\alpha_1$. 
So, for any value of $\alpha_1$ we get an equation for $h_2$ to be solved in the
space of vector-functions on $[0,T]$ satisfying $h_2(T)=h_2(0)$.
Notice that if $h_2$ satisfies \eqref{eq:xi^2}, $h_2+\varepsilon
F(u_0)$ also satisfies \eqref{eq:xi^2}, due to the fact that
$F(u_0)$ is in the kernel of the operator
$\frac{d}{d\tau}-A(\tau)$ and to the linearity of this operator. Now, the orthogonality condition with $v^*$ determines the value of $\varepsilon$ such that we can define $h_2$ as the unique solution of 
\begin{equation}\label{eq:h2_CPC}
\left\{\begin{array}{rcl}
 \dot h_2 - A(\tau) h_2 - B(\tau;v,v) - 2 A v-2F(u_0) + 2 \alpha_1 F(u_0) &=& 0,\ \tau \in [0,T],\\
 h_2(T)-h_2(0)&=&0, \\
 \int_0^T \langle v^*, h_2 \rangle\; d\tau&=&0.
\end{array}\right.
\end{equation}

Collecting the $\xi^3$-terms we finally obtain an equation in
$h_3$ which allows us to determine the normal form coefficient $c$ of \eqref{eq:NF-CPC}
\begin{equation*} 
 \dot h_3 - A(\tau) h_3 = -6 \alpha_2 \dot u_0 - 6 \alpha_1 \dot v + 3 \dot h_2 - 6 c v + 3 B(\tau;h_2,v)+C(\tau;v,v,v).
\end{equation*}
As before, the null-eigenfunction of the adjoint operator
$-\frac{d}{d\tau} - A^{\rm T}(\tau)$ is $\varphi^*$. Thus, the Fredholm
solvability condition implies that
$$
 \int_0^{T} \langle \varphi^*, -6 \alpha_2 \dot u_0 - 6 \alpha_1 \dot v + 3 \dot h_2 - 6 c v + 3 B(\tau;h_2,v)+C(\tau;v,v,v)\rangle\; d\tau=0.
$$
Using \eqref{eq:EigenFunc_CPC}, \eqref{eq:Normo} and
\eqref{eq:Ortho}, we get the expression
\begin{equation} \label{eq:c-CPC}
 c=\frac{1}{6}\int_0^T \langle \varphi^*,- 6 \alpha_1 A(\tau) v + 3 A(\tau) h_2 + 3 B(\tau;v,v) + 6 A(\tau) v + 3 B(\tau;h_2,v) + C(\tau;v,v,v) \rangle\; d\tau,
\end{equation}
where $v$ and $ \varphi^*$ are defined by \eqref{eq:EigenFunc_CPC}
and \eqref{eq:AdjEigenFunc_CPC}, while $h_2$ satisfies \eqref{eq:h2_CPC}.

In what follows we will prove that the choice of $\alpha_1$ does not influence the value of the critical normal form coefficient $c$. Two solutions $h_2$ differ by $h_2^{(2)}-h_2^{(1)}=-2(\alpha_1^{(2)}-\alpha_1^{(1)})v$, from which
\begin{eqnarray*}
	c^{(2)}-c^{(1)} & = & \frac{1}{6} \int_0^T \langle \varphi^*,-6(\alpha_1^{(2)}-\alpha_1^{(1)})A(\tau)v-6A(\tau)(\alpha_1^{(2)}-\alpha_1^{(1)})v-6(\alpha_1^{(2)}-\alpha_1^{(1)}) B(\tau;v,v)\rangle\; d\tau\\
	& = & (\alpha_1^{(2)}-\alpha_1^{(1)}) \int_0^T \langle \varphi^*,-2A(\tau)v-B(\tau;v,v)\rangle\; d\tau\\
	& = & (\alpha_1^{(2)}-\alpha_1^{(1)}) \int_0^T \langle \varphi^*,-2A(\tau)v+2A(\tau)v\rangle\; d\tau\\
	& = & 0,
\end{eqnarray*}
since $b$ vanishes. So, for simplicity we take $\alpha_1=0$ which further simplifies normal form (\ref{eq:NF-CPC}).

Therefore, the critical coefficient $c$ in the periodic normal
form \eqref{eq:NF-CPC} has been computed. The bifurcation is
nondegenerate if $c\neq 0$. 

\subsection{Generalized period-doubling bifurcation} \label{Section:DFlip}
The two-dimensional critical center manifold $W^c(\Gamma)$ at the
{\tt GPD} bifurcation can be parametrized locally by $(\tau,\xi)$
as
\begin{equation} \label{eq:CM_GPD}
 u=u_0 +\xi v + H(\tau,\xi),\ \ \tau \in [0,2T],\ \xi \in {\mathbb R},
\end{equation}
where the function $H$ satisfies $H(2T,\xi)=H(0,\xi)$. It has the
Taylor expansion
\begin{equation} \label{eq:H_GPD}
 H(\tau,\xi)=\frac{1}{2}h_2 \xi^2 +\frac{1}{6} h_3 \xi^3 + \frac{1}{24} h_4 \xi^4+\frac{1}{120} h_5 \xi^5 + O(\xi^6),
\end{equation}
with $h_j(2T)=h_j(0)$, 
%
with
\begin{equation} \label{eq:EigenFunc_GPD}
\left\{\begin{array}{rcl}
\dot{v} - A(\tau)v  & = & 0,\ \tau \in [0,T], \\
v(T)+v(0) & = & 0,\\
\int_{0}^{T} {\langle v,v\rangle d\tau} - 1 & = & 0,
\end{array}
\right.
\end{equation}
and 
\begin{equation*} \label{eq:EigenFunc2T_GPD}
v(\tau +T) = -v(\tau) \mbox{ for } \tau \in [0,T].
\end{equation*}
The function $v$ exists due to Lemma~5 of \cite{Io:88}.

The functions $h_{i}$, $i=1\ldots5$, can be found by solving
appropriate BVPs, assuming that \eqref{eq:P.1} restricted to
$W^c(\Gamma)$ has the periodic normal form \eqref{eq:NF-GPD}. From (\ref{eq:CM_GPD}) and (\ref{eq:H_GPD}) it follows that $h_i(\tau+T) = h_i(\tau)$ for $i$ even and $h_i(\tau+T) = -h_i(\tau)$ for $i$ odd, for $\tau \in [0,T]$. Indeed, since we are at a generalized period-doubling point $u(\tau,\xi) = u(\tau+T,-\xi)$, so 
$$\sum_i \frac{1}{i!} h_i(\tau) \xi^i = \sum_i \frac{1}{i!} h_i(\tau+T) (-1)^i \xi^i,$$
and thus
$$h_i(\tau) = (-1)^ih_i(\tau+T),$$
from which the stated follows. This makes it possible to restrict our considerations to the interval $[0,T]$ instead of $[0,2T]$.

The coefficients $\alpha_1$, $\alpha_2$ and $e$ arise from the
solvability conditions for the BVPs as integrals of scalar
products over the interval $[0,T]$. Specifically, these scalar
products involve among other things the terms up to the fifth order of \eqref{ODE}
near the periodic solution $u_0$, the eigenfunction $v$, the adjoint eigenfunction $\varphi^*$ satisfying
\begin{equation} \label{eq:AdjEigenFunc2T-2_GPD}
\left\{\begin{array}{rcl}
 \dot{\varphi}^*+A^{\rm T}(\tau)\varphi^* & = & 0,\ \tau \in [0,T], \\
 \varphi^*(T)-\varphi^*(0) & = & 0,\\
 \int_0^T \langle \varphi^*,F(u_0)\rangle\; d\tau - 1 & = & 0,
\end{array}
\right.
\end{equation}
and a similar adjoint eigenfunction $v^*$ satisfying
\begin{equation}  \label{eq:AdjEigenFunc_GPD}
\left\{\begin{array}{rcl}
 \dot{v}^*+A^{\rm T}(\tau)v^* & = & 0,\ \tau \in [0,T], \\
 v^*(T)+v^*(0) & = & 0,\\
 \int_{0}^{T} {\langle v^*,v\rangle d\tau} - 1 & = & 0.
\end{array}
\right.
\end{equation}

To derive the normal form coefficient, we proceed as in
Section~\ref{Section:Cusp}, namely, we substitute
(\ref{eq:CM_GPD}) into (\ref{eq:P.1}) and use (\ref{eq:MULT}), as
well as \eqref{eq:NF-GPD} and (\ref{eq:H_GPD}).

Collecting the $\xi^0$-terms in the resulting equation gives us the
identity
$$
\dot{u}_0=F(u_0),
$$
since $u_0$ is the $T$-periodic solution of (\ref{eq:P.1}).

The $\xi^1$-terms provide another identity
$$
\dot{v}=A(\tau)v,
$$
due to (\ref{eq:EigenFunc_GPD}) and (\ref{eq:EigenFunc2T_GPD}).

By collecting the $\xi^2$-terms, we obtain the equation for $h_2$,
\begin{equation} \label{eq:Eq_h_2_GPD}
\dot{h}_2 - A(\tau) h_2 = B(\tau;v,v) - 2\alpha_1 \dot{u}_0 ,
\end{equation}
to be solved in the space of functions satisfying
$h_2(T)=h_2(0)$. In this space, the differential operator
$\frac{d}{d\tau}-A(\tau)$ is singular with null-function $\dot{u}_0$. Thus, the following Fredholm solvability condition is involved
$$
 \int_0^{T} \langle \varphi^*, B(\tau;v,v) - 2\alpha_1 \dot{u}_0 \rangle\; d\tau=0,
$$
which leads to the expression
\begin{equation} \label{eq:a_GPD}
 \alpha_1 = \frac{1}{2}\int_0^{T} \langle \varphi^*,B(\tau;v,v) \rangle\; d\tau,
\end{equation}
where $v$ and $\varphi^*$ are defined by (\ref{eq:EigenFunc_GPD})
and (\ref{eq:AdjEigenFunc2T-2_GPD}), respectively.

With $\alpha_1$ defined in this way, let $h_2$ be a solution of (\ref{eq:Eq_h_2_GPD}) in the space of functions satisfying $h_2(0)=h_2(T)$. Notice also here that if
$h_2$ is a solution of \eqref{eq:Eq_h_2_GPD}, then also
$h_2+\varepsilon_1 F(u_0)$ satisfies
\eqref{eq:Eq_h_2_GPD}, since $F(u_0)$ is in the kernel of
the operator $\frac{d}{d\tau}-A(\tau)$. In
order to obtain a unique solution (without projection on the null
eigenspace) we impose the following orthogonality condition which determines the value of $\varepsilon_1$
$$
\int_0^{T} \langle \varphi^*,h_2\rangle \; d\tau =0.
$$
Thus $h_2$ is the unique solution of the BVP
\begin{equation}  \label{eq:BVP_h_2_GPD}
\left\{\begin{array}{rcl}
\dot{h}_2-A(\tau) h_2 - B(\tau;v,v) + 2\alpha_1 F(u_0) & = & 0,\ \tau \in [0,T], \\
h_2(T)- h_2(0) & = & 0,\\
\int_0^{T} \langle \varphi^*,h_2\rangle \; d\tau & = & 0.
\end{array}
\right.
\end{equation}

Collecting the $\xi^3$-terms, we get the equation for $h_3$,
\begin{equation} \label{eq:Eq_h_3_GPD}
\dot{h}_3 - A(\tau) h_3 = C(\tau;v,v,v)+ 3B(\tau;v,h_2) -
6\alpha_1 \dot{v},
\end{equation}
to be solved in the space of functions satisfying
$h_3(T)=-h_3(0)$. In this space the differential
operator $\frac{d}{d\tau}-A(\tau)$ has a one-dimensional
nullspace, spanned by $v$, and \eqref{eq:Eq_h_3_GPD} is solvable only if the
RHS of this equation lies in the reachable space of that
operator. Using
\eqref{eq:EigenFunc_GPD}, we can
rewrite the right-hand side as
\[
C(\tau;v,v,v)+ 3B(\tau;v,h_2) - 6\alpha_1 A(\tau) v.
\]
But the Fredholm solvability condition
\begin{equation} \label{eq:GPD_c}
\int_0^{T} \langle v^*,C(\tau;v,v,v)+ 3B(\tau;v,h_2) - 6\alpha_1
A(\tau) v\rangle \; d\tau =0
\end{equation}
is trivially satisfied due
to the fact that we are in a generalized period-doubling point and
so the cubic coefficient of the normal form
\[
 c=\frac{1}{6} \int_0^{T} \langle v^*,C(\tau;v,v,v)+ 3B(\tau;v,h_2) - 6\alpha_1 A(\tau) v\rangle \; d\tau
\]
(for it's definition see \cite{KuDoGoDh:05}) vanishes. Since the RHS
of \eqref{eq:Eq_h_3_GPD} is in the range space of the operator
$\frac{d}{d\tau}-A(\tau)$, we can solve the equation in order to find $h_3$
as the unique solution of the BVP
\begin{equation}  \label{eq:BVP_h_3_GPD}
 \left\{\begin{array}{rcl}
  \dot{h}_3-A(\tau) h_3 - C(\tau;v,v,v)- 3B(\tau;v,h_2) + 6\alpha_1 A(\tau) v & = & 0,\ \tau \in [0,T], \\
  h_3(T) + h_3(0) & = & 0,\\
  \int_0^{T} \langle v^*,h_3\rangle \; d\tau & = & 0.
 \end{array}
\right.
\end{equation}

By collecting the $\xi^4$-terms, we get the equation for $h_4$,
\begin{multline*} 
\dot{h}_4-A(\tau) h_4= D(\tau;v,v,v,v)+6 C(\tau;v,v,h_2)+3
B(\tau;h_2,h_2)\\+4 B(\tau;v, h_3)-12 \alpha_1 \dot{h}_2-24
\alpha_2 \dot{u}_0,
\end{multline*}
to be solved in the space of functions satisfying
$h_4(T)=h_4(0)$. Formulation of the Fredholm solvability condition
\[
\int_0^{T} \langle \varphi^*,D(\tau;v,v,v,v)+6 C(\tau;v,v,h_2)+3
B(\tau;h_2,h_2)+4 B(\tau;v, h_3)-12 \alpha_1 \dot{h}_2-24 \alpha_2
\dot{u}_0 \rangle\; d\tau=0
\]
gives us an equation for the parameter $\alpha_2$
\[
\alpha_2=\frac{1}{24}\int_0^{T} \langle
\varphi^*,D(\tau;v,v,v,v)+6 C(\tau;v,v,h_2)+3 B(\tau;h_2,h_2)+4
B(\tau;v, h_3)-12 \alpha_1 \dot{h}_2\rangle\; d\tau
\]
which can be simplified considering \eqref{eq:Eq_h_2_GPD} into
\begin{multline*} 
\alpha_2=\frac{1}{24}\int_0^{T} \langle
\varphi^*,D(\tau;v,v,v,v)+6 C(\tau;v,v,h_2)+3 B(\tau;h_2,h_2)+\\4
B(\tau;v, h_3)-12 \alpha_1 (A(\tau) h_2 + B(\tau;v,v))\rangle\;
d\tau\;+\;\alpha_1^2,
\end{multline*}
where $\alpha_1$ is given by \eqref{eq:a_GPD}, and $h_2$, $h_3$,
$v$ and $\varphi^*$ are the solutions of the BVPs
\eqref{eq:BVP_h_2_GPD}, \eqref{eq:BVP_h_3_GPD},
\eqref{eq:EigenFunc_GPD} and \eqref{eq:AdjEigenFunc2T-2_GPD}, respectively.

Using this value of $\alpha_2$ we can find $h_4$ by
\begin{equation} \label{eq:BVP_h_4_GPD}
\left\{\begin{array}{rcl}
 \dot{h}_4-A(\tau) h_4 - D(\tau;v,v,v,v)-6 C(\tau;v,v,h_2)-3 B(\tau;h_2,h_2)-&& \\
       4 B(\tau;v,h_3)+12 \alpha_1 (A(\tau) h_2 + B(\tau;v,v) - 2\alpha_1 F(u_0))+24 \alpha_2 F(u_0) & = & 0,\ \tau \in [0,T], \\
 h_4(T)- h_4(0) & = & 0,\\
 \int_0^{T} \langle \varphi^*,h_4 \rangle \; d\tau & = & 0.
\end{array}
\right.
\end{equation}

Finally, by collecting the $\xi^5$-terms, we get the equation for
$h_5$,
\begin{multline*} 
\dot{h}_5-A(\tau) h_5=E(\tau;v,v,v,v,v)+10 D(\tau;v,v,v,h_2)+15
C(\tau;v,h_2,h_2)\\+10 C(\tau;v,v,h_3)+10 B(\tau;h_2,h_3)+5
B(\tau;v, h_4)-120 \alpha_2 \dot{v}-20 \alpha_1 \dot{h}_3-120 e v,
\end{multline*}
which has to be solved in the space of functions satisfying $h_5(T)=-h_5(0)$. Since the operator
$\frac{d}{d\tau}-A(\tau)$ has a one-dimensional null-space, we can
write 
\begin{eqnarray*}
\int_0^{T} \langle v^*,E(\tau;v,v,v,v,v)+10 D(\tau;v,v,v,h_2)+15
C(\tau;v,h_2,h_2)+10 C(\tau;v,v,h_3)+\\10 B(\tau;h_2,h_3)+5
B(\tau;v, h_4)-120 \alpha_2 \dot{v}-20 \alpha_1 \dot{h}_3-120 e v
\rangle \; d\tau =0,
\end{eqnarray*}
which makes it possible to compute the parameter $e$ of the normal
form \eqref{eq:NF-GPD}. Using the normalization of \eqref{eq:AdjEigenFunc_GPD}, \eqref{eq:BVP_h_3_GPD} and
\eqref{eq:GPD_c} gives
\begin{multline} \label{eq:c-GPD}
e=\frac{1}{120} \int_0^{T} \langle v^*,E(\tau;v,v,v,v,v)+10
D(\tau;v,v,v,h_2)+15 C(\tau;v,h_2,h_2)+\\10 C(\tau;v,v,h_3)+10
B(\tau;h_2,h_3)+5 B(\tau;v, h_4)-120 \alpha_2 A(\tau) v - 20
\alpha_1 A(\tau) h_3 \rangle \; d\tau.
\end{multline}
A check that this quantity doesn't vanish, guarantees us that the
codim-2 bifurcation is non-degenerate.

\subsection{Chenciner bifurcation}\label{Section:GNS}

The three-dimensional critical center manifold $W^c(\Gamma)$ at
the {\tt CH} bifurcation can be parametrized locally by
$(\tau,\xi)$ as
\begin{equation}
u=u_0+\xi v + \bar{\xi} \bar{v}(\tau) + H(\tau,\xi,\bar{\xi}),\ \
\tau \in [0,T],\ \xi \in {\mathbb C}, \label{eq:CM_GNS}
\end{equation}
where the real function $H$ satisfies
$H(T,\xi,\bar{\xi})=H(0,\xi,\bar{\xi})$, and has the Taylor
expansion
\begin{eqnarray}
H(\tau,\xi,\bar{\xi}) & = & \frac{1}{2}h_{20}\xi^2 + h_{11} \xi
\bar{\xi}
+\frac{1}{2}h_{02} \bar{\xi}^2 \nonumber\\
 & + & \frac{1}{6}h_{30}\xi^3 + \frac{1}{2}h_{21}\xi^2 \bar{\xi} +
\frac{1}{2}h_{12}\xi \bar{\xi}^2 + \frac{1}{6}h_{03}\bar{\xi}^3 \nonumber \\
 & + &\frac{1}{24}h_{40}(\tau)\xi^4 + \frac{1}{6}h_{31}(\tau)\xi^3 \bar{\xi}
 +  \frac{1}{4}h_{22}(\tau)\xi^2 \bar{\xi}^2 +
 \frac{1}{6}h_{13}(\tau)\xi \bar{\xi}^3 + \frac{1}{24}h_{04}(\tau)\bar{\xi}^4 \nonumber \\
 & + &\frac{1}{120}h_{50}(\tau)\xi^5 + \frac{1}{24}h_{41}(\tau)\xi^4 \bar{\xi}
 +  \frac{1}{12}h_{32}(\tau)\xi^3 \bar{\xi}^2 + \frac{1}{12}h_{23}(\tau)\xi^2 \bar{\xi}^3 +
 \frac{1}{24}h_{14}(\tau)\xi \bar{\xi}^4 \nonumber \\ &+& \frac{1}{120}h_{05}(\tau)\bar{\xi}^5
 +  O(|\xi|^6),  \label{eq:H_GNS}
\end{eqnarray}
with $h_{ij}(T)=h_{ij}(0)$ and $h_{ij}=\bar{h}_{ji}$ so that
$h_{ii}$ is real,  while $v$ and its conjugate $\bar v$ are
defined as
\begin{equation}
\left\{\begin{array}{rcl}
\dot{v}(\tau)-A(\tau)v + i\omega\, v & = & 0,\ \tau \in [0,T], \\
v(T)-v(0) & = & 0,\\
\int_{0}^{T} {\langle v,v\rangle d\tau} - 1 & = & 0.
\end{array}
\right. \label{eq:EigenFunc_GNS}
\end{equation}
These functions exist due to Lemma~2 of \cite{Io:88}.

If we assume that (\ref{eq:P.1}) restricted to $W^c(\Gamma)$ has
the periodic normal form (\ref{eq:NF-GNS}), as in the previous cases, we can find the functions $h_{ij}(\tau)$
by solving appropriate BVPs. 

First we introduce the two needed adjoint eigenfunctions. The
first one, namely $\varphi^*$, satisfies \eqref{eq:AdjEigenFunc2T-2_GPD},
and the second one, namely $v^*$, satisfies
\begin{equation}
\left\{\begin{array}{rcl} \dot{v}^*(\tau)+A^{\rm T}(\tau)v^* +
i\omega \,v^* & = & 0,\
\tau \in [0,T], \\
v^*(T)-v^*(0) & = & 0,\\
\int_{0}^{T} {\langle v^*,v\rangle d\tau} - 1 & = & 0.
\end{array}
\right. \label{eq:AdjEigenFunc_GNS}
\end{equation}

As usual, we substitute \eqref{eq:CM_GNS} into \eqref{eq:P.1}, use
\eqref{eq:MULT}, \eqref{eq:NF-GNS}, and (\ref{eq:H_GNS}), as well
as the homological equation
$$
\frac{du}{dt}=\frac{\partial u}{\partial \xi}\frac{d\xi}{dt} +
\frac{\partial u}{\partial \bar{\xi}}\frac{d\bar{\xi}}{dt} +
\frac{\partial u}{\partial \tau}\frac{d\tau}{dt},
$$
and collect the corresponding terms in order to find the needed
coefficients of \eqref{eq:NF-GNS}.

The $\xi$-independent and the linear terms give rise to the usual
identities
$$
 \dot{u}_0=F(u_0), \quad \dot{v}-A(\tau) v+i \omega v=0, \quad \dot{\bar v}-A(\tau) \bar v - i \omega \bar
 v=0.
$$

Collecting the coefficients of the $\xi^2$- or $\bar{\xi}^2$-terms
leads to the equation
\[
\dot h_{20} -A(\tau) h_{20} + 2 i \omega h_{20} = B(\tau;v,v)
\]
or its complex-conjugate. This equation has a unique solution
$h_{20}$ satisfying $h_{20}(T)=h_{20}(0)$, since due to the spectral
assumptions $e^{2i\omega T}$ is not a multiplier of the critical cycle. Thus, $h_{20}$ can be found by solving
\begin{equation}
\left\{\begin{array}{rcl}
 \dot h_{20}-A(\tau) h_{20} +2i \omega h_{20} - B(\tau;v,v) & = & 0,\ \tau \in [0,T], \\
 h_{20}(T)-h_{20}(0) & = & 0.
\end{array}
\right. \label{eq:h20-GNS}
\end{equation}

By collecting the $\xi\bar\xi$-terms we obtain an equation for $h_{11}$, namely
\[
\dot{h}_{11}-A(\tau)h_{11}=B(\tau;v,\bar{v})-\alpha_1\dot{u}_0,
\]
to be solved in the space of the functions satisfying
$h_{11}(T)=h_{11}(0)$. In this space the operator
$\frac{d}{d\tau}-A(\tau)$ has a range space with codimension one. As before, the null-eigenfunction of
the adjoint operator $-\frac{d}{d\tau} - A^{\rm T}(\tau)$ is
$\varphi^*$, given by \eqref{eq:AdjEigenFunc2T-2_GPD}, and thus because of the Fredholm
solvability condition, we can easily obtain the needed value for
$\alpha_1$
\begin{equation}\label{eq:a1-CH}
 \alpha_1=\int_0^T \langle \varphi^*, B(\tau;v,\bar v) \rangle d\tau.
\end{equation}
With $\alpha_1$ defined in this way, let $h_{11}$ be the unique
solution of the BVP
\begin{equation}
\left\{\begin{array}{rcl}
 \dot{h}_{11}-A(\tau) h_{11}-B(\tau;v,\bar{v})+\alpha _1 \dot{u}_0 & = & 0,\ \tau \in [0,T], \\
 h_{11}(T)-h_{11}(0) & = & 0,\\
 \int_{0}^{T} {\langle \varphi^*, h_{11}\rangle d\tau} & = & 0.
\end{array}
\right. \label{eq:h11-CH}
\end{equation}

The coefficient in front of the third order terms in (\ref{eq:NF-GNS}) is purly imaginary since the first Lyapunov coefficient vanishes at a Chenciner point. We are now ready to compute this coefficient. In fact, if we collect the $\xi^2\bar\xi$-terms we
obtain
\[
 \dot h_{21}- A(\tau) h_{21} +i \omega  h_{21}=C(\tau;v,v,\bar{v})+2 B(\tau;v,h_{11})+B(\tau;\bar{v},h_{20}) - 2 i c v - 2 \alpha _1 \dot v,
\]
to be solved in the space of functions satisfying
$h_{21}(T)=h_{21}(0)$. In this space the operator
$\frac{d}{d\tau}-A(\tau)+i \omega$ is singular, since $e^{i\omega
T}$ is
 a multiplier of the critical cycle. So we can impose the usual Fredholm
solvability condition, taking (\ref{eq:AdjEigenFunc_GNS}) into
account
\begin{equation}\label{eq:GNS_Fredholm_21}
\int_0^T \langle v^*,C(\tau;v,v,\bar{v})+2
B(\tau;v,h_{11})+B(\tau;\bar{v},h_{20}) - 2 i c v - 2 \alpha _1
\dot v \rangle d \tau = 0.
\end{equation}
From this equation we can find the value of the coefficient $c$ of
the normal form \eqref{eq:NF-GNS}
\begin{equation}
 c=-\frac{i}{2} \int_0^T \langle v^*,C(\tau;v,v,\bar{v})+2
B(\tau;v,h_{11})+B(\tau;\bar{v},h_{20}) - 2 \alpha _1 A(\tau) v
\rangle d \tau +\alpha_1\omega \label{eq:c-GNS}
\end{equation}
and, with $c$ defined in this way, we can find $h_{21}$ as the
unique solution of the BVP
\begin{equation}
\left\{\begin{array}{rcl}
 \dot h_{21}-A(\tau) h_{21}+i \omega h_{21}-C(\tau;v,v,\bar{v})-2 B(\tau;v,h_{11})&&\\
 - B(\tau;\bar{v},h_{20}) +2 i c v + 2 \alpha _1 (A(\tau) v -i \omega v) & = & 0,\ \tau \in [0,T], \\
 h_{21}(T)-h_{21}(0) & = & 0, \\
 \int_{0}^{T} {\langle v^*,h_{21}\rangle d\tau} & = & 0.
\end{array}
\right. \label{eq:h21-GNS}
\end{equation}

Collecting the $\xi^3$-terms gives us an equation for $h_{30}$
\[
\dot h_{30}-A(\tau) h_{30}+3 i \omega
h_{30}=C(\tau;v,v,v)+3B(\tau;v, h_{20}),
\]
which has a unique solution $h_{30}$ satisfying
$h_{30}(T)=h_{30}(0)$, since $e^{3i\omega T}$ is not a multiplier
of the critical cycle by the spectral assumptions. Thus, $h_{30}$
is the unique solution of the BVP
\begin{equation}
\left\{\begin{array}{rcl}
 \dot h_{30}-A(\tau) h_{30}+3 i \omega h_{30}-C(\tau;v,v,v)-3B(\tau;v, h_{20}) & = & 0,\ \tau \in [0,T], \\
 h_{30}(T)-h_{30}(0) & = & 0.
\end{array}
\right. \label{eq:h30-GNS}
\end{equation}

By collecting the $\xi^3\bar \xi$-terms we obtain an equation for
$h_{31}$
\begin{eqnarray*}
\dot h_{31}-A(\tau) h_{31}+2 i \omega
h_{31}&=&D(\tau;v,v,v,\bar{v})+3 C(\tau;v,v, h_{11})+3
C(\tau;v,\bar{v}, h_{20})+ 3 B(\tau;h_{11}, h_{20})\\&&+3
B(\tau;v, h_{21})+B(\tau;\bar{v},h_{30}) - 6 i c h_{20}-3 \alpha
_1 \dot h_{20}
\end{eqnarray*}
which has a unique solution $h_{31}$ satisfying
$h_{31}(T)=h_{31}(0)$, since $e^{2i\omega T}$ is not a multiplier
of the critical cycle by the spectral assumptions. Thus, $h_{31}$
is the unique solution of the BVP
\begin{equation}
\left\{\begin{array}{rcl}
 \dot h_{31}-A(\tau) h_{31}+2 i \omega  h_{31}-D(\tau;v,v,v,\bar{v})-3 C(\tau;v,v, h_{11})&& \\
 -3 C(\tau;v,\bar{v}, h_{20})- 3 B(\tau;h_{11}, h_{20}) -3 B(\tau;v, h_{21}) &&\\
 -B(\tau;\bar{v},h_{30})+ 6 i c h_{20}+ 3 \alpha _1  (A(\tau) h_{20} - 2 i \omega h_{20} + B(\tau;v,v))& = & 0,\ \tau \in [0,T], \\
 h_{31}(T)-h_{31}(0) & = & 0.
\end{array}
\right. \label{eq:h31-GNS}
\end{equation}

Taking into account the $|\xi|^4$-terms gives an equation for
$h_{22}$
\begin{eqnarray*}
 \dot h_{22}-A(\tau) h_{22}&=& D(\tau;v,v,\bar{v},\bar v)+C(\tau;v,v,h_{02}) \\
 &&+4 C(\tau;v,\bar{v}, h_{11})+C(\tau;\bar{v},\bar v, h_{20})+2 B(\tau;h_{11},h_{11})+2 B(\tau;v,h_{12})\\
 &&+B(\tau;h_{02},h_{20}) +2 B(\tau;\bar{v}, h_{21})- 4\alpha _1 \dot h_{11}- 4 \alpha _2 \dot u_0,
\end{eqnarray*}
to be solved in the space of functions satisfying
$h_{22}(T)=h_{22}(0)$. In this space the operator
$\frac{d}{d\tau}-A(\tau)$ has a range space with codimension one which is orthogonal to $\varphi^*$. So one Fredholm solvability condition is involved, namely
\begin{multline*}
 \int_0^T \langle \varphi^*, D(\tau;v,v,\bar{v},\bar v)+C(\tau;v,v,h_{02})+ 4
C(\tau;v,\bar{v}, h_{11})+C(\tau;\bar{v},\bar v,
h_{20})+2 B(\tau;h_{11},h_{11})\\+2
B(\tau;v,h_{12})+B(\tau;h_{02},h_{20}) +2 B(\tau;\bar{v}, h_{21})- 4 \alpha
_1 \dot h_{11}- 4 \alpha _2 \dot u_0 \rangle d \tau = 0,
\end{multline*}
which allows us to compute the value of the coefficient $\alpha_2$
of our normal form
\begin{multline}\label{eq:a2-GNS}
 \alpha_2=\frac{1}{4} \int_0^T \langle \varphi^*, D(\tau;v,v,\bar{v},\bar v)+C(\tau;v,v,h_{02}) + 4 C(\tau;v,\bar{v}, h_{11}) \\
  +C(\tau;\bar{v},\bar v, h_{20})+2 B(\tau;h_{11},h_{11})+2B(\tau;v,h_{12}) + B(\tau;h_{02},h_{20})\\
  +2 B(\tau;\bar{v}, h_{21})- 4 \alpha_1 (A(\tau) h_{11} + B(\tau;v, \bar v)) \rangle d \tau+ \alpha_1^2.
\end{multline}
Using this value for $\alpha_2$ we can find $h_{22}$ as the
unique solution of the BVP
\begin{equation}
\left\{\begin{array}{rcl}
 \dot h_{22}-A(\tau) h_{22} - D(\tau;v,v,\bar{v},\bar v)- C(\tau;v,v,h_{02}) \\
 - 4 C(\tau;v,\bar{v}, h_{11})-C(\tau;\bar{v},\bar v, h_{20})-2 B(\tau;h_{11},h_{11})-2 B(\tau;v,h_{12})\\
 -B(\tau;h_{02},h_{20}) -2 B(\tau;\bar{v}, h_{21})\\
 + 4 \alpha _1 (A(\tau) h_{11}+ B(\tau;v, \bar v)-\alpha_1 F(u_0)) + 4 \alpha _2 F(u_0)& = & 0,\ \tau \in [0,T], \\
 h_{22}(T)-h_{22}(0) & = & 0,\\
 \int_0^T \langle \varphi^*,h_{22}\rangle d \tau & = & 0.
\end{array}
\right. \label{eq:h22-GNS}
\end{equation}

Finally, by collecting the $\xi^3 \bar \xi^2$-terms we get an
equation for $h_{32}$
\begin{eqnarray*}
 \dot h_{32}-A(\tau) h_{32}+i \omega  h_{32}=E(\tau;v,v,v,\bar{v},\bar v)+D(\tau;v,v,v,h_{02})+6 D(\tau;v,v,\bar{v},h_{11}) \\
 +3 D(\tau;v,\bar{v},\bar v,h_{20})+ 6 C(\tau;v,h_{11},h_{11})+3 C(\tau;v,v,h_{12})+ 3 C(\tau;v, h_{02}, h_{20})\\
 + 6 C(\tau;\bar{v},h_{11}, h_{20})+6 C(\tau;v, \bar{v}, h_{21})+C(\tau;\bar{v},\bar v, h_{30})+3 B(\tau;h_{12},h_{20})+6 B(\tau;h_{11}, h_{21})\\
 +3 B(\tau;v, h_{22})+B(\tau;h_{02}, h_{30})+2B(\tau;\bar{v}, h_{31})\\
 - 12 e v-6 i c h_{21}-12 \alpha _2 \dot v-6 \alpha _1 \dot h_{21}
\end{eqnarray*}
that, since the operator is singular, allows us, using the first of
\eqref{eq:EigenFunc_GNS} as well as the first and the last of
\eqref{eq:h21-GNS} and \eqref{eq:GNS_Fredholm_21}, to compute the critical coefficient $e$ of
\eqref{eq:NF-GNS} imposing the Fredholm solvability condition,
obtaining
\begin{eqnarray}\label{eq:e-GNS}
e=\frac{1}{12} \int_0^T \langle v^*&,& E(\tau;v,v,v,\bar{v},\bar v)+D(\tau;v,v,v,h_{02})+6 D(\tau;v,v,\bar{v},h_{11}) \nonumber \\
 &&+3 D(\tau;v,\bar{v},\bar v,h_{20})+ 6 C(\tau;v,h_{11},h_{11})+3 C(\tau;v,v,h_{12})\nonumber\\
 &&+ 3 C(\tau;v, h_{02}, h_{20})+ 6 C(\tau;\bar{v},h_{11}, h_{20})+6 C(\tau;v, \bar{v}, h_{21})+C(\tau;\bar{v},\bar v, h_{30}) \nonumber\\
 &&+3 B(\tau;h_{12},h_{20})+6 B(\tau;h_{11}, h_{21})+ 3 B(\tau;v, h_{22})+B(\tau;h_{02}, h_{30})\nonumber \\
 &&+2B(\tau;\bar{v}, h_{31})-12 \alpha _2 A(\tau) v - 6 \alpha _1 (A(\tau) h_{21}+2B(\tau;v,h_{11})+C(\tau;v,v,\bar{v})\nonumber\\
 &&+B(\tau;\bar v,h_{20})-2\alpha_1 Av) \rangle d \tau + i \omega \alpha_2+ ic \alpha_1-\alpha_1^2 i\omega.
\end{eqnarray}
As stated in Appendix \ref{Appendix:2}, we define the second Lyapunov coefficient as
\[
L_2(0)= \Re(e).
\]
If this coefficient does not vanish, no more degeneracies happen at this
codim $2$ point.\\

Since we have to check all equations up to the fifth order, we still have to look at the $\xi^4$-terms, the $\xi^5$-terms and the $\xi^4 \bar \xi$-terms, which give respectively
\begin{eqnarray*}
 \dot h_{40}-A(\tau) h_{40}+4i \omega  h_{40}&=&D(\tau;v,v,v,v)+6C(\tau;v,v,h_{20})+3 B(\tau;h_{20},h_{20})\\
 && B(\tau;v,h_{30}),
\end{eqnarray*}
\begin{eqnarray*}
 \dot h_{50}-A(\tau) h_{50}+5i \omega  h_{50}&=&E(\tau;v,v,v,v,v)+10D(\tau;v,v,v,h_{20})+ 10 C(\tau;v,v,h_{30})\\
 &&+15 C(\tau;v,h_{20},h_{20})+ 10 B(\tau;h_{20},h_{30})+5 B(\tau;v, h_{40})
\end{eqnarray*}
and 
\begin{eqnarray*}
 \dot h_{41}-A(\tau) h_{41}+3i \omega  h_{41}&=&E(\tau;v,v,v,v,\bar v)+6 D(\tau;v,v,\bar v,h_{20})+4 D(\tau;v,v, v,h_{11})\\
 &&+ 4 C(\tau;v,\bar v,h_{30})+8 C(\tau;v,h_{20},h_{11})+6 C(\tau;v,v,h_{21})\\
 &&+C(\tau;\bar v,h_{20},h_{20})+ 6 B(\tau;h_{20},h_{21})+  B(\tau;\bar v,h_{40})\\
 &&+ 4 B(\tau;v,h_{31})+ 4 B(\tau;h_{30},h_{11})-3\alpha_1 \dot{h}_{30}-12ich_{30}.
\end{eqnarray*}
No solvability conditions have to be satisfied.\\

Since we are in a complex eigenvalues case, $v$ is determined up to a factor $\gamma$, for which $\gamma^{\rm H}\gamma = 1$. Then $v^*$, $h_{20}$, $h_{21}$, $h_{30}$, $h_{31}$ are replaced by $\gamma v^*$, $\gamma^2h_{20}$, $\gamma h_{21}$, $\gamma^3h_{30}$, $\gamma^2h_{31}$ respectively, but $\alpha_1, \alpha_2, c$ and $e$ are not affected by this factor.

\subsection{Strong resonance 1:1 bifurcation} \label{Section:Res1:1}
The three-dimensional critical center manifold $W^c(\Gamma)$ at
the {\tt R1} bifurcation can be parameterized locally by
$(\tau,\xi)$ as 
\begin{equation} \label{eq:R1_CM}
u=u_0+\xi_1 v_1 + \xi_2 v_2 + H(\tau,\xi),\ \ \tau \in [0,T],\
\xi=(\xi_1,\xi_2) \in {\mathbb R^2},
\end{equation}
where $H$ satisfies $H(T,\xi)=H(0,\xi)$ and has the Taylor
expansion 
\begin{equation} \label{eq:R1_H}
 H(\tau,\xi)=\frac{1}{2}h_{20}\xi_1^2 + h_{11}\xi_1 \xi_2 + \frac{1}{2}h_{02}\xi_2^2 
  + O(\|\xi\|^3) ,
\end{equation}
where the functions $h_{20}$, $h_{11}$ and $h_{02}$ are
$T$-periodic in $\tau$, where $v_1$ and $v_2$ are the generalized
eigenfunctions associated with the trivial multiplier and defined as
the unique solutions of the BVP
\begin{eqnarray} \label{eq:R1_EigenFunc}
&&\left\{\begin{array}{rcl}
\dot{v}_1-A(\tau)v_1 - F(u_0) & = & 0,\ \tau \in [0,T], \\
v_1(T)-v_1(0) & = & 0,\\
\int_{0}^{T} {\langle v_1,F(u_0)\rangle d\tau} & = & 0,
\end{array}
\right.
\end{eqnarray}
and
\begin{eqnarray} \label{eq:R1_EigenFunc_2}
&&\left\{\begin{array}{rcl}
\dot{v}_2-A(\tau)v_2 + v_1 & = & 0,\ \tau \in [0,T], \\
v_2(T) - v_2(0) & = & 0,\\
\int_{0}^{T} {\langle v_2,F(u_0)\rangle d\tau} & = & 0,
\end{array}
\right.
\end{eqnarray}
respectively. The functions $v_1$ and $v_2$ exist and are different due to
Lemma~2 of \cite{Io:88}. Following our approach to find the
value of the normal form constants, we define $\varphi^*$ as
a solution of the adjoint eigenfunction problem
\eqref{eq:AdjEigenFunc}, $v_1^*$ as a solution of the adjoint
generalized eigenfunction problem \eqref{eq:AdjGenEigenFunc} and
$v_2^*$ as a solution of
\begin{equation*}  
\left\{\begin{array}{rcl}
\dot{v_2}^*(\tau)+A^{\rm T}(\tau)v_2^* + v_1^* & = & 0,\ \tau \in [0,T], \\
v_2^*(T)-v_2^*(0) & = & 0.
\end{array}
\right.
\end{equation*}
First, notice that the Fredholm solvability condition gives us immediately the following scalar products
\begin{equation} \label{eq:R1_Ortho}
 \int_0^T\langle\varphi^*,F(u_0)\rangle d \tau =
 \int_0^T\langle\varphi^*,v_1\rangle d \tau =
 \int_0^T\langle F(u_0),v_1^*\rangle d \tau = 0.
\end{equation}
Due to the spectral assumptions at the {\tt R1} point we are free to
assume that
\begin{equation} \label{eq:R1_Normo}
\int_0^T\langle\varphi^*,v_2\rangle d \tau = 1.
\end{equation}
Appending this condition to the eigenproblem, we can find
the eigenfunctions $\varphi^*$ as the unique
solutions of the BVP
\begin{equation}
\left\{\begin{array}{rcl}
\dot{\varphi}^*+A^{\rm T}(\tau)\varphi^* & = & 0,\ \tau \in [0,T], \\
\varphi^*(T)-\varphi^*(0) & = & 0,\\
\int_{0}^{T} {\langle \varphi^*,v_2 \rangle d\tau} - 1 & = & 0.
\end{array}
\right. \label{eq:R1_AdjEigenFunc}
\end{equation}

As already mentioned in the cusp of cycles case, we will choose adjoint generalized eigenfunctions orthogonal to an original eigenfunction. Therefore, $v_1^*$ and $v_2^*$ are obtained as the solution of
\begin{equation}
\left\{\begin{array}{rcl}
\dot{v_1}^*+A^{\rm T}(\tau)v_1^* - \varphi^*& = & 0,\ \tau \in [0,T], \\
v_1^*(T)-v_1^*(0) & = & 0,\\
\int_{0}^{T} {\langle v_1^*, v_2 \rangle d\tau} & = & 0,
\end{array}
\right. \label{eq:R1_AdjGenEigenFunc}
\end{equation}
and
\begin{equation}
\left\{\begin{array}{rcl}
 \dot{v_2}^*(\tau)+A^{\rm T}(\tau)v_2^* + v_1^*& = & 0,\ \tau \in [0,T], \\
 v_2^*(T)-v_2^*(0) & = & 0,\\
 \int_{0}^{T} {\langle v_2^*, v_2 \rangle d\tau} & = & 0,
\end{array}
\right. \label{eq:R1_AdjGenEigenFunc2}
\end{equation}
respectively. Notice that, as in the cusp of cycles case, we have normalized in
\eqref{eq:R1_Normo} the adjoint eigenfunction with the last
generalized eigenfunction, which gives us in addition
\begin{equation*} 
\int_0^T\langle v_1^*,v_1\rangle d \tau = \int_0^T\langle
v_2^*,F(u_0)\rangle d \tau =1.
\end{equation*}

As usual, to derive the value of the normal form coefficients we substitute
\eqref{eq:R1_CM} into \eqref{eq:P.1}, we use \eqref{eq:MULT} as
well as \eqref{eq:NF-11C} and \eqref{eq:R1_H} and get different
equalities for every degree of $\xi$. Remark that in fact the solvability of
all the equations up to the maximal order of the normal form has to be checked. We will pay extra attention to it in this section. 

By collecting the $\xi^0$-terms we get the identity
$$
\dot{u}_0=F(u_0).
$$

The linear terms provide two other identities, namely
$$
\dot{v_1}-A(\tau) v_1-F(u_0)=0, \qquad \dot{v}_2-Av_2 +v_1= 0,
$$
cf. \eqref{eq:R1_EigenFunc} and \eqref{eq:R1_EigenFunc_2}.

By collecting the $\xi_1^2$-terms we find an equation for $h_{20}$
\begin{equation} \label{eq:R1_xi1^2}
 \dot{h}_{20}-A(\tau) h_{20}=-2\alpha \dot{u}_0 + 2\dot{v}_1+B(\tau;v_1,v_1)- 2 a v_2,
\end{equation}
to be solved in the space of periodic functions on $[0,T]$. In
this space, the differential operator $\frac{d}{d\tau}-A(\tau)$ is
singular with a range orthogonal to $\varphi^*$.
Thus a Fredholm solvability condition is involved, namely
\[
\int_0^T\langle \varphi^*,-2\alpha \dot{u}_0 + 2\dot{v}_1+B(\tau;v_1,v_1)- 2 a v_2 \rangle d
\tau = 0.
\]
The equations \eqref{eq:R1_Ortho}, \eqref{eq:R1_Normo}, and \eqref{eq:R1_EigenFunc} let us obtain the following value for $a$
\begin{eqnarray} \label{eq:R1_a}
 a=\frac{1}{2}\int_0^T\langle\varphi^*, 2A(\tau) v_1+B(\tau;v_1,v_1) \rangle d\tau.
\end{eqnarray}
Notice that in the RHS of \eqref{eq:R1_xi1^2} we have no freedom
which could change the value of the coefficient $a$.
This confirms the theoretically proved fact that the $\xi_1^2$-term
of normal form \eqref{eq:NF-11C} is resonant. Notice moreover that
parameter $\alpha$ is undetermined, which gives us two degrees of
freedom for $h_{20}$. In fact, if $h_{20}$ is a solution of
\eqref{eq:R1_xi1^2}, then also $\tilde h_{20} = h_{20} +
\varepsilon^{I}_{20} F(u_0)+\varepsilon^{II}_{20} v_1$ is a
solution, due to the fact that $F(u_0)$ spans the nullspace of the
operator $\frac{d}{dt}-A(\tau)$ and that we can tune $\alpha$ as
desirable:
\begin{equation} \label{tuningh20}
 \dd{\tilde h_{20}}{t}-A(\tau) \tilde h_{20} =  \dd{h_{20}}{t}-A(\tau) h_{20} +  \varepsilon^{II}_{20} \left(\dd{v_1}{t}-A(\tau) v_1\right)
 = \dd{h_{20}}{t}-A(\tau) h_{20} + \varepsilon^{II}_{20} \dot u_0.\end{equation}

By collecting the $\xi_1 \xi_2$-terms we find an equation for
$h_{11}$
\begin{equation}\label{eq:r1_xi1xi2}
 \dot{h}_{11}-A(\tau) h_{11}=B(\tau;v_1, v_2)+\dot{v}_2-h_{20}-b v_2-v_1,
\end{equation}
to be solved in the space of $T$-periodic functions. As in the
previous case, a solvability condition is involved
\[
 \int_0^T\langle \varphi^*,B(\tau;v_1, v_2)+\dot{v}_2-h_{20}-b v_2-v_1\rangle d \tau = 0.
\]
Equation \eqref{eq:R1_Normo} as well as \eqref{eq:R1_EigenFunc_2} and \eqref{eq:R1_Ortho} let us rewrite this
condition as
\[
 b= \int_0^T\langle \varphi^*,B(\tau;v_1, v_2)+A(\tau) v_2\rangle d \tau -\int_0^T\langle \varphi^*,h_{20}\rangle d \tau.
\]
Using \eqref{eq:R1_AdjGenEigenFunc}, \eqref{eq:R1_xi1^2}, \eqref{eq:R1_Ortho} and \eqref{eq:R1_AdjGenEigenFunc} we can
rewrite the second term of the right-hand side as
\begin{multline*}
 \int_0^T\langle \left(\frac{d}{d\tau}+A^T(\tau)\right)v_1^*,h_{20}\rangle d \tau
 = \int_0^T\langle v_1^*,\left(-\frac{d}{d\tau}+A(\tau)\right)h_{20}\rangle d
 \tau\\
 = -\int_0^T\langle v_1^*,-2\alpha \dot{u}_0 + 2\dot{v}_1+B(\tau;v_1,v_1)- 2 a v_2\rangle d
 \tau
 = -\int_0^T\langle v_1^*, 2 Av_1+B(\tau;v_1,v_1) \rangle d \tau
\end{multline*}
obtaining an equation for $b$ which involves only the original and adjoint eigenfunctions
\begin{equation} \label{eq:r1_b}
 b= \int_0^T\langle \varphi^*,B(\tau;v_1, v_2)+A(\tau) v_2\rangle d\tau + \int_0^T\langle v_1^*, 2 Av_1+B(\tau;v_1,v_1) \rangle d \tau.
\end{equation}
Notice that the freedom that we have on $h_{20}$ can not be used
to change the value to coefficient $b$ (and so the $\xi_1
\xi_2$-term of the normal form \eqref{eq:NF-11C} is resonant). Indeed, $h_{20}$ is defined up to a multiple of $F(u_0)$ and $v_1$, but both vectors are orthogonal to $\varphi^*$, see the first two orthogonality conditions in \eqref{eq:R1_Ortho}. However the
presence of $h_{20}$ in the RHS gives us three degrees of freedom for $h_{11}$. In fact, if $h_{11}$ is a solution of
\eqref{eq:r1_xi1xi2}, also $\tilde h_{11} = h_{11} +
\varepsilon^{I}_{11} F(u_0)-\varepsilon^{I}_{20}
v_1+\varepsilon^{II}_{20} v_2$ is a solution, since
\begin{multline*}
 \dd{\tilde h_{11}}{t}-A(\tau) \tilde h_{11} =  \dd{h_{11}}{t}-A(\tau) h_{11} -  \varepsilon^{I}_{20} \left(\dd{v_1}{t}-A(\tau)
 v_1\right) +  \varepsilon^{II}_{20} \left(\dd{v_2}{t}-A(\tau)
 v_2\right) \\
 = \dd{h_{11}}{t}-A(\tau) h_{11} - \varepsilon^{I}_{20} F(u_0)  - \varepsilon^{II}_{20} v_1.
\end{multline*}

Collecting the $\xi_2^2$-terms gives us the following equation for $h_{02}$
\[
 \dot h_{02} - A (\tau) h_{02} =  B(\tau, v_2, v_2) - 2 h_{11},
\]
to be solved in the space of $T$-periodic functions. This
equation should be solvable, so the RHS
should lay in the reachable space of the operator $\frac{d}{dt}-A(\tau)$:
\[
 \int_0^T\langle \varphi^*,B(\tau,v_2,v_2) - 2 h_{11} \rangle d \tau = 0.
\]
This condition can be satisfied by correctly tuning $h_{11}$. In fact, $\varepsilon^{II}_{20}$ is not yet determined, so $h_{11}$ can have a projection on $v_2$. Due to(\ref{eq:R1_Normo}) $v_2$ does
not lay in the reachable space of the $\frac{d}{dt}-A(\tau)$
operator, and therefore we can impose that
\[
 \int_0^T\langle \varphi^*, h_{11} \rangle d \tau  = \frac{1}{2} \int_0^T\langle \varphi^*,B(\tau,v_2,v_2) \rangle d \tau.
\]

This last solvability condition determines $\varepsilon^{II}_{20}$ uniquely, and since $\varepsilon^{II}_{20}$ determines the value of $\alpha$, see (\ref{eq:R1_xi1^2}) and (\ref{tuningh20}), also $\alpha$ is now uniquely determined. So the center manifold expansion \eqref{eq:NF-11C} has now become unique. Note that in fact the value of $\alpha$ is not needed since, as shown in Appendix \ref{Appendix:2}, it does not affect the bifurcation scenario. Remark also that in order to compute the necessary coefficients $a$ and $b$ by equations \eqref{eq:R1_a} and \eqref{eq:r1_b}, the second order expansion of the center manifold
is not needed. Indeed, we have rewritten the formulas of the normal form coefficients in terms of the original and adjoint eigenfunctions. $h_{20}$ or $h_{11}$ are not needed, therefore we don't write down the BVPs for their unique solutions.

\subsection{Strong resonance 1:2 bifurcation} \label{Section:Res1:2}
The three-dimensional critical center manifold $W^c(\Gamma)$ at
the {\tt R2} bifurcation can be parametrized locally by
$(\tau,\xi)$ as
\begin{equation}
u=u_0+\xi_1 v_1 + \xi_2 v_2 + H(\tau,\xi),\ \ \tau \in [0,2T],\
\xi=(\xi_1,\xi_2) \in {\mathbb R^2}, \label{eq:CM_12C}
\end{equation}
where $H$ satisfies $H(2T,\xi)=H(0,\xi)$ and has the Taylor
expansion
\begin{multline}
 H(\tau,\xi)=\frac{1}{2}h_{20}\xi_1^2+ h_{11}\xi_1 \xi_2 + \frac{1}{2}h_{02}\xi_2^2 + \\
    \frac{1}{6}h_{30}\xi_1^3+ \frac{1}{2}h_{21}\xi_1^2\xi_2+\frac{1}{2}h_{12}\xi_1\xi_2^2+\frac{1}{6}h_{03}\xi_2^3 
    + O(\|\xi\|^4) , \label{eq:H_12C}
\end{multline}
where all functions $h_{ij}$ are $2T$-periodic, the eigenfunction corresponding to eigenvalue $-1$ is given by
\begin{eqnarray} \label{eq:EigenFunc_12C}
&&\left\{\begin{array}{rcl}
\dot{v}_1-A(\tau)v_1 & = & 0,\ \tau \in [0,T], \\
v_1(T)+v_1(0) & = & 0,\\
\int_{0}^{T} {\langle v_1,v_1\rangle d\tau} -1 & = & 0,\\
\end{array}
\right. 
\end{eqnarray} 
and the generalized eigenfunction by
\begin{eqnarray} \label{eq:EigenFunc_12C_2}
&&\left\{\begin{array}{rcl}
\dot{v}_2-A(\tau)v_2 + v_1 & = & 0,\ \tau \in [0,T], \\
v_2(T) + v_2(0) & = & 0,\\
\int_{0}^{T} {\langle v_2,v_1\rangle d\tau} & = & 0,\\
\end{array}
\right.
\end{eqnarray}
with
\begin{equation} \label{eq:2TEigenfunctions_12C}
v_1(\tau+T) := -v_1(\tau) \mbox{  and  }
v_2(\tau+T) := -v_2(\tau) \mbox{ for } \tau \in [0,T].
\end{equation}
The functions $v_1$ and $v_2$ exist due to
Lemma~5 of \cite{Io:88}.
The functions $h_{ij}$ of \eqref{eq:H_12C} can be
found by solving appropriate BVPs, assuming that \eqref{eq:P.1}
restricted to $W^c(\Gamma)$ has normal form
\eqref{eq:NF-12C}. As in the generalized period-doubling case, we first deduce a property for these functions $h_{ij}$. More general than in the GPD case, we here have that $u(\tau,\xi_1,\xi_2) = u(\tau+T,-\xi_1,-\xi_2)$. This implies that 
$$\sum_{i,j} \frac{1}{i!j!} h_{ij}(\tau) \xi_1^i\xi_2^j = \sum_{i,j} \frac{1}{i!j!} h_{ij}(\tau+T) (-1)^{i+j} \xi_1^i\xi_2^j,$$
and thus
$$h_{ij}(\tau) = (-1)^{i+j}h_{ij}(\tau+T),$$
from which follows that $h_{ij}(\tau+T) = h_{ij}(\tau)$ for $i+j$ even and $h_{ij}(\tau+T) = -h_{ij}(\tau)$ for $i+j$ odd, for $\tau \in [0,T]$. Taking these periodicity properties into account, we can reduce our observations to the discussion of the interval $[0,T]$ instead of $[0,2T]$.

The coefficients $\alpha$, $a$ and $b$ arise
from the solvability conditions for the BVPs as integrals of
scalar products over the interval $[0,T]$. Specifically, those
scalar products involve among other things the quadratic and cubic terms of
\eqref{eq:MULT} near the periodic solution $u_0$, the
eigenfunction $v_1$. The adjoint eigenfunction $\varphi^*$ associated to the trivial
multiplier is the $T$-periodic solution of \eqref{eq:AdjEigenFunc2T-2_GPD}. The adjoint
eigenfunction $v_1^*$ is the unique solution of the problem
\begin{eqnarray}
&&\left\{\begin{array}{rcl}
 \dot{v}_1^*(\tau)+A^{\rm T}(\tau)v_1^* & = & 0,\ \tau \in [0,T], \\
 v_1^*(T)+v_1^*(0) & = & 0 \\
 \int_{0}^{T} {\langle v_1^*,v_2\rangle d\tau} -1 & = & 0.
\end{array} \right.
\label{eq:AdjEigenFunc_12C}
\end{eqnarray}
Note that we can indeed require this normalization since
$v_2$ is the last generalized eigenfunction of the original problem and therefore not orthogonal to all the eigenfunctions
of the adjoint problem. We further define the generalized adjoint eigenfunction $v_2^*$ as the
unique solution of
\begin{eqnarray}
&&\left\{\begin{array}{rcl}
 \dot{v}_2^*(\tau)+A^{\rm T}(\tau)v_2^* - v_1^* & = & 0,\ \tau \in [0,T], \\
 v_2^*(T)+v_2^*(0) & = & 0 \\
 \int_{0}^{T} {\langle v_2,v_2^*\rangle d\tau} & = & 0,
\end{array} \right.
\label{eq:GenAdjEigenFunc_12C}
\end{eqnarray}
since, as above, $v_1^*$ is not orthogonal to $v_2$. Moreover, we have
\begin{multline*}
 \int_{0}^{T} {\langle v_2^*,v_1 \rangle d\tau} = - \int_{0}^{T} {\langle v_2^*,\left(\dd{}{\tau}-A(\tau)\right) v_2 \rangle
 d\tau}\\
 = \int_{0}^{T} {\langle \left(\dd{}{\tau}+A^T(\tau)\right)v_2^*,v_2 \rangle d\tau} =\int_{0}^{T} {\langle v_1^*,v_2\rangle d\tau} = 1.
\end{multline*}

Note that
\begin{eqnarray}
 && \int_{0}^{T} {\langle v_2,v_1\rangle d\tau} = \int_{0}^{T} {\langle v_1^*,v_1\rangle d\tau} = \int_{0}^{T} {\langle v_2^*,v_2\rangle d\tau} = 0. \label{eq:Ortho2_12C}
\end{eqnarray}

To derive the normal form coefficients, we proceed as in the previous sections, namely, we substitute \eqref{eq:CM_12C} into
\eqref{eq:P.1}, and use \eqref{eq:MULT} as well as
\eqref{eq:NF-12C} and \eqref{eq:H_12C}.

By collecting the $\xi^0$-terms we get the identity
$$
\dot{u}_0=F(u_0).
$$

The linear terms provide two other identities, namely
$$
\dot{v}_1=A(\tau) v_1, \qquad v_1+\dot{v}_2 = A(\tau) v_2,
$$
in correspondance with \eqref{eq:EigenFunc_12C} and \eqref{eq:EigenFunc_12C_2}.

Collecting the $\xi_2^2$-terms gives us an equation for $h_{02}$
\begin{equation*} 
\dot{h}_{02}-A(\tau) h_{02} = B(\tau;v_2,v_2)-2 h_{11},
\end{equation*}
to be solved in the space of functions satisfying
$h_{02}(T)=h_{02}(0)$. In this space, the differential operator
$\frac{d}{d\tau}-A(\tau)$ is singular and its null-space is
spanned by $\dot u_0$. The Fredholm solvability
condition
\[
 \int_0^{T} \langle \varphi^*, B(\tau;v_2,v_2)-2 h_{11} \rangle\; d\tau=0
\]
gives us a normalization condition for function $h_{11}$, i.e.
\begin{equation*}
  \int_0^{T} \langle \varphi^*, h_{11} \rangle\;d\tau = \frac{1}{2} \int_0^{T} \langle \varphi^*, B(\tau;v_2,v_2) \rangle\;d\tau.
\end{equation*}

By collecting the $\xi_1\xi_2$-terms we obtain the differential equation for $h_{11}$
\begin{equation*}
\dot{h}_{11}- A(\tau) h_{11}=B(\tau;v_1,v_2)-h_{20},
\end{equation*}
which must be solved in the space of functions
satisfying $h_{11}(T)=h_{11}(0)$. The orthogonality condition
\[
 \int_0^{T} \langle \varphi^*, B(\tau;v_1,v_2)-h_{20} \rangle\; d\tau=0
\]
gives us a normalization condition for $h_{20}$, i.e.
\begin{equation}\label{eq:norm_h20_12C}
  \int_0^{T} \langle \varphi^*, h_{20} \rangle\;d\tau = \int_0^{T} \langle \varphi^*, B(\tau;v_1,v_2) \rangle\;d\tau.
\end{equation}

By collecting the $\xi_1^2$-terms we find an equation for $h_{20}$
\begin{equation} \label{eq:xi1^2}
\dot{h}_{20}-A(\tau) h_{20} = B(\tau;v_1,v_1)-2\alpha  \dot{u}_0,
\end{equation}
to be solved in the space of functions satisfying
$h_{20}(T)=h_{20}(0)$. In this space, the differential operator
$\frac{d}{d\tau}-A(\tau)$ is singular and its null-space is
spanned by $\dot u_0$. The Fredholm solvability condition
\[
 \int_0^{T} \langle \varphi^*, B(\tau;v_1,v_1) - 2 \alpha \dot u_0\rangle\; d\tau=0
\]
leads to the expression
\begin{equation}\label{eq:alfa_12C}
 \alpha = \frac{1}{2}\int_0^{T} \langle \varphi^*, B(\tau;v_1,v_1) \rangle\;d\tau,
\end{equation}
where $v_1$ is defined in \eqref{eq:EigenFunc_12C}.

With $\alpha$ defined in this way we have to find a normalization
condition which makes the  solution of \eqref{eq:xi1^2} unique. Indeed, if $h_{20}$ is a solution of \eqref{eq:xi1^2} with $h_{20}(T)=h_{20}(0)$,
also $\tilde h_{20}=h_{20}+\varepsilon_1 \dot u_0$ is a solution, since $\dot u_0$ spans the kernel of
the operator $\frac{d}{dt}-A(\tau)$ in the space of $T$-periodic
functions. The projection along the space generated by $\dot
u_0$ is fixed by solvability condition
\eqref{eq:norm_h20_12C}. So $h_{20}$ can be found as the unique solution of the BVP
\begin{equation} \label{eq:h20_12C}
 \left\{\begin{array}{rcl}
  \dot{h}_{20}-A(\tau) h_{20} - B(\tau;v_1,v_1)+2\alpha  F(u_0) &=& 0,\ \tau \in [0,T],\\
  h_{20}(T)-h_{20}(0) & = & 0, \\
  \int_0^{T} \langle \varphi^*, h_{20} \rangle\;d\tau &=& \int_0^{T} \langle \varphi^*, B(\tau;v_1,v_2) \rangle\;d\tau.
 \end{array}\right.
\end{equation}

In the line of the previous observations, we can define $h_{11}$ as the unique solution of the BVP
\begin{equation} \label{eq:h11_12C}
 \left\{\begin{array}{rcl}
  \dot{h}_{11}-A(\tau) h_{11} - B(\tau;v_1,v_2) + h_{20} &=& 0,\ \tau \in [0,T]\\
  h_{11}(T)-h_{11}(0) & = & 0, \\
  \int_0^{T} \langle \varphi^*, h_{11} \rangle\;d\tau &=&\frac{1}{2} \int_0^{T} \langle \varphi^*, B(\tau;v_2,v_2) \rangle\;d\tau,
 \end{array}\right.
\end{equation}
with $h_{20}$ defined in \eqref{eq:h20_12C}.

By collecting the $\xi_1^3$-terms we get an equation for $h_{30}$
\begin{equation} \label{eq:xi1^3}
 \dot{h}_{30}- A(\tau) h_{30} = C(\tau;v_1,v_1,v_1)+3 B(\tau;v_1,h_{20}) -6 a v_2 - 6 \alpha  \dot{v}_1,
\end{equation}
which again must be solved in the space of functions
satisfying $h_{30}(T)=-h_{30}(0)$. Taking the integral condition of \eqref{eq:AdjEigenFunc_12C} into account, we obtain
\begin{equation} \label{eq:a_12C}
 a=\frac{1}{6}\int_0^{T} \langle v_1^*, C(\tau;v_1,v_1,v_1)+3 B(\tau;v_1,h_{20}) - 6 \alpha  A(\tau) v_1 \rangle\;d\tau,
\end{equation}
where $\alpha$ is defined by \eqref{eq:alfa_12C}, $h_{20}$ is the solution of \eqref{eq:h20_12C} and $v_1$ and $v_1^*$ are
defined in \eqref{eq:EigenFunc_12C} and \eqref{eq:EigenFunc_12C_2}, respectively. As remarked before, it is important to note that if $h_{30}$ is a solution of
\eqref{eq:xi1^3} with $h_{30}(T)=h_{30}(0)$, also $\tilde h_{30} = h_{30} + \epsilon^I_{30}
v_1$ is a solution, since $v_1$ spans the nullspace of
the operator $\frac{d}{dt}-A(\tau)$. 

Collecting the $\xi_1^2\xi_2$-terms we get the equation for
$h_{21}$
\begin{equation} \label{eq:r12_h21}
 \dot{h}_{21}-A(\tau) h_{21}=-h_{30}-2 b v_2-2 \alpha\dot{v}_2-2\alpha v_1+C(\tau;v_1,v_1,v_2)+ B(\tau;h_{20},v_2)+ 2 B(\tau;h_{11},v_1),
\end{equation}
to be solved in the space of functions satisfying $h_{21}(T)=-h_{21}(0)$. The solvability of this equation implies
\[
 \int_0^{T} \langle v_1^*,-h_{30}-2b v_2-2 \alpha\dot{v}_2-2\alpha v_1+C(\tau;v_1,v_1,v_2)+ B(\tau;h_{20},v_2)+ 2 B(\tau;h_{11},v_1)\rangle d \tau = 0.
\]
Notice that the $b \xi_1^2\xi_2$-term in normal form
\eqref{eq:NF-12C} is resonant: in fact we cannot use the freedom
on $h_{30}$ to make the normal form parameter $b$ zero since 
\[
 \int_0^{T} \langle v_1^*, \tilde h_{30} \rangle d \tau =  \int_0^{T} \langle v_1^*, h_{30} + \varepsilon^{I}_{30} v_1 \rangle d \tau
 = \int_0^{T} \langle v_1^*, h_{30}\rangle d \tau.
\]

Using the normalization from \eqref{eq:AdjEigenFunc_12C} and (\ref{eq:Ortho2_12C}) gives us the following expression for $b$
\[
 b=\frac{1}{2}\int_0^{T} \langle v_1^*,-2\alpha A(\tau) v_2+C(\tau;v_1,v_1,v_2)+ B(\tau;h_{20},v_2)+ 2B(\tau;h_{11},v_1)\rangle d \tau - \frac{1}{2}\int_0^{T} \langle v_1^*,h_{30} \rangle d \tau.
\]
There is no need to compute explicitly the cubic expansion of the center manifold since the last term of this sum can be rewritten as
\begin{multline*}
 \int_0^T\langle \left(\frac{d}{d\tau}+A^T(\tau)\right)v_2^*,h_{30}\rangle d \tau
 = \int_0^T\langle v_2^*,\left(-\frac{d}{d\tau}+A(\tau)\right)h_{30}\rangle d \tau\\
 = -\int_0^T\langle v_2^*, C(\tau;v_1,v_1,v_1)+3 B(\tau;v_1,h_{20}) -6 a v_2 - 6 \alpha  A v_1 \rangle d \tau\\
 = -\int_0^T\langle v_2^*, C(\tau;v_1,v_1,v_1)+3 B(\tau;v_1,h_{20}) - 6 \alpha  A v_1 \rangle d \tau,
\end{multline*}
obtaining
\begin{multline} \label{eq:b_12C}
 b=\frac{1}{2} \int_0^{T} \langle v_1^*,-2\alpha A(\tau) v_2+C(\tau;v_1,v_1,v_2)+ B(\tau;h_{20},v_2)+ 2B(\tau;h_{11},v_1)\rangle d \tau \\
 + \frac{1}{2}\int_0^T\langle v_2^*, C(\tau;v_1,v_1,v_1)+3 B(\tau;v_1,h_{20}) - 6 \alpha  A v_1 \rangle d \tau,
\end{multline}
where $h_{20}$ is defined in \eqref{eq:h20_12C} and $\alpha$
calculated in \eqref{eq:alfa_12C}. Notice that, since $h_{30}$ appears on the RHS
of equation \eqref{eq:r12_h21}, we have two degrees of freedom on $h_{21}$. In fact, if $h_{21}$ is a solution of \eqref{eq:r12_h21}, also
$\tilde h_{21} = h_{21} + \varepsilon_{21}^I v_1 +
\varepsilon_{30}^{I} v_2$ is a solution since
\[
 \dd{\tilde h_{21}}{t}-A(\tau) \tilde h_{21} =  \dd{h_{21}}{t}-A(\tau) h_{21} +  \varepsilon^{I}_{30} \left(\dd{v_2}{t}-A(\tau) v_2\right)
 = \dd{h_{21}}{t}-A(\tau) h_{21} - \varepsilon^{I}_{30} v_1.
\]

By collecting the $\xi_1 \xi_2^2$-terms we get the equation for
$h_{12}$
\[
 \dot{h}_{12}-A(\tau) h_{12}= C(\tau,v_1,v_2,v_2)+B(\tau,v_1,h_{02})+2 B(\tau,v_2,h_{11}) - 2 h_{21},
\]
to be solved in the space of functions satisfying $h_{12}(T)=-h_{12}(0)$. The Fredholm
solvability condition implies that
\[
 \int_0^{T} \langle v_1^*, C(\tau,v_1,v_2,v_2)+B(\tau,v_1,h_{02})+2 B(\tau,v_2,h_{11})- 2 h_{21}\rangle d \tau = 0.
\]
As mentioned before, $h_{21}$ has a component in the direction of $v_2$, which is not orthogonal to the adjoint eigenfunction $v_1^*$, so it
is possible to impose
\[
 \int_0^{T} \langle v_1^*, h_{21}\rangle d \tau =  \frac{1}{2}\int_0^{T} \langle v_1^*, C(\tau,v_1,v_2,v_2)+B(\tau,v_1,h_{02})+2 B(\tau,v_2,h_{11})\rangle d \tau.
\]
This condition defines $\varepsilon_{30}^{I}$ uniquely; the
freedom of $\varepsilon_{21}^{I}$ gives us as usual another freedom
on $h_{12}$ in the direction of $v_2$.

Finally, collecting the $\xi_2^3$-terms gives
\[
 \dot{h}_{03}-A(\tau) h_{03}= C(\tau, v_2,v_2,v_2) + 3 B(v_2, h_{02}) - 3 h_{12},
\]
to be solved in the space of functions satisfying $h_{03}(T)=-h_{03}(0)$. The Fredholm
solvability condition is
\[
 \int_0^{T} \langle v_1^*, C(\tau, v_2,v_2,v_2) + 3 B(v_2, h_{02}) - 3 h_{12} \rangle d \tau = 0,
\]
which can be satisfied imposing
\[
 \int_0^{T} \langle v_1^*, h_{12} \rangle d \tau = \frac{1}{3} \int_0^{T} \langle v_1^*, C(\tau, v_2,v_2,v_2) + 3 B(v_2, h_{02}) \rangle d \tau.
\]
This last condition determines the value of $\varepsilon_{21}^{I}$ and thus the third order center
manifold expansion is uniquely determined. However, since this third order expansion
of the center manifold is not needed for the computation of the critical coefficients, we
don't write down those conditions.

\subsection{Strong resonance 1:3 bifurcation} \label{Section:Res1:3}
The three-dimensional critical center manifold $W^c(\Gamma)$ at
the {\tt R3} bifurcation can be parametrized locally by
$(\tau,\xi)$ as
\begin{equation}
 u=u_0+\xi v + \bar{\xi} \bar{v}(\tau) + H(\tau,\xi,\bar{\xi}),\ \ \tau \in [0,3T],\ \xi \in {\mathbb C}, \label{eq:CM_R3}
\end{equation}
where the real function $H$ satisfies
$H(3T,\xi,\bar{\xi})=H(0,\xi,\bar{\xi})$ and has the Taylor
expansion
\begin{eqnarray}
 H(\tau,\xi,\bar{\xi}) & = & \frac{1}{2}h_{20} \xi^2
 + h_{11}\xi \bar{\xi}
 +\frac{1}{2}h_{02} \bar{\xi}^2
 + \frac{1}{6} h_{30}\xi^3 +
 \frac{1}{2}h_{21}\xi^2 \bar{\xi} \nonumber \\
 &+& \frac{1}{2}h_{12}\xi \bar{\xi}^2
 + \frac{1}{6}h_{03}\bar{\xi}^3
 +  O(|\xi|^4), \label{eq:H_R3}
\end{eqnarray}
with $h_{ij}(3T)=h_{ij}(0)$ and $h_{ij}=\bar{h}_{ji}$ so that
$h_{ii}$ is real. The eigenfunction $v$ is defined
as the unique solution of the BVP
\begin{equation}
\left\{\begin{array}{rcl}
\dot{v}(\tau)-A(\tau)v & = & 0,\ \tau \in [0,T], \\
v(T)- e^{i \frac{2\pi}{3}} v(0) & = & 0,\\
\int_{0}^{T} {\langle v,v\rangle d\tau} - 1 & = & 0,
\end{array}
\right. \label{eq:EigenFunc_R3}
\end{equation}
and extended on the interval $[0,3T]$ using the equivariance
property of the normal form, i.e.
\begin{eqnarray*}
v(\tau+T):= e^{i \frac{2\pi}{3}} v(\tau) \mbox{ and } v(\tau+2T):= e^{i \frac{4\pi}{3}} v(\tau) \mbox{ for } \tau \in [0,T].
\end{eqnarray*}
%
The definition of the conjugate eigenfunction $\bar v$ follows immediately. These functions exist due to Lemma 2 of \cite{Io:88}.

As usual the functions $h_{ij}$ can be found by solving
appropriate BVPs, assuming that \eqref{eq:P.1} restricted to
$W^c(\Gamma)$ has the periodic normal form \eqref{eq:NF-13C}. Also here we can deduce a property for the functions $h_{ij}$. The definition of $v(\tau)$ in $[0,3T]$ states that $u(\tau,\xi,\bar{\xi}) = u(\tau+T,e^{-i2\pi/3}\xi,e^{i2\pi/3}\bar{\xi})$. Therefore,
$$\sum_{k,l} \frac{1}{k!l!} h_{kl}(\tau) \xi^k\bar{\xi}^l = \sum_{k,l} \frac{1}{k!l!} h_{kl}(\tau+T) (e^{-i2\pi/3})^k \xi^k (e^{i2\pi/3})^l\bar{\xi}^l,$$
and thus
$$h_{kl}(\tau)= h_{kl}(\tau+T) (e^{-i2\pi/3})^k (e^{i2\pi/3})^l,$$
for $\tau \in [0,T]$. This for example implies that $h_{kk}$ is T-periodic. These periodicity properties allow us to just concentrate on the interval $[0,T]$

The adjoint eigenfunction $\varphi^*$ corresponding to the trivial multiplier is the unique $T$-periodic solution of BVP \eqref{eq:AdjEigenFunc2T-2_GPD}. The adjoint eigenfunction $v^*$ satisfies
\begin{equation}
\left\{\begin{array}{rcl}
 \dot{v}^*(\tau)+A^{\rm T}(\tau)v^* & = & 0,\ \tau \in [0,T], \\
 v^*(T)-e^{i \frac{2\pi}{3}} v^*(0) & = & 0,\\
 \int_{0}^{T} {\langle v^*, v\rangle d\tau} - 1 & = & 0.
\end{array}
\right. \label{eq:AdjEigenFunc_R3}
\end{equation}
Similarly, we obtain $\bar v^*$. 

We now write down the homological equation and
compare term by term. The constant and linear terms give us
as usual
\[
 \dot u_0=F(u_0), \qquad \dot v - A(\tau) v = 0, \qquad \dot {\bar v} - A(\tau) \bar v =0.
\]

From the $\xi^2$- or $\bar \xi^2$-terms we obtain the following equation (or
its complex conjugate)
\[
\dot{h}_{20} - A(\tau) h_{20}= B(\tau;v,v)-2 \bar{b} \bar{v},
\]
to be solved in the space of functions satisfying $h_{20}(
T)= e^{i \frac{4\pi}{3}}h_{20}(0)$. In this space the operator
$\frac{d}{d\tau}-A(\tau)$ has a range space with codimension one which is orthogonal to $\bar v^*$. So one Fredholm
solvability condition is involved, namely
\[
 \int_0^{T} \langle \bar v^*, B(\tau;v,v)-2 \bar{b} \bar{v} \rangle d \tau = 0,
\]
which makes it possible to obtain the value of the coefficient $b$. In fact,
\begin{equation}\label{eq:b-R3}
 b=\frac{1}{2} \int_0^T \langle  v^*, B(\tau;\bar v,\bar v) \rangle d\tau.
\end{equation}
Using this value for $b$ we can find $h_{20}$ as the unique
solution of the BVP
\begin{equation}
\left\{\begin{array}{rcl}
 \dot{h}_{20} - A(\tau) h_{20} - B(\tau; v,v) +  2 \bar{b} \bar{v} & = & 0,\ \tau \in [0,T], \\
 h_{20}(T)- e^{i\frac{4 \pi}{3}}h_{20}(0) & = & 0,\\
 \int_{0}^{T} {\langle \bar v^*, h_{20}\rangle d\tau} & = & 0.
\end{array}
\right. \label{eq:h20-R3}
\end{equation}

By collecting the $\xi\bar\xi$-terms we obtain an equation for
$h_{11}$
\[
\dot{h}_{11}-A(\tau) h_{11}=B(\tau;v,\bar{v})-\alpha _1 \dot{u}_0,
\]
to be solved in the space of functions satisfying
$h_{11}(T)=h_{11}(0)$. The Fredholm solvability condition with $\varphi^*$ gives us the value of
$\alpha_1$
\begin{equation}\label{eq:a1-R3}
 \alpha_1=\int_0^T \langle \varphi^*, B(\tau;v,\bar v) \rangle d\tau.
\end{equation}
With $\alpha_1$ defined in this way, let $h_{11}$ be the unique
solution of the BVP
\begin{equation}
\left\{\begin{array}{rcl}
 \dot{h}_{11}-A(\tau) h_{11}-B(\tau;v,\bar{v})+\alpha _1 \dot{u}_0 & = & 0,\ \tau \in [0,T], \\
 h_{11}(T)-h_{11}(0) & = & 0,\\
 \int_{0}^{T} {\langle \varphi^*, h_{11}\rangle d\tau} & = & 0.
\end{array}
\right. \label{eq:h11-R3}
\end{equation}

Finally, collecting the $\xi^2\bar\xi$-terms gives an equation
for $h_{21}$
\[
 \dot{h}_{21}-A(\tau) h_{21} = C(\tau;v,v,\bar{v})+2 B(\tau;v, h_{11})+ B(\tau; \bar{v}, h_{20}) - 2 c v - 2 \bar{b} h_{02}- 2 \alpha_1\dot{v},
\]
to be solved in the space of the functions satisfying
$h_{21}(T)=e^{i \frac{2\pi}{3}}h_{21}(0)$. Therefore, there must hold that
\[
\int_0^{T} \langle v^*, C(\tau;v,v,\bar{v})+2 B(\tau;v, h_{11})+
B(\tau; \bar{v}, h_{20}) - 2 c v - 2 \bar{b} h_{02}- 2 \alpha_1\dot{v}
\rangle d \tau = 0,
\]
such that parameter $c$ of \eqref{eq:NF-13C} is determined by 
\[
 c = \frac{1}{2}\int_0^T \langle v^*, C(\tau;v,v,\bar{v})+2 B(\tau;v, h_{11})+ B(\tau; \bar{v}, h_{20}) - 2 \alpha_1A v \rangle d \tau,
\]
where $\alpha_1$ and $b$ are defined by
\eqref{eq:a1-R3} and \eqref{eq:b-R3}, respectively, and $v$, $h_{11}$ and $h_{20}$
are the unique solutions of the BVPs \eqref{eq:EigenFunc_R3},
\eqref{eq:h11-R3} and \eqref{eq:h20-R3}.\\

Since we have to check the solvability of all the equations up to the maximal order of the normal form, we also collect the $\xi^3$-terms and obtain
\[
 \dot{h}_{30}-A(\tau) h_{30} = C(\tau;v,v,v)+3 B(\tau;v, h_{20}) - 6 \bar{b} h_{11}- 6 \alpha_2\dot{u}_0,
\]
to be solved in the space of the functions satisfying
$h_{30}(T)=h_{30}(0)$. Therefore, there must hold that
\[
\int_0^{T} \langle \varphi^*,C(\tau;v,v,v)+3 B(\tau;v, h_{20}) - 6 \bar{b} h_{11}- 6 \alpha_2\dot{u}_0
\rangle d \tau = 0,
\]
which determines the value of $\alpha_2$, namely
\[\alpha_2=\int_0^{T} \langle \varphi^*,C(\tau;v,v,v)+3 B(\tau;v, h_{20})\rangle d \tau.
\]

Remark that as in the Chenciner case $v$ is not uniquely determined. Indeed, when $v$ is a solution of (\ref{eq:EigenFunc_R3}) and $\gamma \in \mathbb{C}$ with $\gamma^{\rm H}\gamma = 1$, then $\gamma v$ is also a solution. Then the adjoint function is given by $\gamma v^*$, and $b$ and $h_{20}$ are replaced by $\bar{\gamma}^3b$ and $\gamma^2h_{20}$, respectively. The normal form coefficient $c$ stays the same. However, the normal form coefficient $b$ is multiplied with $\bar{\gamma}^3$. This doesn't affect the bifurcation analysis since this normal form coefficient only has to be different from zero, and obviously $\gamma\neq 0$. Moreover, the analysis around the bifurcation point is independent from the sign of $b$.
\medskip

%

\subsection{Strong resonance 1:4 bifurcation} \label{Section:Res1:4}
The three-dimensional critical center manifold $W^c(\Gamma)$ at
the {\tt R4} bifurcation can be parametrized locally by
$(\tau,\xi)$ as
\begin{equation}
 u =u_0+ \xi v + \bar{\xi} \bar{v}(\tau) + H(\tau,\xi,\bar{\xi}),\ \ \tau
\in [0,4T],\ \xi \in {\mathbb C}, \label{eq:CM_R4}
\end{equation}
where the real function $H$ satisfies
$H(4T,\xi,\bar{\xi})=H(0,\xi,\bar{\xi})$ and has the Taylor
expansion
\begin{eqnarray}
H(\tau,\xi,\bar{\xi}) & = & \frac{1}{2}h_{20} \xi^2 + h_{11} \xi
\bar{\xi} +\frac{1}{2}h_{02} \bar{\xi}^2
+  \frac{1}{6}  h_{30}\xi^3 +  \frac{1}{2}h_{21}\xi^2 \bar{\xi} \nonumber \\
&+& \frac{1}{2}h_{12}\xi \bar{\xi}^2 +
\frac{1}{6}h_{03}\bar{\xi}^3
  +  O(|\xi|^4), \label{eq:H_R4}
\end{eqnarray}
with $h_{ij}(4T)=h_{ij}(0)$ and $h_{ij}=\bar{h}_{ji}$ so that
$h_{ii}$ is real, while $v$ is defined by
\begin{equation}
\left\{\begin{array}{rcl}
\dot{v}-A(\tau)v  & = & 0,\ \tau \in [0,T], \\
v(T)-e^{i \frac{\pi}{2}}v(0) & = & 0,\\
\int_{0}^{T} {\langle v,v\rangle d\tau} - 1 & = & 0,
\end{array}
\right. \label{eq:EigenFunc_R4}
\end{equation}
extended on $[0,4T]$ using the equivariance property of the
normal form, i.e.
\begin{eqnarray*}
 v(\tau+T)&:=&e^{i \frac{\pi}{2}} v(\tau) = i v(\tau), \\
 v(\tau+2T)&:=&e^{i \pi} v(\tau) = - v(\tau), \\
 v(\tau+3T)&:=&e^{i \frac{3 \pi}{2}} v(\tau) = -i v(\tau),
\end{eqnarray*}
for $\tau \in [0,T]$.

The definition of the conjugate $\bar v$ follows from this. These functions exist due to Lemma 2 of \cite{Io:88}. As usual the
functions $h_{ij}$ can be found by solving appropriate BVPs,
assuming that \eqref{eq:P.1} restricted to $W^c(\Gamma)$ has the
periodic normal form \eqref{eq:NF-14C}. Similar to the R1:3 case, there holds that 
$$h_{kl}(\tau)= h_{kl}(\tau+T) (e^{-i\pi/2})^k (e^{i\pi/2})^l,$$
for $\tau \in [0,T]$.

The adjoint eigenfunction $\varphi^*$ is defined
by the $T$-periodic solution of \eqref{eq:AdjEigenFunc2T-2_GPD} and $v^*$ satisfies
\begin{equation}
\left\{\begin{array}{rcl}
 \dot{v}^*(\tau)+A^{\rm T}(\tau)v^* & = & 0,\ \tau \in [0,T], \\
 v^*(T)-e^{i \frac{\pi}{2}}v^*(0) & = & 0,\\
 \int_{0}^{T} {\langle v^*,v\rangle d\tau} - 1 & = &
 0.
\end{array}
\right. \label{eq:AdjEigenFunc_R4}
\end{equation}
Similarly, we obtain $\bar v^*$. 

The constant and the linear terms give the identities
\[
 \dot u_0=F(u_0), \qquad \dot v - A(\tau) = 0, \qquad \dot {\bar v} - A(\tau) \bar v  = 0.
\]
From the $\xi^2$- or $\bar \xi^2$-terms the following equation (or
its complex conjugate) is obtained
\begin{eqnarray}\label{R4h20}
\dot h_{20}-A(\tau)h_{20}=B(\tau;v,v).
\end{eqnarray}
Notice that this equation is non-singular in the space of
functions satisfying $h_{20}(T)= -h_{20}(0)$. So $h_{20}$ is given as the unique
solution of the BVP
\begin{equation}
\left\{\begin{array}{rcl}
 \dot h_{20}-A(\tau) h_{20} - B(\tau;v,v) & = & 0,\ \tau \in [0,T], \\
 h_{20}(T)+h_{20}(0) & = & 0.
\end{array}
\right. \label{eq:h20-R4}
\end{equation}

By collecting the $\xi\bar\xi$-terms we obtain an equation for
$h_{11}$
\[
 \dot{h}_{11}-A(\tau) h_{11}=B(\tau;v,\bar{v})-\alpha _1 \dot{u}_0,
\]
to be solved in the space of functions satisfying
$h_{11}(T)=h_{11}(0)$. The Fredholm
solvability condition
\[
 \int_0^{T} \langle \varphi^*,B(\tau;v,\bar{v})-\alpha _1 \dot{u}_0\rangle d \tau = 0
\]
gives us the possibility to obtain the value of $\alpha_1$, namely given by
\eqref{eq:a1-R3}. With this value of $\alpha_1$, $h_{11}$ is the
unique solution of BVP \eqref{eq:h11-R3}.

The $\xi\bar\xi^2$-terms give an equation for $h_{12}$
\[
\dot h_{12}-A(\tau) h_{12}=C(\tau;v,\bar{v},\bar
v)+B(\tau;v,h_{02})+2B(\tau;\bar{v},h_{11})-2 \bar{c} \bar{v} -2
\alpha _1 \dot{\bar{v}},
\]
to be solved in the space of functions satisfying
$h_{12}(T)=-i h_{12}(0)$. The Fredholm
solvability condition leads us to the value
of $\bar c$, namely
\begin{equation}\label{eq:c_R4}
 \bar c = \frac{1}{2}\int_0^T \langle \bar v^*, C(\tau;v,\bar{v},\bar v)+B(\tau;v,h_{02})+2B(\tau;\bar{v},h_{11}) - 2 \alpha _1 A(\tau) \bar{v} \rangle d\tau,
\end{equation} 
where $\alpha_1$ is defined in \eqref{eq:a1-R3}, and $v$, $h_{11}$ and
$h_{02}$ are the unique solutions of the BVPs
\eqref{eq:EigenFunc_R4}, \eqref{eq:h11-R3} and the complex conjugate of \eqref{eq:h20-R4}. Taking the complex conjugate gives us the critical coefficient $c$.

Finally, by collecting the $\bar\xi^3$-terms we obtain an equation
for $h_{03}$
\[
 \dot h_{03} -A(\tau) h_{03} =C(\tau;\bar{v},\bar{v},\bar{v})+3 B(\tau;\bar{v},h_{02})-6 d v,
\]
to be solved in the space of the functions satisfying
$h_{03}(T)=ih_{03}(0)$. The non-trivial Fredholm solvability
condition
\[
\int_0^{T} \langle v^*, C(\tau;\bar{v},\bar{v},\bar{v})+3
B(\tau;\bar{v},h_{02})-6 d v \rangle d \tau = 0
\]
gives us the value of the critical coefficient $d$ of
\eqref{eq:NF-14C}, namely
\begin{equation}
 d = \frac{1}{6}\int_0^T \langle v^*, C(\tau;\bar{v},\bar{v},\bar{v})+3 B(\tau;\bar{v},h_{02}) \rangle
 d\tau.
\end{equation} \label{eq:d_R4}
So we finally obtain the value of 
\[
a=\frac{c}{|d|}
\]
which makes it possible to understand which bifurcation scenario of the {\tt R4} resonance we have.\\

Also in this case $v$ is not uniquely determined, since for every $\gamma \in \mathbb{C}$ with $\gamma^{\rm H}\gamma = 1$,  $\gamma v$ is a solution. Then the adjoint eigenfunction is given by $\gamma v^*$, and $h_{20}$ is replaced by $\gamma^2h_{20}$. The normal form coefficient $c$ stays the same, but instead of $d$ we get $\bar{\gamma}^4d$. However, this again doesn't influence the bifurcation analysis since the study is determined by the above defined $a$ for which we need only $|d|$.
\medskip


\subsection{Fold-Flip bifurcation} \label{Section:FoldFlip}
The three-dimensional critical center manifold $W^c(\Gamma)$ at
the {\tt LPPD} bifurcation can be parametrized locally by
$(\tau,\xi)$ as 
\begin{equation}
u=u_0+\xi_1 v_1 + \xi_2 v_2 + H(\tau,\xi),\ \ \tau \in [0,2T],\
\xi=(\xi_1,\xi_2) \in {\mathbb R^2}, \label{eq:CM_FF}
\end{equation}
where $H$ satisfies $H(2T,\xi)=H(0,\xi)$ and has the Taylor
expansion 
\begin{multline}
 H(\tau,\xi)=\frac{1}{2}h_{20}\xi_1^2 + h_{11}\xi_1 \xi_2+ \frac{1}{2}h_{02}\xi_2^2 
 + \\
    \frac{1}{6}h_{30}\xi_1^3+\frac{1}{2}h_{21}\xi_1^2\xi_2+\frac{1}{2}h_{12}\xi_1\xi_2^2+\frac{1}{6}h_{03}\xi_2^3
    + O(\|\xi\|^4) , \label{eq:H_FF}
\end{multline}
while the eigenfunctions $v_1$ and $v_2$ are given by 
\begin{eqnarray} \label{eq:EigenFunc_FF}
&&\left\{\begin{array}{rcl}
\dot{v}_1-A(\tau)v_1 - F(u_0) & = & 0,\ \tau \in [0,T], \\
v_1(T)-v_1(0) & = & 0,\\
\int_{0}^{T} {\langle v_1,F(u_0)\rangle d\tau} & = & 0,\\
\end{array}
\right. 
\end{eqnarray}
and
\begin{eqnarray}\label{eq:EigenFunc_FF_2}
&&\left\{\begin{array}{rcl}
\dot{v}_2-A(\tau)v_2 & = & 0,\ \tau \in [0,T], \\
v_2(T) + v_2(0) & = & 0,\\
\int_{0}^{T} {\langle v_2,v_2\rangle d\tau} - 1  & = & 0\\
\end{array}
\right.
\end{eqnarray}
with
\begin{equation} \label{eq:2TEigenfunctions_FF}
v_1(\tau+T):=v_1(\tau) \mbox{ and } v_2(\tau +T):= -v_2(\tau) \mbox{ for } \tau \in [0,T].
\end{equation}

The functions $v_1$ and $v_2$ exist because of Lemma~2 and Lemma~5
of \cite{Io:88}. 
The functions $h_{ij}$ can be found by solving appropriate
BVPs, assuming that \eqref{eq:P.1} restricted to $W^c(\Gamma)$ has
the periodic normal form \eqref{eq:NF-FF}. Moreover, similar as before, $u(\tau,\xi_1,\xi_2) = u(\tau+T,\xi_1,-\xi_2)$ such that
$$h_{ij}(\tau)= (-1)^jh_{ij}(\tau+T),$$
for $\tau \in [0,T]$. As before, we will reduce all computations to the interval $[0,T]$.

The coefficients of the normal
form  arise from the solvability conditions for the BVPs as
integrals of scalar products over the interval $[0,T]$.
Specifically, those scalar products involve among other things the quadratic and
cubic terms of \eqref{eq:MULT} near the periodic solution $u_0$,
the generalized eigenfunction $v_1$ and eigenfunction $v_2$, and the
adjoint eigenfunctions $\varphi^*$, $v_1^*$ and $v_2^*$ as solution
of the problems
\begin{eqnarray}
&&\left\{\begin{array}{rcl}
 \dot{\varphi}^*+A^{\rm T}(\tau)\varphi^* & = & 0,\ \tau \in [0,T], \\
 \varphi^*(T)-\varphi^*(0) & = & 0, \label{eq:AdjEigenFunc_FF}\\
 \int_{0}^{T} {\langle \varphi^*,v_1 \rangle d\tau} -1 & = & 0,
 \end{array} \right.
 \end{eqnarray}
 \begin{eqnarray}
 &&\left\{\begin{array}{rcl}
 \dot{v}_1^*+A^{\rm T}(\tau) v_1^*+\varphi^* & = & 0,\ \tau \in [0,T], \\
 v_1^*(T)-v_1^*(0) & = & 0, \label{eq:AdjEigenFunc_FF_2}\\
 \int_{0}^{T} {\langle v_1^*,v_1 \rangle d\tau} & = & 0,
 \end{array} \right. 
 \end{eqnarray}
and
 \begin{eqnarray}
 &&\left\{\begin{array}{rcl}
 \dot{v}_2^*+A^{\rm T}(\tau) v_2^* & = & 0,\ \tau \in [0,T], \\
 v_2^*(T)+v_2^*(0) & = & 0, \label{eq:AdjEigenFunc_FF_3}\\
 \int_{0}^{T} {\langle v_2^*,v_2\rangle d\tau} -1 & = & 0.
 \end{array} \right.
\end{eqnarray}
Note that the integral conditions are possible due to the spectral
assumptions at the {\tt LPPD} point. The following orthogonality conditions hold automatically
\begin{eqnarray}
 \int_{0}^{T} {\langle \varphi^*,F(u_0)\rangle d\tau}
 =\int_{0}^{T} {\langle \varphi^*,v_2\rangle d\tau}
 =\int_{0}^{T} {\langle v_1^*,v_2\rangle d\tau}&&=\\
 \nonumber
 =\int_{0}^{T} {\langle v_2^*,v_1\rangle d\tau}
 =\int_{0}^{T} {\langle v_2^*,F(u_0)\rangle d\tau}&&= 0,
\label{eq:ortho_FF}
\end{eqnarray}
and since we have normalized the adjoint eigenfunction associated
to multiplier 1 with the last generalized eigenfunction, we
have for free that
\begin{multline}\label{eq:FF_normo}
 1 = \int_{0}^{T} {\langle \varphi^*,v_1 \rangle d\tau} = - \int_{0}^{T} {\langle \left( \dd{}{\tau} + A^T(\tau)\right)v_1^*,v_1 \rangle d\tau} \\
 = \int_{0}^{T} {\langle v_1^*,\left( \dd{}{\tau} - A(\tau)\right)v_1 \rangle d\tau}
 = \int_{0}^{T} {\langle v_1^*, F(u_0) \rangle d\tau}.
\end{multline}

As usual, to derive the normal form coefficients we write down the
homological equation and compare term by term. 

By collecting the constant and linear terms we get the identities
$$
 \dot{u}_0=F(u_0), \qquad \dot v_1=A(\tau) v_1 +F(u_0), \qquad \dot{v}_2=A(\tau) v_2,
$$
and the complex conjugate of the last equation.

By collecting the $\xi_1^2$-terms we find an equation for $h_{20}$
\begin{equation}\label{eq:h20_FF}
 \dot h_{20}-A(\tau) h_{20}=B(\tau;v_1,v_1)-2 a_{20} v_1-2 \alpha_{20} \dot u_0+2 \dot v_1,
\end{equation}
to be solved in the space of functions satisfying
$h_{20}(T)=h_{20}(0)$. In this space, the differential operator
$\frac{d}{d\tau}-A(\tau)$ is singular and its null-space is
spanned by $\dot u_0$. The Fredholm solvability
condition
\[
 \int_0^{T} \langle \varphi^*,B(\tau;v_1,v_1)-2 a_{20} v_1-2 \alpha_{20} \dot u_0+2 \dot v_1 \rangle\; d \tau =0
\]
gives us the possibility to calculate parameter $a_{20}$ of
our normal form, i.e.
\begin{equation} \label{eq:a20_FF}
 a_{20}=\frac{1}{2}\int_0^{T} \langle \varphi^*,B(\tau;v_1,v_1)+2 A(\tau) v_1 \rangle\; d \tau.
\end{equation}

With $a_{20}$ tuned in this way equation \eqref{eq:h20_FF} is
solvable, for any value of parameter $\alpha_{20}$. As in the Cusp of cycles case, we are free to choose parameter $\alpha_{20}$ as we want, and we take $\alpha_{20}=0$. This choice will not influence our final conclusion about the kind of situation we are in.

In order to make the solution of \eqref{eq:h20_FF} unique, we have to fix the projection on
the null-space of the operator, more specifically in the direction of $F(u_0)$. Therefore we impose the orthogonality with the adjoint generalized
eigenfunction $v_1^*$, and obtain $h_{20}$ as the unique solution
of the BVP
\begin{equation} \left\{\begin{array}{rcl}
 \dot h_{20}-A(\tau) h_{20}-B(\tau;v_1,v_1)+2 a_{20} v_1+2 \alpha_{20} F(u_0) - 2 A(\tau) v_1 - 2 F(u_0) & = & 0,\ \tau \in [0,T], \\
 h_{20}(T)-h_{20}(0) & = & 0, \\
 \int_{0}^{T} {\langle v_1^*,h_{20}\rangle d\tau} & = & 0.
\end{array}
\right. \label{eq:h20-FF}
\end{equation}

By collecting the $\xi_1 \xi_2$-terms we obtain a singular equation
for $h_{11}$
\[
\dot h_{11} - A(\tau) h_{11}=B(\tau;v_1, v_2)-b_{11} v_2+\dot v_2,
\]
to be solved in the space of the functions that satisfy
$h_{11}(T)=-h_{11}(0)$. The Fredholm solvability condition
\[
 \int_0^{T} \langle v_2^*,B(\tau;v_1, v_2)-b_{11} v_2+\dot v_2 \rangle\; d \tau = 0
\]
gives us the possibility, using (\ref{eq:EigenFunc_FF_2}) and (\ref{eq:AdjEigenFunc_FF_3}), to calculate
coefficient $b_{11}$ of our normal form
\begin{equation} \label{eq:b11_FF}
 b_{11}=\int_0^{T} \langle v_2^*,B(\tau;v_1, v_2)+A(\tau) v_2 \rangle\; d \tau.
\end{equation}
With $b_{11}$ defined in this way we can compute $h_{11}$ as the
unique solution of the BVP
\begin{equation}
\left\{\begin{array}{rcl}
 \dot h_{11}-A(\tau) h_{11}-B(\tau;v_1, v_2)+b_{11} v_2-A(\tau) v_2 & = & 0,\ \tau \in [0,T], \\
 h_{11}(T)+h_{11}(0) & = & 0, \\
 \int_{0}^{T} {\langle v_2^*,h_{11}\rangle d\tau} & = & 0.
\end{array}
\right. \label{eq:h11-FF}
\end{equation}

Collecting the $\xi_2^2$-terms gives a singular equation for
$h_{02}$
\begin{equation} \label{eq:h02_FF}
 \dot h_{02}-A(\tau) h_{02}= B(\tau;v_2,v_2)- 2 a_{02} v_1 - 2 \alpha_{02} \dot u_0
\end{equation}
and, since this equation has to be solvable, the following Fredholm solvability
condition is involved
\[
 \int_0^{T} \langle \varphi^*,B(\tau;v_2,v_2)- 2 a_{02} v_1 - 2 \alpha_{02} \dot u_0 \rangle\; d \tau = 0,
\]
from which we obtain an equation for $a_{02}$
\begin{equation} \label{eq:a02_FF}
 a_{02} = \frac{1}{2}\int_0^{T} \langle \varphi^*,B(\tau;v_2,v_2) \rangle\; d \tau.
\end{equation}
So (\ref{eq:h02_FF}) is
solvable, for any value of the parameter $\alpha_{02}$. For simplicity, we take $\alpha_{02}=0$.\\
Notice that also here, the solution of \eqref{eq:h02_FF} is orthogonal to the adjoint eigenfunction $\varphi^*$. Since we have to fix the projection in the direction of eigenfunction $\dot u_0$, we define $h_{02}$ as the unique solution of 
\begin{equation}\label{eq:h02-FF}
\left\{\begin{array}{rcl}
 \dot h_{02}-A(\tau) h_{02}-B(\tau;v_2,v_2)+2 a_{02} v_1+2 \alpha_{02} F(u_0) & = & 0,\ \tau \in [0,T], \\
 h_{02}(T)-h_{02}(0) & = & 0, \\
 \int_{0}^{T} {\langle v_1^*,h_{02}\rangle d\tau} & = & 0.
\end{array}
\right.
\end{equation}

By collecting the $\xi_1^3$-terms we get a singular equation for
$h_{30}$
\[
 \dot{h}_{30}-A(\tau) h_{30}= C(\tau;v_1,v_1,v_1)+3 B(\tau;h_{20},v_1)-6 a_{20} h_{20}-6 a_{30} v_1+3 \dot{h}_{20}-6 \alpha _{30} \dot{u}_0-6 \alpha _{20} \dot{v}_1,
\]
where the Fredholm solvability condition
\[
 \int_0^{T} \langle \varphi^*,C(\tau;v_1,v_1,v_1)+3 B(\tau;h_{20},v_1)-6 a_{20} h_{20}-6 a_{30} v_1+3 \dot{h}_{20}-6 \alpha _{30} \dot{u}_0-6 \alpha _{20} \dot{v}_1\rangle\; d
 \tau = 0
\]
gives us the value of $a_{30}$
\begin{eqnarray}
 a_{30} = \frac{1}{6}\int_0^{T} \langle \varphi^*&,& C(\tau;v_1,v_1,v_1)+3 B(\tau;h_{20},v_1)-6 a_{20} h_{20}
 \\ \nonumber
 &&+3 (A(\tau) h_{20}+B(\tau;v_1,v_1)) +6 (1-\alpha _{20}) A(\tau) v_1 \rangle\; d
 \tau - a_{20}. \label{eq:a30_FF}
\end{eqnarray}

Similarly, by collecting the $\xi_1^2\xi_2$-terms we get a singular equation for
$h_{21}$
\begin{eqnarray*}
 \dot{h}_{21}-A(\tau) h_{21}&=&C(\tau;v_1,v_1,v_2)+B(\tau;h_{20},v_2)+2 B(\tau;h_{11},v_1) -2 a_{20} h_{11} \\
&-&2 b_{11} h_{11}-2 b_{21} v_2 +2 \dot{h}_{11}-2 \alpha _{20}
\dot{v}_2,
\end{eqnarray*}
where the Fredholm solvability condition
\begin{eqnarray*}
 \int_0^{T} \langle v_2^*&,&C(\tau;v_1,v_1,v_2)+B(\tau;h_{20},v_2)+2 B(\tau;h_{11},v_1) \\
 &&-2 a_{20} h_{11}-2 b_{11} h_{11}-2 b_{21} v_2 +2 \dot{h}_{11}-2 \alpha _{20} \dot{v}_2 \rangle\; d
 \tau = 0
\end{eqnarray*}
produces the value of $b_{21}$
\begin{eqnarray} \label{eq:b21_FF} \\\nonumber
 b_{21} = \frac{1}{2}\int_0^{T} \langle v_2^*, C(\tau;v_1,v_1,v_2)+B(\tau;h_{20},v_2)+2 B(\tau;h_{11},v_1) -2 a_{20} h_{11} &&\\ \nonumber
 -2 b_{11} h_{11}  + 2 (A(\tau) h_{11}+B(\tau;v_1, v_2)) +2 (1-\alpha _{20}) A(\tau) v_2 \rangle&\;& d \tau -b_{11}.
\end{eqnarray}

By collecting the $\xi_1\xi_2^2$-terms we obtain a singular equation for
$h_{12}$
\begin{eqnarray*}
 \dot{h}_{12}-A(\tau) h_{12}=C(\tau;v_1,v_2,v_2)+B(\tau;h_{02},v_1)+2 B(\tau;h_{11},v_2)\\
 -2 b_{11} h_{02}-2 a_{02} h_{20}-2 a_{12} v_1+\dot{h}_{02}-2 \alpha _{12} \dot{u}_0-2 \alpha_{02} \dot{v}_1,
\end{eqnarray*}
where its solvability requires the following equality
\begin{eqnarray*}
 \int_0^{T} \langle \varphi^*&,&C(\tau;v_1,v_2,v_2)+B(\tau;h_{02},v_1)+2 B(\tau;h_{11},v_2)\\
 &&-2 b_{11} h_{02}-2 a_{02} h_{20}-2 a_{12} v_1+\dot{h}_{02}-2 \alpha _{12} \dot{u}_0-2 \alpha_{02} \dot{v}_1 \rangle\; d
 \tau = 0,
\end{eqnarray*}
such that
\begin{eqnarray} \label{eq:a12_FF} \\ \nonumber
 a_{12} = \frac{1}{2}\int_0^{T} \langle \varphi^*,C(\tau;v_1,v_2,v_2)+B(\tau;h_{02},v_1)+2 B(\tau;h_{11},v_2)-2 b_{11} h_{02}  && \\ \nonumber
 -2 a_{02} h_{20} +A(\tau) h_{02}+ B(\tau;v_2,v_2) - 2 \alpha_{02} A(\tau) v_1 \rangle&\;& d
 \tau - a_{02}.
\end{eqnarray}

Finally, the $\xi_2^3$-terms give us the value of the last needed critical coefficient. The equation
for $h_{03}$ is
\[
 \dot{h}_{03}-A(\tau) h_{03}=C(\tau;v_2,v_2,v_2)+3 B(\tau;h_{02},v_2)-6 a_{02} h_{11}-6 b_{03} v_2-6 \alpha _{02} \dot{v}_2,
\]
with Fredholm solvability condition
\[
 \int_0^{T} \langle v_2^*,C(\tau;v_2,v_2,v_2)+3 B(\tau;h_{02},v_2)-6 a_{02} h_{11}-6 b_{03} v_2-6 \alpha _{02} \dot{v}_2 \rangle\; d
 \tau = 0,
\]
and thus
\begin{equation} \label{eq:b03_FF}
 b_{03} = \frac{1}{6} \int_0^{T} \langle v_2^*,C(\tau;v_2,v_2,v_2)+3 B(\tau;h_{02},v_2)-6 a_{02} h_{11}-6 \alpha _{02} A(\tau) v_2 \rangle\; d \tau.
\end{equation}

\section{Implementation issues}
\label{Section:Implementation}
Numerical implementation of the formulas derived in the preceding
sections requires the evaluation of integrals of scalar functions
over $[0,T]$ and the solution of nonsingular linear BVPs with
integral constraints. Such tasks can be carried out with
continuation software such as {\sc auto} \cite{AUTO97}, {\sc
content} \cite{CONTENT}, and {\sc matcont} \cite{sac2003,MATCONT}.
In these software packages, periodic solutions to (\ref{ODE}) are
computed with the method of {\em orthogonal collocation} with
piecewise polynomials applied to properly formulated BVPs.

The standard BVP for the periodic solutions
is formulated on the unit interval $[0,1]$, so that the period $T$ becomes a parameter,
and it involves an integral phase condition:
\begin{eqnarray}\label{BVP_LC}
\left\{ \begin{array}{rcl}
\dot{x}(\tau) - Tf(x(\tau),\alpha) & = & 0,\ \tau \in [0,1], \\
x(0) - x(1) & = & 0,\\
\int_0^1 \langle x(\tau), \dot{\xi}(\tau) \rangle ~d\tau & = & 0,
\end{array} \right.
\end{eqnarray}
where $\xi$ is a previously calculated periodic solution to a nearby problem, rescaled to $[0,1]$.

In the orthogonal collocation method \cite{BVPbook:95}, problem (\ref{BVP_LC}) is replaced by the following discretization:
\begin{equation}
\left\{\begin{array}{rcl}
\displaystyle\sum_{j=0}^m x_{i,j} \dot \ell_{i,j}(\zeta_{i,k})
- T f\left(\sum_{j=0}^m x_{i,j}\ell_{i,j}(\zeta_{i,k}),\alpha\right) & = & 0 ,\qquad \\
x_{0,0} - x_{N-1,m} & = & 0, \\[12pt]
\displaystyle\sum_{i=0}^{N-1} \sum_{j=0}^{m-1} \sigma_{i,j}
\langle x_{i,j},\dot{\xi}_{i,j}\rangle +\sigma_{N,0}
\langle x_{N,0},\dot{\xi}_{N,0} \rangle & = & 0 .
\end{array}\right.
\label{BVP_LC_Num}
\end{equation}
The points $x_{i,j}$ form the approximation of $x(\tau)$ with $m+1$ equidistant mesh points
\begin{eqnarray*}
\tau_{i,j} = \tau_i + \frac{j}{m}(\tau_{i+1}-\tau_i),~~~j=0,1,\ldots,m,
\end{eqnarray*}
in each of the $N$ intervals $[\tau_i,\tau_{i+1}]$, where
\begin{eqnarray*}
0 = \tau_0 < \tau_1 < \cdots < \tau_N = 1.
\end{eqnarray*}

The $\ell_{i,j}(\tau)$'s are the Lagrange basis polynomials, while the points
$\zeta_{i,j}\ (j=1,\ldots,m)$ are Gauss points \cite{BoSw:73},
i.e. the roots of the Legendre polynomial of degree $m$,
all relative to the interval $[\tau_i, \tau_{i+1}]$.

With this choice of collocation points $\zeta_{i,j}$, the approximation error
at the mesh points has order of accuracy $m$,
$$
\|x(\tau_{i,j}) - x_{i,j}\| = \mathcal{O}(h^m),
$$
where $h=\max_{i=1,2,\ldots,N}\{t_i\},\ t_i=\tau_i-\tau_{i-1}\ (i=1,\ldots,N)$,
while for the coarse mesh points $\tau_i$ the error has order of accuracy $2m$,
$$
\|x(\tau_{i}) - x_{i,0}\| = \mathcal{O}(h^{2m})
$$
(``superconvergence").

The integration weight $\sigma_{i,j}$ of $\tau_{i,j}$ is given by
$w_{j+1}t_{i+1}$ for $0\leq i\leq N-1$ and $0<j<m$.
For $i=0,\ldots,N-2$, the integration weight
of $\tau_{i,m}$ ($\tau_{i,m}=\tau_{i+1,0}$) is given by
$\sigma_{i,m}=w_{m+1}t_{i+1}+w_1t_{i+2}$,
and the integration weights of $\tau_0$ and $\tau_N$ are given by
$w_1t_1$ and $w_{m+1}t_N$, respectively. In the above expressions,
$w_{j+1}$ is the Lagrange quadrature coefficient.

The numerical continuation of solutions of (\ref{BVP_LC_Num}) leads to structured,
sparse linear systems,
which in {\sc auto} \cite{AUTO97} and {\sc content} \cite{CONTENT} are
solved by an efficient, specially adapted elimination algorithm that computes
the multipliers as a by-product, without explicitly using the Poincar\'{e} map.
To detect codim 1 bifurcations, one can specify test functions that
are based on computing multipliers \cite{DoKeKe:91b,AUTO97}
or on solving appropriate bordered linear BVPs \cite{DoGoKu}.

\subsection{Discretization symbols}\label{Section: Discretization symbols}

It is convenient to discretize all computed functions using the same mesh as in
(\ref{BVP_LC_Num}). For a given vector function $\eta \in {{\cal C}}^1([0,1],{\mathbb R^n})$
we consider three different discretizations:

\begin{itemize}
\item{$\eta_M \in {\mathbb R^{(Nm+1)n}}$, the vector of the function values at the mesh points;}

\item{$\eta_C \in {\mathbb R^{Nmn}}$, the vector of the function values at the collocation points;}

\item{$\eta_W= [\begin{smallmatrix} \eta_{W_1} \\ \eta_{W_2} \end{smallmatrix}]   
\in {\mathbb R^{Nmn}}\times{\mathbb R^n}$, where $\eta_{W_1}$ is the vector of the function values
at the collocation points multiplied by the Gauss-Legendre weights
and the lengths of the corresponding mesh intervals, and $\eta_{W_2}=\eta(0)$.}
\end{itemize}

Formally we also introduce the structured sparse matrix $L_{C\times M}$ that converts
a vector $\eta_M$ of function values at the mesh points into a vector $\eta_C$ of its
values at the collocation points, namely, $\eta_C=L_{C\times M}\eta_M$. This matrix is never
formed explicitly; its entries are approximated by the $\ell_{i,j}(\zeta_{i,k})$-coefficients
in (\ref{BVP_LC_Num}). We also need a matrix $A_{C\times M}$ such that
$A_{C\times M}\eta_M=(A(t)\eta(t))_C$. Again this matrix need not be formed explicitly.
On the other hand, we do need the matrix $(D-TA(t))_{C\times M}$ explicitly; it is
defined by $(D-TA(t))_{C\times M}\eta_M=(\dot{\eta}(t)-TA(t)\eta(t))_C$.
Finally, let the tensors $B_{C\times M\times M}$ and $C_{C\times M\times M\times M}$ be
defined by $B_{C\times M\times M}\eta_{1M}\eta_{2M}=(B(t;\eta_1(t),\eta_2(t)))_C$ and
$$
C_{C\times M\times M\times M}\eta_{1M}\eta_{2M}\eta_{3M}=(C(t;\eta_1(t),\eta_2(t),\eta_3(t)))_C
$$
for all $\eta_i \in {{\cal C}}^1([0,1],{\mathbb R^n})$.
(These tensors are not formed explicitly.)

Let $f(t),g(t)\in {{\cal C}}^0([0,1],{\mathbb R})$ be two scalar
functions. Then the integral $\int_0^1f(t)dt$ is represented by $
\sum_{i=0}^{N-1}\sum_{j=1}^m
\omega_j(f_C)_{i,j}t_{i+1}=\sum_{i=0}^{N-1}\sum_{j=1}^{m}(f_{W_1})_{i,j},
$ where $(f_C)_{i,j}=f(\zeta_{i,j})$ and $\omega_j$ is the
Gauss-Legendre quadrature coefficient. The integral $\int_0^{1}
f(t)g(t)dt$ is approximated with Gauss-Legendre by $f_{W_1}^{\rm
T}g_C \approx f^{\rm T}_{W_1}L_{C\times M}g_M$, where equality
holds if $g(t)$ is a piecewise polynomial of degree $m$ or less on
the given mesh. For vector functions $f(t),g(t)\in {{\cal
C}}^0([0,1],{\mathbb R^n})$, the integral $\int_0^{1}\langle
f(t),g(t)\rangle\;dt$ is formally approximated by the same
expression: $f_{W_1}^{\rm T}g_C \approx f^{\rm T}_{W_1}L_{C\times
M}g_M$, where again we have equality if $g(t)$ is a piecewise
polynomial of degree $m$ or less on the given mesh. Concerning the
accuracy of the quadrature formulas, we first note that accuracy
is not an important issue for the phase integral in
(\ref{BVP_LC}), as this equation only selects a specific solution
from the continuum of solutions obtained by phase shifts.
Similarly, the discretization of the normalization integrals does not affect the inherent
accuracy, including superconvergence at the main mesh points
$\tau_i$ of the solution of the discretized BVP. Discretization of
integrals, as specified above, follows the standard Gauss
quadrature error, which has order of accuracy $2m$ if, as
mentioned, the function $g(t)$ is a piecewise polynomial of degree
$m$ or less on the given mesh and if $f(t)$ is sufficiently smooth
(in a piecewise sense). Otherwise, still assuming sufficient
piecewise smoothness, the order of accuracy of the numerical
integrals is $m+1$ if $m$ is odd, and $m+2$ if $m$ is even. In
particular, for the often used choice $m=4$, the integrals would
then have order of accuracy $6$.

\subsection{Cusp of cycles bifurcation} \label{cusp_impl} The first task is to rescale the computed functions to the interval $[0,1]$. We start by defining $u_1(t) = u_0(Tt) = u_0(\tau)$ for $t \in [0,1]$. The linear BVP's (\ref{eq:EigenFunc_CPC}) and (\ref{eq:AdjEigenFunc_CPC}) are replaced by
\begin{equation}\label{eq:EigenFunc_CPCdiscr}
\left\{\begin{array}{rcl}
 \dot{v}_1(t)-TA(t) v_1(t) - TF(u_1(t)) & = & 0,\ t \in [0,1], \\
 v_1(1)-v_1(0) & = & 0,\\
 \int_{0}^{1} {\langle v_1(t), F(u_1(t))\rangle dt} & = & 0,\\
\end{array}
\right.
\end{equation}
with $v(\tau) = v_1(\tau/T)$ and
\begin{equation*}
\left\{\begin{array}{rcl}
 \dot{\varphi}_1^*(t)+TA^{\rm T}(t)\varphi^*_1(t) & = & 0,\ t \in [0,1], \\
 \varphi_1^*(1)-\varphi^*_1(0) & = & 0,\\
 \int_{0}^{1} {\langle \varphi_1^*(t),v_1(t) \rangle dt} - 1 & = & 0,
\end{array}
\right.
\end{equation*}
where $\varphi^*(\tau) = \varphi^*_1(\tau/T)/T$, respectively. We also need to rescale the adjoint generalized eigenfunction defined by (\ref{eq:AdjGenEigenFunc_CPC}):
\begin{equation*}
\left\{\begin{array}{rcl}
\dot{v}^*_1(t)+TA^{\rm T}(t) v^*_1(t)+T\varphi^*_1(t) & = & 0,\ t \in [0,1], \\
v^*_1(1)-v^*_1(0) & = & 0,\\
\int_{0}^{1} {\langle v^*_1(t),v_1(t) \rangle dt} & = & 0,
\end{array}
\right.
\end{equation*}
with $v^*(\tau) = v_1^*(\tau/T)/T$. Now, $\alpha_1=0$ and $h_{2,1}$ is the unique solution of the BVP
\begin{equation}\label{eq:h2_CPCdiscr}
\left\{\begin{array}{rcl}
 \dot h_{2,1}(t) - T A(t) h_{2,1}(t) - T B(t;v_1(t),v_1(t)) - 2TA(t)v_1(t)-2TF(u_1(t)) + 2 \alpha_{1} T F(u_1(t)) &=& 0,\ t \in [0,1],\\
 h_{2,1}(1)-h_{2,1}(0)&=&0, \\
 \int_0^1 \langle v^*_1(t), h_{2,1}(t) \rangle\; dt&=&0,
\end{array}\right.
\end{equation}
where $h_2(\tau) = h_{2,1}(\tau/T)$.
Therefore, we obtain
\begin{eqnarray}\label{eq:c-CPCdiscr}
    c&=&\frac{1}{6}\int_0^1 \langle \varphi^*_1(t),- 6 \alpha_{1} A(t) v_1(t) + 3 A(t) h_{2,1}(t) + 3 B(t;v_1(t),v_1(t)) \\
    &&+ 6 A(t) v_1(t) + 3 B(t;h_{2,1}(t),v_1(t)) + C(t;v_1(t),v_1(t),v_1(t)) \rangle\; dt.\nonumber
\end{eqnarray}

We now determine the matrix solutions for the several functions. We compute $v_{1M}$ by solving the discretization of (\ref{eq:EigenFunc_CPCdiscr})
\begin{eqnarray}\label{eq:EigenFunc_CPCimpl}
\left[\begin{array}{cc}
\begin{array}{c}
(D-TA(t))_{C \times M}\\
\delta_0-\delta_1
\end{array} & p\\
g_{W_1}^{\rm T} L_{C \times M} & 0
\end{array}
\right]
\left[\begin{array}{c}
v_{1M}\\
a
\end{array}\right]
= \left[\begin{array}{c}
Tf_{C}\\
0_{n\times 1}\\
0
\end{array}\right]
,\end{eqnarray}
where $a$ equals zero since the $M\times M$ upper left part of the big matrix is singular, $g(t) = F(u_1(t))$, and $p$ is obtained by solving the following system
\begin{eqnarray*}
\begin{bmatrix}p^{\rm T} & b \end{bmatrix} \left[\begin{array}{cc}
\begin{array}{c}
(D-TA(t))_{C \times M}\\
\delta_0-\delta_1
\end{array} & r_1\\
r_2 & 0\\
\end{array}
\right]
= \begin{bmatrix}
0_{M\times 1} & 1\end{bmatrix}
,\end{eqnarray*}
with $r_1$ and $r_2$ random vectors. In the solution of this system $b=0$; in (\ref{eq:EigenFunc_CPCimpl}) we then use the normalized $p$. This technique guarantees that we are working with well-defined systems.\\

We will compute $\varphi_{1W}^*$ instead of $\varphi_{1M}^*$ since $\varphi_{1W}^*$ can be computed with a matrix which is very similar to the matrix from (\ref{eq:EigenFunc_CPCimpl}). Formally, the computation of $\varphi_{1W}^*$ is based on Proposition \ref{Proposition1} from the appendix since
$$\left[\begin{array}{c}
\dot{\varphi}_1^*(t)+TA^{\rm T}(t)\varphi^*_1(t) \\
\varphi_1^*(1)-\varphi^*_1(0) 
\end{array}
\right]= 0,$$
we have that $\left[\begin{array}{c}\varphi_{1}^*\\\varphi_{1}^*(0)\end{array} \right]$ is orthogonal to the range of $\left[\begin{array}{c}
D-TA(t)\\
\delta_0-\delta_1
\end{array}
\right]$. By discretization we obtain
$$(\varphi_{1}^*)_W^{\rm T} \left[\begin{array}{c}
(D-TA(t))_{C \times M}\\
\delta_0-\delta_1
\end{array}
\right] = 0.$$
Therefore, $\varphi_{1W}^*$ can be obtained by solving
$$\begin{bmatrix}
(\varphi_{1}^*)_W^{\rm T} & a
\end{bmatrix}\left[\begin{array}{cc}
\begin{array}{c}
(D-TA(t))_{C \times M}\\
\delta_0-\delta_1
\end{array}&p\\
q^{\rm T} & 0
\end{array}
\right] = \begin{bmatrix}0_{M\times 1} & 1\end{bmatrix},$$
where $a$ equals zero and $q$ is the normalized right null-vector of $\left[\begin{array}{c}(D-TA(t))_{C \times M}\\ \delta_0-\delta_1\end{array}\right]$. We then approximate $I = \int_{0}^{1} {\langle \varphi_1^*(t),v_1(t) \rangle dt}$ by $I_1 = (\varphi_{1}^*)_{W_1}^{\rm T}L_{C\times M}v_{1M}$. $\varphi_{1W}^*$ is then rescaled to ensure that $I_1 = 1$.

It is more efficient to compute $v_{1W}^*$ than $v_{1M}^*$, since $v_{1}^*$ will be used only to compute integrals of the form $\int_0^1 <v_1^*(t),\zeta(t)>dt$. Proposition \ref{Proposition4} learns us how to determine $v_1^*$. Indeed, 
$$
\left\langle\left[\begin{array}{c}v_1^*\\v_1^*(0)\end{array}\right],\left[\begin{array}{c}\dot{h}-TAh\\ h(0) \;\;\;\;\;\; -h(1)\end{array} \right]\right\rangle = -\left\langle\left[\begin{array}{c}-T\varphi_1^*\\0\end{array}\right],\left[\begin{array}{c}h\\ 0\end{array} \right]\right\rangle,$$
for all appropriate functions $h$, thus $v_1^*$ can be obtained by solving
$$\begin{bmatrix}(v_{1}^*)_W^{\rm T} & a \end{bmatrix}
\left[\begin{array}{cc}
(D-TA(t))_{C \times M}& v_{1C}\\
\delta_0-\delta_1 & 0_{n\times 1}\\
q^{\rm T} & 0
\end{array}
\right]
= \begin{bmatrix}(T\varphi^*_{1})_{W_1}^{\rm T}L_{C\times M} & 0\end{bmatrix},$$
where $a$ equals zero and $p$ is defined above. 

Next, $(h_{2,1})_M$ is found by solving the discretization of (\ref{eq:h2_CPCdiscr}), namely,
$$\left[\begin{array}{cc}
\begin{array}{c}
(D-TA(t))_{C \times M}\\
\delta_0-\delta_1
\end{array} & p\\
(v_{1}^*)_{W_1}^{\rm T} L_{C \times M} & 0
\end{array}
\right]
\left[\begin{array}{c}
h_{2,1M} \\
a
\end{array}\right]
=
\left[\begin{array}{c}
    T B_{C\times M \times M}v_{1M}v_{1M}+2TA_{C\times M}v_{1M}+2Tg_C\\
    0_{n \times 1}\\
    0
\end{array}\right],$$
with $a=0$ and $p$ defined above.

Finally, (\ref{eq:c-CPCdiscr}) is approximated by
\begin{eqnarray*}
    c&=&\frac{1}{6}(\varphi^*_{1})_{W_1}^{\rm T}(3 A_{C\times M} h_{2,1M} + 3 B_{C\times M \times M}v_{1M}v_{1M} \\
    &&+ 6 A_{C \times M} v_{1M} + 3 B_{C\times M \times M}h_{2,1M}v_{1M} + C_{C\times M \times M \times M}v_{1M}v_{1M}v_{1M}).
\end{eqnarray*}

\subsection{Generalized period-doubling bifurcation}
As done in Section \ref{cusp_impl}, we first rescale the computed quantities to the interval $[0,1]$. The linear BVPs (\ref{eq:EigenFunc_GPD})  and (\ref{eq:AdjEigenFunc2T-2_GPD}) are replaced by
\begin{equation*}
\left\{\begin{array}{rcl}
\dot{v}_1(t) - T A(t)v_1(t)  & = & 0,\ t \in [0,1], \\
v_1(1)+v_1(0) & = & 0,\\
\int_{0}^{1} {\langle v_1(t),v_1(t)\rangle dt} - 1 & = & 0,
\end{array}
\right.
\end{equation*}
where $v(\tau) = v_1(\tau/T)/\sqrt{T}$, and
\begin{equation*}
\left\{\begin{array}{rcl}
 \dot{\varphi}_1^*(t)+TA^{\rm T}(t)\varphi_1^*(t) & = & 0,\ t \in [0,1], \\
 \varphi^*_1(1)-\varphi^*_1(0) & = & 0,\\
 \int_0^1 \langle \varphi^*_1(t),F(u_1(t))\rangle\; dt - 1 & = & 0,
\end{array}
\right.
\end{equation*}
with $\varphi^*(\tau) = \varphi_1^*(\tau/T)/T$. This leads to the expression
$$ \alpha_{1,1} = \frac{1}{2} \int_0^{1} \langle \varphi_1^*(t),B(t;v_1(t),v_1(t)) \rangle\; dt,$$
with $\alpha_{1,1}=T\alpha_1$.

The adjoint eigenfunction, defined by (\ref{eq:AdjEigenFunc_GPD}), determines the following rescaled $v_1^*$:
\begin{equation}\label{eq:AdjEigenFunc_GPDdiscr}
\left\{\begin{array}{rcl}
 \dot{v}_1^*(t)+T A^{\rm T}(t)v_1^*(t) & = & 0,\ t \in [0,1], \\
 v^*_1(1)+v^*_1(0) & = & 0,\\
 \int_{0}^{1} {\langle v^*_1(t),v_1(t)\rangle dt} - 1 & = & 0,
\end{array}
\right.
\end{equation}
with $v^*(\tau) = v_1^*(\tau/T)/\sqrt{T}$.

Let $h_{2,1}$ be the unique solution of the BVP
\begin{equation*}  
\left\{\begin{array}{rcl}
\dot{h}_{2,1}(t)-T A(t) h_{2,1}(t) - T B(t;v_1(t),v_1(t)) + 2\alpha_{1,1} T F(u_1(t)) & = & 0,\ t \in [0,1], \\
h_{2,1}(1)- h_{2,1}(0) & = & 0,\\
\int_0^{1} \langle \varphi^*_1(t),h_{2,1}(t)\rangle \; dt & = & 0,
\end{array}
\right.
\end{equation*}
where $h_2(\tau) = h_{2,1}(\tau/T)/T$, and $h_{3,1}$ be the unique solution of the BVP
\begin{equation*}  
 \left\{\begin{array}{rcl}
  \dot{h}_{3,1}(t)-T A(t) h_{3,1}(t) - T C(t;v_1(t),v_1(t),v_1(t))- 3T B(t;v_1(t),h_{2,1}(t)) + 6\alpha_{1,1} T A(t) v_1(t) & = & 0,\ t \in [0,1], \\
  h_{3,1}(1) + h_{3,1}(0) & = & 0,\\
  \int_0^{1} \langle v^*_1(t),h_{3,1}(t)\rangle \; dt & = & 0,
 \end{array}
\right.
\end{equation*}
where $h_3(\tau) = h_{3,1}(\tau/T)/(\sqrt{T}T)$.

Therefore, we obtain
\begin{multline*} 
\alpha_{2,1}=\frac{1}{24}\int_0^{1} \langle
\varphi^*_1(t),D(t;v_1(t),v_1(t),v_1(t),v_1(t))+6 C(t;v_1(t),v_1(t),h_{2,1}(t))+3 B(t;h_{2,1}(t),h_{2,1}(t))+\\4
B(t;v_1(t), h_{3,1}(t))-12 \alpha_{1,1} (A(t) h_{2,1}(t) + B(t;v_1(t),v_1(t)))\rangle\;
dt\;+\;\alpha_{1,1}^2,
\end{multline*}
with $\alpha_{2,1}=T^2\alpha_2$.

Now, $h_{4,1}$ is obtained as unique solution of the following BVP
\begin{equation*} 
\left\{\begin{array}{rcl}
 \dot{h}_{4,1}(t)-T A(t) h_{4,1}(t) - T D(t;v_1(t),v_1(t),v_1(t),v_1(t))-6T C(t;v_1(t),v_1(t),h_{2,1}(t))&&\\
 -3 T B(t;h_{2,1}(t),h_{2,1}(t))-4T B(t;v_1(t),h_{3,1}(t))+12 \alpha_{1,1}T (A(t) h_{2,1}(t) + B(t;v_1(t),v_1(t)) &&\\
 - 2\alpha_{1,1} F(u_1(t)))+24 \alpha_{2,1}T F(u_1(t)) & = & 0,\ t \in [0,1], \\
 h_{4,1}(1)- h_{4,1}(0) & = & 0,\\
 \int_0^{1} \langle \varphi^*_1(t),h_{4,1}(t) \rangle \; dt & = & 0,
\end{array}
\right.
\end{equation*}
with $h_4(\tau) = h_{4,1}(\tau/T)/T^2$.

Finally, we can write down the critical normal form coefficient
\begin{multline*} 
e=\frac{1}{120\,T^2} \int_0^{1} \langle v^*_1(t), E(t;v_1(t),v_1(t),v_1(t),v_1(t),v_1(t))+10
D(t;v_1(t),v_1(t),v_1(t),h_{2,1}(t))\\+15 C(t;v_1(t),h_{2,1}(t),h_{2,1}(t))+10 C(t;v_1(t),v_1(t),h_{3,1}(t))+10
B(t;h_{2,1}(t),h_{3,1}(t))\\+5 B(t;v_1(t), h_{4,1}(t))-120 \alpha_{2,1} A(t) v_1(t) -20
\alpha_{1,1} A(t) h_{3,1}(t) \rangle \; dt.
\end{multline*}

We now come to the implementation details in MatCont. We compute $v_{1M}$ by solving
\begin{eqnarray*}
\left[\begin{array}{cc}
\begin{array}{c}
(D-TA(t))_{C \times M} \\
\delta_0+\delta_1
\end{array} & p_1\\
q_1^{\rm T} & 0
\end{array}
\right]
\left[\begin{array}{c}
v_{1M}\\
a_1
\end{array}\right]
=
\left[\begin{array}{c}
    0_{C\times 1}\\0_{n \times 1}\\1
\end{array}\right]
,\end{eqnarray*}
with $p_1$ and $q_1$ the rescaled unit vectors of 
\begin{eqnarray*}
\left[\begin{array}{cc}
\begin{array}{c}
(D-TA(t))_{C \times M} \\
\delta_0+\delta_1
\end{array} & r_1\\
r_2^{\rm T} & 0
\end{array}
\right]
\left[ \begin{array}{c}
    q_1 \\ a_2
\end{array}\right]
= \left[ \begin{array}{c}
    0_{C\times 1} \\ 0_{n\times 1} \\1
\end{array}\right]\end{eqnarray*}
and
\begin{eqnarray*}
\begin{bmatrix}p_1^{\rm T} & a_3\end{bmatrix}
\left[\begin{array}{cc}
\begin{array}{c}
(D-TA(t))_{C \times M} \\
\delta_0+\delta_1
\end{array} & r_1\\
r_2^{\rm T} & 0
\end{array}
\right]
= \begin{bmatrix}0_{M\times 1} & 1 \end{bmatrix} ,\end{eqnarray*}
where $r_1$ and $r_2$ are random vectors. Every $a_i$ is equal to zero. We normalize $v_{1M}$ by requiring
$\sum_{i=0}^{N-1}\sum_{j=0}^m\sigma_j\langle(v_{1M})_{i,j},(v_{1M})_{i,j}\rangle =1,$ where $\sigma_j$ is the Lagrange quadrature coefficient.

The discretization of (\ref{eq:AdjEigenFunc_GPDdiscr}) can be computed with the same matrix, see Proposition \ref{Proposition2} of the appendix,
\begin{eqnarray*}
\begin{bmatrix}(v_{1}^*)_W^{\rm T} & a \end{bmatrix}\left[\begin{array}{cc}
\begin{array}{c}
(D-TA(t))_{C \times M} \\
\delta_0+\delta_1
\end{array} & p_1\\
q_1^{\rm T} & 0
\end{array}
\right]
= \begin{bmatrix}0_{M\times 1} & 1\end{bmatrix},\end{eqnarray*}
where $a=0$. We then approximate $I = \int_{0}^{1} {\langle v_1^*(t),v_1(t) \rangle dt}$ by $I_1 = (v_{1}^*)_{W_1}^{\rm T}L_{C\times M}v_{1M}$. $v_{1W}^*$ is then rescaled to ensure that $I_1 = 1$.

An analogous matrix is used to compute $\varphi^*_{1W}$:
\begin{eqnarray*}
\begin{bmatrix}(\varphi_{1}^*)_W^{\rm T} & a \end{bmatrix}\left[\begin{array}{cc}
\begin{array}{c}
(D-TA(t))_{C \times M} \\
\delta_0-\delta_1
\end{array} & p\\
q^{\rm T} & 0
\end{array}
\right]
=  \begin{bmatrix}0_{M\times 1} & 1\end{bmatrix},\end{eqnarray*}
where $q$ is the normalized right null-vector and $p$ the normalized left null-vector of $\left[\begin{array}{c}
(D-TA(t))_{C\times M}\\
\delta_0-\delta_1
\end{array}
\right]$, and $a$ is equal to zero. In what follows we will use these definitions for $p, q, p_1$ and $q_1$. We then approximate $I = \int_{0}^{1} {\langle \varphi_1^*(t),F(u_1(t)) \rangle dt}$ by $I_1 = (\varphi_{1}^*)_{W_1}^{\rm T}g_C$ and normalize $\varphi_{1W}^*$ to ensure that $I_1 = 1$.

This then leads to the discretization of the expression for
$\alpha_{1,1}$:
\begin{equation} \label{eq:alpha11-GPD}
\alpha_{1,1} = \frac{1}{2} (\varphi_{1}^*)_{W_1}^{\rm T}B_{C
\times M \times M}v_{1M}v_{1M}.
\end{equation}

Now, $h_{2,1}$ and $h_{3,1}$ are found by solving the following systems:
\begin{eqnarray*}
\left[\begin{array}{cc}
\begin{array}{c}
(D-TA(t))_{C \times M} \\
\delta_0-\delta_1
\end{array} & p\\
(\varphi^*_1)_{W_1}^{\rm T} L_{C \times M} & 0
\end{array}
\right]
\left[\begin{array}{c}
h_{2,1M}\\a
\end{array}\right]=
\left[\begin{array}{c}
    T B_{C \times M \times M}v_{1M}v_{1M} - 2\alpha_{1,1} T g_C\\
    0_{n\times 1}\\
    0
\end{array}
\right]
\end{eqnarray*}
and
\begin{eqnarray*}
\left[\begin{array}{cc}
\begin{array}{c}
(D-TA(t))_{C \times M} \\
\delta_0+\delta_1
\end{array} & p_1 \\
(v^*_1)_{W_1}^{\rm T} L_{C \times M} & 0
\end{array}
\right]
\left[\begin{array}{c}
h_{3,1M}\\ b
\end{array}\right]=
\begin{bmatrix}\mbox{rhs} \\ 0_{n\times 1} \\ 0\end{bmatrix}
,\end{eqnarray*}
respectively, with
$$\mbox{rhs} = T C_{C\times M \times M\times M}v_{1M}v_{1M}v_{1M}+ 3T B_{C\times M \times M}v_{1M}h_{2,1M} - 6\alpha_{1,1} T A_{C \times M} v_{1M}.$$
$a$ and $b$ will be zero.

Thus,
\begin{multline}\label{eq:alpha21-GPD}
\alpha_{2,1}=\frac{1}{24} (\varphi^*_{1})_{W_1}^{\rm T}\Bigg(D_{C\times M \times M \times M \times M}v_{1M}v_{1M}v_{1M}v_{1M}+6 C_{C\times M\times M \times M}v_{1M}v_{1M}h_{2,1M}\\+3B_{C\times M \times M}h_{2,1M}h_{2,1M}+4
B_{C\times M \times M}v_{1M}h_{3,1M}\\-12 \alpha_{1,1} ( A_{C\times M}h_{2,1M} + B_{C\times M\times M}v_{1M}v_{1M})\Bigg)+\;\alpha_{1,1}^2.
\end{multline}

Then, $h_{4,1}$ is found by
\begin{eqnarray*}
\left[\begin{array}{cc}
\begin{array}{c}
(D-TA(t))_{C \times M} \\
\delta_0-\delta_1
\end{array} & p \\
(\varphi^*_1)_{W_1}^{\rm T} L_{C \times M} & 0
\end{array}
\right]
\left[\begin{array}{c}
h_{4,1M}\\ a
\end{array}\right]=
\begin{bmatrix}
\mbox{rhs}\\0_{n\times 1} \\0\end{bmatrix}
,\end{eqnarray*}
with
\begin{eqnarray*}
     \mbox{rhs} & = &
     T D_{C\times M \times M \times M \times M}v_{1M}v_{1M}v_{1M}v_{1M}+6T C_{C\times M\times M \times M}v_{1M}v_{1M}h_{2,1M}\\
     &&+3 T B_{C\times M \times M}h_{2,1M}h_{2,1M}+4T B_{C\times M \times M}v_{1M}h_{3,1M}\\
     &&-12 \alpha_{1,1}T (A_{C\times M}h_{2,1M}+ B_{C\times M \times M}v_{1M}v_{1M} - 2\alpha_{1,1} g_C)-24 \alpha_{2,1}T g_C
\end{eqnarray*}
and $a = 0$.

Now, we have all ingredients for the computation of the normal form coefficient
\begin{multline} \label{eq:e1-GPD}
e=\frac{1}{120T^2} (v^*_{1})_{W_1}^{\rm T}\Bigg(E_{C\times M \times M \times M \times M\times
 M}v_{1M}v_{1M}v_{1M}v_{1M}v_{1M}+10D_{C\times M \times M \times M \times M}v_{1M}v_{1M}v_{1M}h_{2,1M}\\+15 C_{C\times M \times M \times M}v_{1M}h_{2,1M}h_{2,1M}+10 C_{C\times M \times M \times M}v_{1M}v_{1M}h_{3,1M}+10
B_{C\times M \times M}h_{2,1M}h_{3,1M}\\+5 B_{C\times M \times M}v_{1M}h_{4,1M}-120 \alpha_{2,1} A_{C\times M} v_{1M} -20
\alpha_{1,1} A_{C\times M}h_{3,1M}\Bigg).
\end{multline}

\subsection{Chenciner bifurcation}
As before, we rescale the computed quantities to the interval $[0,1]$. The linear BVPs (\ref{eq:EigenFunc_GNS}), (\ref{eq:AdjEigenFunc2T-2_GPD}) and (\ref{eq:AdjEigenFunc_GNS}) are replaced by
\begin{equation}
\left\{\begin{array}{rcl}
\dot{v}_1(t)-T A(t)v_1(t) + i\omega T\, v_1(t) & = & 0,\ t \in [0,1], \\
v_1(1)-v_1(0) & = & 0,\\
\int_{0}^{1} {\langle v_1(t),v_1(t)\rangle dt} - 1 & = & 0,
\end{array}
\right. \label{eq:EigenFunc_GNSdiscr}
\end{equation}
with $v(\tau) = v_1(\tau/T)/\sqrt{T}$,
\begin{equation*} 
\left\{\begin{array}{rcl}
 \dot{\varphi}_1^*(t)+T A^{\rm T}(t)\varphi^*_1(t) & = & 0,\ t \in [0,1], \\
 \varphi^*_1(1)-\varphi^*_1(0) & = & 0,\\
 \int_0^1 \langle \varphi^*_1(t),F(u_1(t))\rangle\; dt - 1 & = & 0,
\end{array}
\right.
\end{equation*}
with $\varphi^*(\tau) = \varphi^*_1(\tau/T)/T$, and
\begin{equation*}
\left\{\begin{array}{rcl}
\dot{v}_1^*(t)+T A^{\rm T}(t)v^*_1(t) +i\omega T \,v^*_1(t) & = & 0,\ t \in [0,1], \\
v^*_1(1)-v^*_1(0) & = & 0,\\
\int_{0}^{1} {\langle v^*_1(t),v_1(t)\rangle dt} - 1 & = & 0,
\end{array}
\right. 
\end{equation*}
where $v^*(\tau) = v_1^*(\tau/T)/\sqrt{T}$, respectively.

Then $h_{20}$ is approximated by
\begin{equation}
\left\{\begin{array}{rcl}
 \dot h_{20,1}(t)-T A(t) h_{20,1}(t) +2\ i \omega T h_{20,1}(t) - T B(t;v_1(t),v_1(t)) & = & 0,\ t \in [0,1], \\
 h_{20,1}(1)-h_{20,1}(0) & = & 0,
\end{array}
\right. \label{eq:h20-GNSdiscr}
\end{equation}
where $h_{20}(\tau) = h_{20,1}(\tau/T)/T$.

Before being able to compute the approximation to $h_{11}$ we need $\alpha_1$, defined by (\ref{eq:a1-CH}):
\begin{equation*}
 \alpha_{1,1}= \int_0^1 \langle \varphi^*_1(t), B(t;v_1(t),\bar v_1(t)) \rangle dt,
\end{equation*}
with $\alpha_{1,1}=T\alpha_1$. Therefore, we get
\begin{equation*}
\left\{\begin{array}{rcl}
 \dot{h}_{11,1}(t)-T A(t) h_{11,1}(t)-TB(t;v_1(t),\bar{v}_1(t))+\alpha _{1,1}T F(u_1(t)) & = & 0,\ t \in [0,1], \\
 h_{11,1}(1)-h_{11,1}(0) & = & 0,\\
 \int_{0}^{1} {\langle \varphi^*_1(t), h_{11,1}(t)\rangle dt} & = & 0,
\end{array}
\right. 
\end{equation*}
with $h_{11}(\tau) = h_{11,1}(\tau/T)/T$.

Now, we can compute normal form coefficient $c$:
\begin{equation*}
 c_1=-\frac{i}{2} \int_0^1 \langle v^*_1(t), C(t;v_1(t),v_1(t),\bar{v}_1(t))+2
B(t;v_1(t),h_{11,1}(t))+B(t;\bar{v}_1(t),h_{20,1}(t)) - 2 \alpha _{1,1} A(t) v_1(t) \rangle dt + \alpha_{1,1} \omega, 
\end{equation*}
with $c_1=Tc$. With $c_1$ defined in this way, $h_{21M}$ can be computed as follows
\begin{equation*}
\left\{\begin{array}{rcl}
 \dot h_{21,1}(t)-T A(t) h_{21,1}(t)+i \omega T h_{21,1}(t)-TC(t;v_1(t),v_1(t),\bar{v}_1(t))-2TB(t;v_1(t),h_{11,1}(t))&&\\
 - TB(t;h_{20,1}(t),\bar{v}_1(t)) +2 i c_1 Tv_1(t) + 2 \alpha _{1,1} T(A(t) v_1(t) -i \omega v_1(t)) & = & 0,\ t \in [0,1], \\
 h_{21,1}(1)-h_{21,1}(0) & = & 0, \\
 \int_{0}^{1} {\langle v^*_1(t),h_{21,1}(t)\rangle dt} & = & 0,
\end{array}
\right. 
\end{equation*}
where $h_{21}(\tau) = h_{21,1}(\tau/T)/(\sqrt{T}T)$.

Next, the rescaling of $h_{30}$ gives us
\begin{equation*}
\left\{\begin{array}{rcl}
 \dot h_{30,1}(t)-T A(t) h_{30,1}(t)+3 i \omega T h_{30,1}(t)-TC(t;v_1(t),v_1(t),v_1(t))-3TB(t;v_1(t), h_{20,1}(t)) & = & 0,\ t \in [0,1], \\
 h_{30,1}(1)-h_{30,1}(0) & = & 0,
\end{array}
\right. 
\end{equation*}
with $h_{30}(\tau) = h_{30,1}(\tau/T)/(\sqrt{T}T)$.

Now, we need the rescaled $h_{31,1}$ before being able to compute coefficient $\alpha_{2,1}$ 
\begin{equation}
\left\{\begin{array}{rcl}
 \dot h_{31,1}(t)-T A(t) h_{31,1}(t)+2 i \omega T h_{31,1}(t)- T D(t;v_1(t),v_1(t),v_1(t),\bar{v}_1(t))&&\\
 -3T C(t;v_1(t),v_1(t),h_{11,1}(t))-3T C(t;v_1(t),\bar{v}_1(t), h_{20,1}(t))- 3T B(t;h_{11,1}(t), h_{20,1}(t)) &&\\
 -3T B(t;v_1(t), h_{21,1}(t))-TB(t;\bar{v}_1(t),h_{30,1}(t)) &&\\
 + 6 i c_1 T h_{20,1}(t)+ 3 \alpha _{1,1} T (A(t) h_{20,1}(t) - 2 i \omega h_{20,1}(t) + B(t;v_1(t),v_1(t)))& = & 0,\ t \in [0,1], \\
 h_{31,1}(1)-h_{31,1}(0) & = & 0,
\end{array}
\right. \label{eq:h31-GNSdiscr}
\end{equation}
where $h_{31}(\tau) = h_{31,1}(\tau/T)/T^2$, so
\begin{multline*}
 \alpha_{2,1}=\frac{1}{4} \int_0^1 \langle \varphi^*_1(t),D(t;v_1(t),v_1(t),\bar{v}_1(t),\bar v_1(t))+C(t;v_1(t),v_1(t),h_{02,1}(t)) + 4 C(t;v_1(t),\bar{v}_1(t), h_{11,1}(t)) \\
  +C(t;\bar{v}_1(t),\bar v_1(t), h_{20,1}(t))+2B(t;h_{11,1}(t),h_{11,1}(t))+2B(t;v_1(t),h_{12,1}(t)) + B(t;h_{02,1}(t),h_{20,1}(t))\\
  +2 B(t;\bar{v}_1(t), h_{21,1}(t))- 4 \alpha_{1,1} (A(t) h_{11,1}(t) + B(t;v_1(t), \bar v_1(t))) \rangle dt+ \alpha_{1,1}^2,
\end{multline*}
with $\alpha_{2,1}=T^2\alpha_2$.

Now, we still need $h_{22,1}(t)$:
\begin{equation*}
\left\{\begin{array}{rcl}
 \dot h_{22,1}(t)-T A(t) h_{22,1}(t) - T D(t;v_1(t),v_1(t),\bar{v}_1(t),\bar v_1(t))-T C(t;v_1(t),v_1(t),h_{02,1}(t))\\
 - 4T C(t;v_1(t),\bar{v}_1(t), h_{11,1}(t))-TC(t;\bar{v}_1(t),\bar v_1(t), h_{20,1}(t))-2T B(t;h_{11,1}(t),h_{11,1}(t))\\
 -2T B(t;v_1(t),h_{12,1}(t))-T B(t;h_{02,1}(t),h_{20,1}(t)) -2T B(t;\bar{v}_1(t), h_{21,1}(t))\\
 + 4 \alpha_{1,1}T (A(t) h_{11,1}(t)+ B(t;v_1(t), \bar v_1(t))-\alpha_{1,1} F(u_1(t)))+ 4 \alpha _{2,1} T F(u_1(t)) & = & 0,\ t \in [0,1], \\
 h_{22,1}(1)-h_{22,1}(0) & = & 0,\\
 \int_0^1 \langle \varphi^*_1(t),h_{22,1}(t)\rangle dt & = & 0,
\end{array}
\right. 
\end{equation*}
where $h_{22}(\tau) = h_{22,1}(\tau/T)/T^2$.

Therefore, we can compute the critical coefficient $e$:
\begin{eqnarray*}
e&=&\frac{1}{12T^2} \int_0^1 \langle v^*_1(t), E(t;v_1(t),v_1(t),v_1(t),\bar{v}_1(t),\bar v_1(t))+ D(t;v_1(t),v_1(t),v_1(t),h_{02,1}(t))\nonumber \\
&&+6 D(t;v_1(t),v_1(t),\bar{v}_1(t),h_{11,1}(t))+3D(t;v_1(t),\bar{v}_1(t),\bar v_1(t),h_{20,1}(t)) \nonumber \\
&& + 6 C(t;v_1(t),h_{11,1}(t),h_{11,1}(t))+3 C(t;v_1(t),v_1(t),h_{12,1}(t))\nonumber\\
 &&+ 3 C(t;v_1(t), h_{02,1}(t), h_{20,1}(t)) + 6 C(t;\bar{v}_1(t),h_{11,1}(t), h_{20,1}(t))+6 C(t;v_1(t), \bar{v}_1(t), h_{21,1}(t))\nonumber\\
 &&+C(t;\bar{v}_1(t),\bar v_1(t), h_{30,1}(t)) \nonumber\\
 &&+3 B(t;h_{12,1}(t),h_{20,1}(t))+6 B(t;h_{11,1}(t), h_{21,1}(t))+ 3 B(t;v_1(t), h_{22,1}(t))\nonumber \\
  &&+B(t;h_{02,1}(t), h_{30,1}(t))+2B(t;\bar{v}_1(t), h_{31,1}(t))-12 \alpha _{2,1}A(t)v_1(t)\nonumber\\
  &&-6 \alpha_{1,1}(A(t) h_{21,1}(t)+2 B(t;v_1(t),h_{11,1}(t))+ C(t;v_1(t),v_1(t),\bar{v}_1(t))\nonumber \\
&& +B(t;h_{20,1}(t),\bar{v}_1(t)) -2 \alpha _{1,1} A(t) v_1(t) ) \rangle dt + \alpha_{2,1}i\frac{\omega}{T^2}+\alpha_{1,1}i \frac{c_1}{T^2} - \alpha_{1,1}^2 i \frac{\omega}{T^2}.
\end{eqnarray*}

We now impose the matrix solutions for the functions. We compute $v_{1M}$ by solving the discretization of (\ref{eq:EigenFunc_GNSdiscr})
\begin{eqnarray*}
\left[\begin{array}{cc}
\begin{array}{c}
(D-TA(t)+i\omega T L)_{C \times M}\\
\delta_0-\delta_1
\end{array} & p_2\\
q_2^{\rm H} & 0
\end{array}
\right]
\left[\begin{array}{c}
v_{1,1M}\\
a
\end{array}
\right]
= \left[\begin{array}{c}
0_{C\times 1}\\0_{n\times 1}\\1
\end{array}\right],
\end{eqnarray*}
with $a=0$, where $q_2$ is the normalized right null-vector of the complex matrix $K=\allowbreak\left[\begin{array}{c}(D-TA(t)+i\omega T L)_{C \times M}\\ \delta_0-\delta_1\end{array}\right]$ and $p_2$ the normalized right null-vector of $K^{\rm H}$. This vector is then rescaled so that $\int_0^1<v_1(t),v_1(t)>dt=1$. For the computation of $\varphi_{1M}$, we use Proposition \ref{Proposition1} to obtain
\begin{eqnarray*}
\begin{bmatrix}(\varphi^*_{1})^{\rm T}_W & a \end{bmatrix}
\left[\begin{array}{cc}
\begin{array}{c}
(D-TA(t))_{C \times M}\\
\delta_0-\delta_1
\end{array} & p\\
q^{\rm T} & 0
\end{array}
\right]
=  \begin{bmatrix} 0_{M\times 1} & 1\end{bmatrix},\end{eqnarray*}
where $a$ equals zero. We then approximate $I = \int_0^1 {\langle \varphi_1^*(t),F(u_1(t)) \rangle dt}$ by $I_1=(\varphi^*_{1})_{W_1}^{\rm T}L_{C \times M}g_M$ and we rescale $\varphi^*_{1W}$ so that $I_1=1$.

For the computation of $v_1^*$ we apply Proposition \ref{Proposition3} from the Appendix. Since $v_1^*$ lies in the kernel of the over there defined $\phi_2$, we have that
$\left[\begin{array}{c}v_1^*\\ v_1^*(0)\end{array}\right]\bot~\phi_1({\cal C}^1([0,1],\allowbreak {\mathbb C^n}))$. The eigenfunction $v_1^*$ is thus computed by solving
\begin{eqnarray*}
\begin{bmatrix}(v_{1}^*)_W^{\rm H} & a\end{bmatrix}
\left[\begin{array}{cc}
\begin{array}{c}
(D-TA(t)+i\omega T L)_{C \times M}\\
\delta_0-\delta_1
\end{array} & p_2\\
q_2^{\rm H} & 0
\end{array}
\right]
=  \begin{bmatrix}0_{M\times 1} & 1\end{bmatrix}.\end{eqnarray*}
We then approximate $I = \int_0^1 {\langle v_1^*(t),v_1(t) \rangle dt}$ by $I_1=(v^*_{1})_{W_1}^{\rm H}L_{C \times M}v_{1M}$ and we rescale $v^*_{1W}$ so that $I_1=1$.

(\ref{eq:h20-GNSdiscr}) is approximated by
\begin{eqnarray*}
\left[\begin{array}{c}
(D-TA(t)+2\,i\omega T L)_{C \times M}\\
\delta_0-\delta_1 \\
\end{array}
\right]
h_{20,1M}
=  \left[\begin{array}{c}
    TB_{C \times M \times M}v_{1M}v_{1M}\\0_{n \times 1}
\end{array}\right]
.\end{eqnarray*}

The coefficient $\alpha_{1,1}$ can be approximated as
\begin{equation*}
 \alpha_{1,1}=(\varphi^*_{1})_{W_1}^{\rm T} B_{C \times M \times M}v_{1M}\bar v_{1M}
\end{equation*}
which gives then all the information to determine the real function $h_{11,1}$:
\begin{eqnarray*}
\left[\begin{array}{cc}
\begin{array}{c}
(D-TA(t))_{C \times M}\\
\delta_0-\delta_1
\end{array} & p\\
(\varphi^*_{1})_{W_1}^{\rm T}L_{C\times M} & 0
\end{array}
\right]
\left[\begin{array}{c}
h_{11,1M}\\
a
\end{array}\right]
=  \left[\begin{array}{c}
    TB_{C \times M \times M}v_{1M}\bar{v}_{1M}-\alpha_{1,1}Tg_C\\
    0_{n\times 1}\\0
\end{array}\right]
,\end{eqnarray*}
where $a$ equals zero.

Then an approximation for the rescaled normal form coefficient $c_1$ is given by
\begin{equation*}
 c_1=-\frac{i}{2} (v^*_{1})_{W_1}^{\rm H}(C_{C \times M \times M \times M}v_{1M}v_{1M}\bar{v}_{1M}+2
B_{C \times M \times M}v_{1M}h_{11,1M}+B_{C \times M \times M}\bar{v}_{1M}h_{20,1M} - 2 \alpha _{1,1} A_{C \times M} v_{1M}
)+\alpha_{1,1}\omega.
\end{equation*}

Next, we determine the third order coefficients of the center manifold expansion, namely
\begin{eqnarray*}
\left[\begin{array}{cc}
\begin{array}{c}
(D-TA(t)+i\omega T L)_{C \times M}\\
\delta_0-\delta_1
\end{array} & p_2\\
(v_{1}^*)_{W_1}^{\rm H}L_{C\times M} & 0
\end{array}
\right]
\left[\begin{array}{c}
h_{21,1M}\\a\end{array}\right]
=  \left[\begin{array}{c}
    \mbox{rhs}\\0_{n\times 1}\\0
\end{array} \right]
,\end{eqnarray*}
where
\begin{eqnarray*}
\mbox{rhs} &= &TC_{C \times M \times M \times M}v_{1M}v_{1M}\bar{v}_{1M}+2T B_{C \times M \times M}v_{1M}h_{11,1M}+T B_{C \times M \times M}h_{20,1M}\bar{v}_{1M}-2 i c_1 T L_{C \times M}v_{1M}\\
&&- 2 \alpha _{1,1}T (A_{C \times M}v_{1M}-i \omega L_{C \times M}v_{1M})
\end{eqnarray*}
and $a=0$, and
\begin{eqnarray*}
\left[\begin{array}{c}
(D-TA(t)+3i\omega T L)_{C \times M}\\
\delta_0-\delta_1
\end{array}
\right]
h_{30,1M}
=  \left[\begin{array}{c}
TC_{C \times M \times M \times M}v_{1M}v_{1M}v_{1M}+3T B_{C \times M \times M}v_{1M}h_{20,1M}\\
 0_{n\times 1}
\end{array}\right]
.\end{eqnarray*}

The approximation to (\ref{eq:h31-GNSdiscr}) is given by
\begin{eqnarray*}
\left[\begin{array}{c}
(D-TA(t)+2i\omega T L)_{C \times M}\\
\delta_0-\delta_1
\end{array}
\right]
h_{31,1M}
= \left[ \begin{array}{c}
 \mbox{rhs}\\0_{n\times 1} \end{array}\right]
,\end{eqnarray*}
with
\begin{eqnarray*}
\mbox{rhs} & = & TD_{C \times M \times M \times M \times M}v_{1M}v_{1M}v_{1M}\bar{v}_{1M}+3T C_{C \times M \times M \times M}v_{1M}v_{1M}h_{11,1M}  +3T C_{C \times M \times M \times M}v_{1M}\bar{v}_{1M}h_{20,1M}\\
&&+ 3 T B_{C \times M \times M}h_{11,1M}h_{20,1M}+3T B_{C \times M \times M}v_{1M}h_{21,1M} +TB_{C \times M \times M}\bar{v}_{1M}h_{30,1M}\\
&&-6 i c_1 T L_{C \times M} h_{20,1M}- 3 \alpha _{1,1}T (A_{C \times M}h_{20,1M}- 2 i \omega L_{C \times M}h_{20,1M}+ B_{C \times M \times M}v_{1M}v_{1M})
\end{eqnarray*}
while
\begin{multline*}
 \alpha_{2,1}=\frac{1}{4} (\varphi^*_{1})_{W_1}^{\rm T}(D_{C \times M \times M \times M\times M}v_{1M}v_{1M}\bar{v}_{1M}\bar v_{1M}+C_{C \times M \times M \times M}v_{1M}v_{1M}h_{02,1M}+ 4 C_{C \times M \times M \times M}v_{1M}\bar{v}_{1M}h_{11,1M} \\
  +C_{C \times M \times M \times M}\bar{v}_{1M}\bar v_{1M}h_{20,1M}+2B_{C \times M \times M}h_{11,1M}h_{11,1M}+2B_{C \times M \times M}v_{1M}h_{12,1M}+ B_{C \times M \times M}h_{02,1M}h_{20,1M}\\
  +2 B_{C \times M \times M}\bar{v}_{1M} h_{21,1M}- 4 \alpha_{1,1} (A_{C \times M}h_{11,1M}+ B_{C \times M \times M}v_{1M}\bar v_{1M}))+ \alpha_{1,1}^2.
\end{multline*}

Now, we are able to compute $h_{22,1}$ by
\begin{eqnarray*}
\left[\begin{array}{cc}
\begin{array}{c}
(D-TA(t))_{C \times M}\\
\delta_0-\delta_1
\end{array} & p\\
(\varphi^*_1)^{\rm T}_{W_1}L_{C \times M} & 0
\end{array}
\right]
\left[\begin{array}{c}
h_{22,1M}\\a \end{array}\right]
= \left[\begin{array}{c}
 \mbox{rhs}\\0_{n\times 1}\\0\end{array}\right]
,\end{eqnarray*}
with
\begin{eqnarray*}
\mbox{rhs} & = & TD_{C \times M \times M \times M\times M}v_{1M}v_{1M}\bar{v}_{1M}\bar v_{1M}+ TC_{C \times M \times M \times M}v_{1M}v_{1M}h_{02,1M}- 4 \alpha _{2,1} T g_C  \\
&&+ 4T C_{C \times M \times M \times M}v_{1M}\bar{v}_{1M}h_{11,1M}+2T B_{C \times M \times M}h_{11,1M}h_{11,1M}+2T B_{C \times M \times M}v_{1M}h_{12,1M} \\
 &&+TC_{C \times M \times M \times M}\bar{v}_{1M}\bar v_{1M}h_{20,1M}+T B_{C \times M \times M}h_{02,1M}h_{20,1M}+2T B_{C \times M \times M}\bar{v}_{1M}h_{21,1M} \\
  &&- 4 \alpha_{1,1}T ( A_{C \times M} h_{11,1M}+ B_{C \times M \times M}v_{1M}\bar v_{1M}-\alpha_1 g_C) 
\end{eqnarray*}
and $a=0$. For the fifth order coefficient of the normal form, we then obtain
\begin{eqnarray*}
e&=&\frac{1}{12T^2} (v^*_{1})_{W_1}^{\rm H}(E_{C \times M \times M \times M \times M \times M}v_{1M}v_{1M}v_{1M}\bar{v}_{1M}\bar v_{1M}+D_{C \times M \times M \times M\times M}v_{1M}v_{1M}v_{1M}h_{02,1M}\nonumber \\
&&+6 D_{C \times M \times M \times M\times M}v_{1M}v_{1M}\bar{v}_{1M}h_{11,1M}+3 D_{C \times M \times M \times M\times M}v_{1M}\bar{v}_{1M}\bar v_{1M}h_{20,1M} \nonumber \\
&& + 6 C_{C \times M \times M \times M}v_{1M}h_{11,1M}h_{11,1M}+3 C_{C \times M \times M \times M}v_{1M}v_{1M}h_{12,1M}\nonumber\\
 &&+ 3 C_{C \times M \times M \times M}v_{1M}h_{02,1M}h_{20,1M} + 6 C_{C \times M \times M \times M}\bar{v}_{1M}h_{11,1M}h_{20,1M}+6 C_{C \times M \times M \times M}v_{1M}\bar{v}_{1M}h_{21,1M}\nonumber\\
 &&+ C_{C \times M \times M \times M}\bar{v}_{1M}\bar v_{1M} h_{30,1M}+3 B_{C \times M \times M}h_{12,1M}h_{20,1M} \nonumber\\
 &&+6 B_{C \times M \times M}h_{11,1M}h_{21,1M}+ 3 B_{C \times M \times M}v_{1M}h_{22,1M}\\
  &&+B_{C \times M \times M}h_{02,1M} h_{30,1M}+2B_{C \times M \times M}\bar{v}_{1M}h_{31,1M}-12 \alpha _{2,1}A_{C \times M}v_{1M} \nonumber\\
  &&-6 \alpha_{1,1}(A_{C \times M}h_{21,1M}+2 B_{C \times M \times M}v_{1M}h_{11,1M}+ C_{C \times M \times M \times M}v_{1M}v_{1M}\bar{v}_{1M}\nonumber \\
&& +B_{C \times M \times M}h_{20,1M}\bar{v}_{1M} -2 \alpha _{1,1} A_{C \times M} v_{1M}))+\alpha_{2,1}i \frac{\omega}{T^2}+\alpha_{1,1}i\frac{c_1}{T^2}-\alpha_{1,1}^2i\frac{\omega}{T^2}.
\end{eqnarray*}

\subsection{Strong resonance 1:1 bifurcation}
Again, we rescale the computed quantities to the interval $[0,1]$. The linear BVPs (\ref{eq:R1_EigenFunc}) and (\ref{eq:R1_EigenFunc_2}) are replaced by
\begin{eqnarray*}
\left\{\begin{array}{rcl}
\dot{v}_{1,1}(t)-TA(t)v_{1,1}(t) - T F(u_1(t)) & = & 0,\ t \in [0,1], \\
v_{1,1}(1)-v_{1,1}(0) & = & 0,\\
\int_{0}^{1} {\langle v_{1,1}(t),F(u_1(t))\rangle dt} & = & 0,
\end{array}
\right.
\end{eqnarray*}
where $v_1(\tau) = v_{1,1}(\tau/T)$, and
\begin{eqnarray*}\label{eq:R1_EigenFunc2discr}
\left\{\begin{array}{rcl}
\dot{v}_{2,1}(t)-T A(t)v_{2,1}(t) + T v_{1,1}(t) & = & 0,\ t \in [0,1], \\
v_{2,1}(1) - v_{2,1}(0) & = & 0,\\
\int_{0}^{1} {\langle v_{2,1}(t),F(u_1(t))\rangle dt} & = & 0,
\end{array}
\right.
\end{eqnarray*}
with $v_2(\tau) = v_{2,1}(\tau/T)$, respectively.

The rescaling of the adjoint eigenfunction, defined by (\ref{eq:R1_AdjEigenFunc}), gives us
\begin{equation*}
\left\{\begin{array}{rcl}
\dot{\varphi}^*_1(t)+T A^{\rm T}(t)\varphi^*_1(t) & = & 0,\ t \in [0,1], \\
\varphi^*_1(1)-\varphi^*_1(0) & = & 0,\\
\int_{0}^{1} {\langle \varphi^*_1(t),v_{2,1}(t) \rangle dt} - 1 & = & 0,
\end{array}
\right.
\end{equation*}
where $\varphi^*(\tau) = \varphi^*_1(\tau/T)/T$. This leads to the expression
\begin{eqnarray*}
 a=\frac{1}{2}\int_0^1\langle\varphi^*_1(t), 2A(t) v_{1,1}(t)+B(t;v_{1,1}(t),v_{1,1}(t)) \rangle dt.
\end{eqnarray*}

Definition (\ref{eq:R1_AdjGenEigenFunc}) of the first generalized adjoint eigenfunction is rescaled as
\begin{equation*}
\left\{\begin{array}{rcl}
\dot{v}_{1,1}^*(t)+T A^{\rm T}(t)v_{1,1}^*(t) -T\varphi^*_1(t)& = & 0,\ t \in [0,1], \\
v_{1,1}^*(1)-v_{1,1}^*(0) & = & 0,\\
\int_{0}^{1} {\langle v_{1,1}^*(t),v_{2,1}(t) \rangle dt} & = & 0,
\end{array}
\right.
\end{equation*}
with $v_1^*(\tau) = v_{1,1}^*(\tau/T)/T$, which then gives all the information we need to compute the critical coefficient $b$
\begin{equation}
 b= \int_0^1\langle \varphi_1^*(t),B(t;v_{1,1}(t), v_{2,1}(t))+A(t) v_{2,1}(t)\rangle dt + \int_0^1\langle v_{1,1}^*(t), 2 A(t)v_{1,1}(t)+B(t;v_{1,1}(t),v_{1,1}(t))\rangle dt.\label{eq:r1_bdiscr}
\end{equation}

The generalized eigenfunctions are computed as 
\begin{eqnarray*}
\left[\begin{array}{cc}
\begin{array}{c}
(D-TA(t))_{C \times M}\\
\delta_0-\delta_1
\end{array} & p\\
g_{W_1}^{\rm T} L_{C \times M} & 0
\end{array}
\right]
\left[\begin{array}{c}
v_{1,1M}\\
a
\end{array}\right]
=  \left[\begin{array}{c}
Tg_C\\
    0_{n \times 1}\\
    0
\end{array}\right]\end{eqnarray*}
and
\begin{eqnarray*}
\left[\begin{array}{cc}
\begin{array}{c}
(D-TA(t))_{C \times M} \\
\delta_0-\delta_1 \\
\end{array} & p\\
g_{W_1}^{\rm T} L_{C \times M} & 0
\end{array}
\right]
\left[\begin{array}{c}
v_{2,1M}\\
a
\end{array}\right]
= \left[\begin{array}{c}
    -Tv_{1,1C}\\
    0_{n \times 1}\\
    0
\end{array}\right],\end{eqnarray*}
with $p$ the left null-vector, as before.

As always, we compute the adjoint eigenfunctions by means of  the transposed of the usual matrix, and obtain here $\varphi_{1W}^*$ instead of $\varphi_{1M}^*$. Formally, the computation of $\varphi_{1W}^*$ is based on Proposition \ref{Proposition1} from the appendix and as in the cusp of cycles case we obtain then
$$
\begin{bmatrix} (\varphi_{1}^*)_W^{\rm T} & a\end{bmatrix}\left[\begin{array}{cc}
\begin{array}{c}
(D-TA(t))_{C \times M} \\
\delta_0-\delta_1\\
\end{array}
& p\\
q^{\rm T} & 0
\end{array}
\right] = \begin{bmatrix}0_{M\times 1} & 1\end{bmatrix},$$
where $a$ equals zero. We approximate $I = \int_{0}^{1} {\langle \varphi_1^*(t),v_{2,1}(t) \rangle dt}$ by $I_1 = (\varphi_{1}^*)_{W_1}^{\rm T}\allowbreak L_{C\times M}v_{2,1M}$ and rescale $\varphi_{1W}^*$ to ensure that $I_1 = 1$.

Having found $v_{1,1M}$ and $\varphi_{1W}^*$, the first normal form coefficient of interest can be computed as
\begin{eqnarray*}
 a=\frac{1}{2}(\varphi^*_{1})_{W_1}^{\rm T}(2A_{C\times M}v_{1,1M}+B_{C\times M \times M}v_{1,1M}v_{1,1M}).
\end{eqnarray*}

Now we still need $v_{1,1}^*$. The technique of Proposition \ref{Proposition4} from the appendix is used to obtain $(v_{1,1}^*)_{W}$, namely
$$\begin{bmatrix}(v_{1,1}^*)_{W}^{\rm T} & a\end{bmatrix}\left[\begin{array}{cc}
(D-TA(t))_{C \times M} & v_{2,1C}\\
\delta_0-\delta_1 & 0_{n\times 1}\\
q^{\rm T} & 0
\end{array}
\right] = \begin{bmatrix}-T(\varphi_{1}^*)_{W_1}^{\rm T}L_{C\times M} & 0 \end{bmatrix},$$
where $a = 0$.

Finally (\ref{eq:r1_bdiscr}) is approximated by
\begin{equation*}
 b= (\varphi_{1}^*)_{W_1}^{\rm T}(B_{C \times M \times M}v_{1,1M}v_{2,1M}+A_{C \times M}v_{2,1M}) + (v_{1,1}^*)_{W_1}^{\rm T}(2 A_{C \times M}v_{1,1M}+B_{C \times M \times M}v_{1,1M}v_{1,1M}).
\end{equation*}

\subsection{Strong resonance 1:2 bifurcation}
As before, we rescale the computed quantities to the interval $[0,1]$. The linear BVPs (\ref{eq:EigenFunc_12C}) and (\ref{eq:EigenFunc_12C_2}) are replaced by
\begin{eqnarray} \label{eq:EigenFunc_12Cdiscr}
&&\left\{\begin{array}{rcl}
\dot{v}_{1,1}(t)-T A(t)v_{1,1}(t) & = & 0,\ t \in [0,1], \\
v_{1,1}(1)+v_{1,1}(0) & = & 0,\\
\int_{0}^{1} {\langle v_{1,1}(t),v_{1,1}(t)\rangle dt} -1 & = & 0,\\
\end{array}
\right.
\end{eqnarray}
with $v_1(\tau) = v_{1,1}(\tau/T)/\sqrt{T}$, and
\begin{eqnarray*}
\left\{\begin{array}{rcl}
\dot{v}_{2,1}(t)-TA(t)v_{2,1}(t) + Tv_{1,1}(t) & = & 0,\ t \in [0,1], \\
v_{2,1}(1) + v_{2,1}(0) & = & 0,\\
\int_{0}^{1} {\langle v_{2,1}(t),v_{1,1}(t)\rangle dt} & = & 0,\\
\end{array}
\right.
\end{eqnarray*}
where $v_2(\tau) = v_{2,1}(\tau/T)/\sqrt{T}$.

The rescaling of the adjoint eigenfunction gives us
\begin{equation*}
\left\{\begin{array}{rcl}
\dot{\varphi}^*_1(t)+T A^{\rm T}(t)\varphi^*_1(t) & = & 0,\ t \in [0,1], \\
\varphi^*_1(1)-\varphi^*_1(0) & = & 0,\\
\int_{0}^{1} {\langle \varphi^*_1(t),F(u_1(t)) \rangle dt} - 1 & = & 0,
\end{array}
\right.
\end{equation*}
where $\varphi^*(\tau) = \varphi^*_1(\tau/T)/T$, so we can compute
\begin{equation*}
 \alpha_1 = \frac{1}{2}\int_0^{1} \langle \varphi^*_1(t), B(t;v_{1,1}(t),v_{1,1}(t)) \rangle\;dt,
\end{equation*}
with $\alpha_1=T\alpha$. With $\alpha_1$ defined in this way, let $h_{20,1}$ be the unique solution of the BVP
\begin{equation} \label{eq:h20_12Cdiscr}
 \left\{\begin{array}{rcl}
  \dot{h}_{20,1}(t)-T A(t) h_{20,1}(t) - T B(t;v_{1,1}(t),v_{1,1}(t))+2 \alpha_1 T F(u_1(t)) &=& 0,\ t \in [0,1],\\
  h_{20,1}(1)-h_{20,1}(0) & = & 0, \\
  \int_0^{1} \langle \varphi^*_1(t), h_{20,1}(t) \rangle\;dt &=& \int_0^1 \langle \varphi_1^*(t),B(v_{1,1}(t),v_{2,1}(t))\rangle\;dt,
 \end{array}\right.
\end{equation}
where $h_{20}(\tau) = h_{20,1}(\tau/T)/T$.

The rescaling of the adjoint eigenfunction corresponding with multiplier $-1$ and the adjoint generalized eigenfunction gives us
\begin{eqnarray*}
&&\left\{\begin{array}{rcl}
 \dot{v}_{1,1}^*(t)+T A^{\rm T}(t)v_{1,1}^*(t) & = & 0,\ t \in [0,1], \\
 v_{1,1}^*(1)+v_{1,1}^*(0) & = & 0, \\
 \int_{0}^{1} {\langle v_{1,1}^*(t),v_{2,1}(t)\rangle dt} -1 & = & 0,
\end{array} \right.
\end{eqnarray*}
with $v_1^*(\tau)=v_{1,1}^*(\tau/T)/\sqrt{T}$, and 
\begin{equation*} 
 \left\{\begin{array}{rcl}
  \dot{v}_{2,1}^*(t)+T A^{\rm T}(t) v_{2,1}^*(t) - T v_{1,1}^*(t) &=& 0,\ t \in [0,1],\\
  v_{2,1}^*(1)+v_{2,1}^*(0) & = & 0, \\
  \int_0^{1} \langle v_{2,1}(t), v_{2,1}^*(t) \rangle\;dt &=&0,
 \end{array}\right.
\end{equation*}
where $v_{2}^*(\tau) = v_{2,1}^*(\tau/T)/\sqrt{T}$, respectively.

The rescaled critical coefficient is then
\begin{equation} \label{eq:a_12Cdiscr}
 a_1=\frac{1}{6}\int_0^{1} \langle v_{1,1}^*(t), C(t;v_{1,1}(t),v_{1,1}(t),v_{1,1}(t))+3 B(t;v_{1,1}(t),h_{20,1}(t)) - 6 \alpha_1  A(t) v_{1,1}(t) \rangle\;dt,
\end{equation}
with $a_1=Ta$.

We replace (\ref{eq:h11_12C}) by
\begin{equation*} 
 \left\{\begin{array}{rcl}
  \dot{h}_{11,1}(t)-T A(t) h_{11,1}(t) - T B(t;v_{1,1}(t),v_{2,1}(t)) + T h_{20,1}(t) &=& 0,\ t \in [0,1],\\
  h_{11,1}(1)-h_{11,1}(0) & = & 0, \\
  \int_0^{1} \langle \varphi^*_1(t), h_{11,1}(t) \rangle\;dt &=&\frac{1}{2}\int_0^1 \langle \varphi_1^*(t),B(v_{2,1}(t),v_{2,1}(t))\rangle\;dt,
 \end{array}\right.
\end{equation*}
with $h_{11}(\tau) = h_{11,1}(\tau/T)/T$, to finally obtain that
\begin{multline*} 
 b=\frac{1}{2T}\int_0^{1} \langle v_{1,1}^*(t),-2\alpha_1A(t)v_{2,1}(t) + C(t;v_{1,1}(t),v_{1,1}(t),v_{2,1}(t))\\+ B(t;h_{20,1}(t),v_{2,1}(t))+ 2B(t;h_{11,1}(t),v_{1,1}(t))\rangle dt \\+ \frac{1}{2T}\int_0^{1} \langle v_{2,1}^*(t),C(t;v_{1,1}(t),v_{1,1}(t),v_{1,1}(t))+ 3B(t;v_{1,1}(t),h_{20,1}(t))-6\alpha_1A(t)v_{1,1}(t)\rangle dt.
\end{multline*}

We now come to the second part of the implementation details and define the matrix solutions for the several functions. We compute $v_{1,1M}$ by solving the discretization of (\ref{eq:EigenFunc_12Cdiscr})
\begin{eqnarray}\label{eq:EigenFunc_12Cimpl}
\left[\begin{array}{cc}
\begin{array}{c}
(D-TA(t))_{C \times M}\\
\delta_0+\delta_1
\end{array} & p_1\\
q_1^{\rm T}  & 0
\end{array}
\right]
\left[\begin{array}{c}
v_{1,1M}\\a\end{array}\right]
= \left[\begin{array}{c}
    0_{C\times 1}\\0_{n\times 1}\\1
\end{array}\right],
\end{eqnarray}
where $a = 0$. We then normalize $v_{1,1M}$ by requiring $\sum_{i=0}^{N-1}\sum_{j=0}^m\sigma_j\langle(v_{1,1M})_{i,j},\allowbreak(v_{1,1M})_{i,j}\rangle =1,$ where $\sigma_j$ is the Lagrange quadrature coefficient.

We have now obtained the value of $v_{1,1}$ in the mesh points. However, since $v_{1,1}$ is used in the integral condition for $v_{2,1}$, we have to transfer this vector to the collocation points and multiply it with the Gauss-Legendre weights and the lenghts of the corresponding intervals, giving us vector $(v_{1,1})_{W_1}$. Then, $v_{2,1M}$ can be found by solving the following system
\begin{eqnarray*}
\left[\begin{array}{cc}
\begin{array}{c}
(D-TA(t))_{C \times M}\\
\delta_0+\delta_1
\end{array} & p_1\\
(v_{1,1})_{W_1}^{\rm T} L_{C \times M} & 0
\end{array}
\right]
\left[\begin{array}{c}
v_{2,1M}\\ a
\end{array}\right]
= \left[\begin{array}{c}
    -Tv_{1,1C}\\
    0_{n\times 1}\\
    0
\end{array}
\right]
,\end{eqnarray*}
where $a$ equals zero.

The adjoint eigenfunction corresponding to the trivial eigenvalue is computed with an analogous matrix as in (\ref{eq:EigenFunc_12Cimpl}):
\begin{eqnarray*}
\begin{bmatrix}(\varphi_{1}^*)_W^{\rm T} & a_1\end{bmatrix}
\left[\begin{array}{cc}
\begin{array}{c}
(D-TA(t))_{C \times M}\\
\delta_0-\delta_1
\end{array} & p\\
q^{\rm T} & 0
\end{array}
\right]
=\begin{bmatrix}0_{M\times 1} & 1\end{bmatrix},
\end{eqnarray*}
while the adjoint eigenfunction corresponding to eigenvalue $-1$ can be found by
\begin{eqnarray*}
\begin{bmatrix}(v_{1,1}^*)_W^{\rm T} & a_2\end{bmatrix}
\left[\begin{array}{cc}
\begin{array}{c}
(D-TA(t))_{C \times M}\\
\delta_0+\delta_1
\end{array} & p_1\\
q_1^{\rm T} & 0
\end{array}
\right]
=\begin{bmatrix}0_{M\times 1} & 1\end{bmatrix},
\end{eqnarray*}
where $a_1$ and $a_2$ are equal to zero. $\varphi_{1W}^*$ is then rescaled to make sure that $(\varphi_{1}^*)_{W_1}^{\rm T}\allowbreak L_{C\times M}g_M=1$ and $v_{1,1W}^*$ so that $(v_{1,1}^*)_{W_1}^{\rm T}L_{C\times M}v_{2,1M}=1$.

Making use of Proposition \ref{Proposition5}, we obtain the adjoint generalized eigenfunction by solving
\begin{eqnarray*}
\begin{bmatrix}(v_{2,1}^*)_W^{\rm T} & a\end{bmatrix}
\left[\begin{array}{cc}
(D-TA(t))_{C \times M} & v_{2,1C}\\
\delta_0+\delta_1 & 0_{n\times 1}\\
q_1^{\rm T} & 0
\end{array}
\right]
=\begin{bmatrix}-T(v_{1,1}^*)_{W_1}^{\rm T}L_{C\times M} & 0\end{bmatrix}.
\end{eqnarray*}
Having found $\varphi_{1W_1}^*$ and $v_{1,1M}$, $\alpha_1$ can be computed as
\begin{equation*}
 \alpha_1 = \frac{1}{2}(\varphi^*_{1})_{W_1}^{\rm T}B_{C \times M \times M}v_{1,1M}v_{1,1M}.
\end{equation*}

Now, $h_{20,1M}$ is found by solving the discretization of (\ref{eq:h20_12Cdiscr}), namely,
\begin{eqnarray*}
\left[\begin{array}{cc}
\begin{array}{c}
(D-TA(t))_{C \times M} \\
\delta_0-\delta_1
\end{array}& p \\
(\varphi_{1}^*)_{W_1}^{\rm T} L_{C \times M} & 0
\end{array}
\right]
\left[\begin{array}{c}
h_{20,1M}\\
a
\end{array}\right]
= \left[\begin{array}{c}
    TB_{C \times M \times M}v_{1,1M}v_{1,1M}-2\alpha_1 T g_C\\
    0_{n \times 1}\\
    (\varphi_1^*)_{W_1}^{\rm T}B_{C\times M \times M}v_{1,1M}v_{2,1M}
\end{array}
\right]
,\end{eqnarray*}
and (\ref{eq:a_12Cdiscr}) is approximated by
\begin{equation*}
 a_1=\frac{1}{6}(v_{1,1}^*)_{W_1}^{\rm T}\left(C_{C \times M \times M \times M}v_{1,1M}v_{1,1M}v_{1,1M}+3 B_{C \times M \times M}v_{1,1M}h_{20,1M} - 6 \alpha_1  A_{C \times M} v_{1,1M}\right).
\end{equation*}

Next, $h_{11,1M}$ is found by
\begin{eqnarray*}
\left[\begin{array}{cc}
\begin{array}{c}
(D-TA(t))_{C \times M} \\
\delta_0-\delta_1\\
\end{array} & p\\
(\varphi_{1}^*)_{W_1}^{\rm T} L_{C \times M} & 0
\end{array}
\right]
\left[\begin{array}{c}
    h_{11,1M}\\
    a
\end{array}\right]
= \left[\begin{array}{c}
    TB_{C \times M \times M}v_{1,1M}v_{2,1M} - T h_{20,1C}\\
    0_{n \times 1}\\
    \frac{1}{2}(\varphi_1^*)_{W_1}^{\rm T}B_{C\times M \times M}v_{2,1M}v_{2,1M}
\end{array}
\right]
.\end{eqnarray*}

Finally, we obtain
\begin{multline*} 
 b=\frac{1}{2T} (v_{1,1}^*)_{W_1}^{\rm T}(-2\alpha_1A_{C \times M} v_{2,1M} \\+  C_{C \times M \times M \times M}v_{1,1M}v_{1,1M}v_{2,1M}+ B_{C \times M \times M}h_{20,1M}v_{2,1M}+ 2B_{C \times M \times M}h_{11,1M}v_{1,1M})\\
 +\frac{1}{2T}(v_{2,1}^*)_{W_1}^{\rm T}(C_{C \times M \times M \times M}v_{1,1M}v_{1,1M}v_{1,1M}+ 3B_{C \times M \times M}v_{1,1M}h_{20,1M}-6\alpha_1A_{C\times M}v_{1,1M}).
\end{multline*}

\subsection{Strong resonance 1:3 bifurcation}
As before, we rescale the computed quantities to the interval $[0,1]$. The BVPs for the eigenfunction and its adjoint belonging to eigenvalue $e^{i \frac{2\pi}{3}}$ are replaced by
\begin{eqnarray}
&&\left\{\begin{array}{rcl}
\dot{v}_1(t)-TA(t)v_1(t) & = & 0,\ t \in [0,1], \\
v_1(1)- e^{i \frac{2\pi}{3}} v_1(0) & = & 0,\\
\int_{0}^{1} {\langle v_1(t),v_1(t)\rangle dt} - 1 & = & 0,
\end{array}
\right.\label{eq:EigenFunc_R3discr}
\end{eqnarray}
with $v(\tau) = v_{1}(\tau/T)/\sqrt{T}$, and
\begin{eqnarray*}
\left\{\begin{array}{rcl}
 \dot{v}_1^*(t)+TA^{\rm T}(t)v_1^*(t) & = & 0,\ t \in [0,1], \\
 v_1^*(1)-e^{i \frac{2\pi}{3}} v_1^*(0) & = & 0,\\
 \int_{0}^{1} {\langle v_1^*(t), v_1(t)\rangle dt} - 1 & = & 0,
\end{array}
\right. 
\end{eqnarray*}
where $v^*(\tau) = v^*_{1}(\tau/T)/\sqrt{T}$. The rescaling of the adjoint eigenfunction corresponding to the trivial multiplier gives
\begin{eqnarray}\label{phi*R3}
\left\{\begin{array}{rcl}
 \dot{\varphi}_1^*(t)+TA^{\rm T}(t)\varphi_1^*(t) & = & 0,\ t \in [0,1], \\
 \varphi_1^*(1)-\varphi_1^*(0) & = & 0,\\
 \int_{0}^{1} {\langle \varphi_1^*(t), F(u_1(t))\rangle dt} - 1 & = & 0,
\end{array}
\right.
\end{eqnarray}
with $\varphi^*(\tau) = \varphi^*_{1}(\tau/T)/T$. 

These eigenfunctions make it already possible to compute the following $2$ rescaled normal form coefficients
\begin{eqnarray}
	\alpha_{1,1}= \int_0^1 \langle \varphi_1^*(t), B(v_1(t),\bar v_1(t)) \rangle dt,
	\label{alpha11R3}
\end{eqnarray}
where $\alpha_{1,1} = T\alpha_1$, and
$$b_1=\frac{1}{2} \int_0^1 \langle v^*_1(t), B(\bar v_1(t),\bar v_1(t)) \rangle dt,$$
with $b_1 = \sqrt{T}b$.

The rescaled second order functions in the center manifold expansion are solutions of
\begin{equation*}
\left\{\begin{array}{rcl}
 \dot{h}_{20,1}(t) - TA(t) h_{20,1}(t) - T B(v_1(t),v_1(t)) +  2 \bar{b}_1 T \bar{v}_1(t) & = & 0,\ t \in [0,T], \\
 h_{20,1}(1)-e^{i\frac{4 \pi}{3}}h_{20,1}(0) & = & 0,\\
 \int_{0}^{1} {\langle \bar v_1^*(t), h_{20,1}(t)\rangle dt} & = & 0,
\end{array}
\right.
\end{equation*}
with $h_{20}(\tau) = h_{20,1}(\tau/T)/T$, and 
\begin{equation}\label{h111R3}
\left\{\begin{array}{rcl}
 \dot{h}_{11,1}(t)-TA(t) h_{11,1}-T B(v_1(t),\bar{v}_1(t))+\alpha _{1,1} T F(u_1(t)) & = & 0,\ t \in [0,1], \\
 h_{11,1}(1)-h_{11,1}(0) & = & 0,\\
 \int_{0}^{1} {\langle \varphi^*_1(t), h_{11,1}(t)\rangle dt} & = & 0,
\end{array}
\right.
\end{equation}
with $h_{11}(\tau)=h_{11,1}(\tau/T)/T$. This all results then in
$$c = \frac{1}{2T}\int_0^1 \langle v^*_1(t), C(v_1(t),v_1(t),\bar{v}_1(t))+2 B(v_1(t), h_{11,1}(t))+ B(\bar{v}_1(t), h_{20,1}(t)) - 2 \alpha_{1,1} A v_1(t) \rangle dt.$$

We now come to the implementation details in MatCont. Eigenfunction $v_1$, determined by (\ref{eq:EigenFunc_R3discr}), is computed by
\begin{eqnarray*}
\left[\begin{array}{cc}
\begin{array}{c}
(D-TA(t))_{C \times M}\\
\delta_0 -e^{-i \frac{2\pi}{3}}\delta_1
\end{array} & p_3\\
q_3^{\rm H} & 0
\end{array}
\right] \left[\begin{array}{c}v_{1M}\\a\end{array}\right]
=\left[\begin{array}{c}
    0_{C\times 1}\\
    0_{n\times 1}\\
    1
\end{array}\right],
\end{eqnarray*}
with $a=0$. We then normalize $v_{1M}$ by requiring $\sum_{i=0}^{N-1}\sum_{j=0}^m\sigma_j\langle(v_{1M})_{i,j},(v_{1M})_{i,j}\rangle =1,$ where $\sigma_j$ is the Lagrange quadrature coefficient. $q_3$ is the normalized right null-vector of $K=\left[\begin{array}{c}(D-TA(t))_{C \times M}\\ \delta_0-e^{-i\theta}\delta_1\end{array}\right]$ and $p_3$ the normalized right null-vector of $K^{\rm H}$, with $\theta=\frac{2\pi}{3}$.

To compute the adjoint eigenfunction $v_1^*$, we apply Proposition \ref{Proposition6} from the appendix with $\theta=\frac{2\pi}{3}$. Since $v_1^*\in {\rm Ker}(\phi_2)$, this function can be obtained by solving
\begin{eqnarray}\label{v11*R3discr}
\begin{bmatrix} (v_1^*)^{\rm H}_W & a\end{bmatrix}
\left[\begin{array}{cc}
\begin{array}{c}
(D-TA(t))_{C \times M}\\
\delta_0 -e^{-i \frac{2\pi}{3}}\delta_1
\end{array} & p_3\\
q_3^{\rm H} & 0
\end{array}
\right]
=\begin{bmatrix} 0_{M\times 1} & 1\end{bmatrix}.
\end{eqnarray}
$v_{1W}^*$ is rescaled such that $(v_{1}^*)_{W_1}^{\rm H}L_{C\times M}v_{1M} = 1$. The adjoint eigenfunction corresponding to eigenvalue $1$ is discretized by
\begin{eqnarray}\label{phi*R3disc}
\begin{bmatrix} (\varphi_1^*)^{\rm T}_W & a\end{bmatrix}
\left[\begin{array}{cc}
\begin{array}{c}
(D-TA(t))_{C \times M}\\
\delta_0 -\delta_1
\end{array} & p\\
q^{\rm T} & 0
\end{array}
\right]
=\begin{bmatrix} 0_{M\times 1} & 1\end{bmatrix},
\end{eqnarray}
where $a$ equals zero. $\varphi_{1W}^*$ is rescaled so that $(\varphi_{1}^*)_{W_1}^{\rm T}L_{C\times M}g_M = 1$.

The normal form coefficients $\alpha_{1,1}$ and $b_1$ become
\begin{eqnarray}
	\alpha_{1,1}=(\varphi_1^*)^{\rm T}_{W_1} B_{C\times M \times M}v_{1M}\bar{v}_{1M}	
	\label{alpha11R3discr}
\end{eqnarray}
and
$$b_1 = \frac{1}{2}(v_1^*)^{\rm H}_{W_1} B_{C\times M \times M}\bar{v}_{1M}\bar{v}_{1M}.$$

By computing first the complex conjugate of $h_{20,1}$ we can use the same matrix as in (\ref{v11*R3discr}), except for the last line which represents the integral condition, to get
\begin{eqnarray*}
\left[\begin{array}{cc}
\begin{array}{c}
(D-TA(t))_{C \times M}\\
\delta_0 -e^{-i \frac{2\pi}{3}}\delta_1
\end{array} & p_3\\
(v_1^*)^{\rm H}_{W_1}L_{C \times M} & 0
\end{array}
\right] \left[\begin{array}{c}\bar{h}_{20,1M}\\a\end{array}\right]
=\left[\begin{array}{c}
    T B_{C\times M \times M}\bar v_{1M}\bar v_{1M}-2b_1 T v_{1C} \\
    0_{n\times 1}\\
    0
\end{array}\right].
\end{eqnarray*}

To obtain the discretization of $h_{11,1}$, the following system is solved
\begin{eqnarray}\label{h111R3discr}
\left[\begin{array}{cc}
\begin{array}{c}
(D-TA(t))_{C \times M}\\
\delta_0 -\delta_1
\end{array} & p\\
(\varphi_1^*)^{\rm T}_{W_1}L_{C\times M} & 0
\end{array}
\right] \left[\begin{array}{c}h_{11,1M}\\a\end{array}\right]
=\left[\begin{array}{c}
    T B_{C\times M \times M}v_{1M}\bar{v}_{1M}-\alpha_{1,1} T g_C \\
    0_{n\times 1}\\
    0
\end{array}\right],
\end{eqnarray}
which gives us
$$c =  \frac{1}{2T} (v^*_1)^{\rm H}_{W_1}(C_{C\times M\times M \times M}v_{1M}v_{1M}\bar{v}_{1M}+2 B_{C\times M \times M}v_{1M}h_{11,1M}+ B_{C\times M \times M}\bar{v}_{1M}h_{20,1M} - 2 \alpha_{1,1} A_{C\times M} v_{1M}).$$

\subsection{Strong resonance 1:4 bifurcation}
The eigenfunction and the adjoint eigenfunction corresponding to multiplier $e^{i\frac{\pi}{2}}$ are given by the solution of
\begin{eqnarray*}
&&\left\{\begin{array}{rcl}
\dot{v}_1(t)-TA(t)v_1(t) & = & 0,\ t \in [0,1], \\
v_1(1)- e^{i \frac{\pi}{2}} v_1(0) & = & 0,\\
\int_{0}^{1} {\langle v_1(t),v_1(t)\rangle dt} - 1 & = & 0,
\end{array}
\right.
\end{eqnarray*}
with $v(\tau) = v_{1}(\tau/T)/\sqrt{T}$ and
\begin{eqnarray*}
\left\{\begin{array}{rcl}
 \dot{v}_1^*(t)+TA^{\rm T}(t)v_1^*(t) & = & 0,\ t \in [0,1], \\
 v_1^*(1)-e^{i \frac{\pi}{2}} v_1^*(0) & = & 0,\\
 \int_{0}^{1} {\langle v_1^*(t), v_1(t)\rangle dt} - 1 & = & 0,
\end{array}
\right.
\end{eqnarray*}
where $v^*(\tau) = v^*_{1}(\tau/T)/\sqrt{T}$, respectively. $\varphi^*, h_{11}$ and $\alpha_1$ are replaced by $\varphi_1^*, h_{11,1}$ and $\alpha_{1,1}$, defined by (\ref{phi*R3}), (\ref{h111R3}) and (\ref{alpha11R3}), respectively.

The rescaling of (\ref{eq:h20-R4}) gives
\begin{equation*}
\left\{\begin{array}{rcl}
 \dot h_{20,1}(t)-TA(t) h_{20,1}(t) - T B(v_1(t),v_1(t)) & = & 0,\ t \in [0,1], \\
 h_{20,1}(1)+h_{20,1}(0) & = & 0,
\end{array}
\right.
\end{equation*}
with $h_{20}(\tau)=h_{20,1}(\tau/T)/T$. 

The critical coefficients are then given by
\begin{equation*}
 \bar c = \frac{1}{2T}\int_0^1 \langle \bar v^*_1(t), C(v_1(t),\bar{v}_1(t),\bar{v}_1(t))+B(v_1(t),h_{02,1}(t))+2 B(\bar{v}_1(t),h_{11,1}(t)) - 2 \alpha _{11} A(t) \bar{v}_1(t) \rangle dt
\end{equation*}
and
\begin{equation*}
 d = \frac{1}{6T}\int_0^1 \langle v^*_1(t), C(\bar{v}_1(t),\bar{v}_1(t),\bar{v}_1(t))+3 B(\bar{v}_1(t),h_{02,1}(t)) \rangle
 dt.
\end{equation*}

The eigenfunction and its adjoint, corresponding to the complex eigenvalue, are discretized by
\begin{eqnarray*}
\left[\begin{array}{cc}
\begin{array}{c}
(D-TA(t))_{C \times M}\\
\delta_0 -e^{-i \frac{\pi}{2}}\delta_1
\end{array} & p_3\\
q_3^{\rm H} & 0
\end{array}
\right] \left[\begin{array}{c}v_{1M}\\a\end{array}\right]
=\left[\begin{array}{c}
    0_{C\times 1}\\
    0_{n\times 1}\\
    1
\end{array}\right]
\end{eqnarray*}
and
\begin{eqnarray*}
\begin{bmatrix} (v_1^*)^{\rm H}_W & a\end{bmatrix}
\left[\begin{array}{cc}
\begin{array}{c}
(D-TA(t))_{C \times M}\\
\delta_0 -e^{-i \frac{\pi}{2}}\delta_1
\end{array} & p_3\\
q_3^{\rm H} & 0
\end{array}
\right]
=\begin{bmatrix} 0_{M\times 1} & 1\end{bmatrix},
\end{eqnarray*}
respectively. For the computation of the adjoint function we have applied Proposition \ref{Proposition6} with $\theta = \frac{\pi}{2}$. We then normalize $v_{1M}$ by requiring $\sum_{i=0}^{N-1}\sum_{j=0}^m\sigma_j\langle(v_{1M})_{i,j},\allowbreak(v_{1M})_{i,j}\rangle =1,$ where $\sigma_j$ is the Lagrange quadrature coefficient. $v_{1W}^*$ is rescaled so that $(v_{1}^*)_{W_1}^{\rm H}\allowbreak L_{C\times M}v_{1M}=1$.

(\ref{phi*R3disc}) and the corresponding normalization, (\ref{h111R3discr}) and (\ref{alpha11R3discr}) determine $\varphi_1^*, h_{11,1}$ and $\alpha_{1,1}$, respectively. An approximation to $h_{20}$ is obtained by
\begin{eqnarray*}
\left[\begin{array}{c}
(D-TA(t))_{C \times M}\\
\delta_0 +\delta_1
\end{array}
\right] h_{20,1M}
=\left[\begin{array}{c}
    T B_{C\times M \times M}v_{1M}v_{1M} \\
    0_{n\times 1}
\end{array}\right].
\end{eqnarray*}

We are now able to compute the two needed normal form coefficients:
$$\bar{c} =  \frac{1}{2T} (v^*_1)^{\rm T}_{W_1}(C_{C\times M\times M \times M}v_{1M}\bar{v}_{1M}\bar{v}_{1M}+ B_{C\times M \times M}v_{1M}h_{02,1M}+2 B_{C\times M \times M}\bar{v}_{1M}h_{11,1M} - 2 \alpha_{1,1} A_{C\times M} \bar{v}_{1M})$$
and
$$d =  \frac{1}{6T} (v^*_1)^{\rm H}_{W_1}(C_{C\times M\times M \times M}\bar{v}_{1M}\bar{v}_{1M}\bar{v}_{1M}+3B_{C\times M \times M}\bar{v}_{1M}h_{02,1M}).$$

\subsection{Fold-Flip bifurcation}
The rescaling of the eigenfunctions (\ref{eq:EigenFunc_FF}) and (\ref{eq:EigenFunc_FF_2}) gives us
\begin{eqnarray} \label{eq:EigenFunc_FFdiscr1}
&&\left\{\begin{array}{rcl}
\dot{v}_{1,1}(t)-T A(t)v_{1,1}(t) - T F(u_1(t)) & = & 0,\ t \in [0,1], \\
v_{1,1}(1)-v_{1,1}(0) & = & 0,\\
\int_{0}^{1} {\langle v_{1,1}(t),F(u_1(t))\rangle dt} & = & 0,\\
\end{array}
\right.
\end{eqnarray}
and
\begin{eqnarray} \label{eq:EigenFunc_FFdiscr2}
&&\left\{\begin{array}{rcl}
\dot{v}_{2,1}(t)-T A(t)v_{2,1}(t) & = & 0,\ t \in [0,1], \\
v_{2,1}(1) + v_{2,1}(0) & = & 0,\\
\int_{0}^{1} {\langle v_{2,1}(t),v_{2,1}(t)\rangle dt} - 1  & = & 0,\\
\end{array}
\right.
\end{eqnarray}
respectively, with $v_1(\tau) = v_{1,1}(\tau/T)$ and $v_{2}(\tau) = v_{2,1}(\tau/T)/\sqrt{T}$.

The rescaled adjoint eigenfunctions can be obtained by solving
\begin{eqnarray*}
&&\left\{\begin{array}{rcl}
 \dot{\varphi}^*_1(t)+T A^{\rm T}(t)\varphi^*_1(t) & = & 0,\ t \in [0,1], \\
 \varphi^*_1(1)-\varphi^*_1(0) & = & 0, \\
 \int_{0}^{1} {\langle \varphi^*_1(t),v_{1,1}(t) \rangle dt} -1 & = & 0,
 \end{array} \right. 
 \end{eqnarray*}
\begin{eqnarray*}
 &&\left\{\begin{array}{rcl}
 \dot{v}_{1,1}^*(t)+T A^{\rm T}(t) v_{1,1}^*(t)+T\varphi^*_1(t) & = & 0,\ t \in [0,1], \\
 v_{1,1}^*(1)-v_{1,1}^*(0) & = & 0, \\
 \int_{0}^{1} {\langle v_{1,1}^*(t),v_{1,1}(t) \rangle dt} & = & 0,
 \end{array} \right.
 \end{eqnarray*}
 \begin{eqnarray*}
 &&\left\{\begin{array}{rcl}
 \dot{v}_{2,1}^*(t)+TA^{\rm T}(t) v_{2,1}^*(t) & = & 0,\ t \in [0,1], \\
 v_{2,1}^*(1)+v_{2,1}^*(0) & = & 0, \\
 \int_{0}^{1} {\langle v_{2,1}^*(t),v_{2,1}(t)\rangle dt} -1 & = & 0,
 \end{array} \right. \\ \nonumber
\end{eqnarray*}
with $\varphi^*(\tau)=\varphi^*_1(\tau/T)/T, v_1^*(\tau) = v_{1,1}^*(\tau/T)/T$ and $v_{2}^*(\tau) = v^*_{2,1}(\tau/T)/\sqrt{T}$.

The two coefficients in front of the $\xi_1^2$-terms are given by
\begin{equation*}
 a_{20}=\frac{1}{2}\int_0^{1} \langle \varphi^*_1(t),B(v_{1,1}(t),v_{1,1}(t))+2 A(t) v_{1,1}(t) \rangle\; d t
\end{equation*}
and $\alpha_{20}= 0$.

The second order functions of the center manifolfd expansion are defined by the BVPs
\begin{equation*}
\left\{\begin{array}{rcl}
 \dot h_{20,1}(t)-TA(t) h_{20,1}(t)-TB(v_{1,1}(t),v_{1,1}(t))+2 a_{20} T v_{1,1}(t)+2 \alpha_{20} T F(u_1(t))&& \\
 - 2T A(t) v_{1,1}(t) - 2 TF(u_1(t)) & = & 0,\ t \in [0,1], \\
 h_{20,1}(1)-h_{20,1}(0) & = & 0, \\
 \int_{0}^{1} {\langle v_{1,1}^*(t),h_{20,1}(t)\rangle dt} & = & 0,
\end{array}
\right.
\end{equation*}
with $h_{20}(\tau)=h_{20,1}(\tau/T)$,
\begin{equation*}
\left\{\begin{array}{rcl}
 \dot h_{11,1}(t)-TA(t) h_{11,1}(t)-TB(v_{1,1}(t), v_{2,1}(t))+Tb_{11} v_{2,1}(t)-TA(t) v_{2,1}(t) & = & 0,\ t \in [0,1], \\
 h_{11,1}(1)+h_{11,1}(0) & = & 0, \\
 \int_{0}^{1} {\langle v_{2,1}^*(t),h_{11,1}(t)\rangle dt} & = & 0,
\end{array}
\right.
\end{equation*}
with $h_{11}(\tau)=h_{11,1}(\tau/T)/\sqrt{T}$ and
\begin{equation*}
\left\{\begin{array}{rcl}
 \dot h_{02,1}(t)-TA(t) h_{02,1}(t)-T B(v_{2,1}(t),v_{2,1}(t))+2a_{02,1} T v_{1,1}(t)+2 \alpha_{02,1} T F(u_1(t)) & = & 0,\ t \in [0,1], \\
 h_{02,1}(1)-h_{02,1}(0) & = & 0, \\
 \int_{0}^{1} {\langle v_{1,1}^*(t),h_{02,1}(t)\rangle dt} & = & 0,
\end{array}
\right.
\end{equation*}
with $h_{02}(\tau)=h_{02,1}(\tau/T)/T$, where
\begin{equation*}
 b_{11}=\int_0^{1} \langle v_{2,1}^*(t),B(v_{1,1}(t), v_{2,1}(t))+A(t) v_{2,1}(t) \rangle\; dt,
\end{equation*}
\begin{equation*}
 a_{02,1} = \frac{1}{2}\int_0^{1} \langle \varphi^*_1(t),B(v_{2,1}(t),v_{2,1}(t)) \rangle\; d t
\end{equation*}
with $a_{02,1}=Ta_{02}$ and $\alpha_{02}=0$.

The rescaling of the last four normal form coefficients of interest gives
\begin{eqnarray*}
 a_{30} = \frac{1}{6}\int_0^{1} \langle \varphi^*_1(t)&,& C(v_{1,1}(t),v_{1,1}(t),v_{1,1}(t))+3 B(h_{20,1},v_{1,1}(t))-6 a_{20} h_{20,1}(t)
 \\ \nonumber
 &&+3 (A(t) h_{20,1}(t)+B(v_{1,1}(t),v_{1,1}(t))) +6 (1-\alpha _{20}) A(t) v_{1,1}(t) \rangle\; d
 t - a_{20},
\end{eqnarray*}

\begin{eqnarray*}  \\\nonumber
 b_{21} = \frac{1}{2}\int_0^{1} \langle v_{2,1}^*(t), C(v_{1,1}(t),v_{1,1}(t),v_{2,1}(t))+B(h_{20,1}(t),v_{2,1}(t))+2 B(h_{11,1}(t),v_{1,1}(t)) -2 a_{20} h_{11,1}(t) &&\\ \nonumber
 -2 b_{11} h_{11,1}(t)  + 2 (A(t) h_{11,1}(t)+B(v_{1,1}(t), v_{2,1}(t))) +2 (1-\alpha _{20}) A(t) v_{2,1}(t) \rangle&\;& d t -b_{11},
\end{eqnarray*}

\begin{eqnarray*}  \\ \nonumber
 a_{12} = \frac{1}{2T}\int_0^{1} \langle \varphi^*_1(t),C(v_{1,1}(t),v_{2,1}(t),v_{2,1}(t))+B(h_{02,1}(t),v_{1,1}(t))+2 B(h_{11,1}(t),v_{2,1}(t))-2 b_{11} h_{02,1}(t)  && \\ \nonumber
 -2a_{02,1} h_{20,1}(t) +A(t) h_{02,1}(t)+ B(v_{2,1}(t),v_{2,1}(t)) - 2 \alpha_{02,1}A(t) v_{1,1}(t) \rangle&\;& d
t - \frac{a_{02,1}}{T}
\end{eqnarray*}
and
\begin{equation*}
 b_{03} = \frac{1}{6T} \int_0^{1} \langle v_{2,1}^*(t),C(v_{2,1}(t),v_{2,1}(t),v_{2,1}(t))+3 B(h_{02,1}(t),v_{2,1}(t))-6 a_{02,1} h_{11,1}(t)-6 \alpha _{02,1}A(t) v_{2,1}(t) \rangle\; dt.
\end{equation*}

The discretization of the functions determined in (\ref{eq:EigenFunc_FFdiscr1}) and (\ref{eq:EigenFunc_FFdiscr2}) is given by
\begin{eqnarray*}
\left[\begin{array}{cc}
\begin{array}{c}
(D-TA(t))_{C \times M}\\
\delta_0-\delta_1
\end{array} & p\\
g^{\rm T}_{W_1}L_{C \times M} & 0
\end{array}
\right]
\left[\begin{array}{c}
v_{1,1M}\\a_1 \end{array}\right]
=  \left[\begin{array}{c}
    Tg_C\\0_{n\times 1}\\0
\end{array}\right]
\end{eqnarray*}
and
\begin{eqnarray*}
\left[\begin{array}{cc}
\begin{array}{c}
(D-TA(t))_{C \times M}\\
\delta_0+\delta_1
\end{array} & p_1\\
q_1^{\rm T} & 0
\end{array}
\right]
\left[\begin{array}{c}
v_{2,1M}\\a_2 \end{array}\right]
=  \left[\begin{array}{c}
    0_{C\times 1}\\0_{n\times 1}\\1
\end{array}\right],\end{eqnarray*}
with $a_1=a_2=0$.
We normalize $v_{21M}$ by requiring $\sum_{i=0}^{N-1}\sum_{j=0}^m\sigma_j\langle(v_{2,1M})_{i,j},\allowbreak(v_{2,1M})_{i,j}\rangle =1,$ where $\sigma_j$ is the Lagrange quadrature coefficient. The implementation of the adjoint eigenfunctions is done by
\begin{eqnarray*}
\begin{bmatrix}(\varphi_1^*)_{W}^{\rm T} & a_1 \end{bmatrix} \left[\begin{array}{cc}
\begin{array}{c}
(D-TA(t))_{C \times M}\\
\delta_0-\delta_1
\end{array} & p\\
q^{\rm T} & 0
\end{array}
\right]
=  \begin{bmatrix}0_{M\times 1} & 1
\end{bmatrix},\end{eqnarray*}

\begin{eqnarray*}
\begin{bmatrix}(v_{1,1}^*)_{W}^{\rm T} & a_2 \end{bmatrix} \left[\begin{array}{cc}
(D-TA(t))_{C \times M} & (v_{1,1})_C\\
\delta_0-\delta_1 & 0_{n \times 1}\\
q^{\rm T} & 0
\end{array}
\right]
=  \begin{bmatrix}T(\varphi_1^*)_{W_1}^{\rm T}L_{C\times M} & 0
\end{bmatrix}\end{eqnarray*}
and
\begin{eqnarray*}
\begin{bmatrix}(v_{2,1}^*)_{W}^{\rm T} & a_3 \end{bmatrix} \left[\begin{array}{cc}
\begin{array}{c}
(D-TA(t))_{C \times M}\\
\delta_0+\delta_1
\end{array} & p_1\\
q_1^{\rm T} & 0
\end{array}
\right]
=  \begin{bmatrix}0_{M\times 1} & 1
\end{bmatrix},\end{eqnarray*}
with $a_1=a_2=a_3=0$. $\varphi_{1W}^*$ and $v_{2,1W}^*$ are then rescaled to ensure that $(\varphi_{1}^*)_{W_1}^{\rm T}\allowbreak L_{C\times M}v_{1,1M}= 1$ and $(v_{2,1}^*)_{W_1}^{\rm T}L_{C\times M}v_{2,1M}=1$.

The first needed normal form coefficient is given by
\begin{equation*}
 a_{20}=\frac{1}{2}(\varphi^*_1)_{W_1}^{\rm T}\left(B_{C\times M \times M}v_{1,1M}v_{1,1M}+2 A_{C\times M}v_{1,1M}\right).
\end{equation*}

The second order functions can be obtained by solving
\begin{eqnarray*}
\left[\begin{array}{cc}
\begin{array}{c}
(D-TA(t))_{C \times M}\\
\delta_0-\delta_1
\end{array} & p\\
(v_{1,1}^*)^{\rm T}_{W_1}L_{C \times M} & 0
\end{array}
\right]
\left[\begin{array}{c}
h_{20,1M}\\a \end{array}\right]
=  \left[\begin{array}{c}
    \mbox{rhs}\\0_{n\times 1}\\0
\end{array}\right]
,\end{eqnarray*}
with
$$\mbox{rhs} = TB_{C\times M \times M}v_{1,1M}v_{1,1M}-2 a_{20}T v_{1,1C} + 2T A_{C\times M} v_{1,1M}+ 2 Tg_C,$$
\begin{eqnarray*}
\left[\begin{array}{cc}
\begin{array}{c}
(D-TA(t))_{C \times M}\\
\delta_0+\delta_1
\end{array} & p_1\\
(v_{2,1}^*)^{\rm T}_{W_1}L_{C \times M} & 0
\end{array}
\right]
\left[\begin{array}{c}
h_{11,1M}\\a \end{array}\right]
=  \left[\begin{array}{c}
    \mbox{rhs}\\0_{n\times 1}\\0
\end{array}\right]
,\end{eqnarray*}
with
$$\mbox{rhs} = TB_{C\times M \times M}v_{1,1M}v_{2,1M}-Tb_{11} v_{2,1C}+TA_{C\times M}v_{2,1M},$$
and 
\begin{eqnarray*}
\left[\begin{array}{cc}
\begin{array}{c}
(D-TA(t))_{C \times M}\\
\delta_0-\delta_1
\end{array} & p\\
(v_{1,1}^*)^{\rm T}_{W_1}L_{C \times M} & 0
\end{array}
\right]
\left[\begin{array}{c}
h_{02,1M}\\a \end{array}\right]
=  \left[\begin{array}{c}
    \mbox{rhs}\\0_{n\times 1}\\0
\end{array}\right]
,\end{eqnarray*}
with
$$\mbox{rhs} = TB_{C\times M \times M}v_{2,1M}v_{2,1M}-2a_{02,1} T v_{1,1C}.$$
The in these functions needed coefficients are given by 
\begin{equation*}
 b_{11}= (v_{2,1}^*)_{W_1}^{\rm T}\left(B_{C\times M \times M}v_{1,1M}v_{2,1M}+A_{C\times M} v_{2,1M}\right)
\end{equation*}
and
\begin{equation*}
 a_{02,1} = \frac{1}{2}(\varphi^*_1)_{W_1}^{\rm T}B_{C\times M \times M}v_{2,1M}v_{2,1M}.
\end{equation*}

At last, we obtain
\begin{eqnarray*}
 a_{30} = \frac{1}{6} (\varphi^*_1)_{W_1}^{\rm T}\left( C_{C\times M \times M\times M}v_{1,1M}v_{1,1M}v_{1,1M}+3 B_{C\times M \times M}h_{20,1M}v_{1,1M}-6 a_{20} h_{20,1C} \right.\\ \nonumber
 +\left.3 (A_{C\times M}h_{20,1M}+B_{C\times M \times M}v_{1,1M}v_{1,1M}) +6 A_{C\times M}v_{1,1M}\right)- a_{20},
\end{eqnarray*}

\begin{eqnarray*}  \\\nonumber
 b_{21} = \frac{1}{2}(v_{2,1}^*)_{W_1}^{\rm T} \left(C_{C\times M \times M\times M}v_{1,1M}v_{1,1M}v_{2,1M}+B_{C\times M \times M}h_{20,1M}v_{2,1M}+2 B_{C\times M \times M}h_{11,1M}v_{1,1M} -2 a_{20} h_{11,1C}\right. &&\\ \nonumber
\left. -2 b_{11} h_{11,1C}  + 2 (A_{C\times M} h_{11,1M}+B_{C\times M \times M}v_{1,1M}v_{2,1M}) +2 A_{C\times M} v_{2,1M}  \right) -b_{11},
\end{eqnarray*}

\begin{eqnarray*}  \\ \nonumber
 a_{12} = \frac{1}{2T} (\varphi^*_1)_{W_1}^{\rm T}\left(C_{C\times M \times M\times M}v_{1,1M}v_{2,1M}v_{2,1M}+B_{C\times M \times M}h_{02,1M}v_{1,1M}+2 B_{C\times M \times M}h_{11,1M}v_{2,1M}-2 b_{11} h_{02,1C}  \right.&& \\ \nonumber
 \left.-2 a_{02,1} h_{20,1C} +A_{C\times M} h_{02,1M}+ B_{C\times M \times M}v_{2,1M}v_{2,1M}\right) - \frac{a_{02,1}}{T},
\end{eqnarray*}
and
\begin{equation*}
 b_{03} = \frac{1}{6T} (v_{2,1}^*)_{W_1}^{\rm T}\left(C_{C\times M \times M\times M}v_{2,1M}v_{2,1M}v_{2,1M}+3 B_{C\times M \times M}h_{02,1M}v_{2,1M}-6 a_{02,1}h_{11,1C}\right).
\end{equation*}

\section{Examples}
\label{Section:Examples}
The computations in this section are done with {\sc matcont} \cite{MATCONT}. In particular, the bordering methods from
\cite{DoGoKu} are used to continue the codim 1 bifurcations of
limit cycles in two parameters. The algorithms described above for
computing the normal form coefficients are also implemented in the
current version of {\sc matcont}.

\subsection{Periodic predator-prey model}
Our first model is a periodically forced predator-prey model, studied in
\cite{KuMuRi:91}, and described by the following differential
equations
\begin{equation} \label{eq:PreyPredator1}
\begin{cases}
\dot x = r \left(1-\frac{x}{K}\right) x - p(x,t) y, \\
\dot y = e p(x,t) y -d y,
\end{cases}
\end{equation}
where $x$ and $y$ are the numbers of individuals respectively of
prey and predator populations or suitable (but equivalent)
measures of density or  biomass. The parameters present in system
\eqref{eq:PreyPredator1} are the intrinsic growth rate $r$, the
carrying capacity $K$, the efficiency $e$ and the death rate $d$ of the
predator. The function $p(x,t)$ is a functional
response, for which the Holling type II is choosen, with constant
attack rate $a$ and half saturation $b(t)$ that varies
periodically with as period one year, i.e.
\[
 p(x,t) = \frac{a x}{b(t) +x }, \qquad b(t) = b_0 (1+\varepsilon \cos 2 \pi t).
\]

Notice that this system can be made autonomous by adding  the
following two differential equations
\[
 \dot{\alpha}=\mu\alpha-\omega\beta-(\alpha^2+\beta^2)\alpha, \qquad
 \dot{\beta}=\omega\alpha+\mu\beta-(\alpha^2+\beta^2)\beta.
\]
In fact, if $\mu$ is positive then this system has as asymptotic
behavior a stable circular limit cycle of radius $\mu$ with angular velocity $\omega$. Therefore, we can set $\mu=1$ and
$\omega=2 \pi$ in order to obtain the forcing function $\sin 2 \pi
t$ as the flow of one of the variables of this subsystem with a
particular phase shift that depends on the initial conditions.
The system becomes
\begin{equation} \label{eq:PreyPredator}
\begin{cases}
\dot x = r \left(1-\frac{x}{K}\right) x - \frac{a xy}{b_0(1+\epsilon \alpha)+x}, \\
\dot y = e \frac{axy}{b_0(1+\epsilon \alpha)+x} -d y,\\
\dot \alpha = \alpha - 2\pi \beta -(\alpha^2+\beta^2)\alpha,\\
\dot \beta = 2\pi\alpha + \beta -(\alpha^2+\beta^2)\beta.\\
\end{cases}
\end{equation}

We have chosen this system as first example since it allows us to
check if the computation of the $\alpha_i$ normal form
coefficients is correct. In fact, in a periodically forced system
the time of the normal form should not depend on the coordinate,
i.e. $d\tau /d t = 1$, and so all the $\alpha_i$ coefficients must
vanish. With fixed $r=2 \pi$, $K=e=1$, $a=4 \pi$ and $d=2 \pi$ we
perform a bifurcation analysis in the remaining parameters
$(\varepsilon, b_0)$ obtaining the bifurcation diagram reported in
Figure \ref{fig:forcedpreypredator}. Since the system is
periodically forced, no equilibria are present. The blue curve,
with label {\tt LPC2}, is a limit point of cycles bifurcation curve of the second iterate,
the magenta curves are supercritical Neimark-Sacker bifurcations
(of the first or of the second iterate, respectively labeled with
{\tt NS1} and {\tt NS2}) while the brown and green curves are
period-doubling bifurcations, brown when subcritical and green
when supercritical (with notation {\tt PD1}, {\tt PD2}, {\tt PD4} and {\tt PD8}).
\begin{figure}[htbp]
\centering
 \footnotesize
 \psfrag{BB}[][c]{$b_0$}
 \psfrag{epsilo}[][c]{$\varepsilon$}
 \includegraphics[width=\textwidth]{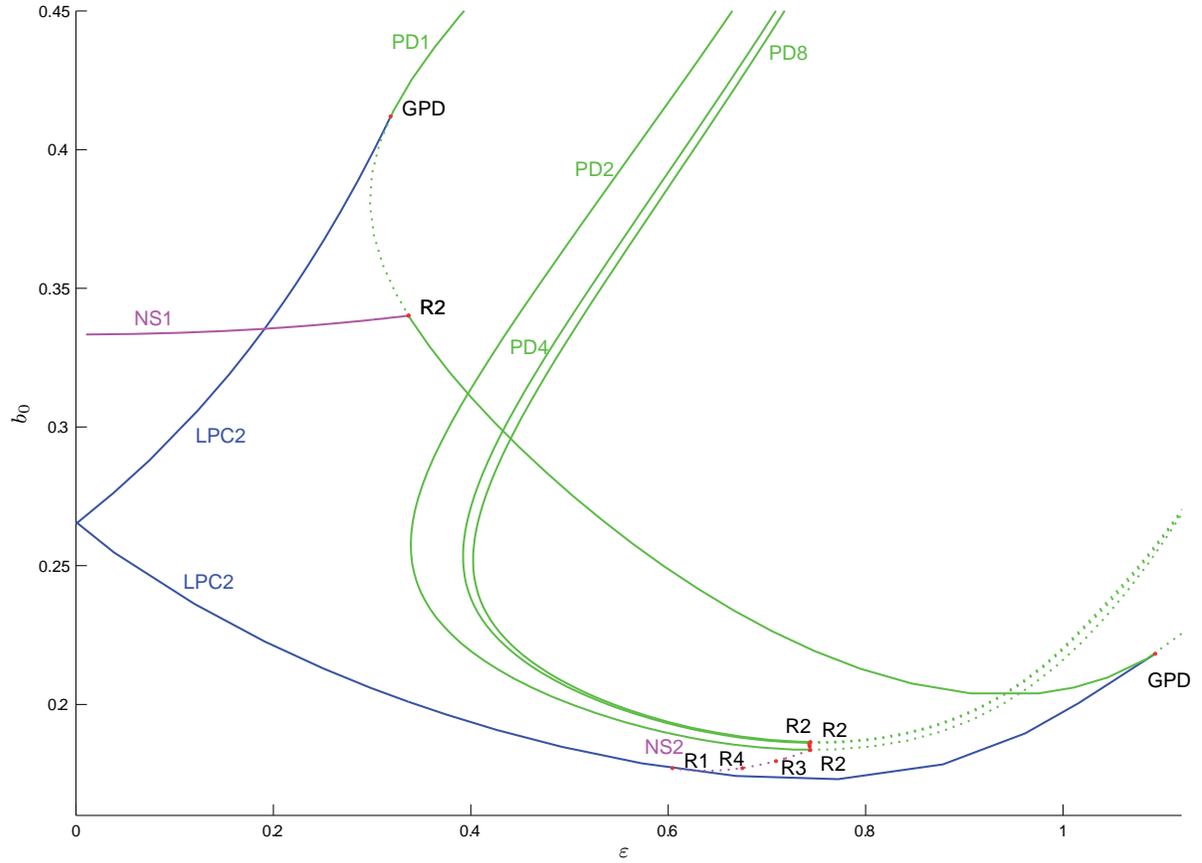}
 \caption{Bifurcation diagram of limit cycles in model \eqref{eq:PreyPredator}. Blue are limit point of cycles bifurcations,
 green period doubling bifurcations and magenta Neimark-Sacker birucations. Continue/dotted curves correspond to
 supercritical/subcritical bifurcations.} \label{fig:forcedpreypredator}
\end{figure}
\bigskip

We now analyze in detail all the detected codimension two points,
reporting the scalar computed coefficients explained in Section
\ref{Section:Implementation}.

\subsubsection{The two {\tt GPD points}}
In Figure \ref{fig:forcedpreypredator} the {\tt LPC2} curve is
tangent to the {\tt PD1} curve in two different {\tt GPD} points.
In the first one, in $(\epsilon,b_0)=(0.319,0.412)$, the limit
point of cycles curve is tangent to the subcritical
period-doubling curve (type presented in Figure
\ref{fig:NF_GPD}-(b)), while in the second one, in
$(\epsilon,b_0)=(1.09,0.218)$, the {\tt LPC2} curve is tangent to
the supercritical part of the {\tt PD} bifurcation curve (i.e. the
type presented in Figure \ref{fig:NF_GPD}-(a)).

Performing the computation of the {\tt GPD} normal form coefficients at
the first point we obtain:
\begin{itemize}
 \item for the first equation of normal form \eqref{eq:NF-GPD}
 the two coefficients $\alpha_1$ and $\alpha_2$, up to a scaling term $T$ and $T^2$ computed through the formula
 \eqref{eq:alpha11-GPD} and \eqref{eq:alpha21-GPD} are zero, up to the accuracy of the computation.
 \item the normal form coefficient of the second equation,
 computed through formula \eqref{eq:e1-GPD} equals
 $e=-58.2867$.
\end{itemize}
Notice that these results are in agreement with what we expected,
i.e. that since the system is periodically forced the time doesn't
depend on the distance from the critical limit cycle, and since we
are in the case presented in Figure \ref{fig:NF_GPD}-(b) the normal form
coefficient $e$ is negative.

From the computation of the {\tt GPD} normal form coefficients at
the second critical point we obtain:
\begin{itemize}
 \item for the first equation of normal form \eqref{eq:NF-GPD}
 the two coefficients equal zero.
 \item the normal form coefficient of the second equation has value
 $e=41.5442$.
\end{itemize}
Also in this case the obtained results are in agreement with the
theory.

\subsubsection{The 1:1 and 1:2 resonance points}
We divide the 1:1 and 1:2 resonance points present in this model into two
groups, namely the {\tt R2} point at $(\epsilon,b_0)=(0.337,0.34)$ and
the cascade of resonance points in the lower part of the graph.

The isolated {\tt R2} point forms the intersection of the
{\tt NS1} curve, the supercritical Neimark-Sacker curve of a limit
cycle with period approximately equal to $1$, and {\tt PD1}. The situation is thus the one
depicted in Figure \ref{fig:NF_R2}-(a). Performing the normal form
coefficient computation we obtain:
\begin{itemize}
 \item for the first equation of normal form \eqref{eq:NF-12C} holds that $\alpha=0$.
 \item for the last equation of the normal form \eqref{eq:NF-12C}
 we have $(a,b)=(3.401426, -12.90745)$.
\end{itemize}
Notice that the obtained results are in accordance with the theory
(no secondary Neimark-Sacker curve implies that $a>0$ and supercritical
Neimark-Sacker curve implies that $b<0$).

In the lower part of the bifurcation diagram a
resonance cascade is present, which accumulates on the sequence of period-doubling curves. A zoom of this part is shown in Figure
\ref{fig:resonance_cascade}.
\begin{figure}[htbp]
\centering
 \footnotesize
 \psfrag{BB}[][c]{$b_0$}
 \psfrag{epsilo}[][c]{$\varepsilon$}
 \includegraphics[width=\textwidth]{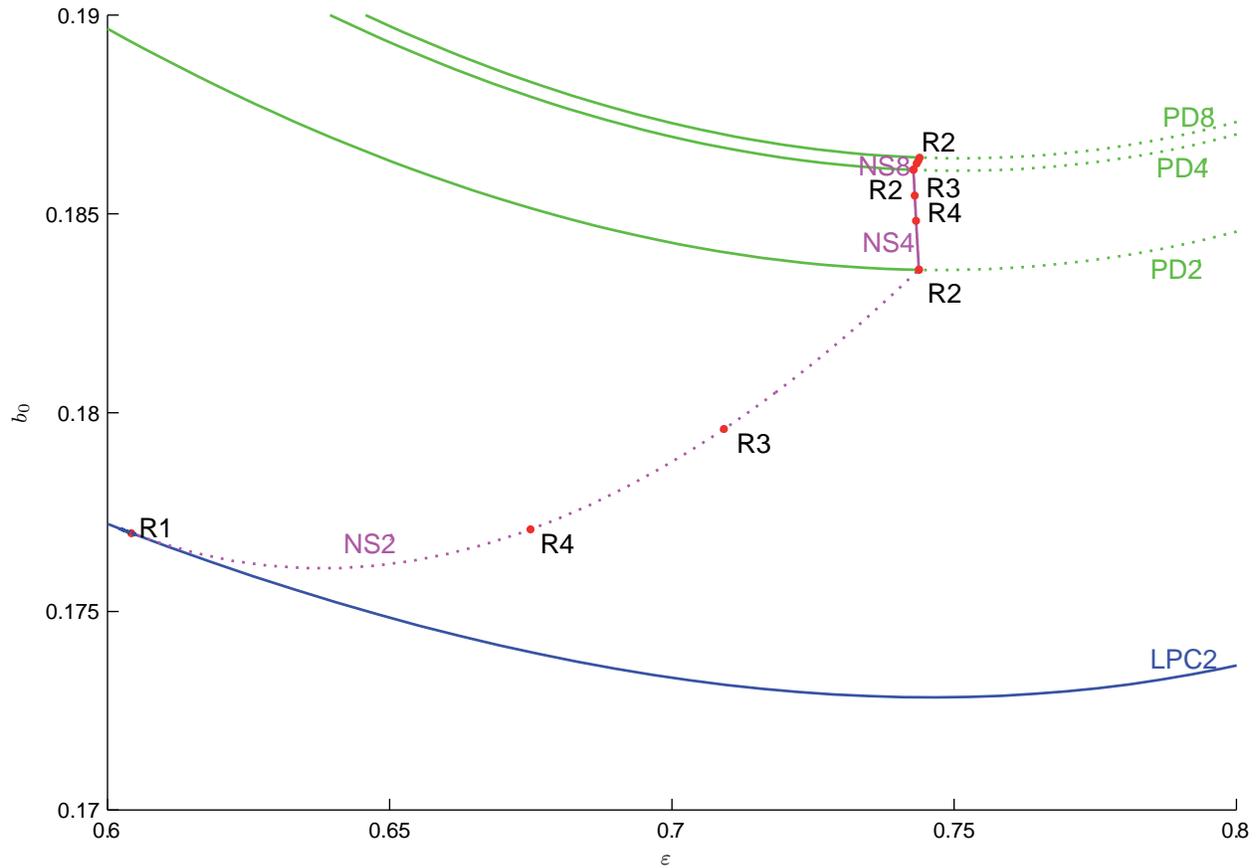}
 \caption{The resonance cascade in model \eqref{eq:PreyPredator}. In blue are limit point of cycles bifurcations,
 green period doubling bifurcations and magenta Neimark-Sacker birucations. Continue/dotted curves correspond to
 supercritical/subcritical bifurcations.} \label{fig:resonance_cascade}
\end{figure}
Each resonance point of this cascade (except for the first
{\tt R1} point on the {\tt LPC2} curve) can be seen in two ways: the
first (and more natural way) is to see them as {\tt R2} points of the
type represented in Figure \ref{fig:NF_R2}-(b) (so with $a<0$ and
the sign of $b$ dependent on the criticality of the incoming
Neimark-Sacker curve), the second way is to see them as {\tt R1} points represented in Figure \ref{fig:NF_R1}-(b).
%
%
Notice that the criticality of each {\tt NS} curve of the cascade
changes at the {\tt R2} point (as depicted in Figure \ref{fig:NF_R2}-(b)).

As first general result we see
that for the first equation of normal form \eqref{eq:NF-12C} there holds that $\alpha=0$ for all points (as expected
since the system is periodically forced). We remark that for the computation of the normal form coefficients of the second equation, the tolerances have to be strong enough. The results are
\begin{itemize}
\item[(on {\tt LPC2})]
 The {\tt R1} point is in $(\epsilon,b_0)=(0.6044021,0.1769608)$. The period 2 limit cycle Neimark-Sacker curve {\tt NS2} starts tangentially to
 the {\tt LPC2} curve and is subcritical. In this situation we therefore expect the product of the two normal form coefficients of the
 last equation of \eqref{eq:NF-11C} to be positive; the
 computed coefficients are $(a,b)=(2.005489e-6,
 6.4806e+9)$, such that $ab=1.299679e+4$.

\item[(on {\tt PD2})] The {\tt R2} point is in $(\epsilon,b_0)=(0.7437713,0.1835935)$.
 On the left side of the {\tt R2} point the {\tt PD2} curve is supercritical, on the
 right side it is subcritical. The {\tt NS2} curve incoming
 in the {\tt R2} point is subcritical, while the {\tt NS4} curve
 outgoing at the {\tt R2} point is supercritical. We are thus in
 the case depicted in Figure \ref{fig:NF_R2}-(b) with time
 reversed. So we expect $b>0$ (subcritical incoming Neimarck-Sacker curve) and
 $a<0$ (there is an outgoing secondary Neimarck-Sacker curve).
 The computed coefficients at the {\tt R2} point are $(a,b)=(-65.76676,
 16.26708)$. Notice that this point is also a degenerate
 {\tt R1} point for the {\tt NS4} curve. In fact, when we
 compute the normal form coefficient at this 1:1 resonance bifurcation point, we obtain $(a,b)=(-1.113774823354237e-4, 6.116846870167980e+12)$.
 In this case the product $ab=-6.812790e+8$ is negative, in accordance with the
 fact that the {\tt NS4} curve (for which we have an {\tt R1} bifurcation) is
 supercritical.

\item[(on {\tt PD4})]
 The {\tt R2} point is in $(\epsilon,b_0)=(0.7427991,0.1861098)$.
 On the left side of the {\tt R2} point the {\tt PD4} curve is supercritical, on the
 right side it is subcritical. The {\tt NS4} curve incoming
 in the {\tt R2} point is supercritical, while the {\tt NS8} curve
 outgoing at the {\tt R2} point is subcritical. We are therefore in
 the case depicted in Figure \ref{fig:NF_R2}-(b). We expect $b<0$
 (supercritical incoming Neimarck-Sacker curve) and
 $a<0$ (there is an outgoing secondary Neimarck-Sacker curve).
 The computed coefficients at the {\tt R2} point are $(a,b)=(-269.3681, -18.15061)$.

\item[(on {\tt PD8})]
 The {\tt R2} point is in $(\epsilon,b_0)=(0.7439079,0.1864190)$.
 On the left side of the {\tt R2} point the {\tt PD8} curve is supercritical, on the
 right side it is subcritical. The {\tt NS8} curve incoming
 in the {\tt R2} point is subcritical: we are thus in
 the case depicted in Figure \ref{fig:NF_R2}-(b) with time
 reversed. Thus, we expect $b>0$ (subcritical incoming Neimarck-Sacker curve) and
 $a<0$ (there is an outgoing secondary Neimarck-Sacker curve, since the cascade continues).
 The computed coefficients of the {\tt R2} point are $(a,b)=(-921.7011, 16.58059)$.

\end{itemize}
All the obtained results are in agreement with the
theory.

\subsubsection{The 1:3 resonance points}
There are two 1:3 resonance points, one on NS2, the other one on
NS4, as can be seen in Figure \ref{fig:resonance_cascade}. These
two points behave in a different way. The Neimark-Sacker curve
corresponding with the first point at $(\epsilon,b_0)=(0.709,0.179)$
is subcritical, so we expect $\Re(c)$ to be positive. The
Neimark-Sacker curve of the second point at
$(\epsilon,b_0)=(0.743,0.185)$ is supercritical, so $\Re(c)$ should
be negative. To check whether we are in a non degenerate case, we
also have to look at $b$, however, as mentioned before, the sign
of $b$ is not relevant. We obtain
\begin{itemize}
    \item for the first {\tt R3} point we have that $(b,\Re(c))=(4.5567 - 4.4567i,  9.155003)$.
    \item for the second {\tt R3} point we have that $(b,\Re(c))=(0.4049 +12.1425i, -8.819864)$.
\end{itemize}
These results are in accordance with the theory.

\subsubsection{The 1:4 resonance points}
There are two 1:4 resonance points, one on NS2, the other one on
NS4, as can be seen in Figure \ref{fig:resonance_cascade}. Also
these two points behave in the same way as the 1:3 resonance bifurcation
points. The Neimark-Sacker curve corresponding with the first
point at $(\epsilon,b_0)=(0.675,0.177)$ is subcritical, so here we
expect $\Re(A)$ to be positive. The Neimark-Sacker curve of the
second point at $(\epsilon,b_0)=(0.743,0.185)$ is supercritical, so
$\Re(A)$ should be negative. Moreover, since those points are part
of a resonance cascade, we should not have limit point bifurcations
of non trivial equilibria, so we are in region I of Figure
\ref{fig:bifR4}. In order to assure that we are not in a
degenerate case, we also need to check that $d\neq 0$. We
obtain
\begin{itemize}
    \item for the first {\tt R4} point we have that $(c,d)=(11.624 -84.897 i, 65 + 92.254 i)$, so $A=0.102999 -0.752278 i$.
    \item for the second {\tt R4} point we have that $(c,d)=(-8.5796 -414.71 i, -416.64 -489.17 i)$, so $A=-0.013352 -0.645406 i$.
\end{itemize}
So for both bifurcation points the value of $A$ belongs to region I, and thus the results are in accordance with the theory.

\subsection{The Steinmetz-Larter model}
The following model of the peroxidase-oxidase reaction was studied
by Steinmetz and Larter \cite{StLa:91}:
\begin{equation}
\left\{\begin{array}{rcl}
\dot A & = & -k_1 ABX- k_3 ABY+ k_7-k_{-7}A,\\
\dot B & = & -k_1 ABX- k_3 ABY+ k_8 ,\\
\dot X & = & k_1 ABX -2k_2 X^2 +2k_3 ABY -k_4X+k_6,\\
\dot Y & = & -k_3 ABY+2k_2X^2-k_5 Y,
\end{array}\right.
\label{eq:StLr}
\end{equation}
where $A,B,X,Y$ are state variables and $k_1$, $k_2$, $k_3$,
$k_4$, $k_5$, $k_6$, $k_7$, $k_8$, and $k_{-7}$ are parameters. We
fix all parameters as reported in the following table
\begin{center}
\begin{tabular}{|l|l|l|l|l|l|l|l|}
\hline
{Par.} & {Value} & {Par.} & {Value} & {Par.} & {Value} & {Par.} & {Value} \\
\hline
$k_1$ & 0.1631021 & $k_2$ & 1250  & $k_3$ & 0.046875 & $k_4$ & 20 \\
$k_5$ & 1.104 & $k_6$ & 0.001 & $k_{-7}$ & 0.1175 &&\\
\hline
\end{tabular}
\end{center}
and we perform a bifurcation analysis in the parameter space
$(k_7,k_8)$. A few curves are reported in Figure
\ref{fig:bif_SL}.
\begin{figure}[htbp]
\centering
 \footnotesize
 \psfrag{k7}[][c]{$k_7$}
 \psfrag{k8}[][c]{$k_8$}
 \includegraphics[width=\textwidth]{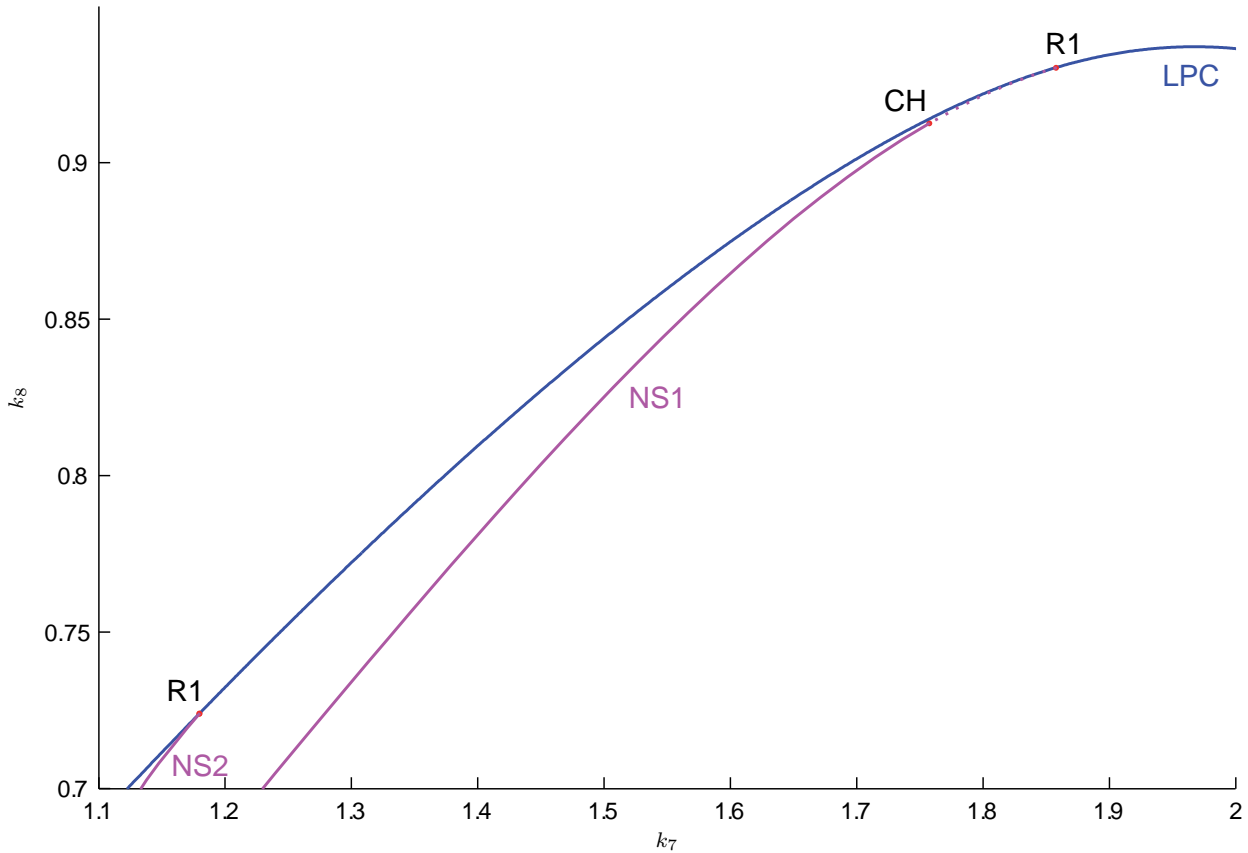}
 \caption{Bifurcation diagram of a limit cycle in model \eqref{eq:StLr}.
 In blue are limit point of cycles bifurcations, green period doubling bifurcations and magenta Neimark-Sacker
 birucations. Continue/dotted curves correspond to supercritical/subcritical bifurcations.} \label{fig:bif_SL}
\end{figure}

\subsubsection{The 1:1 resonance points}
The two 1:1 resonance points have different nature, since in one
the Neimark-Sacker curve rooted at the bifurcation point is supercritical,
while in the other one it is subcritical.

We obtain:
\begin{itemize}
 \item for the {\tt R1} point in $(k_7,k_8)= (1.179554,
 0.7239571)$, the two coefficients of the last equation of \eqref{eq:NF-11C} are
 equal to $(a,b)=( -0.003654362200739, 0.735048055230916)$.
 Their product $ab=-2.686132e-3$ is negative, which corresponds with the fact that the NS
 curve rooted at the {\tt R1} point is supercritical.
 \item for the {\tt R1} point in $(k_7,k_8)= (1.857676, 0.9304220)$, the two coefficients of the last equation of \eqref{eq:NF-11C}
 are equal to $(a,b)=(-0.066429738171756, -2.156596806473489)$.
 Their product $ab=0.1432622$ is positive, and indeed the NS
 curve rooted at the {\tt R1} point is subcritical.
\end{itemize}
So we can conclude that the results are in accordance with the theory.

\subsubsection{The Chenciner points}
As can be seen in Figure \ref{fig:bif_SL} we have detected a {\tt CH} point at
$(k_7,k_8)= (1.757356, 0.9125773)$. The normal form
coefficient at that bifurcation point equals $\Re(e)=1.391931$, hence
positive so the unfolding is the one depicted in
Figure \ref{fig:NF_GNS}-(b).
\begin{figure}[tbhp]
\footnotesize \centering
 \psfrag{e}[][c]{$k_7$}
 \psfrag{d}[][c]{$k_8$}
 \psfrag{a}[][c]{\circledcoloreditem{1}}
 \psfrag{b}[][c]{\circledcoloreditem{2}}
 \psfrag{c}[][c]{\circledcoloreditem{3}}
 \includegraphics[width=.7\textwidth]{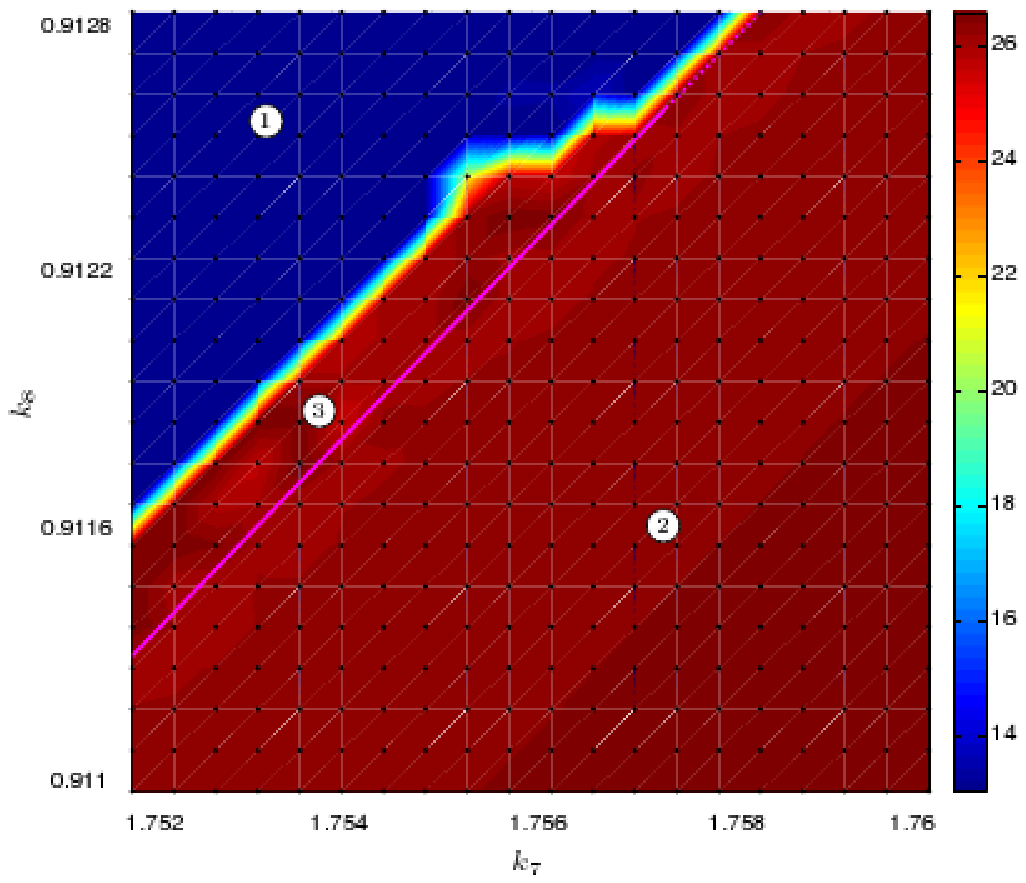}
 \caption{Simulations on a parameter grid (black points) of system \eqref{eq:StLr}. The magenta continue/dotted
 line is the supercritical/subbritical Neimark-Sacker curve. The color represents
 the value of the maximum of the first coordinate of the attractor reached through
 simulation from a point close to the limit cycle. A sketch of the state portrait
 is reported in Figure \ref{fig:NF_GNS}-(b).}  \label{fig:testCH}
\end{figure}
In order to verify if the normal form computation is
correct, one should use tori continuation techniques
\cite{Br:03}, either recurring to Poincar\'e maps
\cite{Ke:85,DoTa:06} or to the so-called {\it invariance
equation} \cite{Di:91,Sc:05,LaCv:06,RaBr:08}. However, these
techniques are not stable, especially in critical cases like
the one we have. In order to validate our result we
thus have to do simulations.

The obtained result is shown in Figure \ref{fig:testCH}. The
indicated regions correspond with the regions from Figure
\ref{fig:NF_GNS}. The green curve between regions 2 and 3 is the
supercritical Neimark-Sacker curve, the red one between region 1
and 2 is the subcritical Neimark-Sacker curve. For each point of
the grid, we have started time integration from a point close to
the orginal limit cycle (a 1 \% perturbation) until an attractor was found. The $1$-norm of the $X$-coordinate of an orbit with time length $1000$ along the attractor is shown in the colormap. In region 2 this attractor is
the original limit cycle, in region 3 it is the inner torus arisen
through the supercritical Neimark-Sacker curve. In region 1 the
original limit cycle is unstable, and so the trajectory which
starts nearby converges to another attractor. Between region 1 and
3 and region 1 and 2 happens a catastrophic bifurcation, i.e. a
drastic  change of the attractor, identified from the change of
color which varies from blue to red. Right above the Chenciner point,
the catastrophic bifurcation is the subcritical NS curve, while left
below it is the limit point of tori ($T_c$) curve. Figure
\ref{fig:testCH} shows that we obtain the scenario which
corresponds with a positive second Lyapunov coefficient.

\subsection{The Lorenz1984 system} This model, taken from \cite{Lo:84}, is a
meteorological model written by Lorenz in 1984 in order to
describe the atmosphere. The equations of the model are
\begin{equation} \label{eq:Lo84model}
\begin{cases}
\dot x = -y^2 - z^2 - ax + aF, \\
\dot y = xy - bxz - y + G, \\
\dot z = bxy + xz - z,
\end{cases}
\end{equation}
where $(a, b, F, G)$ are parameters, with $a=0.25, b=4$. This model, as depicted in
\cite{Sh:95}, has most of the analyzed codimension two
bifurcations of limit cycles. We report in Figure \ref{fig:bifcompleto} a bifurcation diagram
obtained with {\sc matcont}
\begin{figure}[htbp]
\centering
 \footnotesize
 \psfrag{F}[][c]{$F$}
 \psfrag{G}[][c]{$G$}
 \includegraphics[width=\textwidth]{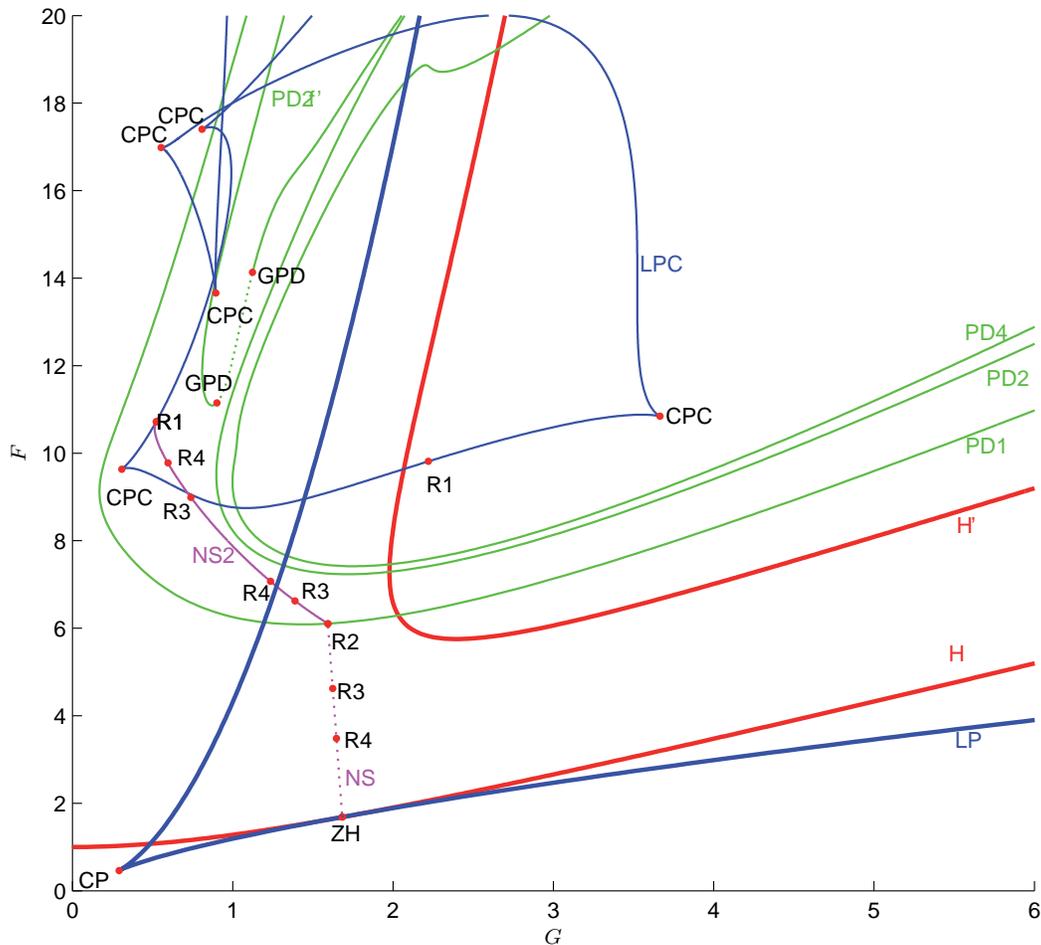}
 \caption{Bifurcation diagram of model \eqref{eq:Lo84model}. The
 thicker curves are bifurcation curves of equilibria, the thin curves are
 bifurcation curves of limit cycles and invariant tori (blue, limit point of cycles,
 green period doubling and magenta Neiamrk-Sacker). Continue/dotted curves correspond to
 supercritical/subcritical bifurcations.} \label{fig:bifcompleto}
\end{figure}
in which bifurcations of equilibria ({\tt LP} stands for limit
point, {\tt H} for Andronov-Hopf) are thicker and limit cycle
bifurcations are thin. In particular, the blue curve is a limit
point of cycles ({\tt LPC}) bifurcation curve, the green ones are
period-doubling ({\tt PD}) bifurcations curves and the magenta
ones are Neimark-Sacker ({\tt NS}) curves. The codimension two
points are marked with a red dot, and, as can be seen in the
figure, almost all cases (except the Chenciner bifurcation and the
fold-flip bifurcation) discussed in Section \ref{Section:Theory}
are present in this model. In the sequel of this section we will
investigate the normal form coefficients of each bifurcation.

\subsubsection{The Swallowtail bifurcation}
\begin{figure}[b]
\footnotesize \centering
 \includegraphics[width=\textwidth]{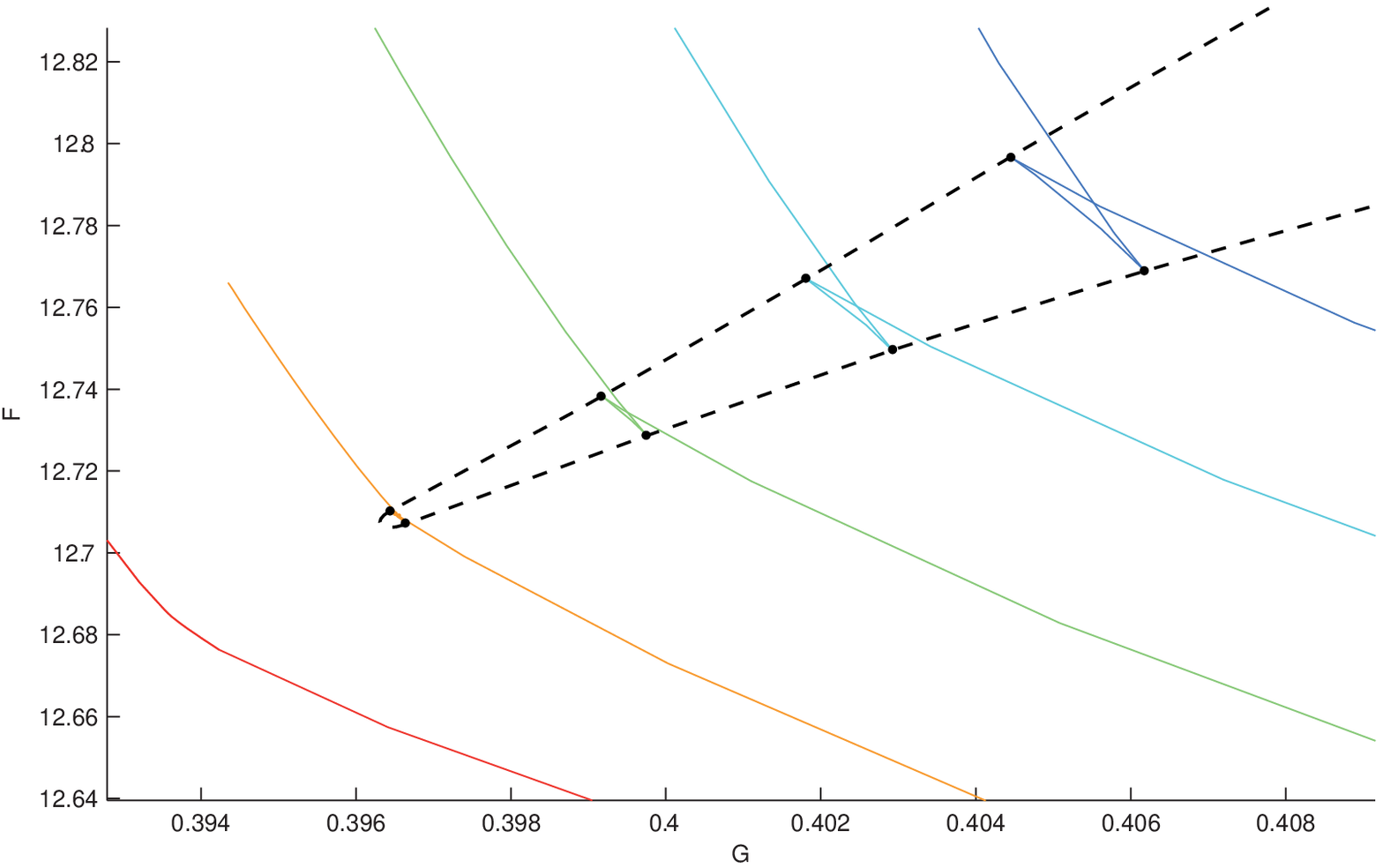} \\
 \medskip
 \definecolor{color1}{rgb}{0,0.5,1}
 \definecolor{color2}{rgb}{0,1,1}
 \definecolor{color3}{rgb}{0.5,1,0.5}
 \definecolor{color4}{rgb}{1.0,0.6,0}
 \definecolor{color5}{rgb}{1,0,0}
 \centering
 \begin{tabular}{c|c|c|c|c}
 $b$ & $(F,G)_{1}$ & $c_1$ & $(F,G)_{2}$ & $c_2$ \\ \hline
 \textcolor{color1}{2.95} & \textcolor{color1}{(12.76893,0.406176)} & \textcolor{color1}{16.4570} & \textcolor{color1}{(12.79664,0.40445)} & \textcolor{color1}{-8.83567} \\
 \textcolor{color2}{2.94} & \textcolor{color2}{(12.74929,0.402934)} & \textcolor{color2}{13.3315} & \textcolor{color2}{(12.76712,0.40183)} & \textcolor{color2}{-7.77721} \\
 \textcolor{color3}{2.93} & \textcolor{color3}{(12.72891,0.399746)} & \textcolor{color3}{10.0534} & \textcolor{color3}{(12.73840,0.39917)} & \textcolor{color3}{-6.47271} \\
 \textcolor{color4}{2.92} & \textcolor{color4}{(12.70759,0.396624)} & \textcolor{color4}{6.30503} & \textcolor{color4}{(12.71071,0.39643)} & \textcolor{color4}{-4.67510} \\
 \textcolor{color5}{2.91} &  &  &  & \\
 \end{tabular}
 \caption{Different limit point of cycles bifurcation curves in the $(F,G)$-plane for different values of the third parameter $b$. The parameter values are reported in the table.} \label{fig:shallowtail}
\end{figure}
The first degeneracy we want to analyze is the vanishing of the
coefficient $c$ in the cusp of cycles normal form (\ref{eq:NF-CPC}). This
bifurcation, named Swallowtail bifurcation, is characterized, in
our case, by the collision and disappearance of two cusp points of
limit cycle. In order to get this codimension three bifurcation we
analyze part of the blue dashed curve in Figure \ref{fig:bifcompleto} for
different parameter values of $b$. The result is shown in Figure
\ref{fig:shallowtail}: in the graph part
of the limit point of cycles manifold is plotted in the $(F,G)$-plane for
different values of parameter $b \in [2.91,2.95]$ (from blue
to red). In the table we can see the behavior of the
critical normal form coefficient $c$, where it exists (the colors of the lines correspond with the bifurcation diagram). Notice how the
behavior of this codim $3$ bifurcation is captured by a smooth vanishing
of the normal form coefficient.

\subsubsection{The degenerate generalized period doubling
bifurcation} On the green curve {\tt PD4} of Figure
\ref{fig:bifcompleto} there are two generalized period-doubling ({\tt GPD}) points. In the first one the flip bifurcation
curve passes from subcritical to supercritical, in the second one
the opposite happens. Computing the normal form coefficient
in the first case gives $e=-1.317656e-3<0$,
therefore there is a limit point of cycles bifurcation curve that starts rightward tangent to the
supercritical part of the period-doubling manifold, while in the
second case $e=2.895460e-3>0$, and so the limit point of cycles
bifurcation starts leftward tangent to the subcritical part of the
{\tt PD} curve. These conclusions can be seen in Figure \ref{fig:GPD1_2},
where the period-doubling curve is black, dotted when
supercritical; in the upper panels are sketched the Poincar\'e map
of the limit cycle involved in the bifurcation. On the yellow
curve the limit cycles sketched in green and red collide and
disappear, while on the violet curve the two involved limit cycles
are sketched in red and blue.
\begin{figure}[htb]
\footnotesize \centering
 \psfrag{F}{$F$}
 \psfrag{G}{$G$}
 \psfrag{p0}[][c]{\circledcoloreditem{0}}
 \psfrag{p1}[][c]{\circledcoloreditem{1}}
 \psfrag{p2}[][c]{\circledcoloreditem{2}}
 \psfrag{p3}[][c]{\circledcoloreditem{3}}
 \includegraphics[width=\textwidth]{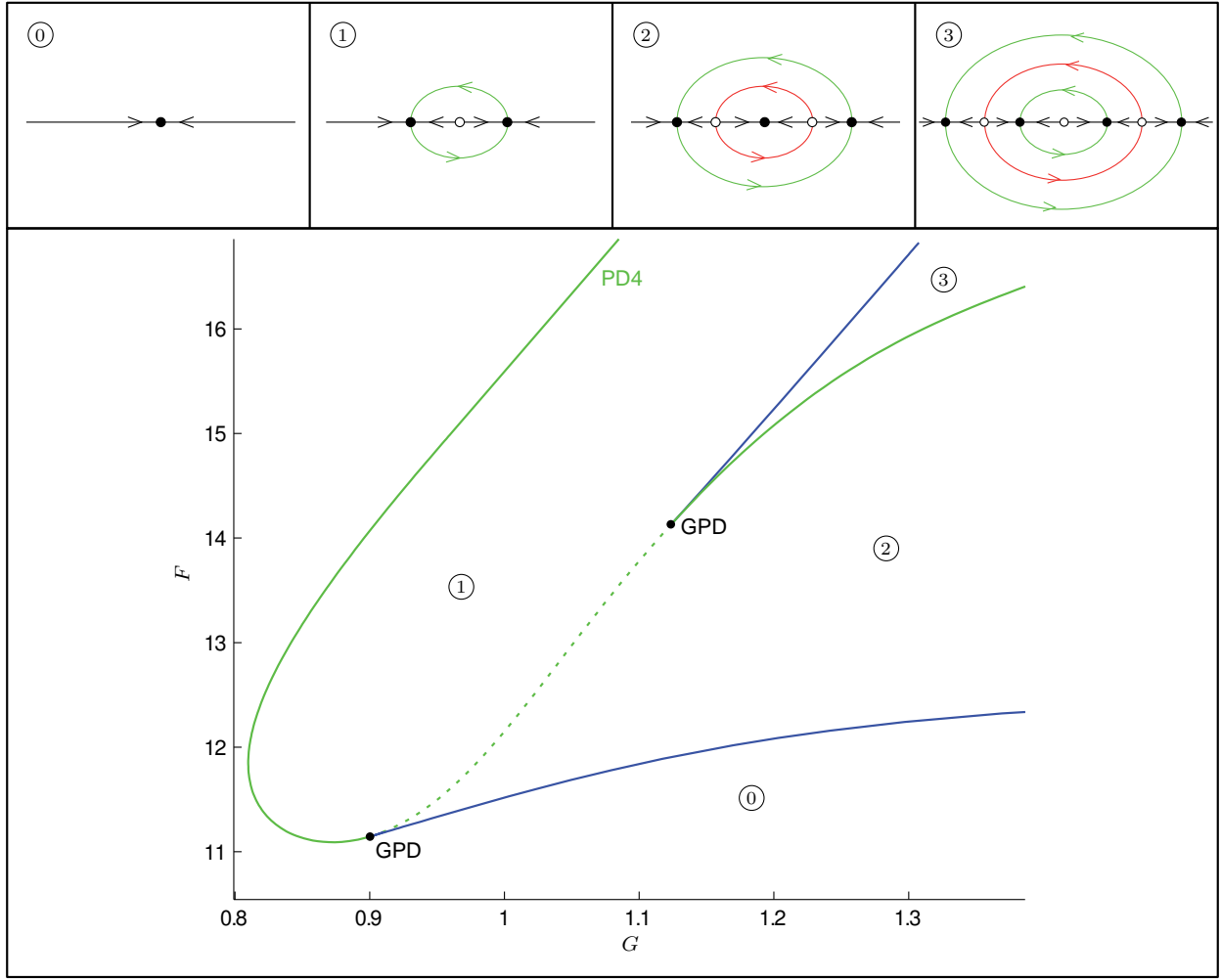}
 \caption{Two generalized period-doubling points with different normal form
 coefficients on the period-doubling bifurcation curve {\tt PD4} of Figure \ref{fig:bifcompleto}.} \label{fig:GPD1_2}
\end{figure}

\subsubsection{The 1:1 resonance points}
Two {\tt R1} points are located on the {\tt LPC} curve. Those two points should have different product of the  normal form coefficients. In fact,
in the first one (in $(F,G)=(10.72,0.522)$) the Neimark-Sacker
curve rooted at the bifurcation point is supercritical (i.e. the system
has the behavior depicted in \ref{fig:NF_R1}-(a)), while in the
second one (in $(F,G)=(9.81,2.22)$) the NS curve is
subcritical (behavior similar to Figure \ref{fig:NF_R1}-(b)). If we apply our analysis we
obtain:
\begin{itemize}
 \item for the first {\tt R1} point $(a,b)=(2.577,-1.2659)$, so
 the product $ab=-3.26237$ is negative.
 \item for the second {\tt R1} point $(a,b)=(-9.887, -2.005)$, so
 the product $ab=19.81858e$ is positive.
\end{itemize}
These results are in accordance with the theory. The blue curve on Figure \ref{fig:bif_LorR1} is the limit point of cycles curve. The violet curves are the Neimark-Sacker curves of first iterate, the green curve is the Neimark-Sacker curve of second iterate.
\begin{figure}[tbhp]
\footnotesize \centering
 \psfrag{FF}{$F$}
 \psfrag{GG}{$G$}
 \subfigure{
 \includegraphics[width=.9\textwidth]{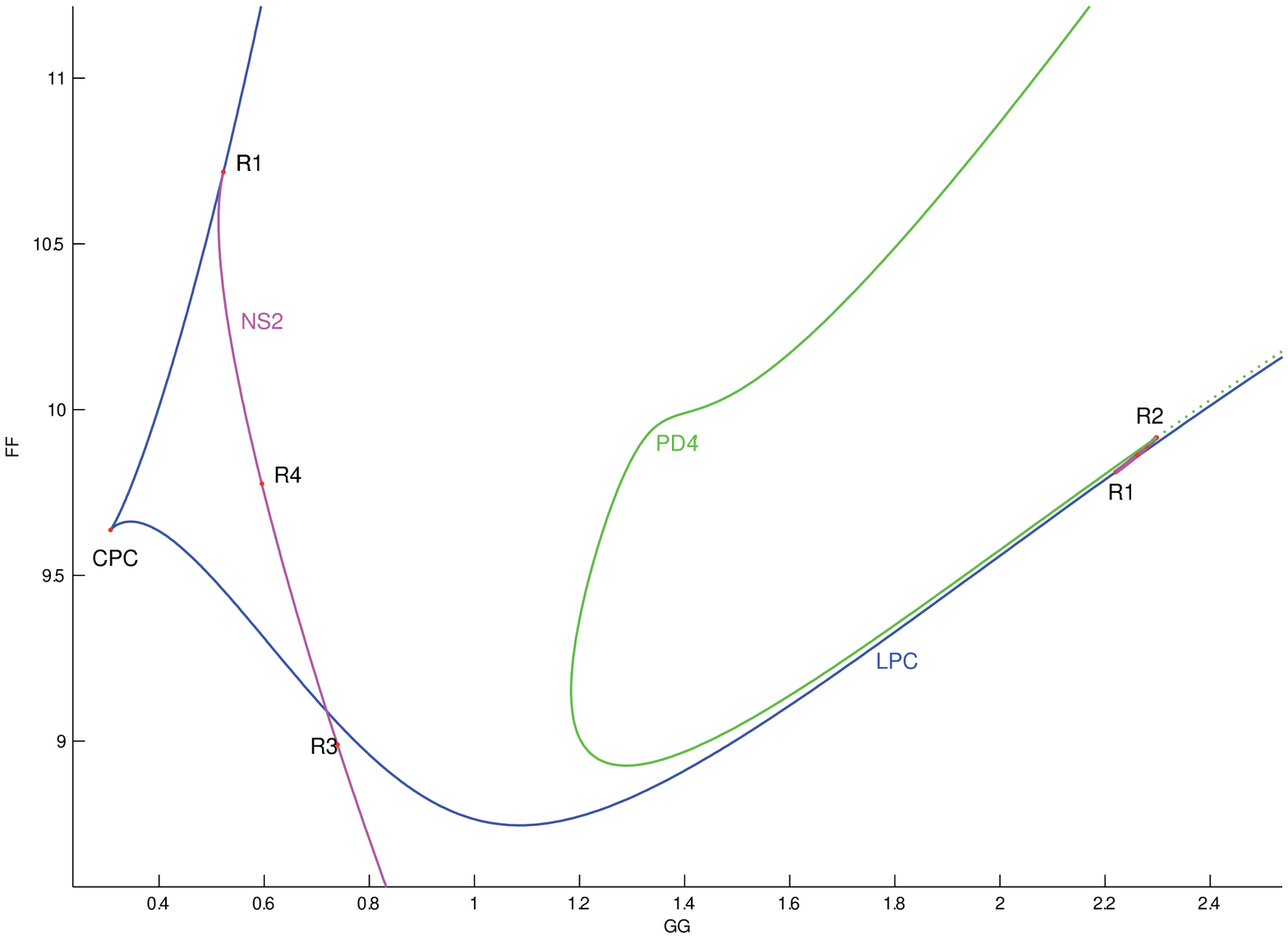}}
 \subfigure{
 \includegraphics[width=.8\textwidth]{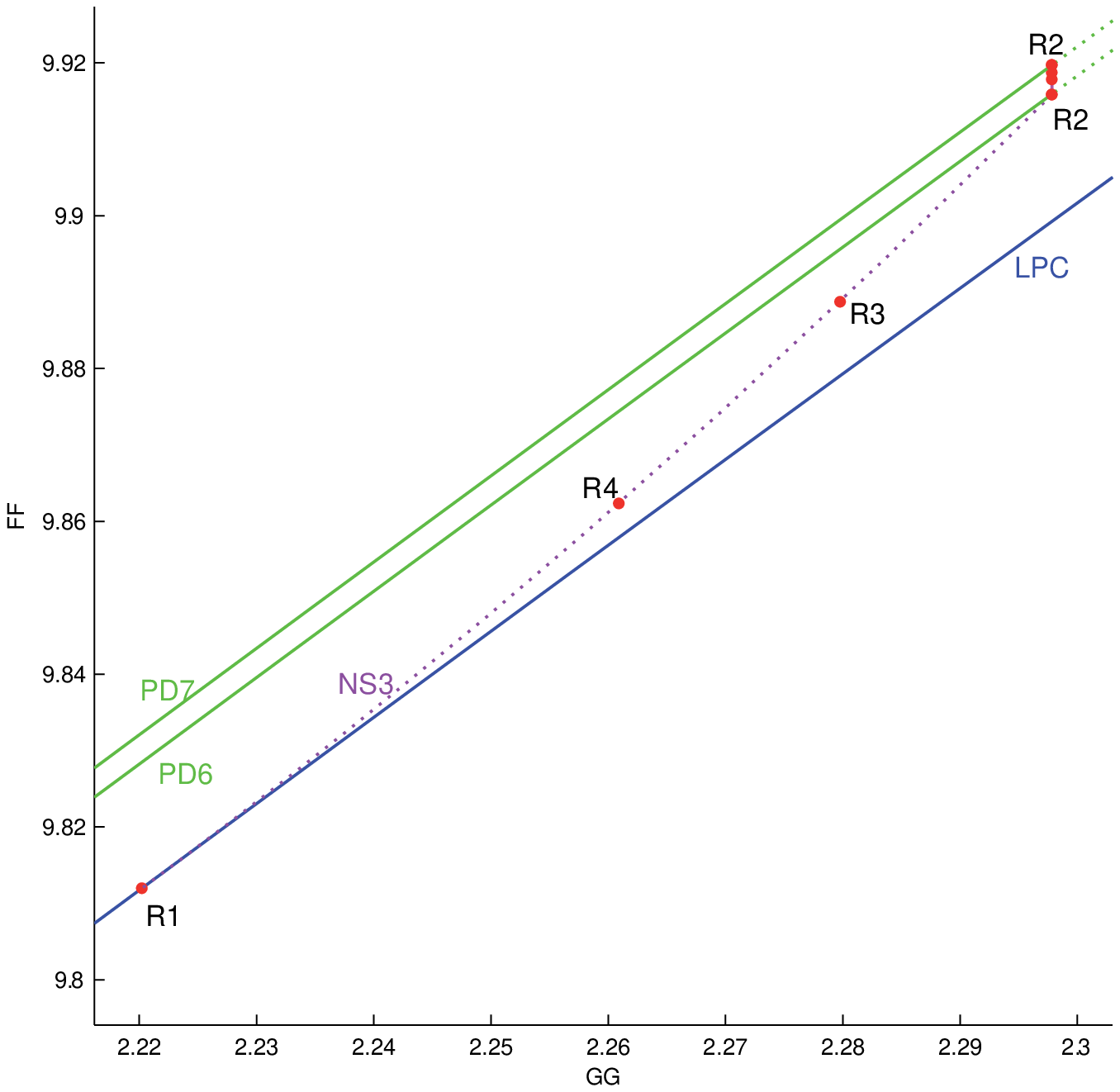}}
 \caption{(a) Two 1:1 resonance points with different behavior on the LPC curve of Figure \ref{fig:bifcompleto}. (b) Zoom on the Lorenz
 cascade that starts at the right {\tt R1} point. In blue
 are the limit point of cycles bifurcation curves, in green the period-doubling curves, in
 violet the Neimark-Sacker curves. Continue/dotted curves correspond to supercritical/subcritical
 curves.} \label{fig:bif_LorR1}
 \end{figure}

\subsubsection{The 1:2 resonance points}
At the unique {\tt R2} point at $(F,G)=(10.72,0.522)$ shown in Figure \ref{fig:bifcompleto}, the incoming
Neimark-Sacker curve, namely {\tt NS}, is subcritical (and so
$b>0$), while the outgoing curve (which exists and so $a<0$), namely {\tt
NS2}, is supercritical, i.e. we are in the case reported in Figure
\ref{fig:NF_R2}-(b) with time reversed. The coefficients
computed at the 1:2 resonance point are $(a,b)=(-0.6329965, 0.1785283)$,
in accordance with the theory.

At the {\tt R1} point located at $(F,G)=(9.81,2.22)$ starts a
resonance cascade, as shown in
Figure \ref{fig:bif_LorR1}-(b). On that cascade we can
find many resonance points which we will analyze in what follows. In particular, since the {\tt R2} points belong to a cascade they are of the same type as presented in Figure \ref{fig:NF_R2}-(b) (so $a<0$), with at each step a change of criticality of
the incoming NS curve. As mentioned before, the first NS curve, born
at the {\tt R1} point, is subcritical, so for the first {\tt R2} point we expect that $b>0$, while for the second
one $b<0$. The obtained numerical results are
\begin{itemize}
 \item for the first {\tt R2} point at $(F,G) = ( 9.9158, 2.2978)$ we have that $(a,b)=(-1.3157, 0.11076)$.
 \item for the second {\tt R2} point at $(F,G) = ( 9.9197, 2.2978)$ we have that $(a,b)=(-2.6228, -0.0564)$.
\end{itemize}
Results are in accordance with the theory.

\subsubsection{The 1:3 resonance points}
We have several 1:3 resonance points at which we can have a closer
look. There is one {\tt R3} point located on the NS curve and two {\tt R3} points on the NS2 curve. The {\tt R3} point corresponding to the first
iterate is at $(F,G)=(4.628,1.624)$, with a positive normal form
coefficient of the Neimark-Sacker curve such that we are in the case represented in Figure \ref{fig:NF_R3}-(b). The {\tt R3} points
corresponding to the second iterate are at $(F,G)=(7.072,1.235)$
and $(F,G)=(8.989,0.7394)$, where the Neimark-Sacker curve is both
times supercritical, so we are in the situation depicted in
Figure \ref{fig:NF_R3}-(a).
\begin{itemize}
    \item for the {\tt R3} point at $(F,G)=(4.628,1.624)$ we have that $(b,\Re(c))=(0.1913 - 0.5464i, 6.185900e-2)$
    \item for the {\tt R3} point at $(F,G)=(7.072,1.235)$ we have that $(b,\Re(c))=(-0.4461 - 0.1901i, -3.612302e-2)$.
    \item for the {\tt R3} point at $(F,G)=(8.989,0.7394)$ we have that $(b,\Re(c))=(-0.1285 + 0.0168i, -1.950822e-2)$.
\end{itemize}
These results are in accordance with the theory.

There are also {\tt R3 }points on the cascade (see Figure \ref{fig:bif_LorR1}-(b)). The first one corresponds with a subcritical NS curve, while
the second one corresponds with a supercritical NS curve.
\begin{itemize}
 \item for the first {\tt R3} point, in $(F,G) = (9.8888, 2.2798)$ we have that $(b,\Re(c))=(-2.9582 - 0.3599i, 0.7383)$.
 \item for the second {\tt R3} point, in $(F,G) = (9.9187, 2.2978)$ we have that $(b,\Re(c))=(2.7447 + 3.5391i, -0.3847)$.
\end{itemize}
Also in this case all results are in accordance with the theory.

\subsubsection{The 1:4 resonance points}
There are $5$ 1:4 resonance points at which we will we have a look. One is
located on the NS curve, two others on the NS2 curve and the last two lie on the resonance cascade
(see Figure \ref{fig:bif_LorR1}-(b)). We obtain
\begin{itemize}
 \item for the {\tt R4} point at $(F,G)=(3.376,1.647)$ we have that $(c,d)=(0.0501 - 0.0746i, -0.0577 - 0.5422i )$ and so $A=0.0918 -0.1368 i$ (subcritical NS curve, case I).
 \item for the {\tt R4} point at $(F,G)=(9.777,0.595)$ we have that $(c,d)=(-0.0151 - 0.1348i, -0.0266 - 0.0411i)$ and so $A=-0.307829 -2.752998 i$ (supercritical NS curve, case VIII).
 \item for the {\tt R4} point at $(F,G)=(6.620,1.390)$ we have that $(c,d)=(-0.0417 - 0.9915i, -0.4283 - 1.0826i)$ and so $A=-0.035841 -0.851653 i$ (supercritical NS curve, case I).
\end{itemize}
For the first and the last point no further bifurcation analysis
is possible to confirm the correctness of the results (present
curves rooted at the point are global bifurcations of limit
cycles). Instead it is possible to continue all local
bifurcations of limit cycles rooted at the second {\tt R4} point,
obtaining the result shown in Figure \ref{fig:R4_8}. Note that we haven't made the distinction between region VII and VIII since we have not computed the fold of torus curve typical for region VIII.

\begin{figure}[tbhp]
\footnotesize \centering
 \psfrag{FF}{$F$}
 \psfrag{GG}{$G$}
 \includegraphics[width=.8\textwidth]{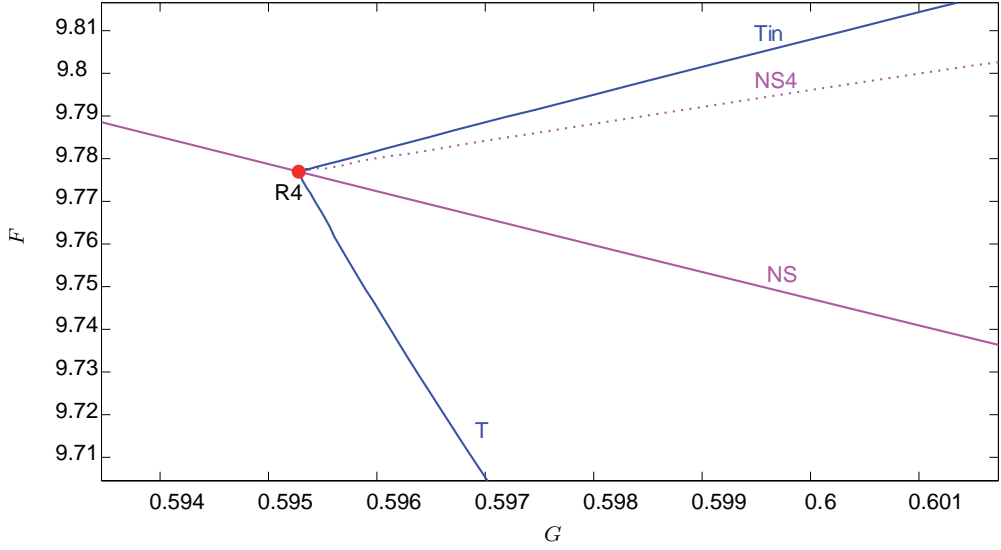}
 \caption{Bifurcation diagram at the {\tt R4} point at $(F,G)=(9.777,0.595)$.
 In blue are the limit point of cycles bifurcation curves, in
 violet the Neimark-Sacker curves. Continue/dotted curves correspond to supercritical/subcritical
 curves.} \label{fig:R4_8}
 \end{figure}

For the cascade (see Figure \ref{fig:bif_LorR1}-(b))
we have that the first point is on a subcritical NS curve, while
the second one is on a supercritical NS curve. Moreover, since they lie on a cascade, they should be
of type I.
\begin{itemize}
 \item for the first {\tt R4} point at $(F,G) = (9.9159, 2.2978)$ we have that $(c,d)=(0.0518 - 1.7633i, -2.0143 + 0.4546i )$ and so $A=0.025102 -0.853919 i$ (subcritical NS curve, case I).
 \item for the first {\tt R4} point at $(F,G) = (9.9197, 2.2978)$ we have that $(c,d)=(-0.0282 - 6.8152i, -10.8446 + 2.1458i )$ and so $A=-0.002550 -0.616491 i$ (supercritical NS curve, case I).
\end{itemize}
Also in this case the results are in accordance with the theory.

\subsection{The Extended Lorenz1984 system}
As done in \cite{KuMeVe:2004}, it is possible to extend the
Lorenz1984 system \eqref{eq:Lo84model} by adding a fourth variable
which takes the influence on the jet stream and the
baroclinic waves of external parameters like the temperature of
the sea surface into account. The obtained system is
\begin{equation} \label{eq:ExtLo84model}
\begin{cases}
 \dot x = -y^2 - z^2 - a x + a F - \gamma u^2, \\
 \dot y = xy - b xz - y + G, \\
 \dot z = b xy + xz - z, \\
 \dot u = -\delta u + \gamma u x + K.
\end{cases}
\end{equation}
We use the parametervalues mentioned in \cite{KuMeVe:2004}, i.e.
$$a=0.25,\, b=1,\, G=0.2,\, \delta=1.04,\, \gamma=0.987, F=1.75,\, K=0.0003.$$
Simulating this system from the trivial initial condition leads to
a limit cycle. In a continuation in $K$ that limit cycle undergoes
a subcritical period-doubling bifurcation. Now, we can do a two
parameter continuation in $(F,K)$ and draw the bifurcation diagram
reported in Figure \ref{fig:bif_ELorenz}.
\begin{figure}[htbp]
\centering \footnotesize
 \psfrag{FF}[][c]{$F$}
 \psfrag{TT}[][c]{$T$}
 \normalsize
 \psfrag{LPPD}[][c]{{\tt LPPD}}
 \includegraphics[width=.6\textwidth]{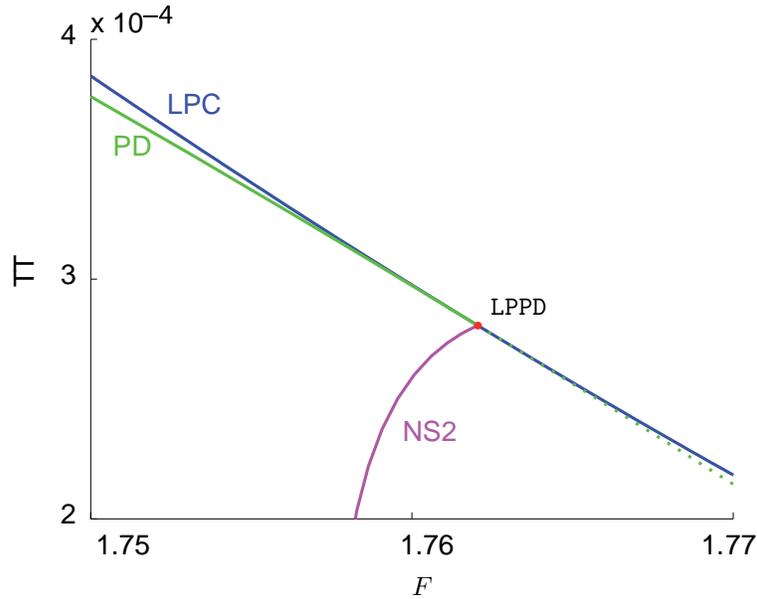}
 \caption{Bifurcation diagram of a limit cycle in model \eqref{eq:ExtLo84model}.
 The blue curve is a limit point of cycles bifurcation, the green curve is a period-doubling curve (continue/dotted curve correspond to supercritical/subcritical
 curves) and the magenta curve is a supercritical Neimark-Sacker bifurcation curve
 of the period doubled limit cycle.} \label{fig:bif_ELorenz}
\end{figure}

\subsubsection{The fold-flip point}
As can be seen in Figure \ref{fig:bif_ELorenz}, a fold-flip
point is detected for $(F,K) = (1.7620,0.2806 \times 10^{-3})$. Since there is a Neimark-Sacker curve of the period doubled
limit cycle rooted at the bifurcation point and the NS curve and the LPC curve lie on different sides of the PD curve, we are in the case represented in Figure
\ref{fig:NF_FF}-(a), i.e. we have $a_{20} b_{11}<0$ and $a_{02}
b_{11}<0$. Moreover, since the NS curve is supercritical, $L_{NS}$
should be negative. Numerically, we obtain that $b_{11}=562.2215,\, a_{20}=-0.5756,\, a_{02}=-0.1036 e-3,\, L_{NS}=-178.8596$. Hence, these results are in accordance with the theory.

\appendix
\section{Derivation of the normal forms}
\label{Appendix:1}
\subsection{Notation}\label{sectionNotation}
Let $M \in \R^{n \times n}$ be the monodromy matrix. In all codimension $2$ cases all critical multipliers, i.e. all multipliers with modulus $1$, have non-degenerate Jordan blocks. Let $M_0$ be the critical Jordan structure, i.e. the block diagonal matrix consisting of the critical Jordan blocks, starting with the block of the trivial multiplier $1$. Let $\mu_k = e^{i\theta_k} (0\leq \theta_k < 2\pi)$ be a critical multiplier with multiplicity $m_k$. The matrix $L_k \in \R^{m_k \times m_k}$ is
defined as
\[
L_k=\begin{pmatrix}
 \sigma_k & 1  & \ldots & 0 \\
 0 & \sigma_k  & \ldots & 0 \\
 \vdots  & \ddots &\ddots & 1\\
 0  & \ldots & 0 & \sigma_k
\end{pmatrix},
\]
where $\sigma_k$ is the \textit{Floquet
exponent of the multiplier} $\mu_k$, with $\sigma_k = i\theta_k/T$ in the case of a positive real multiplier or a complex multiplier $\mu_k$ and $\sigma_k = 0$ for $\mu_k=-1$. The matrix $L_0$ is the block diagonal matrix formed from
the blocks $L_k$ for which $\left|\mu_k\right|=1$, starting with the block that corresponds with multiplier $1$. The matrix $\tilde L_0$ is the
matrix $L_0$ without the first row and the first column.

\subsection{Bifurcations with 2 critical eigenvalues}
\subsubsection{\tt CPC} At the CPC bifurcation the monodromy matrix
has the critical Jordan structure
\[
M_0=\begin{pmatrix} 1 & 1 \\ 0 & 1
\end{pmatrix},
\]
i.e. the multiplier $\mu=1$ is double non semi-simple. According
to the proposed notation $\sigma=0$ and thus
\[
L_0=\begin{pmatrix} 0 & 1 \\ 0 & 0 \end{pmatrix}, \quad \tilde
L_0=0.
\]
We are in a situation where we can apply Theorem 2 from
\cite{Io:88}. In particular, we have that
\[
 \ds\dd{\tau}{t}=1+\xi+p(\tau,\xi), \quad \ds\dd{\xi}{\tau}=\tilde L_0 \xi + P(\tau,\xi),
\]
where $\tau$ plays the role of phase coordinate along the orbit. The polynomials $p$ and $P$ are $T$-periodic in $\tau$ and at least quadratic in $\xi$
such that
\[
\ds\dd{}{\tau}p(\tau,\xi)-\ds\dd{}{\xi}p(\tau,\xi)\tilde L_0^*\xi=0, \quad
\ds\dd{}{\tau}P(\tau,\xi)+\tilde L_0^* P(\tau,\xi)-\ds\dd{}{\xi}P(\tau,\xi)\tilde L_0^*\xi=0\]
for all $\tau$ and $\xi\in \R$. Putting $\tilde
L_0=0$ we obtain
\[
\ds\dd{}{\tau}p(\tau,\xi)=0, \quad \ds\dd{}{\tau}P(\tau, \xi)=0,
\]
i.e. the two polynomials are independent from $\tau$. So, by a Taylor expansion of the two polynomials the normal form becomes
\[
\begin{cases}
 \dd{\tau}{t}= 1+\xi+p(\xi)=1+\xi+\alpha_1\xi^2+\alpha_2'\xi^3+\ldots, \smallskip \\
 \dd{\xi}{\tau}=P(\xi) = b\xi^2+c \xi^3 + \ldots.
\end{cases}
\]
Applying the chain rule gives 
\begin{eqnarray*}
\dd{\xi}{t} & = & (b\xi^2+c \xi^3 + \ldots)(1+\xi+\alpha_1\xi^2+\alpha_2'\xi^3+\ldots)\\
 & = & b\xi^2+b\xi^3+c \xi^3+\ldots\\
 & = & c\xi^3+\ldots,
 \end{eqnarray*}
because of the cusp degeneracy condition. In order to obtain the proposed normal form \eqref{eq:NF-CPC}, we do the substitution $\xi \mapsto -\xi$ and find
\[
\begin{cases}
 \dd{\tau}{t}= 1-\xi+\alpha_1\xi^2+\alpha_2\xi^3+\ldots, \smallskip \\
 \dd{\xi}{t}=c\xi^3+\ldots,
\end{cases}
\]
with $\alpha_2=-\alpha_2'$.
\bigskip
\subsubsection{\tt GPD} At the GPD bifurcation the matrices described
in Section \ref{sectionNotation} are
\[
M_0=\begin{pmatrix} 1 & 0 \\ 0 & -1
\end{pmatrix}, \quad
L_0=\begin{pmatrix} 0 & 0 \\ 0 & 0
\end{pmatrix}, \quad
\tilde L_0 =0.
\]
We are in a case in which we can apply Theorem 3 from
\cite{Io:88}. So we have the following $2 T$-periodic normal form
(using the formula of Theorem 1 from \cite{Io:88})
\[
 \ds\dd{\tau}{t}=1+p(\tau,\xi), \quad \ds\dd{\xi}{\tau}=\tilde L_0 \xi + P(\tau,\xi),
\]
with polynomials $p$ and $P$ $2 T$-periodic in
$\tau$ and at least quadratic in $\xi$ such that
\begin{eqnarray*}
\ds\dd{}{\tau}p(\tau,\xi)-\ds\dd{}{\xi}p(\tau,\xi)\tilde L_0^*\xi=0, &\quad&
\ds\dd{}{\tau}P(\tau,\xi)+\tilde L_0^* P(\tau,\xi)-\ds\dd{}{\xi}P(\tau,\xi)\tilde L_0^*\xi=0, \\
p(\tau+T,\xi)=p(\tau,-\xi), &\quad& P(\tau+T,-\xi)=-P(\tau,\xi).
\end{eqnarray*}
Putting $\tilde L_0=0$ in the first two formulas brings us in the same situation as of the previous case
\[
\ds\dd{}{\tau}p(\tau,\xi)=0, \quad \ds\dd{}{\tau}P(\tau, \xi)=0,
\]
i.e. the two polynomials are independent of $\tau$. This makes it possible to rewrite the last two formulas as
\[
p(\xi)=p(-\xi), \quad P(-\xi)=-P(\xi),
\]
so polynomial $p$ is even ($p=\phi(\xi^2)$) and polynomial
$P$ is odd ($P=\xi \psi(\xi^2)$). Therefore, taking the {\tt GPD} degenerate condition into account, we can write down the first
approximation of our normal form
\[
\begin{cases}
\dd{\tau}{t}=\ds1+\phi(\xi^2)=\ds1+\alpha_1 \xi^2+\alpha_2 \xi^4 + \ldots,\smallskip \\
\dd{\xi}{\tau}=\ds\xi \psi(\xi^2)=c\xi^3+e \xi^5 + \ldots.
\end{cases}
\]
Applying the chain rules gives 
\begin{eqnarray*}
	\dd{\xi}{t} & = & (c\xi^3+e \xi^5 + \ldots)(\ds1+\alpha_1 \xi^2+\alpha_2 \xi^4 + \ldots)\\
	& = & c\xi^3+\alpha_1c\xi^5 +e \xi^5 + \ldots\\
	& = & e \xi^5 + \ldots,
\end{eqnarray*}
because of the GPD degeneracy condition. So we obtain the normal form presented in (\ref{eq:NF-GPD}), namely
\[
\begin{cases}
\dd{\tau}{t}=\ds1+\alpha_1 \xi^2+\alpha_2 \xi^4+\ldots,\smallskip \\
\dd{\xi}{t}=e \xi^5 + \ldots.
\end{cases}
\]
 
\subsection{Bifurcations with 3 critical eigenvalues}

\subsubsection{\tt CH} In the CH case the Jordan block
associated to the trivial multiplier is one-dimensional. We have
\[
M_0=\begin{pmatrix} 1 & 0 & 0 \\ 0 & e^{i \omega T} & 0 \\
0 & 0 & e^{-i \omega T}
\end{pmatrix}, \quad
L_0=\begin{pmatrix} 0&0&0\\0 & i \omega & 0 \\ 0 & 0 & -i \omega
\end{pmatrix}\quad
\tilde L_0=\begin{pmatrix} i \omega & 0 \\ 0 & -i \omega
\end{pmatrix}.
\]
This puts us in a situation in which we can
apply Theorem 1 from \cite{Io:88}. If we assume that
$\frac{\omega T}{2 \pi} \not\in \mathbb Q$, then it follows immediately from the results of Example III.9 from \cite{IoAd:92} that the normal form is given by
\[
\begin{cases}
 \dd{\tau}{t}= 1+\phi(|\xi|^2), \smallskip \\
 \dd{\xi}{\tau}= i \omega \xi + \xi \psi(|\xi|^2),
\end{cases}
\]
where the polynomials $\phi$ and $\psi$ are at least linear in their
argument. $\phi$ is real, while $\psi$ takes values in $\C$. We epand the polynomials up to the fifth order, namely
\[
\begin{cases}
 \dd{\tau}{t}= 1+\alpha_1 |\xi|^2+\alpha_2 |\xi|^4 + \ldots, \smallskip \\
 \dd{\xi}{\tau}= i \omega \xi + c' \xi |\xi|^2+ e' \xi |\xi|^4 + \ldots,
\end{cases}
\]
such that the chain rule gives
\begin{eqnarray*}
	\dd{\xi}{t} & = & (i \omega \xi + c' \xi |\xi|^2+ e' \xi |\xi|^4 + \ldots)(1+\alpha_1 |\xi|^2+\alpha_2 |\xi|^4 + \ldots)\\
	& = & i \omega \xi + i\omega \alpha_1 \xi |\xi|^2+c'\xi |\xi|^2+i\omega \alpha_2 \xi|\xi|^4+\alpha_1c'\xi |\xi|^4 + e'\xi|\xi|^4+\ldots\\
	& = & i \omega \xi + ic\xi |\xi|^2+e\xi|\xi|^4+\ldots.\\
\end{eqnarray*}
Here $i\omega \alpha_1+c'=ic$ since the Lyapunov coefficient of the Neimark-Sacker bifurcation is purely imaginary and $i\omega\alpha_2+\alpha_1c'+e' =e$. This gives us normal form \eqref{eq:NF-GNS}.

\subsubsection{\tt R1} At the R1 bifurcation the matrices described
in Section \ref{sectionNotation} are
\[
M_0=\begin{pmatrix} 1 & 1 & 0 \\ 0 & 1 & 1 \\ 0 & 0 & 1
\end{pmatrix}, \quad
L_0=\begin{pmatrix} 0 & 1 & 0 \\ 0 & 0 & 1 \\ 0 & 0 & 0
\end{pmatrix}, \quad
\tilde L_0 =\begin{pmatrix}  0 & 1 \\ 0 & 0
\end{pmatrix}.
\]
We are in a case in which we can apply Theorem 2 from
\cite{Io:88}. So we can define a $T$-periodic normal form
\[
\ds\dd{\tau}{t}=1+\xi_1+p(\tau,\xi), \quad \ds\dd{\xi}{\tau}=\tilde L_0
\xi + P(\tau,\xi),
\]
where $\xi=(\xi_1,\xi_2)$. The polynomials $p$ and $P$ are $T$-periodic in $\tau$ and
at least quadratic in $(\xi_1,\xi_2)$ such that
\begin{eqnarray*}
\ds\dd{}{\tau}p(\tau,\xi)-\ds\dd{}{\xi}p(\tau,\xi)\tilde L_0^*\xi=0, \quad
\ds\dd{}{\tau}P(\tau,\xi)+\tilde L_0^* P(\tau,\xi)-\ds\dd{}{\xi}P(\tau,\xi)\tilde L_0^*\xi=0.
\end{eqnarray*}
If we write the polynomials in a Fourier expansion, namely
\begin{equation*}
p(\tau,\xi)= \sum_{l=-\infty}^{\infty} p_l(\xi) e^{i \frac{2 \pi l
\tau}{T}}, \quad P(\tau,\xi)= \sum_{l=-\infty}^{\infty} P_l(\xi) e^{i
\frac{2 \pi l \tau}{T}},
\end{equation*}
we obtain for any $l\in \mathbb Z$ the following differential equations
\begin{eqnarray*}
 &\dd{}{\xi}p_l(\xi) \tilde L_0^* \xi-i\frac{2 \pi l}{T} p_l(\xi)=0,& \\
 &\dd{}{\xi}P_l(\xi) \tilde L_0^* \xi-i\frac{2 \pi l}{T} P_l(\xi)- \tilde L_0^*
 P_l(\xi)=0.&
\end{eqnarray*}
Putting our $\tilde L_0$ into the equations and writing
$P_l(\xi_1,\xi_2)=(P^{(1)}_l(\xi_1,\xi_2),P^{(2)}_l(\xi_1,\xi_2))$
we can rewrite them as a set of differential equations in variable $\xi_2$
\begin{eqnarray*}
 \dd{}{\xi_2}p_l(\xi_1,\xi_2) &=& i \frac{2 \pi l}{T \xi_1} p_l(\xi_1,\xi_2), \\
 \dd{}{\xi_2}P^{(1)}_l(\xi_1,\xi_2) &=&  i \frac{2 \pi l}{T \xi_1} P^{(1)}_l(\xi_1,\xi_2), \\
 \dd{}{\xi_2}P^{(2)}_l(\xi_1,\xi_2) &=& \frac{1}{\xi_1} \left(i \frac{2 \pi l}{T} P^{(2)}_l(\xi_1,\xi_2)+P^{(1)}_l(\xi_1,\xi_2)\right).
\end{eqnarray*}
Since
$p_l(\xi_1,\xi_2),P^{(1)}_l(\xi_1,\xi_2)$ and $P^{(2)}_l(\xi_1,\xi_2)$
are polynomials, if $l\neq 0$ the only solution 
is the trivial one. So $l$ equals zero and thus the polynomials are $\tau$ independent.
We obtain
\[
 \dd{}{\xi_2}p_0(\xi_1,\xi_2) =
 \dd{}{\xi_2}P^{(1)}_0(\xi_1,\xi_2) = 0, \qquad
 \dd{}{\xi_2}P^{(2)}_0(\xi_1,\xi_2) =
 \frac{1}{\xi_1}P^{(1)}_0(\xi_1,\xi_2).
\]
The first two equations show that $p_0$ and $P^{(1)}_0$ are independent from $\xi_2$, thus
\[
p_0(\xi_1)=\phi_0(\xi_1), \qquad P^{(1)}_0(\xi_1)=\xi_1 \chi(\xi_1).
\]
Integrating the last differential equation gives
\[
P^{(2)}_0(\xi_1,\xi_2)= \xi_2 \chi(\xi_1)+\psi(\xi_1).
\]
Now we can further simplify our normal form. In fact, since the homological operator associated with $\tilde L_0$ has a
two-dimensional null-space, we can make a change of
variables such that polynomial $P_0^{(1)}$ vanishes (see \cite{IoAd:92}). So we can write
\[
\tilde P_0^{(1)}(\xi_1)=0, \qquad \tilde P_0^{(2)}(\xi_1,\xi_2)=
\xi_2\phi_1(\xi_1) +\phi_2(\xi_1),
\]
where $\phi_1$ and $\phi_2$ are polynomials satisfying
$\phi_1(0)=\phi_2(0)=\left.\dd{\phi_2}{\xi_1}\right|_{\xi_1=0}=0$.

Assembling all the information gives us the following normal form
\[
\begin{cases}
 \ds\dd{\tau}{t}=1+\xi_1+\phi_0(\xi_1)=1+\xi_1+\alpha \xi_1^2+\ldots, \smallskip \\
 \ds\dd{\xi_1}{\tau}= \xi_2, \smallskip \\
 \ds\dd{\xi_2}{\tau}= \xi_2\phi_1(\xi_1) +\phi_2(\xi_1)=a \xi_1^2+ b \xi_1 \xi_2 +\ldots.
\end{cases}
\]
Note that the polynomials $\phi_0$ and $\phi_2$ are at least
quadratic in $\xi_1$, while $\phi_1$ is at least linear in its
argument. In order to obtain the normal form presented in
(\ref{eq:NF-11C}) we apply the chain rule which gives
\begin{eqnarray*}
 \ds\dd{\xi_1}{t} & = & \xi_2(1+\xi_1+\alpha \xi_1^2+\ldots)\\
 & = & \xi_2+\xi_1\xi_2+\ldots
\end{eqnarray*}
and 
\begin{eqnarray*}
 \ds\dd{\xi_2}{t} & = & (a \xi_1^2+ b \xi_1 \xi_2 +\ldots)(1+\xi_1+\alpha \xi_1^2+\ldots)\\
 & = & a \xi_1^2+ b \xi_1 \xi_2 +\ldots.
\end{eqnarray*}

\subsubsection{\tt R2} At the R2 bifurcation the matrices described
in Section \ref{sectionNotation} are
\[
M_0=\begin{pmatrix} 1 & 0 & 0 \\ 0 & -1 & 1 \\ 0 & 0 & -1
\end{pmatrix}, \quad
L_0=\begin{pmatrix} 0 & 0 & 0 \\ 0 & 0 & 1 \\ 0 & 0 & 0
\end{pmatrix}, \quad
\tilde L_0 =\begin{pmatrix}  0 & 1 \\ 0 & 0
\end{pmatrix}.
\]
We are in a case in which we can apply Theorem 3 from
\cite{Io:88}. So we have a $2 T$-periodic normal form
\[
 \ds\dd{\tau}{t}=1+p(\tau,\xi), \quad \ds\dd{\xi}{\tau}=\tilde L_0 \xi + P(\tau,\xi),
\]
where $\xi=(\xi_1,\xi_2)$. The polynomials $p$ and $P$ are $2 T$-periodic in
$\tau$ and at least quadratic in their argument with
\begin{eqnarray}\label{eq:R2_proprieties1}
\ds\dd{}{\tau}p(\tau,\xi)-\ds\dd{}{\xi}p(\tau,\xi)\tilde L_0^*\xi=0, &\quad&
\ds\dd{}{\tau}P(\tau,\xi)+\tilde L_0^* P(\tau,\xi)-\ds\dd{}{\xi}P(\tau,\xi)\tilde L_0^*\xi=0, 
\\ \label{eq:R2_proprieties2} p(\tau+T,\xi)=p(\tau,-\xi),
&\quad& P(\tau+T,-\xi)=-P(\tau,\xi).
\end{eqnarray}
Similar as in the R1 case (since the $\tilde L_0$ matrix is the same) we obtain that all
polynomials are independent from $\tau$, $l$ has to be equal to zero and the first two polynomials
$p$ and $P^{(1)}$ are independent from $\xi_2$.

Since the polynomials obtained are independent from $\tau$, we can
rewrite \eqref{eq:R2_proprieties2} as:
\[
 p(\xi)=p(-\xi), \quad P(-\xi)=-P(\xi),
\]
obtaining that polynomial $p$ is even
($p(\xi_1)=\phi_0(\xi_1^2)$) and the polynomials $P(\xi_1,\xi_2)$
are odd ($P^{(1)}(\xi_1)=\xi_1 \tilde\phi_1(\xi_1^2)$ and
$P^{(2)}(\xi_1,\xi_2)= \xi_2
\tilde\phi_1(\xi_1^2)+\xi_1\tilde\phi_2(\xi_1^2)$). Now we can simplify our normal form by changing variables as discussed in the previous section. So we can write
\[
\tilde P^{(1)}(\xi_1)=0, \qquad \tilde P^{(2)}(\xi_1,\xi_2)=
\xi_2\phi_1(\xi_1^2) +\xi_1\phi_2(\xi_1^2).
\]
Putting all information in the normal form equations gives
the system
\[
\begin{cases}
 \ds\dd{\tau}{t}=1+\phi_0(\xi_1^2)=1+\alpha \xi_1^2+ \ldots, \smallskip \\
 \ds\dd{\xi_1}{\tau}= \xi_2, \smallskip \\
 \ds\dd{\xi_2}{\tau}= \xi_2\phi_1(\xi_1^2) +\xi_1\phi_2(\xi_1^2)=a \xi_1^3+b\xi_1^2\xi_2 + \ldots.
\end{cases}
\]
In order to obtain normal form (\ref{eq:NF-12C}) we
apply the chain rule
\begin{eqnarray*}
\ds\dd{\xi_1}{t} & = & \xi_2(1+\alpha \xi_1^2+ \ldots)\\
& = & \xi_2+\alpha \xi_1^2\xi_2+ \ldots,\\
\ds\dd{\xi_2}{t} & = & (a \xi_1^3+b\xi_1^2\xi_2 + \ldots)(1+\alpha \xi_1^2+ \ldots)\\
& = & a \xi_1^3+b\xi_1^2\xi_2 + \ldots.
\end{eqnarray*}

\subsubsection{\tt R3} This is a simple case, since the Jordan
block associated with the trivial multiplier is one-dimensional and
-1 is not a multiplier of the critical limit cycle. So we can
write
\[
M_0=\begin{pmatrix} 1 & 0 & 0 \\ 0 & e^{i\frac{2\pi}{3}} & 0 \\ 0 &
0 & e^{-i\frac{2\pi}{3}}
\end{pmatrix}, \quad
L_0=\begin{pmatrix} 0 & 0 & 0 \\ 0 & i\frac{2\pi}{3 T} & 0 \\ 0 &
0 & -i\frac{2\pi}{3 T}
\end{pmatrix}, \quad
\tilde L_0=\begin{pmatrix} i\frac{2\pi}{3 T} & 0 \\ 0 & -i\frac{
2\pi}{3 T}
\end{pmatrix}.
\]
We are in a case in which we can apply Theorem 1 from
\cite{Io:88}. So we can define the following $T$-periodic normal form
\[
\ds\dd{\tau}{t}=1+p(\tau,z), \quad \ds\dd{z}{\tau}=\tilde L_0 z + P(\tau,z),
\]
where $z=(z_1,\bar z_1)$. The polynomials $p$ and $P$ are $T$-periodic in $\tau$ and at
least quadratic in their argument such that
\begin{eqnarray*}
\ds\dd{}{\tau}p(\tau,z)-\ds\dd{}{z}p(\tau,z)\tilde L_0^*z=0, \quad
\ds\dd{}{\tau}P(\tau,z)+\tilde L_0^* P(\tau,z)-\ds\dd{}{z}P(\tau,z)\tilde L_0^*z=0.
\end{eqnarray*}

We apply the results obtained in Example III.9 from \cite{IoAd:92} with $\omega T
/2\pi = 1/3$ to obtain
\[
\begin{cases}
 \ds\dd{\tau}{t}= 1 + P_0(|z_1|^2,\bar z_1^3 e^{i 2 \pi \tau/T},z_1^3 e^{-i 2 \pi \tau/T}), \smallskip \\
 \ds\dd{z_1}{\tau}= \frac{i 2 \pi}{3 T} z_1 + z_1 Q_0(|z_1|^2,z_1^3 e^{-i 2 \pi
 \tau/T})+\bar z_1^2 e^{i 2 \pi \tau/T} Q_1(|z_1|^2,\bar z_1^3 e^{i 2 \pi
 \tau/T}).
 \end{cases}
\]
Defining a new variable $\xi= e^{-i \frac{2 \pi \tau}{3T}} z_1$, this
system can be rewritten as
\[
 \begin{cases}
 \ds\dd{\tau}{t}=1 + \phi_0(|\xi|^2,\bar \xi^3 ,\xi^3), \smallskip \\
 \ds\dd{\xi}{\tau}= \xi \phi_1(|\xi|^2,\xi^3)+\bar \xi^2 \phi_2(|\xi|^2,\bar \xi^3),
\end{cases}
\]
with polynomials $\phi_{0}$ and $\phi_{1}$ at least linear in their
arguments, while $\phi_2(0)\neq 0$. Notice that this system is
autonomous and equivariant under the
rotations of angle $2 \pi /3$. Expanding the polynomials gives 
\[
 \begin{cases}
 \ds\dd{\tau}{t}=1+\alpha_1 |\xi|^2+\alpha_2 \xi^3+\alpha_3 \bar \xi^3+ \ldots, \smallskip \\
 \ds\dd{\xi}{\tau}= b \bar \xi^2+  c \xi |\xi|^2 + \ldots,
\end{cases}
\]
and thus
\begin{eqnarray*}
\ds\dd{\xi}{t} & = & (b \bar \xi^2+  c \xi |\xi|^2 + \ldots)(1+\alpha_1 |\xi|^2+\alpha_2 \xi^3+\alpha_3 \bar \xi^3+ \ldots)\\
& = & b \bar \xi^2+  c \xi |\xi|^2 + \ldots,
\end{eqnarray*}
so normal form
\eqref{eq:NF-13C} is obtained.

\subsubsection{\tt R4} As in the previous case the Jordan
block associated with the trivial multiplier is one-dimensional. The
matrices in Section \ref{sectionNotation} are
\[
M_0=\begin{pmatrix} 1 & 0 & 0 \\ 0 & e^{i\frac{\pi}{2}} & 0 \\ 0 &
0 & e^{-i\frac{\pi}{2}}
\end{pmatrix}, \quad
L_0=\begin{pmatrix} 0 & 0 & 0 \\ 0 & i\frac{\pi}{2 T} & 0 \\ 0 & 0
& -i\frac{\pi}{2 T}
\end{pmatrix}, \quad
\tilde L_0=\begin{pmatrix} i\frac{\pi}{2 T} & 0 \\ 0 & -i\frac{
\pi}{2 T}
\end{pmatrix}.
\]
We can apply Theorem 1 from
\cite{Io:88} and define a $T$-periodic normal form
\[
\ds\dd{\tau}{t}=1+p(\tau,z), \quad \ds\dd{z}{\tau}=\tilde L_0 z + P(\tau,z),
\]
where $z=(z_1,\bar z_1)$. The polynomials $p$ and $P$ are $T$-periodic in $\tau$ and at
least quadratic in their argument such that
\begin{eqnarray*}
\ds\dd{}{\tau}p(\tau,z)-\ds\dd{}{z}p(\tau,z)\tilde L_0^*z=0, \quad
\ds\dd{}{\tau}P(\tau,z)+\tilde L_0^* P(\tau,z)-\ds\dd{}{z}P(\tau,z)\tilde L_0^*z=0.
\end{eqnarray*}

Again, we use Example III.9 from \cite{IoAd:92} with $\omega T
/2\pi = 1/4$ and obtain
\[
\begin{cases}
 \ds\dd{\tau}{t}= 1 + P_0(|z_1|^2,\bar z_1^4 e^{i 2 \pi \tau/T},z_1^4 e^{-i 2 \pi \tau/T}), \smallskip \\
 \ds\dd{z_1}{\tau}= \frac{i \pi}{2 T} z_1 + z_1 Q_0(|z_1|^2,z_1^4 e^{-i 2 \pi
 \tau/T})+\bar z_1^3 e^{i 2 \pi \tau/T} Q_1(|z_1|^2,\bar z_1^4 e^{i 2 \pi
 \tau/T}).
 \end{cases}
\]
Defining a new variable $\xi=e^{-i \frac{\pi \tau}{2T}} z_1$, the
system can be rewritten as
\[
 \begin{cases}
 \ds\dd{\tau}{t}=1 + \phi_0(|\xi|^2,\bar \xi^4 ,\xi^4), \smallskip \\
 \ds\dd{\xi}{\tau}= \xi \phi_1(|\xi|^2,\xi^4)+\bar \xi^3 \phi_2(|\xi|^2,\bar \xi^4),
\end{cases}
\]
with polynomials $\phi_{0}$ and $\phi_{1}$ at least linear in their
arguments, while $\phi_2(0)\neq 0$. Notice that this system is
autonomous and equivariant under the
rotations of angle $\pi /2$. Expanding the polynomials gives
\[
 \begin{cases}
 \ds\dd{\tau}{t}=1+\alpha_1 |\xi|^2+\alpha_2 \xi^4+\alpha_3 \bar \xi^4+ \ldots, \smallskip \\
 \ds\dd{\xi}{\tau}=c \xi |\xi|^2 + d \bar \xi^3 + \ldots,
\end{cases}
\]
We still have to transform the second equation, namely 
\begin{eqnarray*}
\ds\dd{\xi}{t} & = & (c \xi |\xi|^2 + d \bar \xi^3 + \ldots)(1+\alpha_1 |\xi|^2+\alpha_2 \xi^4+\alpha_3 \bar \xi^4+ \ldots)\\
& = & c \xi |\xi|^2 + d \bar \xi^3+ \ldots,
\end{eqnarray*}
so normal form \eqref{eq:NF-14C} is obtained.

\subsubsection{\tt LPPD} At the LPPD bifurcation the matrices from Section \ref{sectionNotation} are
\[
M_0=\begin{pmatrix} 1 & 1 & 0 \\ 0 & 1 & 0 \\ 0 & 0 & -1
\end{pmatrix}, \quad
L_0=\begin{pmatrix} 0 & 1 & 0 \\ 0 & 0 & 0 \\ 0 & 0 & 0
\end{pmatrix}, \quad
\tilde L_0 =\begin{pmatrix}  0 & 0 \\ 0 & 0
\end{pmatrix}.
\]

We are in a case in which we can apply Theorem 3 from
\cite{Io:88}. So we can define a $2 T$-periodic normal form
\[
 \ds\dd{\tau}{t}=1+\xi_1+p(\tau,\xi), \quad \ds\dd{\xi}{\tau}=\tilde L_0 \xi + P(\tau,\xi),
\]
where $\xi=(\xi_1,\xi_2)$. The polynomials $p$ and $P$ are $2 T$-periodic in
$\tau$ and at least quadratic in their argument such that
\begin{eqnarray}\label{eq:FF_proprieties1}
&&\ds\dd{}{\tau}p(\tau,\xi)-\ds\dd{}{\xi}p(\tau,\xi)\tilde L_0^*\xi=0, \quad
\ds\dd{}{\tau}P(\tau,\xi)+\tilde L_0^* P(\tau,\xi)-\ds\dd{}{\xi}P(\tau,\xi)\tilde L_0^*\xi=0, \\ \label{eq:FF_proprieties2} &&p(\tau+T,\xi_1,\xi_2)=p(\tau,\xi_1,-\xi_2),\\
\label{eq:FF_proprieties3} &&P^{(1)}(\tau+T,\xi_1,-\xi_2)=P^{(1)}(\tau,\xi_1,\xi_2), \quad P^{(2)}(\tau+T,\xi_1,-\xi_2)=-P^{(2)}(\tau,\xi_1,\xi_2).
\end{eqnarray}
By putting $\tilde L_0$ into \eqref{eq:FF_proprieties1},
we obtain
\[
\ds\dd{}{\tau}p(\tau,\xi_1,\xi_2)= \dd{}{\tau}
P^{(1)}(\tau,\xi_1,\xi_2)=\dd{}{\tau} P^{(2)}(\tau,\xi_1,\xi_2)=0,
\]
so our polynomials are independent from $\tau$. Then, using
\eqref{eq:FF_proprieties2} and \eqref{eq:FF_proprieties3}, there holds
\[
p(\xi_1,\xi_2)=p(\xi_1,-\xi_2), \quad
P^{(1)}(\xi_1,-\xi_2)=P^{(1)}(\xi_1,\xi_2), \quad
P^{(2)}(\xi_1,-\xi_2)=-P^{(2)}(\xi_1,\xi_2),
\]
so the polynomials are of the following form
\begin{eqnarray*}
 &&p=\chi_1(\xi_1)+\chi_2(\xi_2^2)(1+\chi_3(\xi_1)), \\
 &&P^{(1)}=\psi_1(\xi_1)+\psi_2(\xi_2^2)(1+\psi_2(\xi_1)), \\
 &&P^{(2)}=\xi_2 \varphi_1(\xi_1)+ \xi_2\varphi_2(\xi_2^2)(1+\varphi_3(\xi_1)),
\end{eqnarray*} 
with $\chi_1$ and $\psi_1$ at least quadratic in their argument and all
the other polynomials at least linear in their argument.

Assembling all the information gives the following system
\[\left\{
\begin{array}{lcl}
	\ds\dd{\tau}{t}& =& 1+\xi_1+\chi_1(\xi_1)+\chi_2(\xi_2^2)(1+\chi_3(\xi_1)) \\
	& = & 1+\xi_1 +\alpha_{20} \xi_1^2 +\alpha_{02} \xi_2^2 + \alpha_{30}'\xi_1^3+\alpha_{12}'\xi_1\xi_2^2+\ldots,\\
	\ds\dd{\xi_1}{\tau}& = & \psi_1(\xi_1)+\psi_2(\xi_2^2)(1+\psi_2(\xi_1)) \\
	& =& a_{20}' \xi_1^2 +a_{02}' \xi_2^2 + a_{30}'\xi_1^3+a_{12}'\xi_1\xi_2^2+\ldots,\\
	\ds\dd{\xi_2}{\tau}& = & \xi_2 \varphi_1(\xi_1)+ \xi_2\varphi_2(\xi_2^2)(1+\varphi_3(\xi_1))\\
	& = & b_{11}'\xi_1\xi_2+b_{21}'\xi_1^2\xi_2+b_{03} \xi_2^3+\ldots.
\end{array}\right.\]
Applying the chain rule gives
\begin{eqnarray*}
\ds\dd{\xi_1}{t}& =&  a_{20}' \xi_1^2 +a_{02}' \xi_2^2 + a_{30}'\xi_1^3+ a_{20}'\xi_1^3+a_{12}'\xi_1\xi_2^2+a_{02}'\xi_1\xi_2^2+\ldots\\
& =& a_{20}' \xi_1^2 +a_{02}' \xi_2^2 + a_{30}\xi_1^3+ a_{12}\xi_1\xi_2^2+\ldots,
\end{eqnarray*}
with $a_{30}=a_{30}'+a_{20}', a_{12}=a_{12}'+a_{02}'$, and 
\begin{eqnarray*}
\ds\dd{\xi_2}{t}& =&  b_{11}'\xi_1\xi_2+b_{21}'\xi_1^2\xi_2+b_{11}'\xi_1^2\xi_2+b_{03} \xi_2^3+\ldots\\
& =& b_{11}'\xi_1\xi_2+b_{21}\xi_1^2\xi_2+b_{03} \xi_2^3+\ldots,
\end{eqnarray*}
with $b_{21} = b_{21}'+b_{11}'$. Now, we do the substitution $\xi_1\mapsto -\xi_1$ which gives
\[\begin{cases}
 \ds\dd{\tau}{t}=1-\xi_1 +\alpha_{20} \xi_1^2 +\alpha_{02} \xi_2^2 - \alpha_{30}'\xi_1^3-\alpha_{12}'\xi_1\xi_2^2+\ldots, \smallskip \\
 \ds\dd{\xi_1}{t}= -a_{20}' \xi_1^2 -a_{02}' \xi_2^2 + a_{30}\xi_1^3+a_{12}\xi_1\xi_2^2+\ldots, \smallskip \\
 \ds\dd{\xi_2}{t}= -b_{11}'\xi_1\xi_2+b_{21}\xi_1^2\xi_2+b_{03} \xi_2^3+\ldots.
\end{cases}\]
By putting $\alpha_{30}=-\alpha_{30}', \alpha_{12}=-\alpha_{12}', a_{20}=-a_{20}', a_{02}=-a_{02}', b_{11}=-b_{11}'$, we obtain (\ref{eq:NF-FF}).

\section{Poincar\'{e} maps of the periodic normal forms}
\label{Appendix:2}
\subsection{Bifurcations with 2 critical eigenvalues}

\subsubsection{\tt CPC} If we reparametrize time, we obtain the
system
\[
\begin{cases}
 \dd{\tau}{\tau}= 1, \smallskip \\
 \dd{\xi}{\tau}={\displaystyle \frac{c \xi^3}{1-\xi+\alpha_1 \xi^2 + \alpha_2 \xi^3}} + \cdots.
\end{cases}
\]
Note that we can define a Poincar\'e map out of system
\eqref{eq:NF-CPC} evaluating the solution of this system starting
from a point $(0,\eta)$ at time $t=T$. Since the two equations are
uncoupled we can consider only the second one. Making one Picard
iteration \cite{Ku:2004}, we can construct an approximation of
this map as follows
\[
\xi_0=\eta, \qquad \xi_1=\xi_0+\int_{0}^{T} \frac{c
\xi_0^3}{1-\xi_0+\alpha_1 \xi_0^2 + \alpha_2 \xi_0^3} dt
\]
and expanding $\xi_1$ in a Taylor series we obtain
\begin{eqnarray}
 \eta \mapsto \eta +c T \eta ^3 + O(\eta^4).
\label{poincaremapcpc}
\end{eqnarray}
Further iterations do not change this expansion and so (\ref{poincaremapcpc}) is the Poincar\'e map of system \eqref{eq:NF-CPC}. Note that this Poincar\'{e} map is similar to the one for the cusp point of
fixed points \cite{Ku:2004}. Therefore, we can conclude that the behavior
of the system in the neighborhood of the bifurcation is the same.
In particular, referring always to \cite{Ku:2004}, we can draw
the bifurcation diagram of the Poincar\'{e} map and obtain
Figure \ref{fig:NF_CPC}. On the two drawn curves, with label $T_1$ and $T_2$,
two limit cycles collide and disappear. The output given by MatCont is the normal form coefficient $c$.
\begin{figure}[htb]
\centering \subfigure[$c<0$]{
 \includegraphics[width=.48\textwidth]{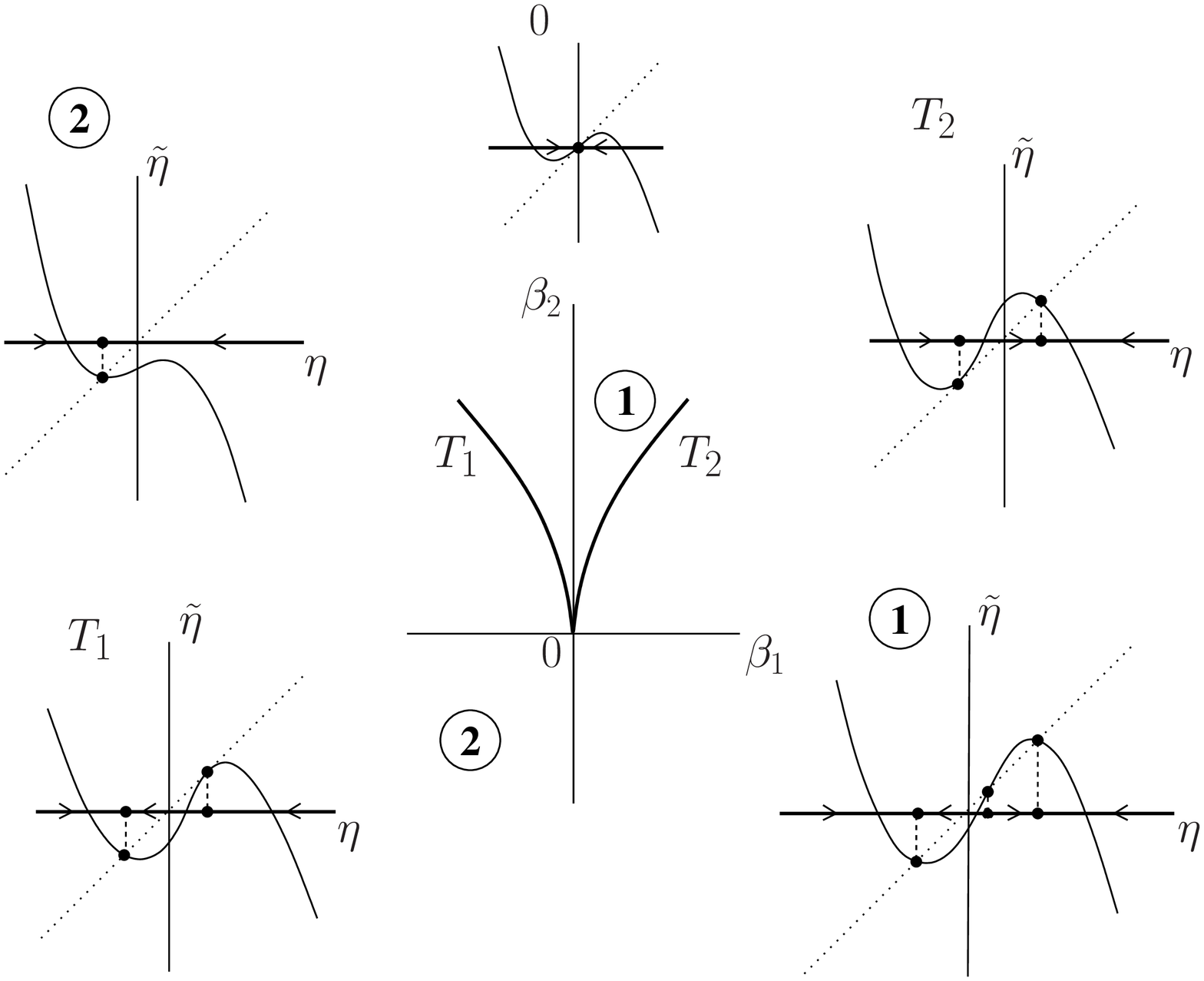}}
\subfigure[$c>0$]{
 \includegraphics[width=.48\textwidth]{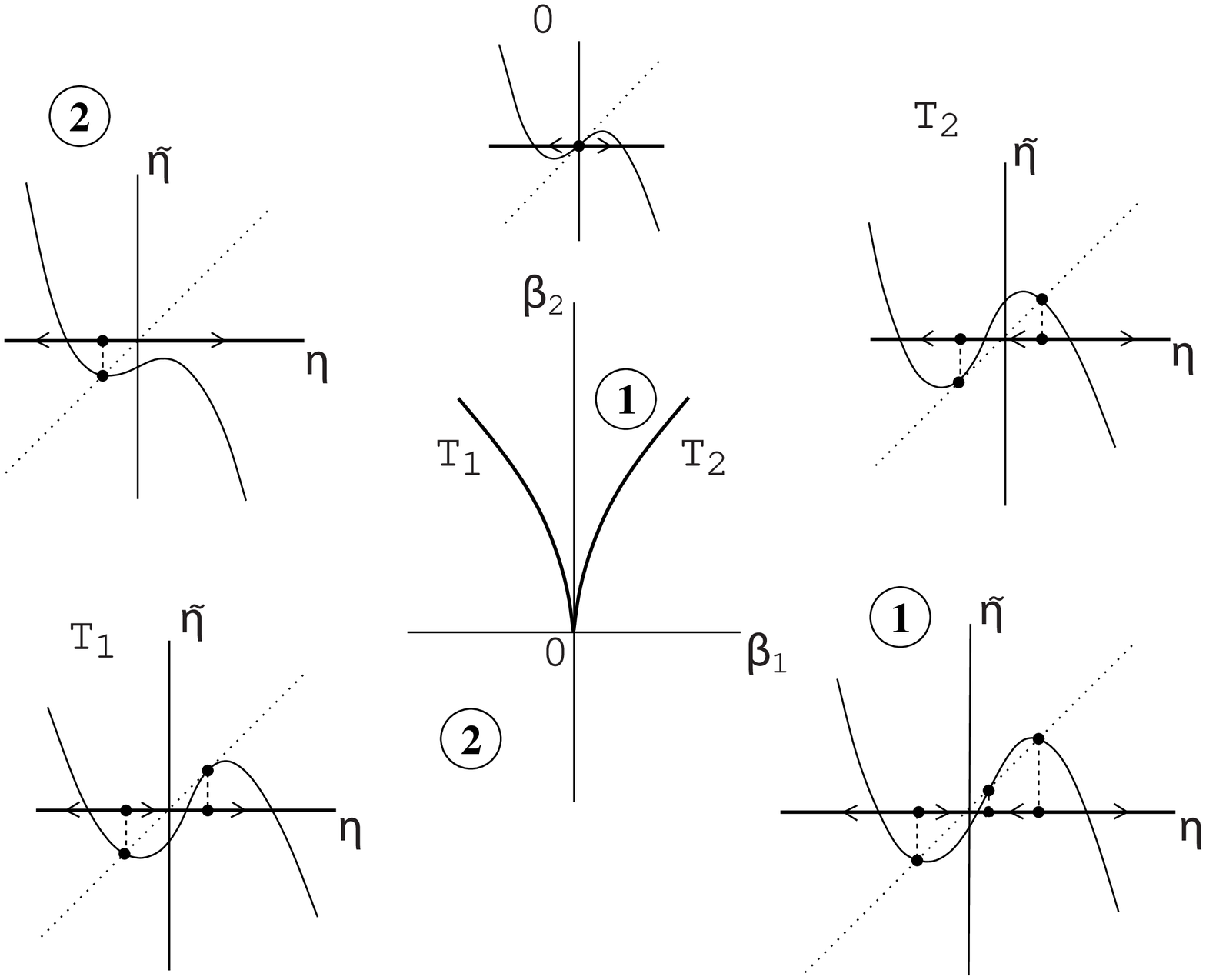}}
 \caption{Bifurcation diagram of the cusp bifurcation of the fixed point normal form.} \label{fig:NF_CPC}
\end{figure}

\subsubsection{\tt GPD} First, we have to calculate the second
iterate of the normal form of the Generalized flip bifurcation of
fixed points:
\begin{equation}\label{eq:GenFlip}
g(v)= -v+\frac{1}{120} \tilde e v^5 \quad\Longrightarrow\quad
g^2(\omega)= v-\frac{1}{60}\tilde e v^5.
\end{equation}
Reparametrizing the time of \eqref{eq:NF-GPD} and doing Taylor expansion up to the fifth order gives
\[
\begin{cases}
\dd{\tau}{\tau}=\ds 1, \smallskip \\
\dd{\xi}{\tau}=\ds e \xi^5 + \ldots,
\end{cases}
\]
where the dots are $O(\xi^6)$ terms $2T$-periodic in $\tau$. Doing
one Picard iteration of the second equation up to $2 T$ we obtain
the Poincar\'e map of our normal form
\[
\xi\mapsto \xi + 2T e \xi^5 + \ldots
\]
and see that it is the same map as the second iterate of the
generalized period-doubling normal form of fixed points. In
particular, since the coefficient of the fifth order term of the second iterate of the normal form of fixed points has opposite sign than the one of limit cycles,
we can conclude that the behavior of the system at the
bifurcation is the same but with opposite sign of the normal form
coefficient. Moreover, since $T$ is nonzero, the
non-degeneracy condition is $e \neq 0$. In fact, if $e>0$ we
obtain the bifurcation diagram reported in Figure
\ref{fig:NF_GPD}-(a), in which the limit point bifurcation of the period doubled
limit cycles ($T^{(2)}$) is tangent to the
supercritical period-doubling branch (the one in which the normal
form coefficient is positive, so labeled as
$F_+^{(1)}$), while if $e$ is negative we are in the opposite
situation, reported in Figure \ref{fig:NF_GPD}-(b). The output given by MatCont is the normal form coefficient $e$.
\begin{figure}[htb]
\centering \subfigure[$e>0$]{
 \includegraphics[width=.49\textwidth]{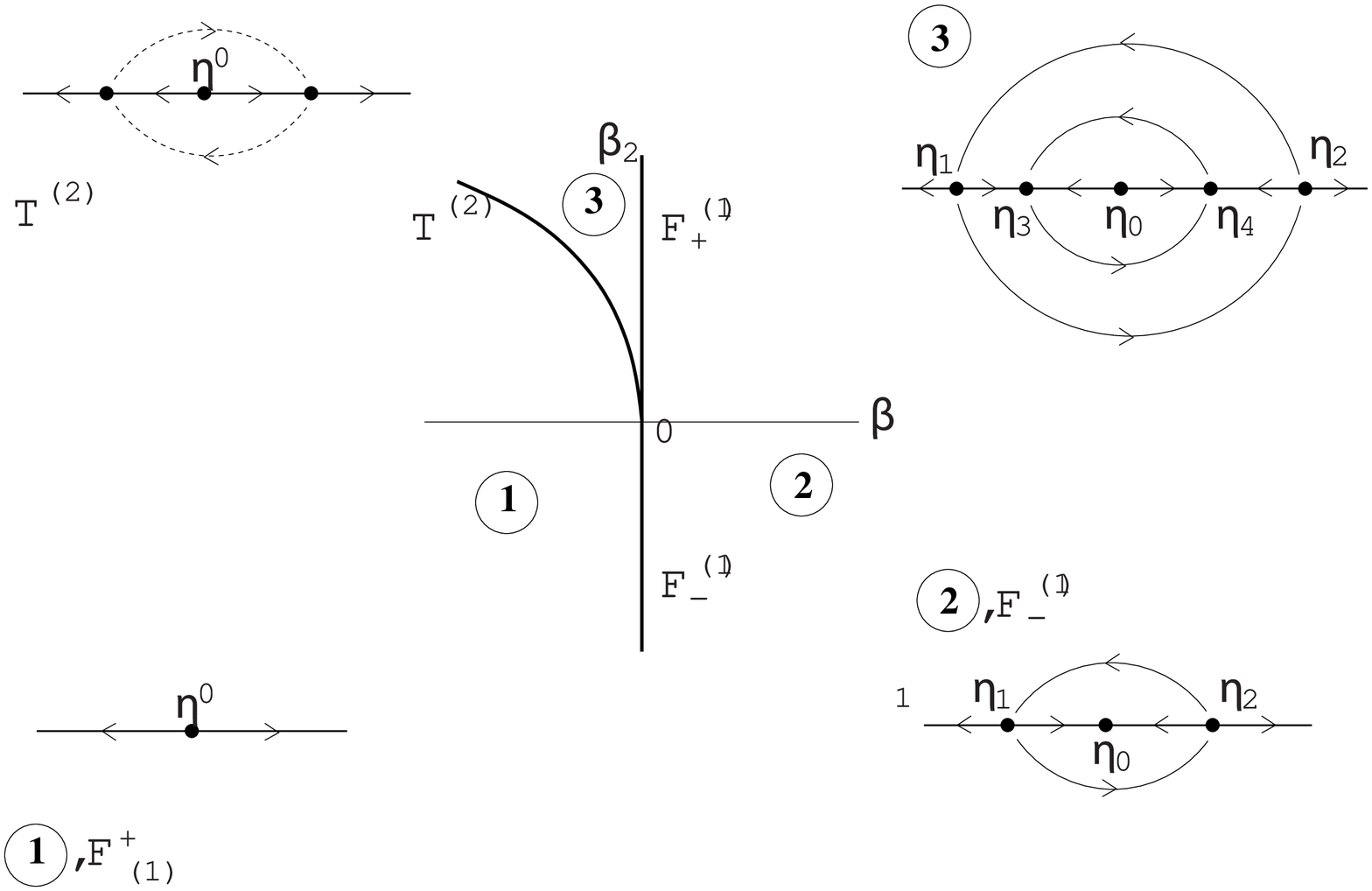}}\subfigure[$e<0$]{
 \includegraphics[width=.49\textwidth]{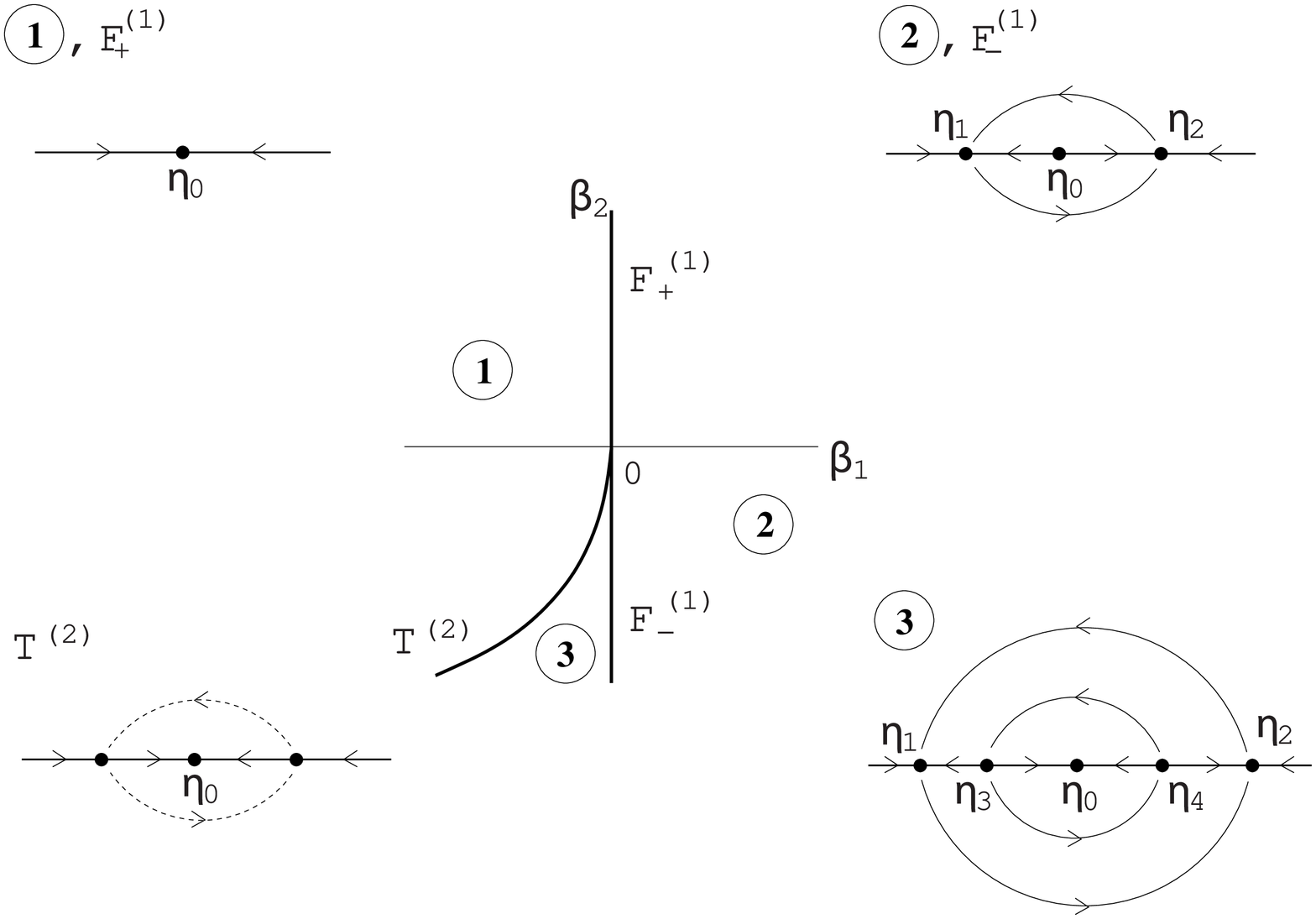}}
 \caption{Bifurcation diagram of the degenerate period-doubling point bifurcation of the fixed point normal form.} \label{fig:NF_GPD}
\end{figure}

Remark that the normal form coefficient of the period-doubling
bifurcation of limit cycles, computed through periodic
normalization as done in \cite{KuDoGoDh:05}, has opposite sign of
the one for fixed points. Also in this case, the normal form
coefficient of the GPD bifurcation of limit cycles has opposite
sign of the one for fixed points.

\subsection{Bifurcations with 3 critical eigenvalues}
\subsubsection{{\tt CH}} In this section we will show how the
periodic normal form \eqref{eq:NF-GNS} is related to the normal
form of the Chenciner bifurcation of fixed points, and how the non-degeneracy condition of the Chenciner bifurcation of limit cycles is related to the Chenciner bifurcation of maps. First note that, if we scale the time,
we can rewrite system \eqref{eq:NF-GNS} as
\[
\begin{cases}
 \dd{\tau}{\tau}=\ds 1,\smallskip \\
 \dd{\xi}{\tau}=\ds i \omega \xi + i (c-\alpha_1 \omega) \xi |\xi|^2+(e-i(\alpha_1 c- \alpha_1^2 \omega +\alpha_2\omega))\xi |\xi|^4+  \ldots. \\
\end{cases}
\]

We need to make a change
of variables in order to obtain a quasi-identity flow. Introducing
the new complex variable $z=e^{-i \omega \tau}\xi$, the second
equation can be rewritten as
\[
\dd{z}{\tau}=i (c -\alpha_1 \omega) z^2 \bar{z}+ \left(e -i
(\alpha_1 c-\alpha_1^2 \omega + \alpha_2 \omega)\right)z^3\bar
z^2+\ldots .
\]

Doing two Picard iterations up to time $T$, we obtain the
rotating Poincar\'e map of the system
\[
\xi \mapsto \xi +iT(c-\alpha_1 \omega )\xi ^2 \bar{\xi }+ T\left(e
-\frac{c^2 T}{2}+\alpha_1 c T \omega -\frac{1}{2} \alpha_1^2 T
\omega ^2+i\left(\alpha_1^2 \omega -\alpha_1 c-\alpha_2 \omega
\right)\right)\xi ^3 \bar{\xi }^2+\ldots.
\]
Note that this system has the same Poincar\'{e} map as the
Chenciner bifurcation in the fixed point case \cite{Ku:2004}, so we
expect the same bifurcation scenario on the Poincar\'{e} map of
the system. Notice that the real part of the first Lyapunov
coefficient is 0, since the Neimark-Sacker bifurcation is
degenerate. The sign of the second Lyapunov coefficient $L_2$ (as
defined on page 420 of \cite{Ku:2004}) determines the bifurcation
scenario. However, from (\ref{eq:NF-GNS}) we can derive that
$\Re(e)<0$ corresponds with a stable critical limit cycle and
$\Re(e)>0$ with an unstable critical limit cycle. Therefore, the case
$\Re(e)<0$ corresponds with the case $L_2<0$ and $\Re(e)>0$
corresponds with $L_2>0$. So $\Re(e)$ and the second
Lyapunov coefficient $L_2$ as defined in \cite{Ku:2004} have the
same sign and vanish at the same time. Since both coefficients have
the same effect and $L_2$ requires more computations, we
compute $\Re(e)$ to determine the bifurcation scenario, and
in this paper we will call $\Re(e)$ the second Lyapunov
coefficient. The bifurcation diagram in the neighborhood of this
codim 2 point can be found by looking at the sign of this value. When
$\Re(e)<0$ the outer invariant curve is stable and the limit point
of tori curve is tangent to the subcritical Neimark-Sacker branch, as
shown in Figure \ref{fig:NF_GNS}-(a). When $\Re(e)>0$ the outer
invariant curve is unstable and the limit point of tori curve is tangent
to the supercritical Neimark-Sacker branch, see Figure
\ref{fig:NF_GNS}-(b). The output given by MatCont is $\Re(e)$.
\begin{figure}[htb]
\centering \subfigure[$\Re(e)<0$]{
 \includegraphics[width=.49\textwidth]{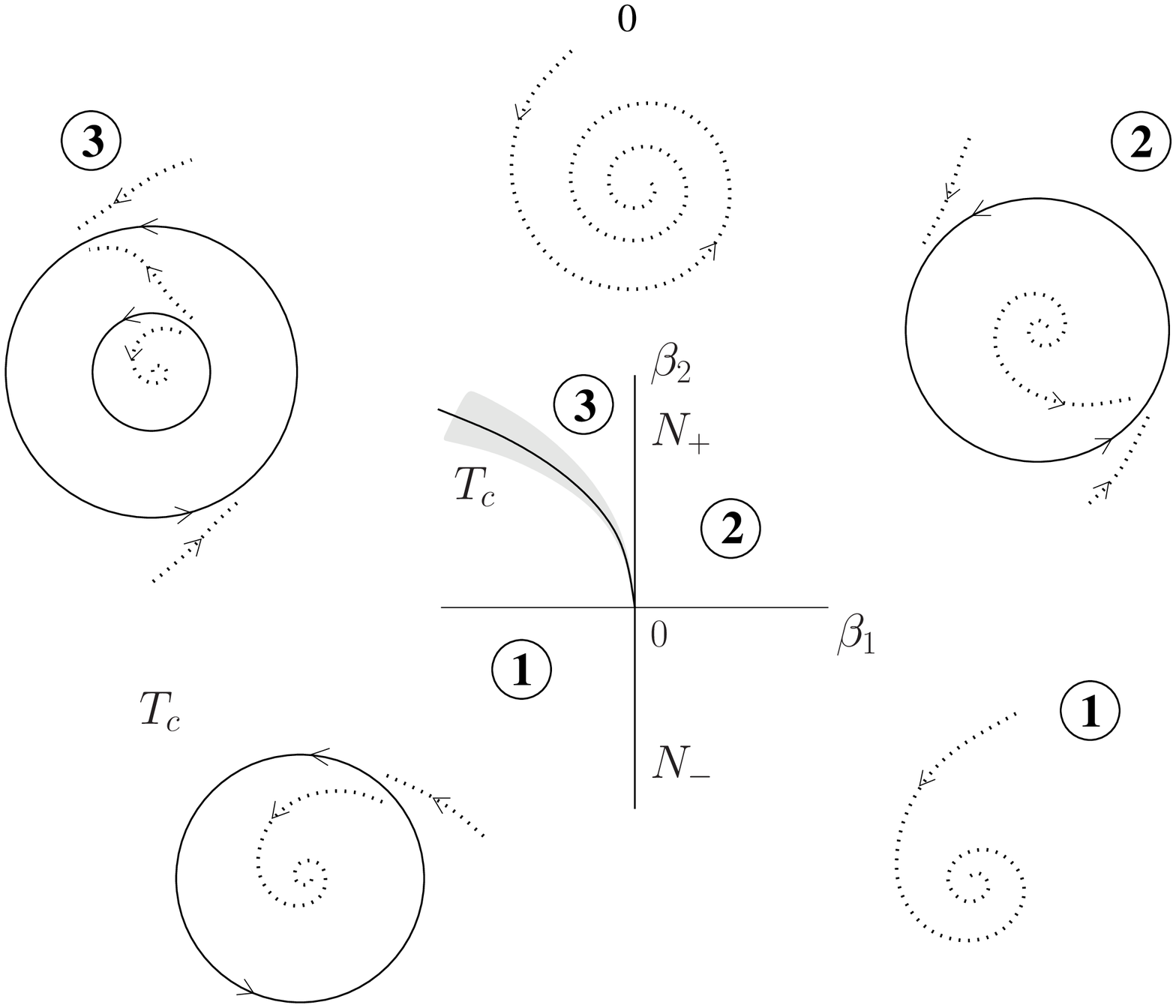}}\subfigure[$\Re(e)>0$]{
 \includegraphics[width=.49\textwidth]{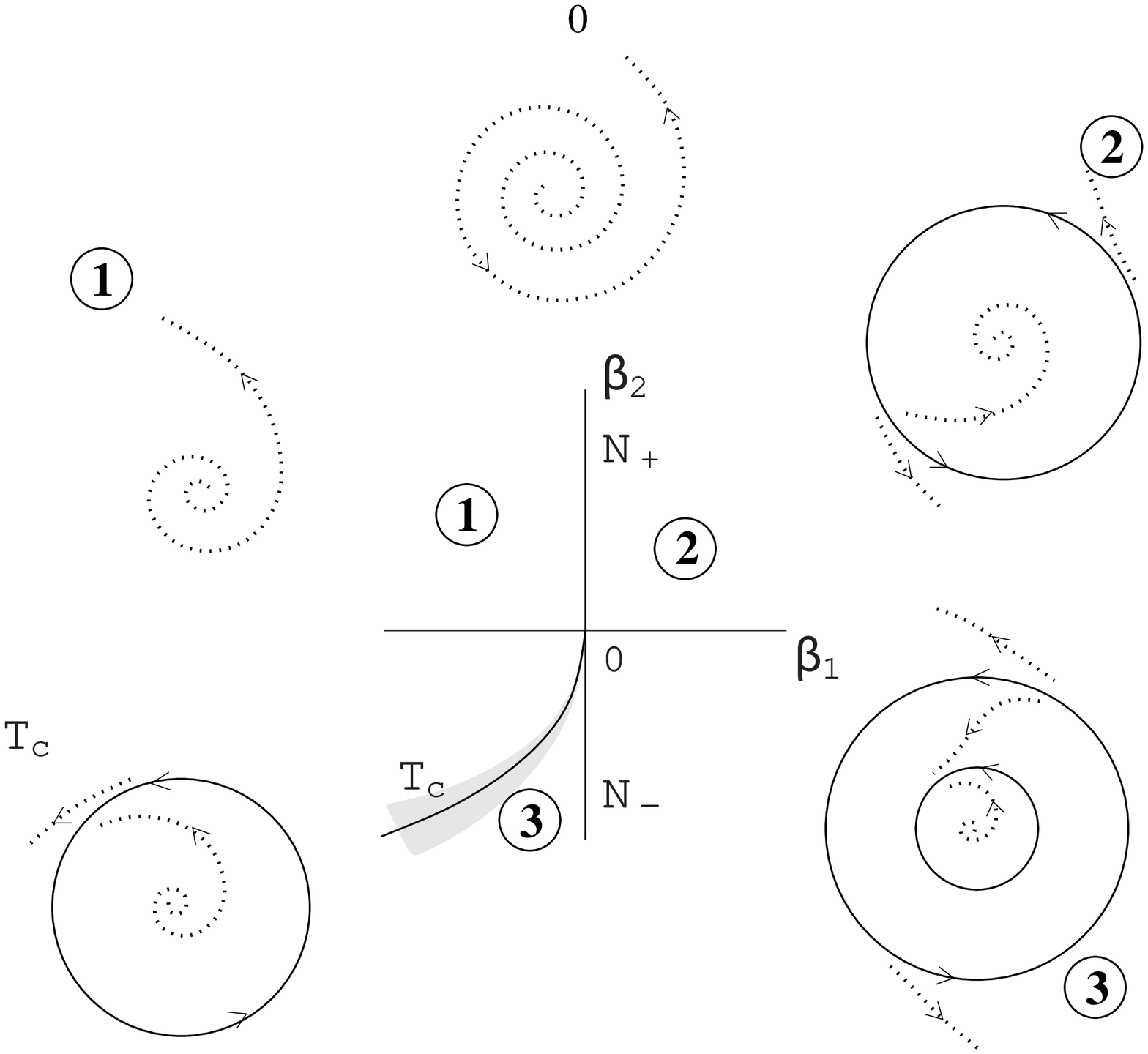}}
 \caption{Bifurcation diagram of the generalized Neimark-Sacker bifurcation of the fixed point normal form.} \label{fig:NF_GNS}
\end{figure}

\subsubsection{\tt R1} In this section we will show how the
periodic normal form \eqref{eq:NF-11C} is related to the normal
form for the 1:1 resonance of fixed points,  which allows us to formulate the
non-degeneracy condition. First note that, if we scale the time,
we can rewrite system \eqref{eq:NF-11C} as
\[
\begin{cases}
 \dd{\tau}{\tau}=\ds 1,\smallskip \\
 \dd{\xi_1}{\tau}=\ds \frac{\xi_2+\xi_1\xi_2}{1 -\xi_1+ \alpha  \xi_1^2}+ \ldots, \smallskip \\
 \dd{\xi_2}{\tau}=\ds \frac{a\xi_1^2+b \xi_1 \xi_2}{1 -\xi_1+ \alpha  \xi_1^2}+ \ldots.
\end{cases}
\]
Doing Picard iterations up to time $T$, it's possible to show
that the generated map is the same, up to second order terms, of the
one generated by Picard iterations up to time 1 of the truncated
system
\[
\begin{cases}
 \dot\xi_1= T (\xi _2+ 2\xi _1\xi _2), \smallskip \\
 \dot\xi_2= T ( a \xi _1^2+ b \xi _1 \xi _2).
\end{cases}
\]
This last system is topological equivalent to the one-time shift ODE which
represents the first iteration of the normal form of 1:1 resonance
of fixed points, i.e. the system
\[
\begin{cases}
\dot \zeta_1=\zeta_2, \\
\dot \zeta_2=a_1 \zeta_1^2+b_1 \zeta_1 \zeta_2,
\end{cases}
\]
since we can transform one in the other using the following change of
variables
\[
\begin{cases}
 \zeta_1= \xi_1- \xi_1^2, \\
 \zeta_2= T \xi_2.
\end{cases}
\]
The bifurcation phenomena in the 1:1 resonance of cycles are the
same of those which appear in the 1:1 resonance of fixed points. Notice
that
\[
a_1= T^2 \ds a, \quad b_1=  T \ds b,
\]
so the non-degeneracy conditions are
\[
a\neq 0, \qquad b\neq 0,
\]
and the cases depend on the sign of the product of $a$
and $b$. In particular, as shown in figure \ref{fig:NF_R1}, if the
two coefficients have different sign the Neimark-Sacker
bifurcation (labeled $H$) is supercritical (with negative normal
form coefficient), while in the other case it is subcritical. The output given by MatCont is the product of the coefficients $a$ and $b$.
\begin{figure}[htb]
\centering \subfigure[$a>0,\, ab<0$]{
 \includegraphics[width=.49\textwidth]{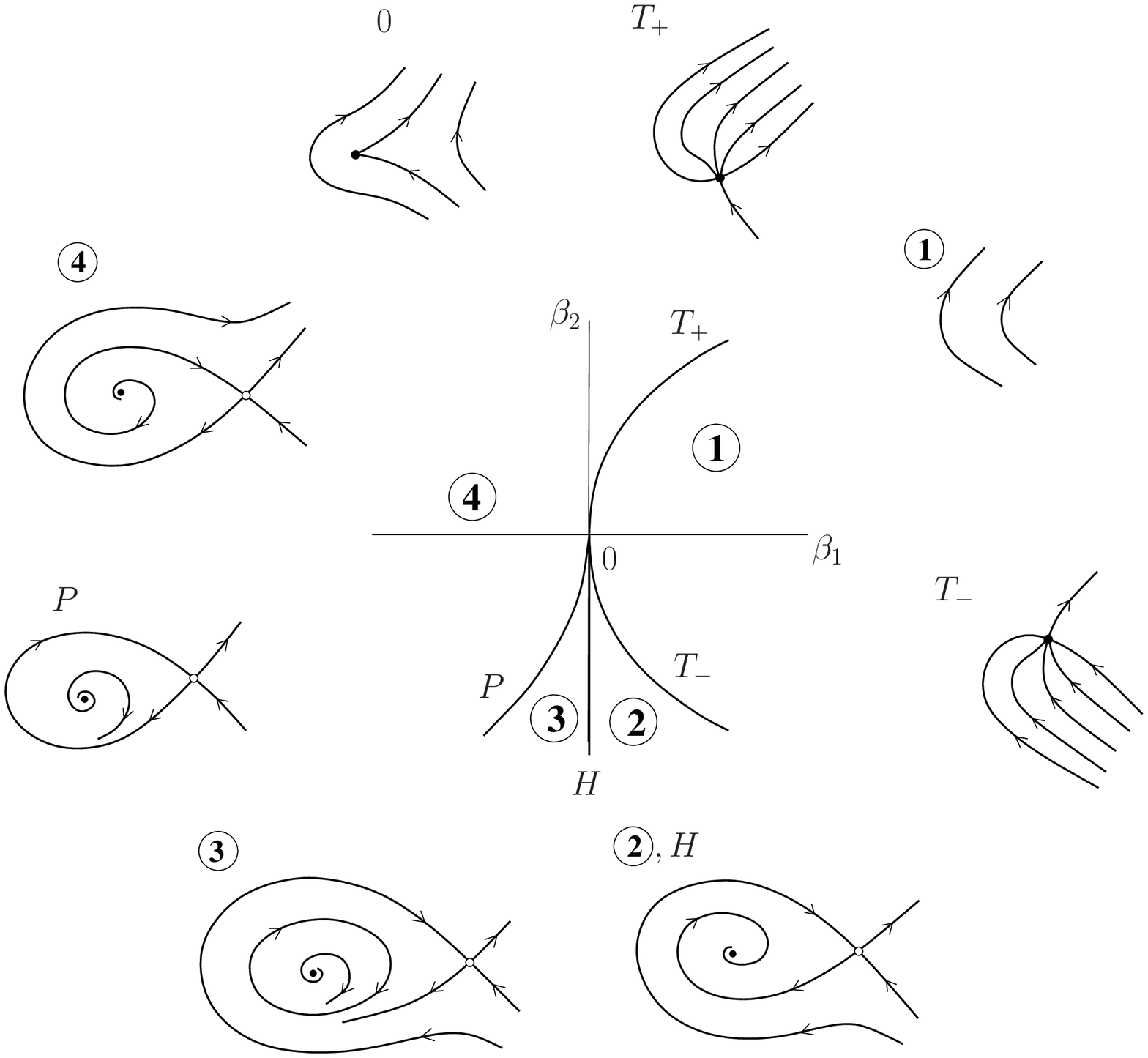}}\subfigure[$a>0,\, ab>0$]{
 \includegraphics[width=.435\textwidth]{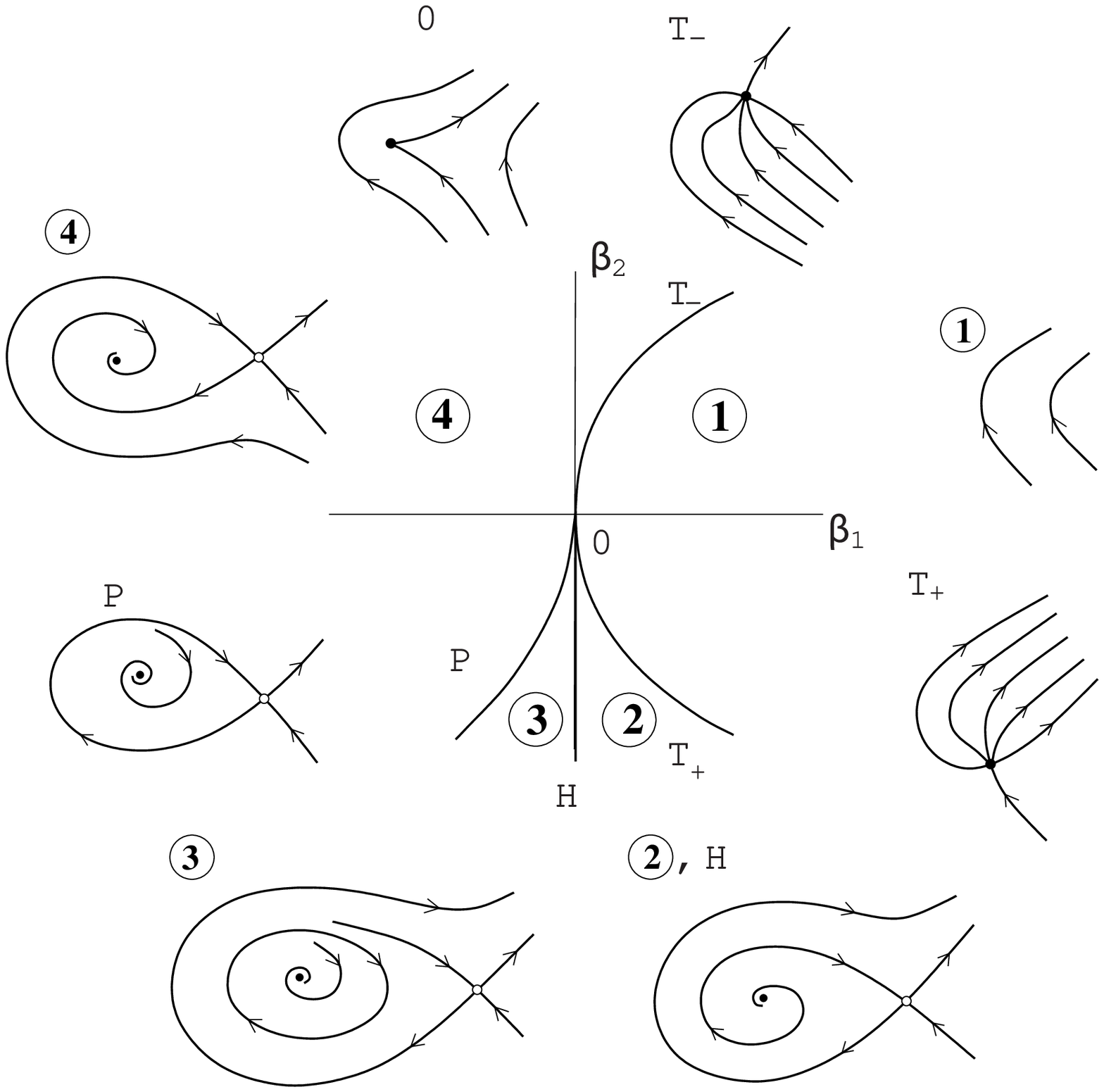}}
 \caption{Bifurcation diagram of the 1:1 resonance bifurcation of the fixed point normal form.
 The other two cases in which $a<0$ can be obtained by a reflection around the origin of the state portraits and a horizontal flip of the bifurcation diagrams.} \label{fig:NF_R1}
\end{figure}

\subsubsection{\tt R2} In this section we will show how the
periodic normal form \eqref{eq:NF-12C} is related with the normal
form for the 1:2 resonance of fixed points, which allows us to formulate the
non-degeneracy condition. First note that, if we scale the time,
we can rewrite system \eqref{eq:NF-12C} as
\[
\begin{cases}
 \dd{\tau}{\tau}=\ds 1,\smallskip \\
 \dd{\xi_1}{\tau}=\ds  \frac{\xi_2+\alpha \xi_1^2\xi_2}{1+\alpha \xi_1^2}+ \ldots, \smallskip \\
 \dd{\xi_2}{\tau}=\ds \frac{a \xi_1^3+ b \xi_1^2\xi_2}{1+\alpha
 \xi_1^2}+ \ldots.
\end{cases}
\]
Doing Picard iterations up to time $2T$, it's possible to show
that the generated map is the same, up to third order terms, of the
one generated by Picard iterations up to time 1 of the truncated
system
\[
\begin{cases}
 \dot\xi_1 = 2 T \xi _2, \smallskip \\
 \dot\xi_2 = 2 T (a \xi_1^3+ b \xi_1^2\xi_2).
\end{cases}
\]
This last system is topological equivalent to the one-time shift ODE which
represents the second iteration of the normal form of 1:2 resonance
of fixed points, i.e. the system
\[
\begin{cases}
\dot \zeta_1=\zeta_2, \\
\dot \zeta_2=a_1 \zeta_1^3+b_1 \zeta_1^2 \zeta_2,
\end{cases}
\]
since we can transform one in the other using the following change of
variables
\[
\begin{cases}
 \zeta_1=\xi _1, \\
 \zeta_2=2 T \xi _2.
\end{cases}
\]
The bifurcation phenomena in the 1:2 resonance of cycles are the
same that appear in the corresponding bifurcation of fixed points. Notice
that in this case
\[
a_1=4 T^2 \ds a, \quad b_1=  2 T \ds b.
\]
so the non-degeneracy conditions are
\[
a\neq 0, \qquad b\neq 0.
\]
We have four different unfoldings as possible bifurcation diagrams, determined by the signs of the
coefficients. The ones with negative $b$ are reported in Figure \ref{fig:NF_R2}. The other two cases can be obtained
by reversing the arrows of the phase portraits and making a vertical flip both of the state portraits and of the bifurcation diagrams. The primary
Neimark-Sacker bifurcation (labeled $H^{(1)}$) is supercritical
(with negative normal form coefficient) if our coefficient $b$ is
negative, subcritical otherwise. Moreover if $a<0$ a secondary
Neimark-Sacker bifurcation ($H^{(2)}$) is rooted at the 1:2 resonance point with opposite criticality of the primary one. The output given by MatCont is $(a,b)$.
\begin{figure}[htb]
\centering \subfigure[$b<0, \,a>0$]{
 \includegraphics[width=.49\textwidth]{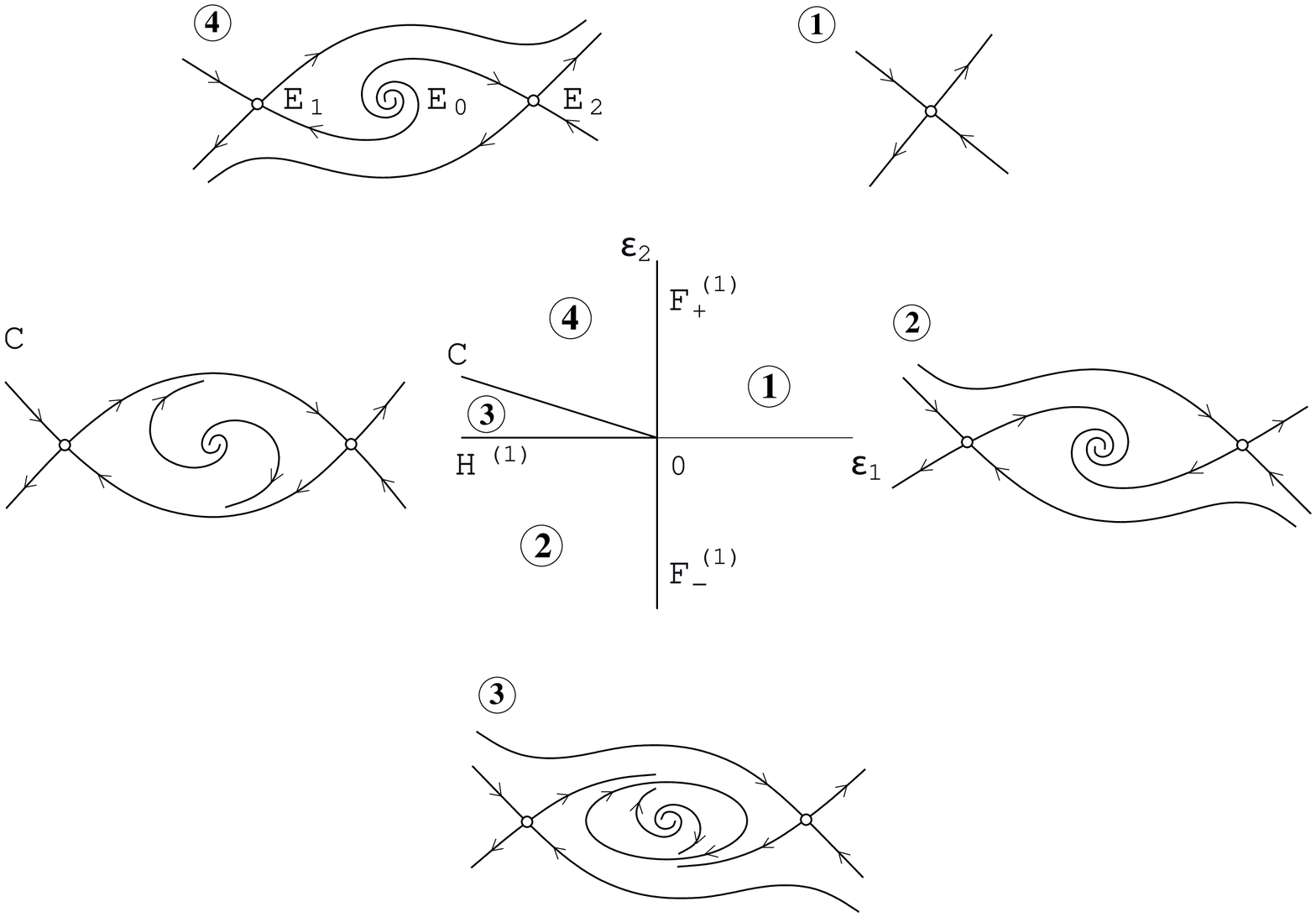}}\subfigure[$b<0,\, a<0$]{
 \includegraphics[width=.49\textwidth]{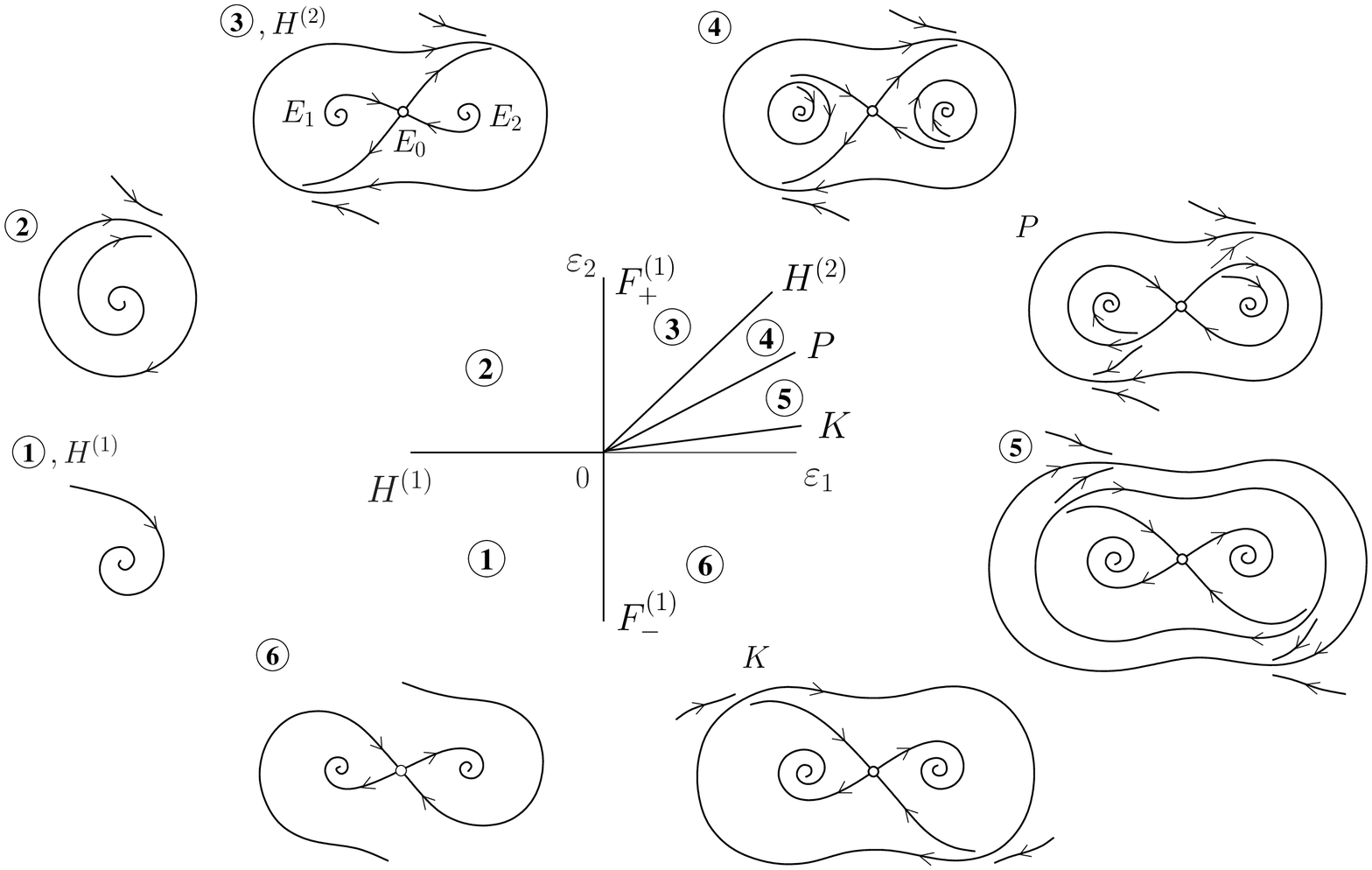}}
 \caption{Bifurcation diagram of the 1:2 resonance bifurcation of the fixed point normal form. The other two
 possible cases in which $b>0$ can be obtained by reversing time and making a
 vertical flip both of the state portraits and of the bifurcation diagrams.} \label{fig:NF_R2}
\end{figure}

\subsubsection{\tt R3}  In this section we will show how the
periodic normal form \eqref{eq:NF-13C} is related with the normal
form for the 1:3 resonance of fixed points, which allows us to formulate the
non-degeneracy condition. First note that, if we scale the time,
we can rewrite system \eqref{eq:NF-13C} as
\[
\begin{cases}
 \dd{\tau}{\tau}=\ds 1 ,\smallskip \\
 \dd{\xi}{\tau}=\ds b \bar \xi^2 + c \xi |\xi|^2 + \ldots. \\
\end{cases}
\]
Doing Picard iterations up to time $3T$, it's possible to show
that the generated map is the same, up to cubic terms, of the
one generated by Picard iterations up to time 1 of the truncated
system
\[
 \dot \xi = 3T b \bar \xi^2 + 3 T c \xi |\xi|^2.
\]
This system is topological equivalent to the one-time shift ODE which
represents the third iteration of the normal form of 1:3 resonance
of fixed points, i.e. the system
\[
\dot \zeta =b_1 \bar\zeta^2 + c_1 \zeta |\zeta|^2.
\]
The bifurcation phenomena in the 1:3 resonance of cycles are the
same of the ones which appear in the 1:3 resonance of fixed points, if the
non-degeneracy conditions are satisfied, i.e.
\[
 b\neq 0, \qquad \Re(c)\neq 0.
\]
As can be seen in Figure \ref{fig:NF_R3}, if $\Re(c) < 0$ the Neimark-Sacker bifurcation (labeled $N$) is supercritical (with negative normal form coefficient), while in the other case it is subcritical. The output given by MatCont is $(b,\Re(c))$.
\begin{figure}[htb]
\centering \subfigure[$\Re(c)<0$]{
 \includegraphics[width=.48\textwidth]{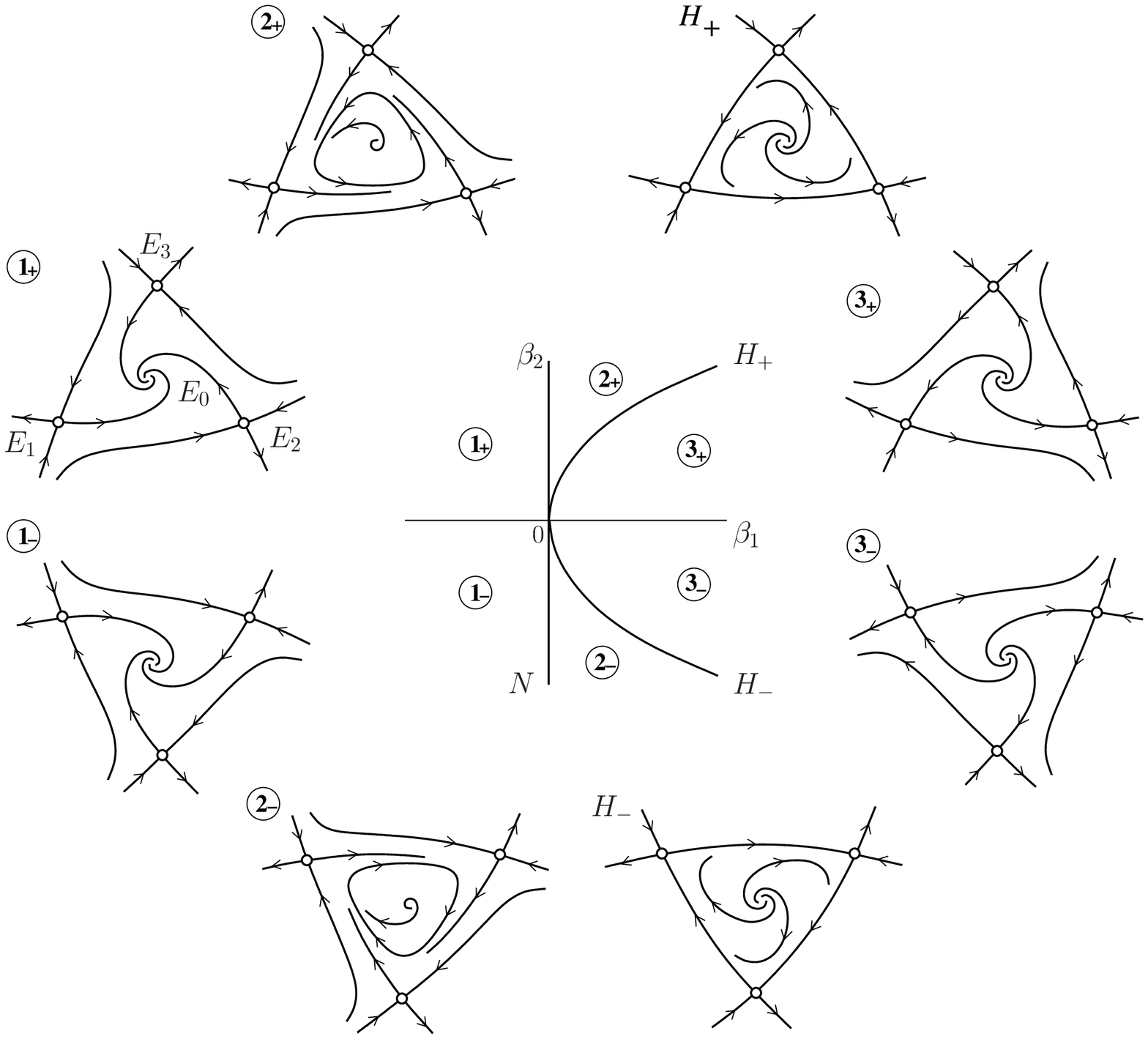}}
 \subfigure[$\Re(c)>0$]{
 \includegraphics[width=.48\textwidth]{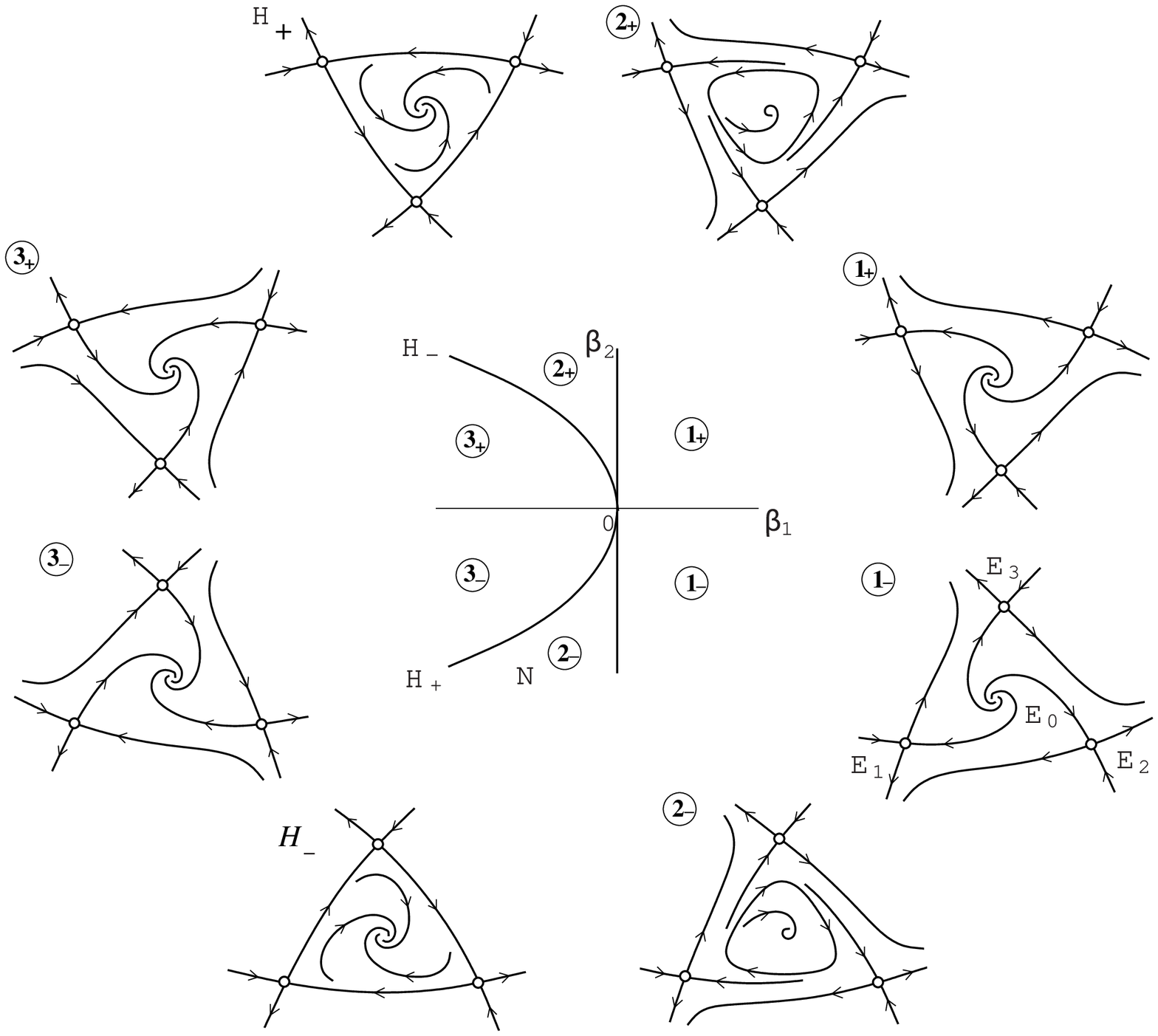}}
 \caption{Bifurcation diagram of the 1:3 resonance bifurcation of the fixed point normal form.} \label{fig:NF_R3}
\end{figure}

\subsubsection{{\tt R4}} In this section we will show how the
periodic normal form \eqref{eq:NF-14C} is related with the normal
form for the 1:4 resonance of fixed points, which allows us to formulate the
non-degeneracy condition. First note that, if we scale the time,
we can rewrite system \eqref{eq:NF-14C} as
\[
\begin{cases}
 \dd{\tau}{\tau}=\ds 1,\smallskip \\
 \dd{\xi}{\tau}=\ds c \xi |\xi|^2 + d \bar \xi^3 +  \ldots. \\
\end{cases}
\]
Doing Picard iterations up to time $4T$, it's possible to show
that the generated map is the same, up to cubic terms, of the
one generated by Picard iterations up to time 1 of the truncated
system
\[
 \dot \xi= 4 T c \xi |\xi|^2+4T d \bar \xi^3.
\]
This system is topologically equivalent to the one-time shift ODE which
represents the fourth iteration of the normal form of 1:4 resonance
of fixed points, i.e. the system
\[
\dot \zeta = c_1 \zeta |\zeta|^2 + d_1 \bar\zeta^3.
\]
The bifurcation phenomena in the 1:4 resonance of cycles are the ones which appear in the corresponding bifurcation of fixed points, if the
non-degeneracy condition is satisfied, i.e.
\[
d\neq 0.
\]
After defining
\[
A=\frac{c}{|d|},
\]
we can determine which bifurcations occur by looking at the place of $A$ in the Gauss plane (see Figure \ref{fig:bifR4}). We
here report the different possible
bifurcation diagrams. A complete investigation can be found in the
literature \cite{Kr:94,Ku:2004}.

\begin{figure}[htb]
\centering \footnotesize
 \psfrag{r}[Bc][Bc]{$\Re(A)$}
 \psfrag{ia}[Bc][Bc]{$\Im(A)$}
 \psfrag{c1}[Bc][Bc]{$c_1$}
 \psfrag{c2}[Bc][Bc]{$c_2$}
 \psfrag{d}[Bc][Bc]{$d$}
 \psfrag{e}[Bc][Bc]{$e$}
 \psfrag{i}[Bc][Bc]{\rectitem{I}}
 \psfrag{ti}[Bc][Bc]{\rectitem{$\widetilde{\mathrm{I}}$}}
 \psfrag{2}[Bc][Bc]{\rectitem{II}}
 \psfrag{t2}[Bc][Bc]{\rectitem{$\widetilde{\mathrm{II}}$}}
 \psfrag{3}[Bc][Bc]{\rectitem{III}}
 \psfrag{t3}[Bc][Bc]{\rectitem{$\widetilde{\mathrm{III}}$}}
 \psfrag{4}[Bc][Bc]{\rectitem{IV}}
 \psfrag{t4}[Bc][Bc]{\rectitem{$\widetilde{\mathrm{IV}}$}}
 \psfrag{5}[Bc][Bc]{\rectitem{V}}
 \psfrag{t5}[Bc][Bc]{\rectitem{$\widetilde{\mathrm{V}}$}}
 \psfrag{6}[Bc][Bc]{\rectitem{VI}}
 \psfrag{t6}[Bc][Bc]{\rectitem{$\widetilde{\mathrm{VI}}$}}
 \psfrag{7}[Bc][Bc]{\rectitem{VII}}
 \psfrag{t7}[Bc][Bc]{\rectitem{$\widetilde{\mathrm{VII}}$}}
 \psfrag{8}[Bc][Bc]{\rectitem{VIII}}
 \psfrag{t8}[Bc][Bc]{\rectitem{$\widetilde{\mathrm{VIII}}$}}
 \psfrag{3a}[Bc][Bc]{\rectitem{$\mathrm{III_a}$}}
 \psfrag{t3a}[Bc][Bc]{\rectitem{$\widetilde{\mathrm{III_a}}$}}
 \psfrag{4a}[Bc][Bc]{\rectitem{$\mathrm{IV_a}$}}
 \psfrag{t4a}[Bc][Bc]{\rectitem{$\widetilde{\mathrm{IV_a}}$}}
 \psfrag{5a}[Bc][Bc]{\rectitem{$\mathrm{V_a}$}}
 \psfrag{t5a}[Bc][Bc]{\rectitem{$\widetilde{\mathrm{V_a}}$}}
 \psfrag{5b}[Bc][Bc]{\rectitem{$\mathrm{V_b}$}}
 \psfrag{t5b}[Bc][Bc]{\rectitem{$\widetilde{\mathrm{V_b}}$}}
 \includegraphics[width=.9\textwidth]{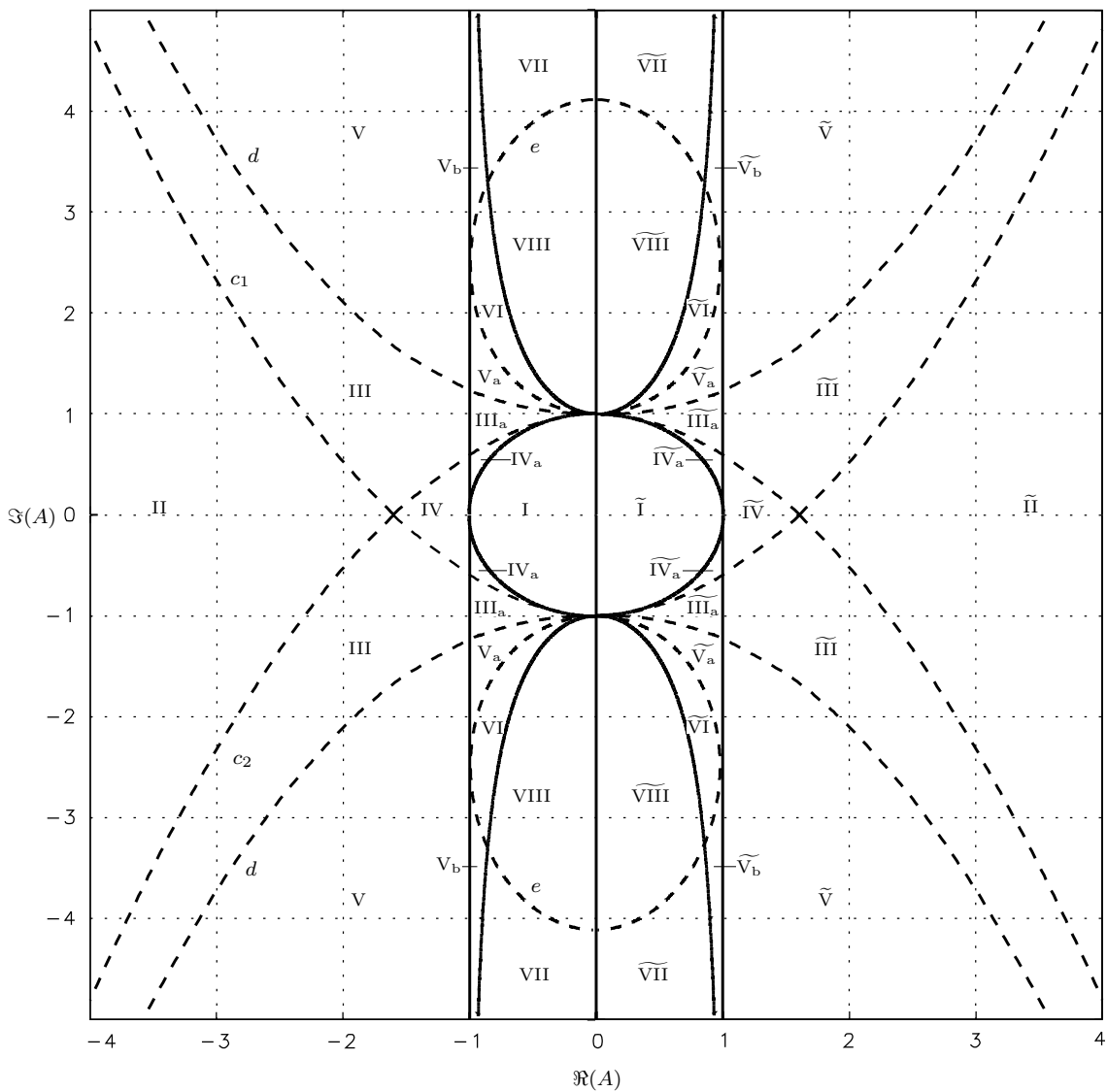}
 \caption{Partitioning of the $A$ plane into topologically different regions.} \label{fig:bifR4}
\end{figure}

Many topologically different bifurcation diagrams can be found
near a 1:4 resonance point. The analysis, if one excludes higher
codimension situations, can be reduced to 22 different
cases, which, as mentioned before, depend on the value of A. First of all, analyzing the normal
form, one can divide the Gauss plane into two big regions: in the semiplane $\Re(A)<0$ the
primary Neimark-Sacker bifurcation is supercritical, in the semiplane $\Re(A)>0$ it
is subcritical. What happens in the semiplane $\Re(A)>0$ can therefore be
obtained by inverting the direction of the vector fields. We can
further reduce the analysis to the third quadrant of the Gauss
plane, since the 12 possible cases are topologically
equivalent paired through the transformation $\zeta \mapsto \bar\zeta$.
The different regions are reported in Figure
\ref{fig:bifR4}, in which only some curves (the continuous lines)
are known analytically, the dashed curves are computed
numerically.

Figure \ref{fig:bifR4_cases} shows the possible bifurcation diagrams
with the sketches of the phase portraits for the Poincar\'e maps in the
case that $\Re(A)<0$. Many
local and global bifurcations are involved in the different
scenarios. We use the following notation, consistent with the
rest of the text:
\begin{itemize}
 \item[{\tt N}:] Neimark-Sacker bifurcation. In regions VII and VIII we also have a Neimark-Sacker bifurcation of the period-4 limit cycle.
 \item[{\tt T}:] Fold bifurcation of the period-4 limit cycles. There are three possibilities. Superscript {\tt in}, {\tt on} or {\tt out}
 means that the bifurcation happens inside, on or outside the invariant curve.
 \item[{\tt H}:] Homoclinic connection of the period-4 saddle
 limit cycle. Superscript {\tt S} means that the born invariant
 curve is smaller than the limit cycle (a square looking homoclinic connection), {\tt C} that it is bigger (a clover looking homoclinic connection),
 and {\tt L} means that the born invariant curve is around the period-4
 limit cycle; subscript {\tt +} ({\tt -}) means that the
 saddle quantity is positive (negative), so the borning invariant curve is
 repelling (attracting).
 \item[{\tt F}:] Fold bfirucation of the tori.
\end{itemize}
The output given by MatCont is $(A,d)$.
\begin{figure}[p]
\centering \footnotesize
 \psfrag{p1}[Bc][Br]{\rectitem{I}}
 \psfrag{p2}[Bc][Br]{\rectitem{II}}
 \psfrag{p3}[Bc][Br]{\rectitem{III}}
 \psfrag{p4}[Bc][Br]{\rectitem{IV}}
 \psfrag{p5}[Bc][Br]{\rectitem{V}}
 \psfrag{p6}[Bc][Br]{\rectitem{VI}}
 \psfrag{p7}[Bc][Br]{\rectitem{VII}}
 \psfrag{p8}[Bc][Br]{\rectitem{VIII}}
 \psfrag{p3a}[Bc][Br]{\rectitem{$\mathrm{III_a}$}}
 \psfrag{p4a}[Bc][Br]{\rectitem{$\mathrm{IV_a}$}}
 \psfrag{p5a}[Bc][Br]{\rectitem{$\mathrm{V_a}$},\rectitem{$\mathrm{V_b}$}}
 \includegraphics[width=0.9\textwidth]{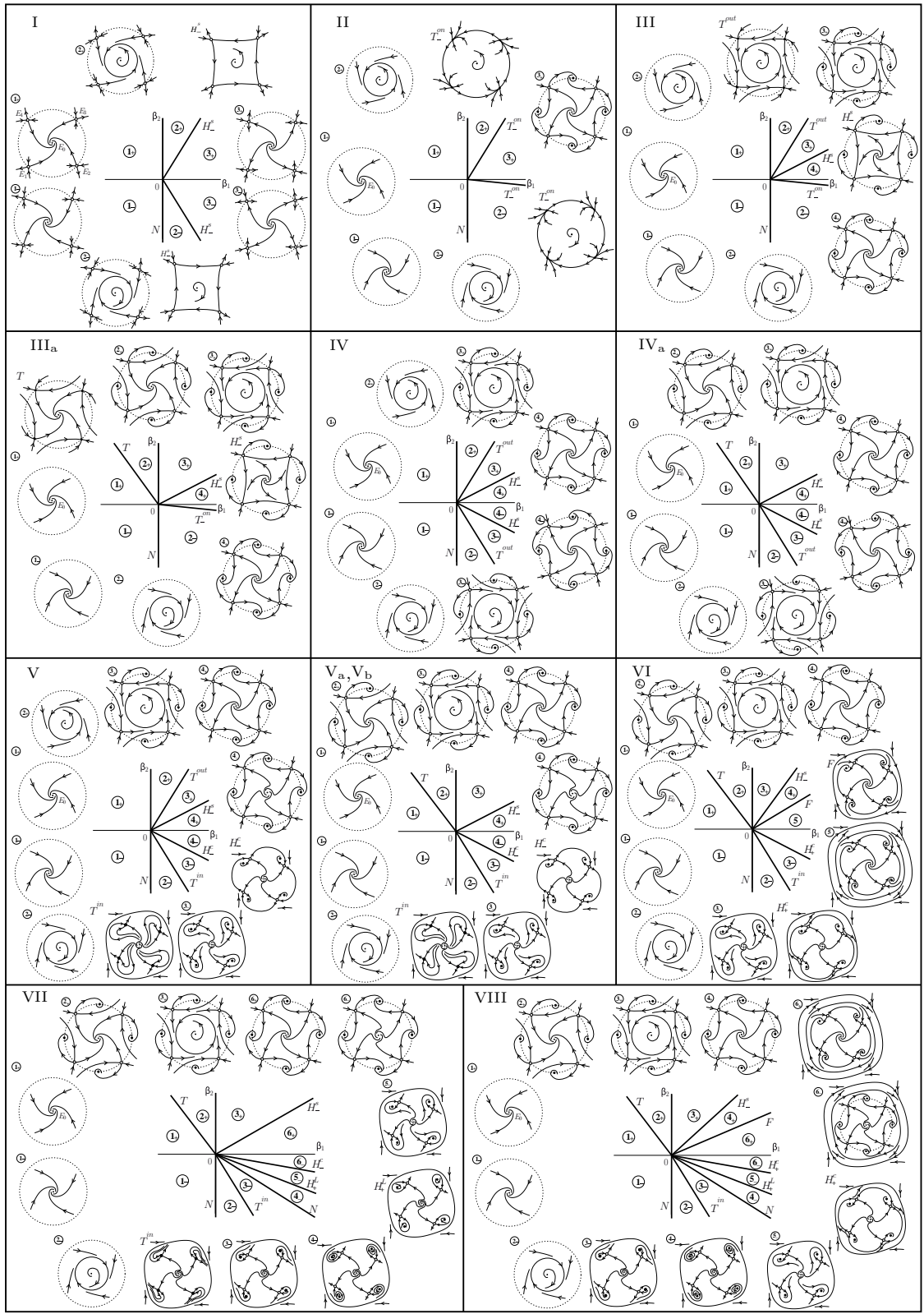}
 \caption{Bifurcation diagrams locally to the 1:4 resonance bifurcation in the different
 regions of figure \ref{fig:bifR4}. The cases in which $Re(A)>0$ can be obtained with the transformation
 $t\rightarrow -t$, $\beta\rightarrow -\beta$.} \label{fig:bifR4_cases}
\end{figure}


\subsubsection{\tt LPPD}
In this section we will show how the periodic normal form
\eqref{eq:NF-FF} is related to the normal form for the fold-flip
of fixed points, which allows us to formulate the non-degeneracy condition. As for the
{\tt R2} case we show the topological equivalence of the $2T$-shift map of our system with the $1$-shift map of the approximating vector field of the fold-flip normal form for fixed
points. First of all we can prove, by scaling the time and using
Picard iterations, that the Poincar\'e map of our normal form is the
same, up to cubic terms, as the one-shift map of the truncated system
\[
 \begin{cases}
  \ds\dot \xi_1= 2T a_{20}\xi_1^2 + 2 T a_{02}  \xi_2^2 + 2 T (a_{30}+a_{20})\xi_1^3 + 2 T (a_{12}+a_{02})\xi_1 \xi_2^2, \smallskip \\
  \ds\dot \xi_2=  2 T b_{11} \xi_1 \xi_2 + 2 T (b_{21}+b_{11})\xi_1^2 \xi_2 + 2 T b_{03} \xi_2^3.
 \end{cases}
\]
This system is topologically equivalent with the one-time shift ODE which represents the second iteration of the normal form of the fold-flip bifurcation of fixed points, i.e. the system
\begin{equation}
\begin{cases}
 \ds\dot \zeta_1 = a_1 \zeta_1^2 + b_1 \zeta_2^2 + (c_1-a_1^2) \zeta_1^3+ (d_1-a_1 b_1 + b_1) \zeta_1 \zeta_2^2, \smallskip \\
 \ds\dot \zeta_2 = -\zeta_1 \zeta_2 + \frac{1}{2}(a_1-1) \zeta_1^2 \zeta_2+ \frac{1}{2} b_1 \zeta_2^3,
\end{cases} \label{eq:shift_FF}
\end{equation}
since this second system can be obtained (neglecting higher order
terms) from the first one using the transformation
\[
\begin{cases}
 \zeta_1= -2 b_{11} T \xi _1-2 T\left(b_{11}+b_{21}+a_{20} b_{11} T+b_{11}^2 T \right) \xi_1^2 - 2 T (b_{03}+a_{02} b_{11} T)\xi_2^2, \\
 \zeta_2= 2 b_{11} T \xi _2.
\end{cases}
\]
This transformation should be invertible, so one non-degeneracy
condition is involved, namely
\[
 b_{11}\neq 0.
\]
If this condition is satisfied, then the system can be put in the
form \eqref{eq:shift_FF}, where the constants are defined as
\begin{gather*}
 a_1=-\frac{a_{20}}{b_{11}},\qquad b_1= -\frac{a_{02}}{b_{11}}, \qquad c_1 = \frac{a_{20}+a_{30}+ 2 a_{20}^2 T}{2 b_{11}^2T},\\
 d_1 = \frac{-2 a_{20} b_{03}+3 a_{02} b_{11}+a_{12} b_{11}+2 b_{03} b_{11}+2 a_{02} b_{21}+2 a_{02} a_{20} b_{11} T+6 a_{02} b_{11}^2 T}{2 b_{11}^3 T}
\end{gather*}
and from those values we can understand which types of bifurcation
the system has. In particular (see \cite{Ku:2004,KuMeVe:2004} for
more details) three more non-degeneracy conditions are involved
\begin{itemize}
 \item if $a_{20}\neq 0$ there are two limit cycles that collide and disappear
 \item if $a_{02}\neq 0$ a period doubled limit cycle born in this point.
\end{itemize}
Moreover if
$\frac{a_{02}}{b_{11}}<0$ a torus bifurcation occurs on the period
doubled orbit, with Lyapunov coefficient
\[
 L_{NS}=-2 a_{20}^2 b_{03}-3 a_{02}a_{30} b_{11}+a_{20}(a_{12} b_{11}+2 b_{03} b_{11}+2 a_{02} b_{21})
\]
and so, in order to avoid degeneracy, we also assume $L_{NS}\neq
0$.

In Figure \ref{fig:NF_FF} the four possible scenarios are reported
depending on the sign of the normal form coefficients. The output given by MatCont is $(b_{11},a_{20},a_{02},L_{NS})$.
\begin{figure}[htb]
\centering
 \subfigure[$a_{20} b_{11}<0, \,a_{02} b_{11}<0$]{
 \includegraphics[width=.54\textwidth]{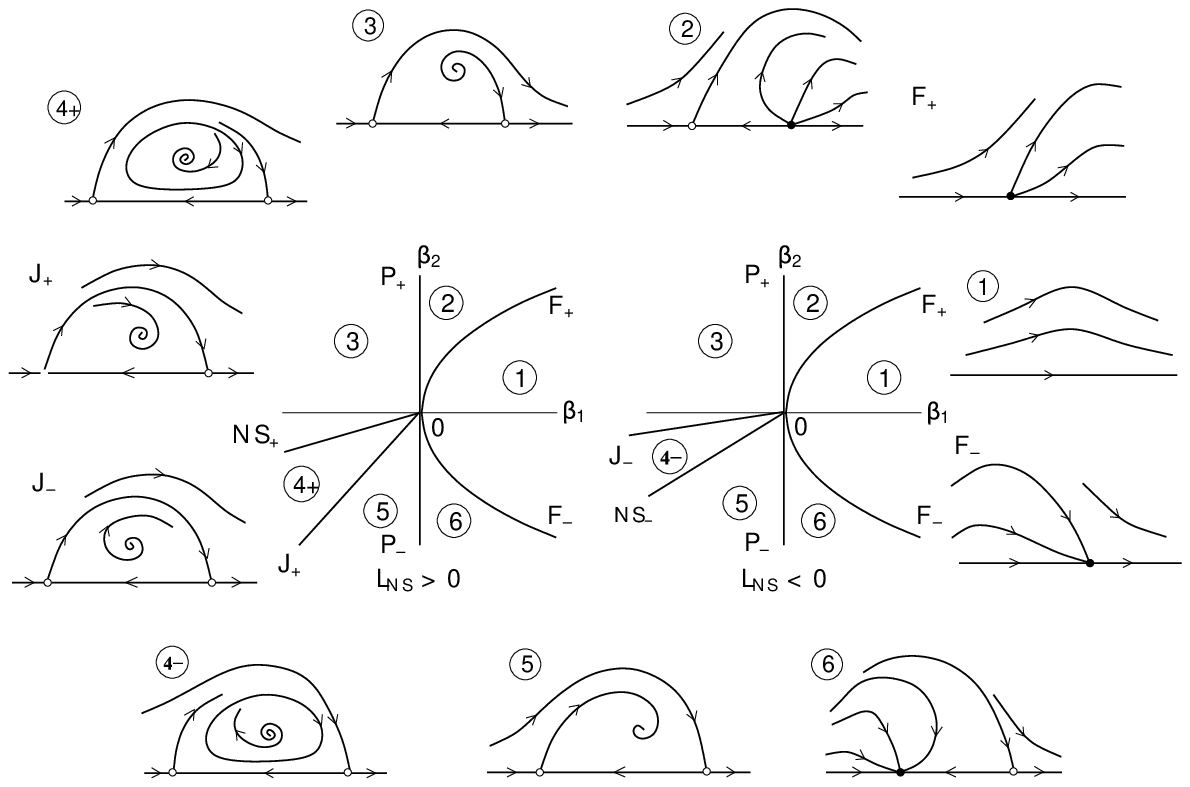}}\subfigure[$a_{20} b_{11}<0, \,a_{02} b_{11}>0$]{
 \includegraphics[width=.44\textwidth]{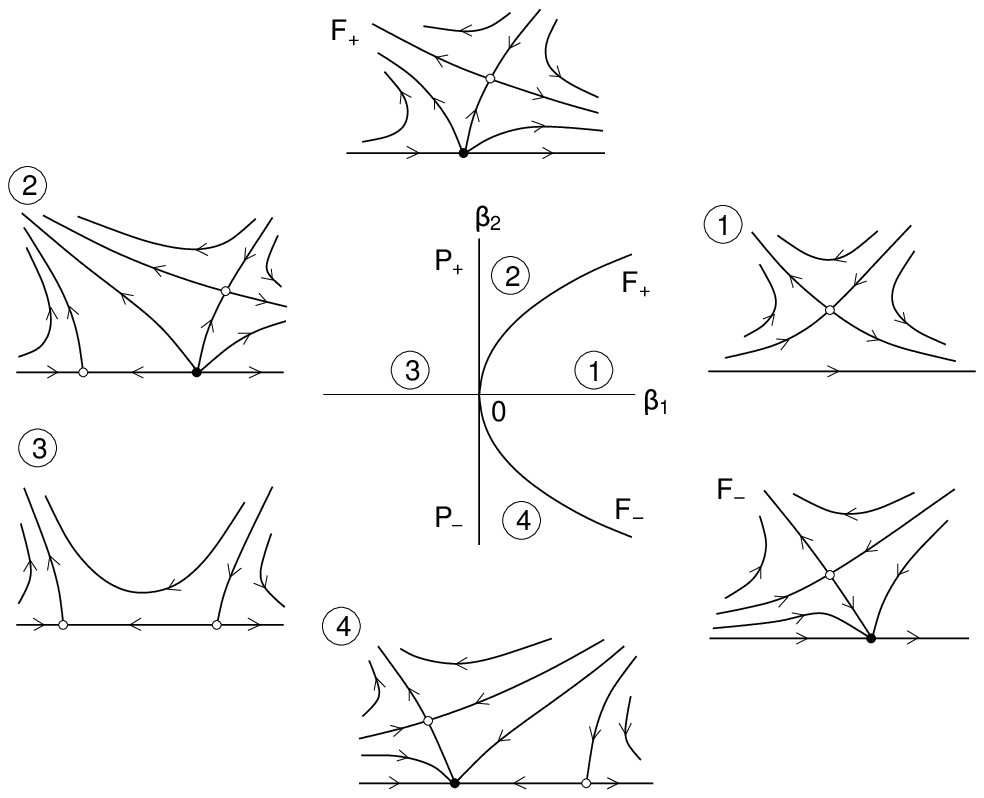}}
 \subfigure[$a_{20} b_{11}>0, \,a_{02} b_{11}<0$]{
 \includegraphics[width=.54\textwidth]{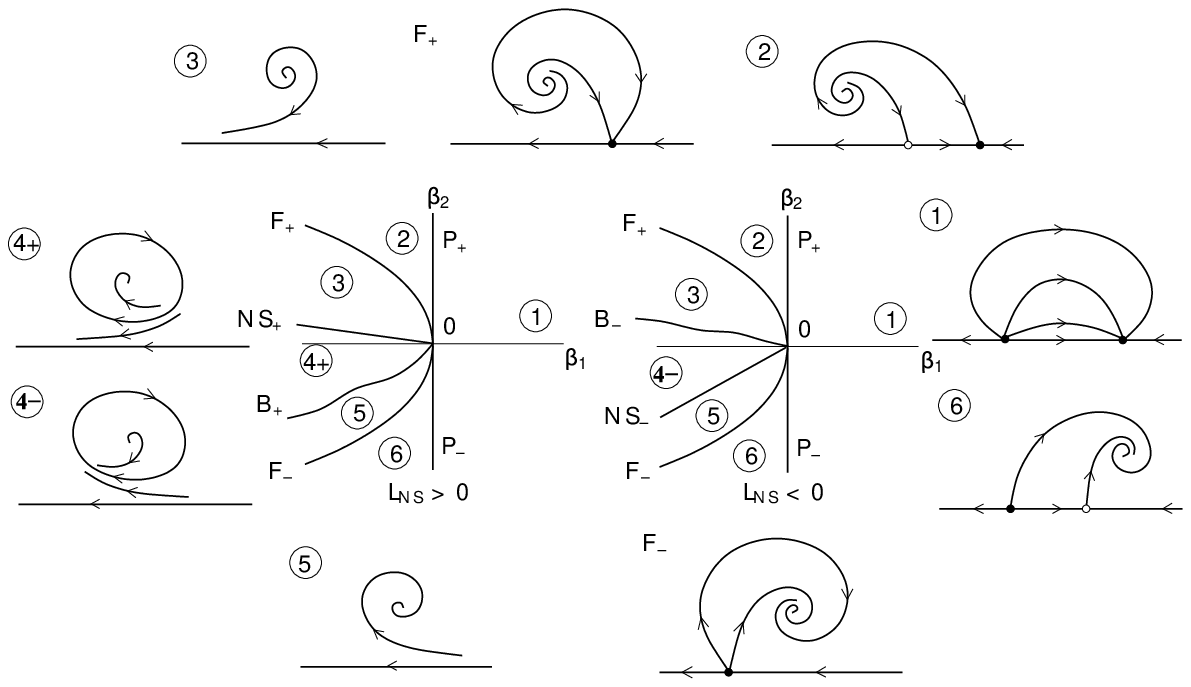}}\subfigure[$a_{20} b_{11}>0, \,a_{02} b_{11}>0$]{
 \includegraphics[width=.44\textwidth]{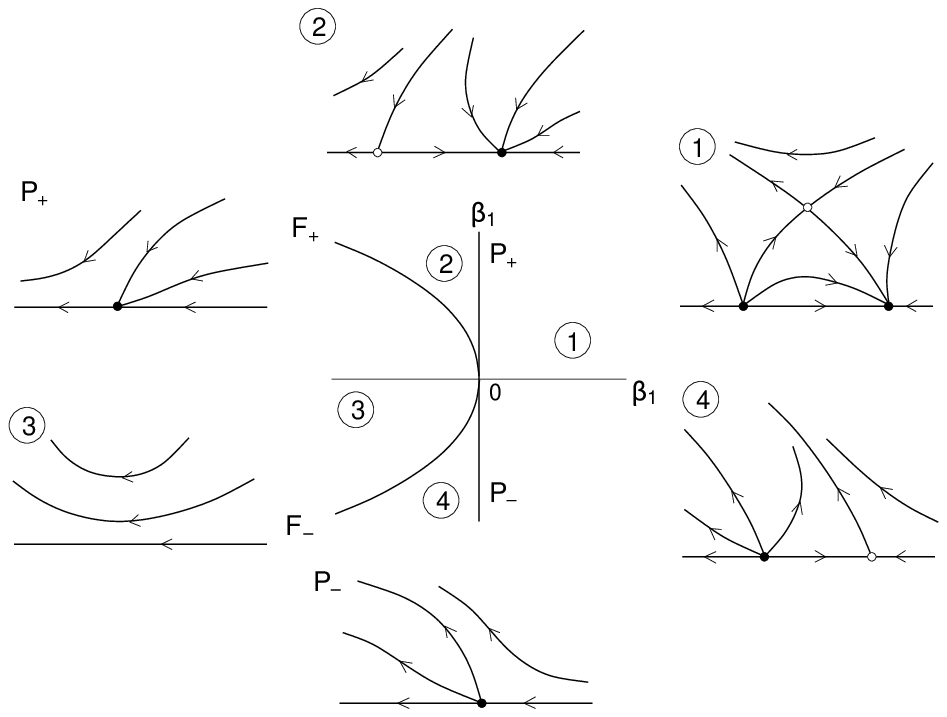}}
 \caption{Bifurcation diagrams of a fold-flip bifurcation of the fixed point normal form.} \label{fig:NF_FF}
\end{figure}

\section{Kernels of some differential-difference operators}
\label{Appendix:3}

In Section~\ref{Section:Implementation} we used the orthogonality with respect to
the following inner product: if $\zeta_1,\zeta_2 \in {\cal C}^0([0,1],{\mathbb C^n})$ and 
$\eta_1,\eta_2\in {\mathbb C}^n$, then 
$$
\left\langle\left[\begin{array}{c}\zeta_1\\\eta_1\end{array}\right],
\left[\begin{array}{c}\zeta_2\\\eta_2
\end{array}\right]\right\rangle=\int_0^1 \langle \zeta_1(t),\zeta_2(t)\rangle ~dt
+\langle \eta_1,\eta_2 \rangle
=\int_0^1\zeta_1^{\rm H}(t)\zeta_2(t)dt+\eta_1^{\rm H}\eta_2.
$$
If this inner product vanishes, then we say that the corresponding vectors are orthogonal
and write 
$$
\left[\begin{array}{c}\zeta_1\\\eta_1\end{array}\right] \bot
\left[\begin{array}{c}\zeta_2\\\eta_2
\end{array}\right].
$$
In Section~\ref{Section:Implementation} we used some propositions, from which we will give the proof in this appendix.

\begin{proposition} \label{Proposition1}
Consider two differential-difference operators 
$$
\phi_{1,2}:{\cal C}^1([0,1],{\mathbb R^n}) \rightarrow 
{\cal C}^0([0,1],{\mathbb R^n}) \times {\mathbb R}^n,
$$
with
$$
\phi_1(\zeta)=\left[\begin{array}{c}\dot{\zeta}-TA(t)\zeta \\\zeta(0)-\zeta(1)\end{array}\right],
\phi_2(\zeta)=\left[\begin{array}{c}\dot{\zeta}+
TA^{\rm T}(t)\zeta \\\zeta(0)-\zeta(1)\end{array}\right].
$$
If $\zeta \in {\cal C}^1([0,1],{\mathbb R^n})$, then
$\zeta\in {\rm Ker}(\phi_1)$ if and only if
$$
\left[\begin{array}{c}\zeta\\\zeta(0)\end{array}\right]\bot~\phi_2({\cal C}^1([0,1],
{\mathbb R^n})),
$$ 
and
$\zeta\in {\rm Ker}(\phi_2)$ if and only if
$$\left[\begin{array}{c}\zeta\\\zeta(0)\end{array}\right]\bot~\phi_1({\cal C}^1([0,1],
{\mathbb R^n})).
$$
\end{proposition}

\begin{proof}
We will focus on the first assertion. If $\zeta$ is in the kernel of $\phi_1$, then $\dot{\zeta}-TA(t)\zeta=0$ and $\zeta(0)-\zeta(1)=0$.
For all $g \in {{\cal C}}^1([0,1],{\mathbb R^n})$ we have
$$
\begin{array}{ll}
 & \int_0^1g^{\rm T}(t)\dot{\zeta}(t)dt-\int_0^1Tg^{\rm T}(t)A(t)\zeta(t)dt=0 \\
\Rightarrow & g^{\rm T}(t)\zeta(t)|_0^1-\int_0^1\dot{g}^{\rm T}(t)\zeta(t)dt-
\int_0^1Tg^{\rm T}(t)A(t)\zeta(t)dt=0\\
\Rightarrow & g^{\rm T}(1)\zeta(1)-g^{\rm T}(0)\zeta(0)-\int_0^1(\dot{g}(t)+
TA^{\rm T}(t)g(t))^{\rm T}\zeta(t)dt=0\\
\Rightarrow & -(g(0)-g(1))^{\rm T}\zeta(0)-\int_0^1(\dot{g}(t)+TA^{\rm T}(t)g(t))^{\rm T}
\zeta(t)dt=0\\
\Rightarrow & \left\langle\left[\begin{array}{c}\dot{g}+TA^{\rm T}(t)g\\g(0)-g(1)\end{array}\right],
\left[\begin{array}{c}\zeta\\\zeta(0)\end{array}\right]\right\rangle=0.
\end{array}
$$
Conversely, assume that
$\langle\left[\begin{matrix} \zeta\\ \zeta(0) \end{matrix}\right],
\left[\begin{matrix} \dot{g}+TA^{\rm T}(t)g\\ g(0)-g(1) \end{matrix}\right]\rangle=0$
for all $g\in{{\cal C}}^1([0,1],{\mathbb R^n})$.
Then,
$$
\begin{array}{ll}
 & \int_0^1\zeta^{\rm T}(t)(\dot{g}(t)+TA^{\rm T}(t)g(t))dt+\zeta^{\rm T}(0)
(g(0)-g(1))=0\\
\Rightarrow & \zeta^{\rm T}(1)g(1)-\zeta^{\rm T}(0)g(0)-\int_0^1(\dot{\zeta}(t)-TA(t)\zeta(t))^{\rm T}g(t)dt+\zeta^{\rm
T}(0)(g(0)-g(1))=0\\
\Rightarrow & -(\zeta(0)-\zeta(1))^{\rm T}g(1)-\int_0^1(\dot{\zeta}(t)-
TA(t)\zeta(t))^{\rm T}g(t)dt=0.\\
\end{array}
$$
If $\dot{\zeta}(t)-TA(t)\zeta(t)\neq 0$, then there exists a $g(t)$ with $g(1)=0$ such that
$$
\int_0^1(\dot{\zeta}(t)-TA(t)\zeta(t))^{\rm T}g(t)dt\neq 0.
$$ 
This is impossible, so 
$\dot{\zeta}(t)-TA(t)\zeta(t)=0$. Hence $(\zeta(0)-\zeta(1))^{\rm T}g(1)=0$ for all $g$;
and thus there must hold that $\zeta(0)-\zeta(1)=0.$ From both observations it follows that $\zeta\in {\rm Ker}(\phi_1)$.

The proof of the second assertion is similar.
\qquad\end{proof}

\begin{proposition} \label{Proposition2}
Consider 
$\phi_{1,2}: {\cal C}^1([0,1],{\mathbb R}^n) \rightarrow 
{\cal C}^0([0,1],{\mathbb R^n}) \times {\mathbb R}^n,$
where
$$
\phi_1(\zeta)=\left[\begin{array}{c}\dot{\zeta}-TA(t)\zeta\\\zeta(0)+\zeta(1)\end{array}\right],
\phi_2(\zeta)=\left[\begin{array}{c}\dot{\zeta}+TA^{\rm T}(t)\zeta\\\zeta(0)+\zeta(1)\end{array}\right].
$$
If $\zeta \in {\cal C}^1([0,1],{\mathbb R^n})$, 
then $\zeta\in {\rm Ker}(\phi_1)$ if and only if
$$
\left[\begin{array}{c}\zeta\\\zeta(0)\end{array}\right]\bot~
\phi_2({\cal C}^1([0,1],{\mathbb R^n})),$$ 
and
$\zeta\in {\rm Ker}(\phi_2)$ if and only if
$$
\left[\begin{array}{c}\zeta\\\zeta(0)\end{array}\right]\bot~
\phi_1({\cal C}^1([0,1],
{\mathbb R^n})).
$$
\end{proposition}
\unskip

\begin{proof}
The proof is similar to the proof of Proposition~\ref{Proposition1}.
\qquad\end{proof}

\begin{proposition} \label{Proposition3}
Consider 
$\phi_{1,2}: {\cal C}^1([0,1],{\mathbb C^n}) \rightarrow 
{\cal C}^0([0,1],{\mathbb C^n}) \times {\mathbb C^n},$
where
$$
\phi_1(\zeta)=\left[\begin{array}{c}\dot{\zeta}-TA(t)\zeta 
+i\theta \zeta \\\zeta(0)-\zeta(1)\end{array}\right],
\phi_2(\zeta)=\left[\begin{array}{c}\dot{\zeta}
+TA^{\rm T}(t)\zeta +i\theta \zeta\\\zeta(0)-\zeta(1)\end{array}\right].
$$
If $\zeta \in {\cal C}^1([0,1],{\mathbb C^n})$, then
$\zeta\in {\rm Ker}(\phi_1)$ if and only if
$$
\left[\begin{array}{c}\zeta\\\zeta(0)\end{array}\right]\bot~
\phi_2({\cal C}^1([0,1],{\mathbb C^n})),
$$ 
and
$\zeta\in {\rm Ker}(\phi_2)$ if and only if
$$
\left[\begin{array}{c}\zeta\\\zeta(0)\end{array}\right]\bot~
\phi_1({\cal C}^1([0,1],{\mathbb C^n})).
$$
\end{proposition}
\unskip

\begin{proof}
If $\zeta$ is in the kernel of $\phi_1$, then $\dot{\zeta}-TA(t)\zeta+i\theta \zeta=0$ and $\zeta(0)-\zeta(1)=0$.
For all $g \in {{\cal C}}^1([0,1],{\mathbb C^n})$ we have
$$
\begin{array}{ll}
 & \int_0^1g^{\rm H}(t)\dot{\zeta}(t)dt-\int_0^1Tg^{\rm H}(t)A(t)\zeta(t)dt+\int_0^1i\theta g^{\rm H}(t)\zeta(t)=0 \\
\Rightarrow & g^{\rm H}(t)\zeta(t)|_0^1-\int_0^1\dot{g}^{\rm H}(t)\zeta(t)dt-
\int_0^1Tg^{\rm H}(t)A(t)\zeta(t)dt+\int_0^1 i\theta g^{\rm H}(t)\zeta(t)=0\\
\Rightarrow & g^{\rm H}(1)\zeta(1)-g^{\rm H}(0)\zeta(0)-\int_0^1(\dot{g}(t)+
TA^{\rm T}(t)g(t)+i\theta g(t))^{\rm H}\zeta(t)dt=0\\
\Rightarrow & -(g(0)-g(1))^{\rm H}\zeta(0)-\int_0^1(\dot{g}(t)+TA^{\rm T}(t)g(t)+i\theta g(t))^{\rm H}
\zeta(t)dt=0\\
\Rightarrow & \left\langle\left[\begin{array}{c}\dot{g}+TA^{\rm T}(t)g+i\theta g\\g(0)-g(1)\end{array}\right],
\left[\begin{array}{c}\zeta\\\zeta(0)\end{array}\right]\right\rangle=0.
\end{array}
$$
The proofs of the reverse implication and the second assertion are similar.
\qquad\end{proof}

\begin{proposition} \label{Proposition6}
Consider 
$\phi_{1,2}: {\cal C}^1([0,1],{\mathbb C^n}) \rightarrow 
{\cal C}^0([0,1],{\mathbb C^n}) \times {\mathbb C^n},$
where
$$
\phi_1(\zeta)=\left[\begin{array}{c}\dot{\zeta}-TA(t)\zeta \\\zeta(0)-e^{-i\theta}\zeta(1)\end{array}\right],
\phi_2(\zeta)=\left[\begin{array}{c}\dot{\zeta}
+TA^{\rm T}(t)\zeta \\\zeta(0)-e^{-i\theta}\zeta(1)\end{array}\right].
$$
If $\zeta \in {\cal C}^1([0,1],{\mathbb C^n})$, then
$\zeta\in {\rm Ker}(\phi_1)$ if and only if
$$
\left[\begin{array}{c}\zeta\\\zeta(0)\end{array}\right]\bot~
\phi_2({\cal C}^1([0,1],{\mathbb C^n})),
$$ 
and
$\zeta\in {\rm Ker}(\phi_2)$ if and only if
$$
\left[\begin{array}{c}\zeta\\\zeta(0)\end{array}\right]\bot~
\phi_1({\cal C}^1([0,1],{\mathbb C^n})).
$$
\end{proposition}
\unskip

\begin{proof}
If $\zeta$ is in the kernel of $\phi_1$, then $\dot{\zeta}-TA(t)\zeta=0$ and $\zeta(0)-e^{-i\theta}\zeta(1)=0$.
For all $g \in {{\cal C}}^1([0,1],{\mathbb C^n})$ we have
$$
\begin{array}{ll}
 & \int_0^1g^{\rm H}(t)\dot{\zeta}(t)dt-\int_0^1Tg^{\rm H}(t)A(t)\zeta(t)dt=0 \\
\Rightarrow & g^{\rm H}(t)\zeta(t)|_0^1-\int_0^1\dot{g}^{\rm H}(t)\zeta(t)dt-
\int_0^1Tg^{\rm H}(t)A(t)\zeta(t)dt=0\\
\Rightarrow & g^{\rm H}(1)\zeta(1)-g^{\rm H}(0)\zeta(0)-\int_0^1(\dot{g}(t)+
TA^{\rm T}(t)g(t))^{\rm H}\zeta(t)dt=0\\
\Rightarrow & -(g(0)-e^{-i\theta}g(1))^{\rm H}\zeta(0)-\int_0^1(\dot{g}(t)+TA^{\rm T}(t)g(t))^{\rm H}
\zeta(t)dt=0\\
\Rightarrow & \left\langle\left[\begin{array}{c}\dot{g}+TA^{\rm T}(t)g\\g(0)-e^{-i\theta}g(1)\end{array}\right],
\left[\begin{array}{c}\zeta\\\zeta(0)\end{array}\right]\right\rangle=0.
\end{array}
$$
The proofs of the reverse implication and the second assertion are similar.
\qquad\end{proof}

\begin{proposition} \label{Proposition4}
Consider two differential-difference operators $\phi_{1,2}:{\cal C}^1([0,1],{\mathbb R^n}) \rightarrow 
{\cal C}^0([0,1],{\mathbb R^n}) \times {\mathbb R}^n,$
where
$$
\phi_1(\zeta)=\left[\begin{array}{c}\dot{\zeta}-TA(t)\zeta \\\zeta(0)-\zeta(1)\end{array}\right],
\ \ \phi_2(\zeta)=\left[\begin{array}{c}\dot{\zeta}+
TA^{\rm T}(t)\zeta \\\zeta(0)-\zeta(1)\end{array}\right].
$$
If $\zeta \in {\cal C}^1([0,1],{\mathbb R^n})$, then
$$\phi_1(\zeta) = \left[\begin{array}{c}g \\0\end{array} \right]$$
if and only if
$$
\left\langle\left[\begin{array}{c}\zeta\\\zeta(0)\end{array}\right],\left[\begin{array}{c}\dot{h}+TA^{\rm T}(t)h\\ h(0) -h(1)\end{array} \right]\right\rangle = -\left\langle\left[\begin{array}{c}g\\0\end{array}\right],\left[\begin{array}{c}h\\ 0\end{array} \right]\right\rangle,$$
for all $h \in {{\cal C}}^1([0,1],{\mathbb R^n})$. Furthermore
$$\phi_2(\zeta) = \left[\begin{array}{c}g \\0\end{array} \right]$$
if and only if
$$
\left\langle\left[\begin{array}{c}\zeta\\\zeta(0)\end{array}\right],\left[\begin{array}{c}\dot{h}-TA(t)h\\ h(0) -h(1)\end{array} \right]\right\rangle = -\left\langle\left[\begin{array}{c}g\\0\end{array}\right],\left[\begin{array}{c}h\\ 0\end{array} \right]\right\rangle,$$
for all $h \in {{\cal C}}^1([0,1],{\mathbb R^n})$.\end{proposition}

\begin{proof}
We focus on the first assertion. Suppose that $\dot{\zeta}(t)-TA(t)\zeta(t)=g(t)$ and $\zeta(0)-\zeta(1)=0$.
For all $h \in {{\cal C}}^1([0,1],{\mathbb R^n})$ we have
$$
\begin{array}{ll}
 & \int_0^1h^{\rm T}(t)\dot{\zeta}(t)dt-\int_0^1Th^{\rm T}(t)A(t)\zeta(t)dt=\int_0^1h^{\rm T}(t)g(t)dt \\
\Rightarrow & h^{\rm T}(t)\zeta(t)|_0^1-\int_0^1\dot{h}^{\rm T}(t)\zeta(t)dt-
\int_0^1Th^{\rm T}(t)A(t)\zeta(t)dt=\int_0^1h^{\rm T}(t)g(t)dt\\
\Rightarrow & h^{\rm T}(1)\zeta(1)-h^{\rm T}(0)\zeta(0)-\int_0^1(\dot{h}(t)+
TA^{\rm T}(t)h(t))^{\rm T}\zeta(t)dt=\int_0^1h^{\rm T}(t)g(t)dt\\
\Rightarrow & \int_0^1\zeta^{\rm T}(t)(\dot{h}(t)+TA^{\rm T}(t)h(t))dt +\zeta^{\rm T}(0)(h(0)-h(1))=-\int_0^1g^{\rm T}(t)h(t)dt\\
\Rightarrow & \left\langle\left[\begin{array}{c}\zeta\\\zeta(0)\end{array}\right],\left[\begin{array}{c}\dot{h}+TA^{\rm T}(t)h\\ h(0) -h(1)\end{array} \right]\right\rangle = -\left\langle\left[\begin{array}{c}g\\0\end{array}\right],\left[\begin{array}{c}h\\ 0\end{array} \right]\right\rangle.
\end{array}
$$
The proofs of the reverse implication and the second assertion are similar.
\qquad\end{proof}

\begin{proposition} \label{Proposition5}
Consider two differential-difference operators $\phi_{1,2}:{\cal C}^1([0,1],{\mathbb R^n}) \rightarrow 
{\cal C}^0([0,1],{\mathbb R^n}) \times {\mathbb R}^n,$
where
$$
\phi_1(\zeta)=\left[\begin{array}{c}\dot{\zeta}-TA(t)\zeta \\\zeta(0) +\zeta(1)\end{array}\right],
\ \ \phi_2(\zeta)=\left[\begin{array}{c}\dot{\zeta}+
TA^{\rm T}(t)\zeta \\\zeta(0) +\zeta(1)\end{array}\right].
$$
If $\zeta \in {\cal C}^1([0,1],{\mathbb R^n})$, then
$$\phi_1(\zeta) = \left[\begin{array}{c}g \\0\end{array} \right]$$
if and only if
$$
\left\langle\left[\begin{array}{c}\zeta\\\zeta(0)\end{array}\right],\left[\begin{array}{c}\dot{h}+TA^{\rm T}(t)h\\ h(0) +h(1)\end{array} \right]\right\rangle = -\left\langle\left[\begin{array}{c}g\\0\end{array}\right],\left[\begin{array}{c}h\\ 0\end{array} \right]\right\rangle,$$
$\forall h \in {{\cal C}}^1([0,1],{\mathbb R^n})$. Furthermore
$$\phi_2(\zeta) = \left[\begin{array}{c}g \\0\end{array} \right]$$
if and only if
$$
\left\langle\left[\begin{array}{c}\zeta\\\zeta(0)\end{array}\right],\left[\begin{array}{c}\dot{h}-TA(t)h\\ h(0) +h(1)\end{array} \right]\right\rangle = -\left\langle\left[\begin{array}{c}g\\0\end{array}\right],\left[\begin{array}{c}h\\ 0\end{array} \right]\right\rangle,$$
$\forall h \in {{\cal C}}^1([0,1],{\mathbb R^n})$.\end{proposition}

\begin{proof}
Suppose that $\dot{\zeta}(t)-TA(t)\zeta(t)=g(t)$ and $\zeta(0)+\zeta(1)=0$.
For all $h \in {{\cal C}}^1([0,1],{\mathbb R^n})$ we have
$$
\begin{array}{ll}
 & \int_0^1h^{\rm T}(t)\dot{\zeta}(t)dt-\int_0^1Th^{\rm T}(t)A(t)\zeta(t)dt=\int_0^1h^{\rm T}(t)g(t)dt \\
\Rightarrow & h^{\rm T}(t)\zeta(t)|_0^1-\int_0^1\dot{h}^{\rm T}(t)\zeta(t)dt-
\int_0^1Th^{\rm T}(t)A(t)\zeta(t)dt=\int_0^1h^{\rm T}(t)g(t)dt\\
\Rightarrow & h^{\rm T}(1)\zeta(1)-h^{\rm T}(0)\zeta(0)-\int_0^1(\dot{h}(t)+
TA^{\rm T}(t)h(t))^{\rm T}\zeta(t)dt=\int_0^1h^{\rm T}(t)g(t)dt\\
\Rightarrow & \int_0^1\zeta^{\rm T}(t)(\dot{h}(t)+TA^{\rm T}(t)h(t))dt +\zeta^{\rm T}(0)(h(0)+h(1))=-\int_0^1g^{\rm T}(t)h(t)dt\\
\Rightarrow & \left\langle\left[\begin{array}{c}\zeta\\\zeta(0)\end{array}\right],\left[\begin{array}{c}\dot{h}+TA^{\rm T}(t)h\\ h(0) +h(1)\end{array} \right]\right\rangle = -\left\langle\left[\begin{array}{c}g\\0\end{array}\right],\left[\begin{array}{c}h\\ 0\end{array} \right]\right\rangle.
\end{array}
$$
The proofs of the reverse implication and the second assertion are similar.
\qquad\end{proof}

\bibliographystyle{abbrv}
\bibliography{LC2NF}

\end{document}